\documentclass{amsart}
 \usepackage[top=1in, bottom=1in, left=1.375in, right=1.375in]{geometry}
 \usepackage{amsmath,amssymb}
 \usepackage{algpseudocode}
 \usepackage{algorithm}
 \usepackage{algorithmicx}
 \usepackage{caption}
 \usepackage{subcaption}
 \usepackage{amsthm}
 \numberwithin{equation}{section}
 \usepackage{amsfonts}
 \usepackage{graphicx}
   \usepackage{epstopdf}
    \usepackage{hyperref}

\usepackage{amsmath}
\usepackage{amssymb}
\usepackage{amsthm}
\usepackage{amsmath}
\usepackage{setspace}
\usepackage{xcolor}
\usepackage{fancyhdr}
\usepackage{amssymb}
\usepackage{amsthm}
\usepackage{listings}
    \usepackage{cite}
    \usepackage{tabularx}
    \usepackage{graphicx}
    \usepackage{epstopdf}    
    \usepackage{epsfig}
    \usepackage{float}
    \usepackage{ listings}
    \usepackage{appendix}
 \theoremstyle{plain}
  \newtheorem{thm}{Theorem}
 \newtheorem{prop}{Proposition}[section]
 \newtheorem{lem}[prop]{Lemma}
 
 \theoremstyle{definition}
 \newtheorem{definition}[prop]{Definition}
 
 \theoremstyle{remark}
 \newtheorem{remark}[prop]{Remark}

 \let\pa=\partial
 \let\al=\alpha
 \let\b=\beta
 \let\d=\delta
 \let\g=\gamma
 \let\e=\varepsilon

 \let \kp = \kappa
 \let\lam=\lambda

 \let\f=\frac
 
 \let \les = \lesssim
 \let\om=\omega
 
 \let \th = \theta
 \let \vp = \varphi
 
\let\B = \Big

 \let\Om=\Omega
 \let\td = \tilde

 \let\teq \triangleq
 
 \let\pa=\partial

 \def\la{\langle}
 \def\ra{\rangle}
\def\lt{\left}
\def\rt{\right}

 \newcommand{\beq}{\begin{equation}}
 \newcommand{\eeq}{\end{equation}}
  \newcommand{\bal}{\begin{aligned} }
  \newcommand{\eal}{\end{aligned}}
 \newcommand{\ben}{\begin{eqnarray}}
 \newcommand{\een}{\end{eqnarray}}
 \newcommand{\beno}{\begin{eqnarray*}}
 \newcommand{\eeno}{\end{eqnarray*}}

 \newcommand{\RR}{\mathbf{R}}
 \newcommand{\R}{\mathbb{R}}
 \newcommand{\N}{\mathbb{N}}

\newcommand{\sgn}{\mathrm{sgn}}
 \newcommand{\supp}{\mathrm{supp}}
\newcommand{\domega}{\boldsymbol{\omega}}

 \author{Jiajie Chen, Thomas Y. Hou, and De Huang}
 \address{Applied and Computational Mathematics, California Institute of Technology, Pasadena, CA 91125, USA}
 \date{Feb 1, 2021} 
\title[Finite time blowup of the De Gregorio Model]{On the Finite Time Blowup of the De Gregorio Model for the 3D Euler Equations}
 
\begin{document}

 \begin{abstract}

We present a novel method of analysis and prove finite time asymptotically self-similar blowup of the De Gregorio model \cite{DG90,DG96} for some smooth initial data on the real line with compact support. We also prove self-similar blowup results for the generalized De Gregorio model \cite{OSW08} for the entire range of parameter on $\R$ or $S^1$ for H\"older continuous initial data with compact support. Our strategy is to reformulate the problem of proving finite time asymptotically self-similar singularity into the problem of establishing the nonlinear stability of an approximate self-similar profile with a small residual error using the dynamic rescaling equation. We use the energy method with appropriate singular weight functions to extract the damping effect from the linearized operator around the approximate self-similar profile and take into account cancellation among various nonlocal terms to establish stability analysis. We remark that our analysis does not rule out the possibility that the original De Gregorio model is well posed for smooth initial data on a circle.
The method of analysis presented in this paper provides a promising new framework to analyze finite time singularity of nonlinear nonlocal systems of partial differential equations. 
\end{abstract}
 \maketitle

 \section{Introduction}

In the absence of external forcing, the three-dimensional Navier-Stokes equations for incompressible fluid read:
\begin{equation}
	\label{otarg}
	\mathbf{u}_t+\mathbf{u}\cdot\nabla \mathbf{u}=-\nabla p+\nu\Delta \mathbf{u},\quad \nabla\cdot \mathbf{u}=0.
\end{equation}
Here $\mathbf{u}(x,t): \R^3 \times [0, T)\to \R^3$ is the 3D velocity vector of the fluid, and $p(x,t): \R^3 \times [0, T)\to \R$ describes the scalar pressure. The viscous term $\nu \Delta \mathbf{u}$ models the viscous forcing in the fluid. In the case of $\nu=0$, equations \eqref{otarg} are referred to as the Euler equations. The divergence-free condition $\nabla \cdot \mathbf{u}=0$ enforces the incompressibility of the fluid. The  Navier-Stokes equations are among the most fundamental nonlinear partial differential equations. The fundamental question regarding the global regularity of the 3D Euler and Navier-Stokes equations for general smooth initial data with finite energy remains open, and it is generally viewed as one of the most important open questions in mathematical fluid mechanics, see the surveys \cite{fefferman2006existence,majda2002vorticity, constantin2007euler, gibbon2008three, hou2009blow}. 

Define vorticity ${\domega}=\nabla\times {\bf u}$, then ${\domega}$ is
governed by
  \begin{eqnarray}
\label{eqn-nse-w}
    {\domega}_t+({\bf u}\cdot\nabla) {\domega} = \nabla
    {\bf u} \cdot {\domega} + \nu \Delta {\domega}.
  \end{eqnarray}
The term $\nabla {\bf u} \cdot {\domega}$ on the right hand side is referred to as the vortex stretching term, which is absent in the two dimensional case. Note that $\nabla {\bf u}$ is formally of the same order as ${\domega}$. In fact, if ${\bf u}$ decays sufficiently fast in the far field, one can show that
 $ c_p \|{\domega}\|_{L^p} \leq   \|\nabla {\bf u}\|_{L^p} \leq C_p \|{\domega}\|_{L^p} $ for $1 < p < \infty$ with constants $c_p, C_p>0 $ depending on $p$. Thus the vortex stretching term scales quadratically as a function of vorticity, i.e.  $\nabla {\bf u}\cdot {\domega} \approx {\domega}^2$. The vortex stretching term in the 3D Navier-Stokes or Euler equations is the main source of difficulty in obtaining global regularity. 

\subsection{The De Gregorio model and its variant}
 In this paper, we study the finite time singularity of the 1D De Gregorio model \cite{DG90,DG96} and its generalization. The De Gregorio model is a simplified model to study the effect of advection and vortex stretching in the 3D incompressible Euler equations.
Specifically, the inviscid De Gregorio model is given below
\beq\label{eq:DG}
\bal
\om_t + a u \om_x = u_x \om\;, \quad
        u_x =  H \om \; ,
\eal
\eeq
where $H$ is the Hilbert transform and $a \in \R $ is a parameter. In this 1D model, 
$\om$ models the vorticity ${\domega}$ in the 3D Euler equations \eqref{eqn-nse-w} with $\nu=0$. The nonlinear terms $ u \om_x$ and $u_x \om$ model the advection term $({\bf u}\cdot\nabla) {\domega} $ and the vortex stretching term $ \nabla{\bf u} \cdot {\domega} $, respectively. 
The Biot-Savart law is modeled by $u_x =  H \om$, which preserves the same scaling as that of the original Biot-Savart law.
The case of $a=0$ is reduced to the well-known Constantin-Lax-Majda model \cite{CLM85}, in which the authors proved the finite time singularity formation for a class of smooth initial data. The case of $a = 1$ was proposed by De Gregorio in \cite{DG90} and its generalization to $a \in \R$ was proposed by Okamoto et. al. in \cite{OSW08}. Throughout this paper, we call equation \eqref{eq:DG} the De Gregorio (DG) model. There are various 1D models proposed in the literature. We refer to \cite{Elg17,kiselev2018} for excellent surveys of other 1D models for the 3D Euler equations and the surface quasi-geostrophic equation.

One important feature of the De Gregorio model is that it captures the competition between the advection term and the vortex stretching term. It is not hard to see that
when $a<0$, the advection effect would work together with the vortex stretching effect to produce a singularity. Indeed, Castro and C\'ordoba \cite{Cor10}  proved the finite time blow-up for $a < 0$ based on a Lyapunov functional argument. For $a > 0$, there are competing nonlocal stabilizing effect due to the advection and the destabilizing effect due to vortex stretching, which are of the same order in terms of scaling. Even for arbitrarily small $a>0$, in which case we expect that the advection effect is much weaker than the vortex stretching, using the same Lyapunov functional argument in \cite{Cor10} would fail to prove a finite time singularity since the control of the solution through the Lyapunov functional is not strong enough. We remark that the stabilizing effect of advection has also been studied by Hou-Li in \cite{hou2008dynamic} for an exact 1D model of the 3D axisymmetric Navier-Stokes equations along the symmetry axis and by Hou-Lei for a 3D model of the axisymmetric Navier-Stokes equations in \cite{lei2009stabilizing}.

The question of whether the De Gregorio model would develop a finite time singularity for $a>0$ has remained unsolved for some time, especially in the case of $a=1$. 
In a recent paper by Elgindi and Jeong \cite{Elg17}, they constructed a smooth self-similar profile for small $|a|$ and a $C^{\alpha}$ self-similar profile for all $a \in \R$ using a power series expansion and an iterative construction. We note that the self-similar profiles constructed in \cite{Elg17} decay slowly in the far field and the corresponding velocity $u$ does not have finite energy. In \cite{martineznonlinear}, Castro performed some preliminary study on \eqref{eq:DG} with $a=1$ both analytically and numerically and obtained finite time blowup from $C_c^{\infty}$ initial data under some convexity and  monotonicity assumptions on the solution.

\subsection{Main results}
Let $\Om, c_l, c_{\om}$ be the solution of the self-similar equation of \eqref{eq:DG} given below
\beq\label{eq:self_similar_eqn}
( c_l x + a U ) \Om_x = ( c_{\om} + U_x) \Om, \quad U_x = H \Om,
\eeq
with $c_{\om} < 0$ and a self-similar profile $\Om \neq 0$ in some weighted $H^1$ space. Then for some given $T>0$, 
\beq\label{eq:self_similar_solu}
\om(x, t) = \f{1}{ (T-t) |c_{\om}|}  \Om\lt( \f{x}{(T-t)^{ \g }}\rt),  \quad \g = - \f{c_l}{ c_{\om}},
\eeq
is a self-similar singular solution of \eqref{eq:DG}.

We define some notions about the self-similar singularities to be used in this paper.

\begin{definition}[Two types of asymptotically self-similar singularities]
We say that a singular solution $\om$ of \eqref{eq:DG} is asymptotically self-similar if there exists a solution of \eqref{eq:self_similar_eqn} $(\Om, c_l, c_{\om})$ with $\Om \neq 0$ in some weighted $H^1$ space and $c_{\om} <0$ such that the following statement holds true. By rescaling $\om$ dynamically, i.e. $C_{\om}(t) \om(C_l(t) x, t)$ for some time dependent scaling factors $C_{\om}(t), C_l(t) >0$, it converges to $\Om$ as $t \to T^-$ in some weighted $L^2$ norm, where $T > 0$ is the blowup time. In addition, we say that the asymptotically self-similar singularity is of the \textit{expanding} type if the self-similar solution \eqref{eq:self_similar_solu} associated to $(\Om, c_l, c_{\om})$ satisfies $\g <0$ and of the \textit{focusing} type if $ \g > 0$. We call $\g$ the scaling exponent.
\end{definition}



\begin{remark}
We will specify in later Sections the weighted $L^2$ norm in which the dynamically rescaled function of $\om$ converges to the self-similar profile $\Om$ in the following Theorems. We will also specify in later Sections the stronger weighted $H^1$ norm that the self-similar profile $\Om$ belongs to, so that the Hilbert transform $U_x = H \Om$ is well defined and $(\Om, c_l, c_{\om})$ is a solution of \eqref{eq:self_similar_eqn}. 
In the case of small $|a|$, we refer to Propositions \ref{prop:stab_a0}, \ref{prop:cvg_a0} and Section \ref{sec:conv_a0} for more precise statements. Similar statements also apply to other cases.
\end{remark}

Our first main result is regarding the finite time singularity of the original De Gregorio model.
\begin{thm}\label{thm:blowup_a1}
   There exist some $C_c^{\infty}$ initial data on $\R$ such that the solution of \eqref{eq:DG} with $a= 1$ develops an expanding and asymptotically self-similar singularity in finite time with scaling exponent $\g = -1$ and compactly supported self-similar profile $\Om \in H^1(\R)$.
 \end{thm}

Although the initial data and the self-similar profile $\Omega$ have compact support, due to the expanding nature of the blowup, the support of the solution will become unbounded at the blowup time.

\begin{remark} \label{rem:wx}
Surprisingly, the blowup solution in Theorem \ref{thm:blowup_a1} satisfies the property that $|| \om(x, t) / x ||_{L^{\infty}}$ is uniformly bounded up to the blowup time, $\sup_{ t\in [0, T )} || \om(x,t) / x||_{\infty} < +\infty$, which can be seen from the special scaling exponent $\g = -1$ and the proof of Theorem \ref{thm:blowup_a1}. 
\end{remark}

\begin{remark}
The uniform boundedness of $|| \om(t)/ x||_{L^{\infty}}$ over $[0, T)$ implies that $\om(x, t)$ cannot blowup at any finite $x$, which is consistent with the expanding nature of the blowup.
\end{remark}

The second result is finite time blowup of \eqref{eq:DG} for small $|a|$ with $C_c^{\infty}$ initial data. 
 \begin{thm}\label{thm:blowup_a0}
There exists a positive constant $\d >0 $ such that for $|a| < \d$, the solution of \eqref{eq:DG} with some $C_c^{\infty}$ initial data develops a focusing and asymptotically self-similar singularity in finite time with self-similar profile $\Om \in H^1(\R)$.
 \end{thm}

The third result is finite time blowup of \eqref{eq:DG} for all $a$ with $C_c^{\al}$ initial data. 
 \begin{thm}\label{thm0:unif}
 There exists $C_1 >0 $ such that for $0 < \al < \min( 1/4, \;C_1 / |a|)$,  the solution of \eqref{eq:DG} with some $C_c^{\al}$ initial data develops a focusing and asymptotically self-similar singularity in finite time  with self-similar profile $\Om$ satisfying $|x|^{-1/2}\Om \in L^2$ and $|x|^{1/2}\Om_x \in L^2$.
 \end{thm}

The blowup results in Theorem \ref{thm:blowup_a0} and Theorem \ref{thm0:unif} also hold for the De Gregorio model on the circle.

\begin{thm} \label{thm:circle}
Consider \eqref{eq:DG} on the circle. (1) There exists $C_1 > 0$ such that if $|a|  < C_1$, the solution of \eqref{eq:DG} develops a singularity in finite time for some $C_c^{\infty}$ initial data. (2) If $0 < \al < \min( 1/4, \;C_1/|a|)$, then the solution of \eqref{eq:DG} develops a finite time singularity for some initial data $\om_0 \in C^{\al} $ with compact support. 
\end{thm}

\begin{remark}
Due to the fact that \eqref{eq:DG} on a circle does not enjoy the perfect spatial scaling symmetry, we do not establish the result on the asymptotically self-similar singularity in the above theorem. 
\end{remark}

The initial data $\om_0$ we constructed for the previous theorems all satisfy the property that $\om_0$ is odd and $\om_0 \leq 0$ for $x > 0 $. 
The following theorem implies that for $a>0$, the H\"older regularity for $\om_0$ in this class is crucial for the focusing self-similar blow-up.

\begin{thm}\label{thm:crit}
There exists a universal positive constant $a_0$ such that the following statement holds true. Suppose that $a>0, 1 \geq \al > a_0 / a$, and $\om_0 \in C^{\al}$  is odd,  non-positive for $x >0$ and has compact support. Then  the solution of \eqref{eq:DG} with initial data $\om_0$ cannot develop an asymptotically self-similar singularity with blowup scaling $c_l  > - \al^{-1}$ in finite time. In particular, the solution of \eqref{eq:DG} with initial data $\om_0$ cannot develop a focusing and asymptotically  self-similar singularity in finite time.
\end{thm}

Theorem \ref{thm0:unif} and Theorem \ref{thm:crit} show that the critical H\"older exponent of the initial data for a finite time, focusing and asymptotically self-similar singularity is $\al  \approx \f{1}{a}$ for large positive $a$. For \eqref{eq:DG} on the circle, we can prove stronger results. See Theorem \ref{thm:nonblow_S1} and Proposition \ref{prop:crit_a1} in Section \ref{sec:crit_S1}.  In particular, for a class of initial data $\om_0$ that vanishes to the order $|x|^{\kp}$ near $x=0$ with $\kp$ larger than certain threshold, we show that $u_x(0,t)$ cannot blow up at the first singularity time $T$ if it exists. In contrast, $u_x(0,t)$ actually blows up at the first singularity time for the blowup solution that we construct in Theorem \ref{thm:circle}.

Recently, the first author established finite time blowup of \eqref{eq:DG} on the circle with $1-\d < a < 1$ from smooth initial data for some $\d>0$ in \cite{chen2020slightly}. This resolves the endpoint case of the conjecture made in \cite{Elg17,okamoto2014steady} that equation \eqref{eq:DG} develops a finite time singularity for $a < 1$ from smooth initial data in the case of a circle. We remark that Theorems \ref{thm:blowup_a1}, \ref{thm:circle} and the result in \cite{chen2020slightly} do not rule out the possibility that the De Gregorio model \eqref{eq:DG} with $a=1$ is globally well-posed for smooth initial data on the circle. In a recent paper by Jia, Stewart and Sverak  \cite{Sve19}, they studied the De Gregorio model with $a=1$ on a circle and proved the nonlinear stability of the equilibrium $A \sin( 2(\th -\th_0))$ of \eqref{eq:DG} for periodic solutions with period $\pi$. In \cite{lei2019constantin}, Lei, Liu and Ren proved global well-posedness of the solution of \eqref{eq:DG} with $a=1$ on the real line or a circle for initial data $\om_0$ that does not change sign and $|\om_0|^{1/2} \in H^1(S^1)$. These results shed useful light on the DG model on $S^1$ for smooth solutions.

We remark that an important observation made by Elgindi and Jeong in \cite{Elg17} is that the advection term can be substantially weakened by choosing $C^{\al}$ data with small $\al$.
We use this property in the proof of Theorem \ref{thm0:unif}. After we completed our work, we learned from Dr. Elgindi that results similar to Theorems \ref{thm:blowup_a0} and \ref{thm0:unif} have recently been established independently by Elgindi, Ghoul and Masmoudi \cite{Elg19} on the asymptotically self-similar solutions of \eqref{eq:DG} with finite energy and the stability of the asymptotically self-similar blowup.

\subsection{A novel method of analysis}\label{Sec:novel}
One of the main contributions of this paper is that we introduce a novel method of analysis that enables us to prove finite time singularity for the original De Gregorio model with $C_c^\infty$ initial data. Our method of analysis consists of several steps. The first step is to construct an approximate self-similar profile for the De Gregorio model with a small residual error in some energy norm. The second step is to perform linear stability analysis around this approximate self-similar profile in the dynamic rescaling equation with some appropriately chosen normalization conditions and energy norm. The third step is to establish nonlinear stability using a bootstrap argument. See Section \ref{outline-stability} for more details on these steps.

Finally, we choose an initial perturbation sufficiently small in the energy norm so that the initial condition of the De Gregorio model has compact support and show that the solution develops a singularity in finite time. Moreover, we prove that the solution of the dynamic rescaling equation converges to the exact self-similar solution exponentially fast in time in the weighted $L^2$ norm. This enables us to show that by rescaling the solution of \eqref{eq:DG} dynamically, it converges to the exact self-similar profile at the blowup time in the weighted $L^2$ norm and the singularity is asymptotically self-similar.

The method of analysis presented in this paper provides a promising new framework to analyze potential finite time singularity of a nonlinear and nonlocal system of partial differential equations. We have been able to generalize this method of analysis in several aspects.  The first author of this paper has generalized this framework to prove finite time asymptotically self-similar blowup of \eqref{eq:DG} with dissipation for certain range of $a$ in \cite{chen2020singularity}. We have also established finite time self-similar blowup of the HL model proposed in \cite{hou2013finite,luo2013potentially-2} with $C_c^{\infty}$ initial data (see also a recent paper in \cite{choi2014on}). Recently the first two authors of this paper have been able to generalize this framework to prove finite time blowup of the 2D Boussinesq and 3D axisymmetric Euler equations with $C^{1,\alpha}$ velocity and boundary in \cite{chen2019finite2}, which share the same symmetry and sign property as the Luo-Hou scenario \cite{luo2013potentially-2,luo2013potentially-1}. The analysis of the HL model, 2D Boussinesq equations or the 3D Euler equations is much more challenging than that of the De Gregorio model since it is a nonlinear nonlocal system. We are currently working to extend our method of analysis to prove the finite time blowup of the 2D Boussinesq system with smooth initial data.

\vspace{0.1in}
\paragraph{\bf{Organization of the paper}}

In Section \ref{outline-stability}, we outline our general strategy that we use to prove nonlinear stability for various cases. In Section \ref{sec:a0}, we study the De Gregorio model with small $|a|$. In Section \ref{sec:a1}, we construct an approximate self-similar profile with a small residual error numerically for the case of $a=1$  and apply our method of analysis to prove the finite time self-similar blowup for $C_c^{\infty}$ initial data. In Section \ref{sec:alpha}, we study the case with any $a \in \R$ and 
prove finite time singularity for any $a \in \R$ on both $\R$ and $S^1$ for some $C^{\al}$ initial data with compact support. Finally, in Section \ref{sec:DG_neg}, we use a Lyapunov functional argument to prove finite time blowup for all $a<0$ with smooth initial data. 
In the Appendix, we prove several useful properties of the Hilbert transform and some functional inequalities. 

\vspace{0.1in}
\paragraph{\bf{Notations}}
Since the functions that we consider in this paper, e.g. $\om, u$, have odd or even symmetry, we just need to consider $R^{+}$.
The inner product is defined on $R^+$, i.e.
\[
\la f , g \ra  \teq \int_0^{\infty} f g dx ,\quad || f||_{L^p} \teq \lt(\int_0^{\infty} |f|^p dx \rt)^{1/p} .
\]
In Section \ref{sec:a1}, we further restrict the inner product and the norm to the interval $[0, L]$, e.g 
$\la f, g \ra = \int_0^L fg  dx$, since the support of $\om ,\bar{\om}$ lies in $[-L, L]$. 

We use $C, C_i$ to denote absolute constants and $C(A,B,..,Z)$ to denote constant depending on $A,B ,..,Z$. These constants may vary from line to line, unless specified. We also use the notation $A \les B$ if there is some absolute constant $C$ such that $A \leq CB$, and denote $A \asymp B$ if $A \les B$ and $B\les A$. We use $\rightarrow$ to denote strong convergence and $\rightharpoonup$ to denote weak convergence in some norm. The upper bar notation is reserved for the approximate profile, e.g. $\bar{\om}$. The letters $e,f, a_1, a_2, a_3$ are reserved for some parameters that we will choose in Section \ref{sec:a1}.

\section{Outline of the general strategy in establishing nonlinear stability}\label{outline-stability}

Our general strategy in establishing nonlinear stability is to first construct an approximate self-similar profile with a small residual error for the De Gregorio model \eqref{eq:DG}, then prove linear and nonlinear stability of this profile in the dynamic rescaling equation (see equation \eqref{eqn-dyn-rescale} below).
We use both analytic and numerical approaches to construct the approximate self-similar profile in various cases. The analytic approach is based on a class of self-similar profiles of the Constantin-Lax-Majda model (CLM) \cite{CLM85}, or equivalent \eqref{eq:DG} with $a=0$, 
which are derived in \cite{Elg17}. In \cite{Elg17}, the exact self-similar profiles of \eqref{eq:DG} with $a \neq 0 $ are also constructed in various cases. We remark that our analysis {\it does not} rely on these profiles of \eqref{eq:DG} with $a \neq 0 $.

In general, it is very difficult to construct a self-similar profile analytically. An important observation is that the self-similar profile is equivalent to the steady state of the dynamic rescaling equation. If we can solve the dynamic rescaling equation  for long enough time numerically to obtain an approximate steady state with a small residual error, this will give an approximate self-similar profile. Due to this connection, we will not distinguish the approximate steady state of the dynamic rescaling equation and the approximate self-similar profile of the De Gregorio model throughout this paper. We will use this approach to obtain a piecewise smooth approximate self-similar profile $\bar{\om}$ with a small residual error for \eqref{eq:DG} in the case of $a=1$.

A very essential part of our analysis is to prove linear and nonlinear stability of the approximate steady state of the dynamic rescaling equation. The dynamic rescaling equation of \eqref{eq:DG} is given below
\begin{equation}
\label{eqn-dyn-rescale}
\om_t + (c_l(t) x + au) \om_x = (c_{\om}(t) + u_x ) \om \; ,
\end{equation}
where $c_l(t)$ and $c_{\om}(t)$ are time-dependent scaling parameters. See \eqref{eq:rescal1}-\eqref{eq:DGdy01} in subsection \ref{sec:dsform} for more discussion on the dynamic rescaling formulation.
Let $(\bar{\omega}, \bar{u}, \bar{c_l}, \bar{c_{\om}})$ be an approximate steady state of the dynamic rescaling equation. We define the linearized operator $L(\om)$
\beq\label{eq:lin_abstract}
L(\om) = - ( \bar{c}_l x +a \bar{u} ) \om_x + ( \bar{c}_{\om} + \bar{u}_x ) \om  + ( u_x + c_{\om} )  \bar{\om} - ( a u + c_l x ) \bar{\om}_x, \quad u_x = H \om,
\eeq
where the scaling factors $c_l$ and $c_{\om}$, which depend on $\om$, will be chosen later. 
Let $\om$ be the perturbation around the approximate steady state $\bar\omega$. The stability around $\bar\omega$ is reduced to analyzing the nonlinear stability of the dynamic equation
\begin{equation}\label{eq:illustration}
\om_t = L(\om) + N(\om) + F
\end{equation} 
around $\om=0$. The perturbation $\om$ lies in $\mathcal{H}(\Omega)$, a Hilbert space on a domain $\Omega$. Here $F =  (\bar c_{\om} + \bar{u}_x ) \bar{\om} -(\bar{c_l} x + a\bar{u}) \bar{\om}_x$ is the residual error and $N(\om) = ( c_{\om} + u_x ) \om - (c_l x + u) \om_x$ is the non-linear operator. We remark that $L(\om)$ and $N(\om)$ are nonlocal operators since $u_x = H(\om)$ is nonlocal. Due to the presence of the non-linear operator $N$ and the error term $F$, it is not sufficient to only show that the spectrum of $L$ has negative real parts.

Our approach is to first perform the weighted $L^2$ estimate with appropriate weight function $\varphi$ to establish the linear stability (we drop the terms $N(\om)$ and $F$ to illustrate the main ideas)
\beq\label{eq:coer}
\f{1}{2} \f{d}{dt} \la  \vp \om, \om  \ra = 
 \langle\varphi \om ,L(\om)\rangle \leq -\lambda \langle\varphi \om,\om\rangle ,
 \quad \om \in \mathcal{H}(\Omega)
\eeq
for some $\lambda >0$ and then extend the above estimates to the weighted $H^1$ estimates. We can use a bootstrap argument to establish the nonlinear stability of \eqref{eq:illustration}, provided that $F$ is sufficiently small in the energy norm.  

We will focus on the linear stability \eqref{eq:coer} to illustrate the main ideas. The linearized equation around some approximate self-similar profile $(\bar{\om}, \bar{u}, \bar{c}_l, \bar{c}_{\om})$ reads
\[
\om_t  = -   ( \bar{c}_l x +a \bar{u} ) \om_x + ( \bar{c}_{\om} + \bar{u}_x ) \om  + ( u_x + c_{\om} )  \bar{\om} - ( a u + c_l x ) \bar{\om}_x  \; .
\]
The linear stability of the profile is mainly due to the damping effect from some local terms and cancellation among several nonlocal terms.

\subsection{Derivation of the damping term}\label{sec:damp}
The damping effect of the equation comes from two parts that depend locally on $\om$: the stretching term $( \bar{c}_l x +a \bar{u} )\om_x$ and the vortex stretching term $( \bar{c}_{\om} + \bar{u}_x ) \om $. An important observation of the approximate profile is that $( \bar{c}_{\om} + \bar{u}_x )$ is negative for large $|x|$, thus the vortex stretching term $( \bar{c}_{\om} + \bar{u}_x ) \om $ is a damping term for large $|x|$. This is the main source of the damping effect for large $|x|$. However, $( \bar{c}_{\om} + \bar{u}_x ) \om$ is not a damping term for $x$ near $0$ since $ \bar{c}_{\om} + \bar{u}_x$ is positive.

For $x$ close to $0$, we choose a singular weight $x^{-k}, k \in \N_+$ to take advantage of the stretching term.
Performing the weighted $L^2$ estimate, we get
\beq\label{eq:lin_idea0}
\bal
\f{1}{2}\f{d}{dt} \la \om^2, x^{-k} \ra
 &= \la - (\bar{c}_l x + a \bar{u} ) \om_x + (\bar{c}_{\om} + \bar{u}_x) \om , \om x^{-k} \ra \\
 &+ \la ( u_x + c_{\om} )  \bar{\om} - ( a u + c_l x ) \bar{\om}_x , \om x^{-k} \ra \teq I + II.
 \eal
\eeq

The profile we constructed satisfies $\bar c_l x +  a \bar u > 0$ for all $x >0$ and $\bar c_l x + a \bar u \approx C x$ near $x = 0$ for some $C>0$, which can be seen in later sections. Hence, we make a simplified assumption that $\bar{c}_l x + a \bar{u} = C x$ for some $ C > 0$ to illustrate the idea. Using integration by parts, we obtain
\[
\bal
I = \B\la - \f{C(k-1)}{2}  + (\bar{c}_{\om} + \bar{u}_x)  ,\om^2 x^{-k} \B\ra \teq \la D, \om^2 x^{-k}\ra \; .
\eal
\]
We will choose $k$ so that the coefficient $D$ is negative (we choose $k=4$ for $a=1$ and small $|a|$). In our analysis, the main damping term for $x$ near $0$ is obtained from $( \bar c_l x + a \bar u) \om_x$. In addition, for large $|x|$, under the assumption $(\bar c_l x + a \bar u ) \om_x = C x \om_x$, we can obtain a damping term from $(\bar c_l x + a \bar u ) \om_x$ in the $L^2(x^{-k})$ energy estimate after performing integration by parts provided that $ k-1 > 0$. In this case, the weight $x^{-k}$ decays faster than $x^{-1}$.
Similar analysis and results also hold for $(\bar c_l x + a \bar u) \om_x$ in the range of large $|x|$ without assuming $\bar c_l x + a \bar u = Cx$ for all $x$ and some $C>0$. 
In order to control the perturbation $\om$ in the far field, we have to choose a weight $\vp$ that decays more slowly than $x^{-1}$ so that the weighted $L^2$ norm of $\om$ is not too weak for large $|x|$. As a result, $( \bar c_l x + a \bar u  ) \om_x$ does not produce a damping term for large $|x|$ in our weighted $L^2$ estimate after performing integration by parts. This is one of the subtleties in our analysis.

The above derivations also apply to the case of $a = 1$, where the approximate steady state $\bar \om$ and the perturbation $\om$ have finite support $[-L, L]$. The damping term near $x=0$ is mainly from $ (\bar c_l x + \bar u) \om_x$, while the damping term near $x= \pm L$  is mainly from $ (\bar c_{\om} + \bar u_x) \om$. 

Another subtlety in our analysis is that we do not use a singular weight to derive a damping term from $(\bar c_lx + \bar u )\om_x$ in all cases with different $a$. In the case of $a = 1$, we need to estimate the perturbation near the endpoints $x =0, \pm L$ carefully. We choose a singular weight $\vp$ of order $O( (x-L)^{-2})$ near $x = L$ in order to obtain a sharp estimate of $u$. See more discussions in next Section.

\subsection{Estimates of the nonlocal terms}
The II term in \eqref{eq:lin_idea0} consists of several nonlocal terms that are difficult to control. To estimate the vortex stretching term $ ( u_x + c_{\om} )  \bar{\om}$ in \eqref{eq:lin_idea0}, we take full advantage of the cancellation between $u_x$ and $\om$, see Lemmas \ref{lem:vel_a1} and \ref{lem:vel}. 
To control the last term  $- ( a u + c_l x ) \bar{\om}_x $  in \eqref{eq:lin_idea0}, we have to 
choose appropriate functional spaces $(X, Y)$ and develop several functional inequalities $ || u ||_{X} \leq C_{XY} || \om ||_{Y}$ with a sharp constant $C_{XY}$. For example, we need to make use of the isometry property of the Hilbert transform. We remark that an overestimate of the constant $C_{XY}$ could lead to the failure of the linear stability analysis since the effect of the advection term can be overestimated. To implement the above ideas in obtaining the damping term and estimating the nonlocal terms, we need to design the singular weight very carefully. See \eqref{eq:wg_a0} and \eqref{eq:wg_a1} for some singular weights that are used in our analysis.

We remark that some weighted Sobolev spaces with singular weights have been used in \cite{lei2019constantin,Sve19} for the nonlinear stability analysis of the steady state of \eqref{eq:DG} with $a=1$ on the circle. Singular weights similar to those in Sections \ref{sec:a0}, \ref{sec:alpha} and in the form of linear combination of $|x|^{-k}$ have also been designed independently in \cite{elgindi2019finite,Elg19} for the stability analysis.

\subsection{ Energy estimates with computer assistance}

In the case of $a=1$, we use computer-assisted analysis in the following aspects. 
As we discuss at the beginning of Section \ref{outline-stability}, we construct the approximate self-similar profile numerically. We use numerical analysis with rigorous error control to verify that the residual error is small in the energy norm. The key part of the stability analysis is to use energy estimates to establish the linear stability. In the energy estimates, instead of bounding several coefficients by some absolute constants, which leads to overestimates, we keep track of these coefficients. Since these coefficients depend on the approximate self-similar profile constructed numerically, we use numerical computation with rigorous error control to verify several inequalities that involve these coefficients. See Sections \ref{subsec:Computation}, \ref{Numerical-verification} and \ref{sec:non_a1} for more discussions.

There is another computer-assisted approach to establish the stability by tracking the spectrum of a given operator and quantifying the spectral gap; see, e.g. \cite{castro2020global}. The key difference between this approach and our approach is that we \textit{do not} use computation to quantify the spectral gap of the linearized operator $L$ in \eqref{eq:lin_abstract}. In fact, the linearized operator $L$ is not a compact operator due to the Hilbert transform $u_x = H\om$ and the non-compact part of $L$ cannot be treated as a small perturbation. Thus we cannot approximate the linearized operator by a finite rank operator which can be estimated using numerical computation. We refer to \cite{gomez2019computer} for an excellent survey of other computer assisted proofs in PDE.

\section{Finite Time Self-Similar Blowup for Small $|a|$}\label{sec:a0}

In this section, we will present the proof of Theorem \ref{thm:blowup_a0}. We use this example to illustrate the main ideas in our method of analysis by carrying stability analysis around an approximate self-similar profile with a small residual error by using a dynamic rescaling formulation. In this case, we have an analytic expression for the approximate steady state $\bar{\om}$.

\subsection{Dynamic rescaling formulation}\label{sec:dsform}
We will prove Theorem \ref{thm:blowup_a0} by using a dynamic rescaling formulation. 
Let $ \om(x, t), u(x,t)$ be the solutions of the original equation \eqref{eq:DG}, then it is easy to show that 
\beq\label{eq:rescal1}
  \td{\om}(x, \tau) = C_{\om}(\tau) \om(   C_l(\tau) x,  t(\tau) ), \quad   \td{u}(x, \tau) = C_{\om}(\tau) 
C_l(\tau)^{-1}u(C_l(\tau) x, t(\tau))
\eeq
are the solutions to the dynamic rescaling equations
 \beq\label{eq:DGdy00}
\bal
\td{\om}_{\tau}(x, \tau) + ( c_l(\tau) x + a \td{u} ) \td{\om}_x(x, \tau)  &=   c_{\om}(\tau) \td{\om} + \td{u}_x \om \quad \td{u}_x &= H \td{\om} ,
\eal
\eeq
where  
\beq
\label{eq:DGdy01}
\bal
  C_{\om}(\tau) = \exp\lt( \int_0^{\tau} c_{\om} (s)  d \tau\rt), \ C_l(\tau) = \exp\lt( \int_0^{\tau} -c_l(s) ds    \rt) , \  t(\tau) = \int_0^{\tau} C_{\om}(\tau) d\tau \;.
\eal
\eeq
We have the freedom to choose the time-dependent scaling parameters $c_l(\tau)$ and $c_{\om}(\tau)$ according to some normalization conditions. After we determine the normalization conditions for $c_l(\tau)$ and $c_{\om}(\tau)$, the dynamic rescaling equation is completely determined and the solution of the dynamic rescaling equation is equivalent to that of the original equation using the scaling relationship described in \eqref{eq:rescal1}-\eqref{eq:DGdy01}, as long as $c_l(\tau)$ and $c_{\om}(\tau)$ remain finite. We remark that the dynamic rescaling formulation was introduced in \cite{mclaughlin1986focusing,  landman1988rate} to study the self-similar blowup of the nonlinear Schr\"odinger equations. This formulation is also called the modulation technique in the literature and has been developed by Merle, Raphael, Martel, Zaag and others.
It has been a very effective tool to analyze the formation of singularities for many problems like the nonlinear Schr\"odinger equation \cite{kenig2006global,merle2005blow}, the nonlinear wave equation \cite{merle2015stability}, the nonlinear heat equation \cite{merle1997stability}, the generalized KdV equation \cite{martel2014blow}, and other dispersive problems.

If there exist $C,c>0$ such that for any $\tau > 0$, $c_{\om}(\tau) \leq -C <0$ and $|\td{\om}|$ is bounded from below $|| \td \om( \cdot , \tau)||_{L^{\infty}} \geq c$ for all $\tau >0$, we then have 
\[
\bal
C_{\om}(\tau) &\leq e^{-C\tau}, \ t(\infty) \leq \int_0^{\infty}  e^{-C \tau } d \tau =C^{-1} <+ \infty \; ,
\eal
\]
and that $ || \om(\cdot, t(\tau) ||_{L^{\infty}} =  || \om(   C_l(\tau)  \cdot,  t(\tau) ) ||_{L^{\infty}} = C_{\om}(\tau)^{-1}  ||\td{\om}(x, \tau) ||_{L^{\infty}} \geq c e^{C\tau} $ blows up in finite time $T = t(\infty)$. Suppose that $\td{\om}(\tau)$ converges to $\Om_{\infty}$ in some weighted $L^2$ norm and $c_{l}(\tau), c_{\om}(\tau)$ converge to $c_{l, \infty}, c_{\om,\infty}$, respectively, as $\tau \to \infty$, with $( \Om_{\infty}, c_{l, \infty}, c_{\om,\infty}) $ being a steady state of \eqref{eq:DGdy00} and $\Om_{\infty} \neq 0$ in some weighted $H^1$ space. Since the steady state equation of \eqref{eq:DGdy00} is the same as the self-similar equation \eqref{eq:self_similar_eqn}, we can use \eqref{eq:self_similar_solu} to obtain a self-similar singular solution of \eqref{eq:DG}. We refer to Propositions \ref{prop:stab_a0}, \ref{prop:cvg_a0} and Section \ref{sec:conv_a0} for more details about the convergence and the regularity of $\Om_{\infty}$ in the case of small $|a|$. Similar statements apply to other cases.

To simplify our presentation, we still use $t$ to denote the rescaled time in the rest of the paper.

\subsection{Nonlinear stability of the approximate self-similar profile} \label{sec:nonsta_a}

 Consider the dynamic rescaling equation
 \beq\label{eq:DGdy}
 \bal
 \om_t + (c_l x + a u ) \om_x& = (c_{\om} + u_x) \om \;, \\
          u_x &= H \om \;.\\
 \eal
 \eeq
For $a= 0$, we have the following analytic steady state obtained in \cite{Elg17}
\beq\label{eq:profile_orig}
\bar{\om} = \f{-x}{b^2 + x^2},  \quad \bar{u}_x = \f{b }{b^2 + x^2},  \ c_l =  1,  \ c_{\om} = -  1 \;,
\eeq
where $b = 1/2$. The above steady state can also be obtained by using the exact formula of the solution of \eqref{eq:DG} with $a=0$ given in \cite{CLM85} and analyzing the profile for smooth solution near the blowup time.

We will use the strategy and the general ideas outlined in Section \ref{outline-stability} to establish the linear and nonlinear stability of the approximate self-similar profile.

Choosing an appropriate singular weight function plays a crucial role in the stability analysis. 
We will use the following weight functions in later $L^2$ and $H^1$ estimates :,
\begin{align}
\vp &=  -\f{1}{\bar{\om} x^3}  - \f{1}{ b^2 \bar{\om} x } = \f{ (b^2 + x^2)^2 }{ b^2 x^4} \label{eq:wg_a0} , \\
\psi & = x^2 \vp  =  -\f{1}{\bar{\om} x} - \f{x}{b^2 \bar{\om}} 
=\f{ (b^2 + x^2)^2 }{ b^2 x^2} \;  \label{eq:wg_a02} ,
\end{align}
where $\bar{\om}$ is defined in \eqref{eq:profile_orig} and $b = 1/2$. Note that $\vp \asymp x^{-4} + 1$ and $\psi \asymp x^{-2} + x^2$. 

Theorem \ref{thm:blowup_a0} is the consequence of the following two propositions.

\begin{prop}\label{prop:stab_a0}
Let $\bar{\om}, \vp,\psi$ be the function and weights defined in \eqref{eq:profile_orig}, \eqref{eq:wg_a0} and \eqref{eq:wg_a02}. There exist some absolute constants $ a_0, \mu, c > 0$, such that if $|a|<a_0$ and the initial data $\bar{\om} + \om_0$ of \eqref{eq:DGdy} ($\om_0$ is the initial perturbation) satisfies that $\om_0$ is odd, $\om_0 \in H^2$, $\om_{0,x}(0) = 0$ and $E(0) < c |a|$, where 
\[
E^2 (t) \teq  \la \om^2(t), \vp \ra + \mu \la \om_x^2(t) , \psi \ra,
\]
then we have (a) In the dynamic rescaling equation \eqref{eq:DGdy}, the perturbation remains small for all time: $E(t) < c |a|$ for all $t>0$; (b) The physical equation \eqref{eq:DG} with initial data $ \bar{\om}+\om_0$ develops a singularity in finite time.
\end{prop}

\begin{prop}\label{prop:cvg_a0}
There exists some universal constant $\d$ with $0 < \d < a_0$ such that, if $ |a|< \d$ and the initial perturbation $\om_0$ satisfies the assumptions in Proposition \ref{prop:stab_a0}, then the solution of the dynamic rescaling equation \eqref{eq:DGdy}, $(\bar{\om} + \om, \bar c_l + {c}_l, \bar{c}_{\om} + c_{\om})$, converges to $ (\Om_{\infty}, c_{l, \infty},c_{\om ,\infty}) $ with  $\Om_{\infty} - \bar \om \in L^2( \vp), \Om_{\infty, x} -\bar \om_x\in  L^2 (\psi)$, $c_{l, \infty} > 0, c_{\om ,\infty} <0 $. 
Moreover, $\bar{\om}+ \om $ converges to $\Om_{\infty}$ in $L^2(\vp)$ exponentially fast
 and  $(\Om_{\infty}, c_{l, \infty}, c_{\om ,\infty})$ is the steady state of \eqref{eq:DGdy}. In particular, the physical equation \eqref{eq:DG} with initial data $ \bar{\om}+\om_0$ develops a focusing and asymptotically self-similar singularity in finite time with self-similar profile $\Om_{\infty} \in H^1(\R)$.
\end{prop}

In the Appendix \ref{app:hil}, we describe some properties of the Hilbert transform. We will use these properties to estimate the velocity.

\begin{proof}[Proof of Proposition \ref{prop:stab_a0}]
For any $ |a| \leq a_0$, where $a_0 > 0$ is to be determined. We consider the following approximate self-similar profile by perturbing $c_l$ in \eqref{eq:profile_orig} :
\beq\label{eq:profile_a0}
\bal
\bar{\om} &= \f{-x}{b^2 + x^2},  \quad \bar{u}_x = H\bar{\om} = \f{b }{b^2 + x^2}, \quad    \bar{u} = \arctan \f{x}{b} \;, \\
\bar{c}_l   &= 1 - a \bar{u}_x(0) = 1 - 2a , \quad \bar{c}_{\om} = -  1 \;,
\eal
\eeq
where $ b = 1/2$.  We consider the equation for any perturbation $\om, u $ around the above approximate self-similar profile
\beq\label{eq:lin_a0}
\om_t +  ( \bar{c}_l x +a \bar{u} ) \om_x = ( \bar{c}_{\om} + \bar{u}_x ) \om  + ( u_x + c_{\om} )  \bar{\om} - ( a u + c_l x ) \bar{\om}_x +N(\om) + F(\bar{\om})\;,
\eeq
where $N$ and $F$ are the nonlinear terms and the error, respectively, and are defined below:
\beq\label{eq:NF_a0}
N(\om) = ( c_{\om} + u_x ) \om - (c_l x + a u) \om_x, \quad  F(\bar{\om}) = - a (  \bar{u} - \bar{u}_x(0) x ) \bar{\om}_x \; .
\eeq
We choose the following normalization condition for $c_l$ and $c_{\om}$
\beq\label{eq:normal_a0}
c_l(t) = - au_x(t, 0),  \quad c_{\om}(t) =  - u_x(t, 0) .
\eeq
Note that $\bar{\om}$ is smooth and odd, the initial data $\om_0 +\bar{\om} \in H^2$ and the evolution of \eqref{eq:DGdy} preserves the odd symmetry of the solution.
Standard local well-posedness results imply that $\om(t, \cdot) + \bar{\om}$ remains in $H^2$ locally in time, so does $\om(t, \cdot)$. Using the above normalization condition, the original equation \eqref{eq:DGdy} and the fact that $\om,u$ are odd, we can derive the evolution equation for $\om_x(t,0)$ as follows 
\[
\bal
\f{d}{dt} (\om_x(t, 0) + \bar{\om}_x(0)  ) 
=& [ ( c_{\om} + \bar{c}_{\om} + u_x + \bar{u}_x ) (\bar{\om} + \om) ]_x- 
[ ( \bar{c}_l x+a \bar{u} + c_l x + au ) (\om_x+\bar{\om}_x) ]_x \B|_{x=0} \\
 =&[( c_{\om} + \bar{c}_{\om} + u_x + \bar{u}_x )  -( \bar{c}_l + c_l + a\bar{u}_x + a u_x)    ] (\bar{\om}_x + \om_x) \B|_{x=0} \\
 =&[(  \bar{c}_{\om} + \bar{u}_x )  -( \bar{c}_l + a\bar{u}_x )    ] (\bar{\om}_x + \om_x) \B|_{x=0} 
  = 0,
 \eal
\]
where we have used \eqref{eq:profile_a0} and $\bar{u}_x(0) =2$ to obtain the last equality. It follows
\beq\label{eq:vanish_a0}
\f{d}{dt} \om_x(t, 0) = \f{d}{dt} (\om_x(t, 0) + \bar{\om}_x(0)  ) = 0,
\eeq
which implies $\om_x(t, 0 ) \equiv  \om_{0,x}(0)$.

In the following discussion, our goal is to construct an energy functional
$ E^2(\om) \teq  \la \om^2, \vp \ra +  \mu \la \om_x^2 , \psi \ra$
for some universal constant $\mu$ and show that $E$ satisfies an ODE inequality
\[
\f{1}{2}\f{d}{dt} E^2(\om) \leq C  E^3  - (1/4 - C|a| ) E^2 + C |a| E,
\]
for some universal constant $C$. Then we will use a bootstrap argument to establish nonlinear stability.

\paragraph{\bf{Linear Stability }}
We use $\vp$ defined in \eqref{eq:wg_a0} for the following weighted $L^2$ estimates. Note that $\vp$ is singular and is of order $O(x^{-4})$ near $x=0$. For an initial perturbation $\om_0 \in H^2$ that is odd and satisfies $\om_{0,x}(0) = 0$, $\om(t, \cdot)$ preserves these properties locally in time (see \eqref{eq:vanish_a0}). We will choose $\om_0(x)$ that has $O(|x|^{-1})$ decay as $|x| \rightarrow \infty$ (same decay as $\bar{\om}$). Hence, $\la \om^2, \vp \ra$ is finite. We perform the weighted $L^2$ estimate
\beq\label{eq:lin_a00}
\bal
& \f{1}{2} \f{d}{dt} \la \om^2, \vp \ra  = \la - ( \bar{c}_l x +a \bar{u} ) \om_x +    ( \bar{c}_{\om} + \bar{u}_x ) \om, \om \vp \ra
 + \la ( u_x + c_{\om} )  \bar{\om} , \om  \vp \ra  \\
 &- \la  (a  u + c_l x ) \bar{\om}_x ,\om \vp\ra  + \la  N(\om) ,\om \vp\ra + \la F(\bar{\om}) , \om \vp \ra  \teq I + II + III + N_1 + F_1 .
\eal
\eeq
For $I$, we use integration by parts to obtain
\[
I = \B\la \f{1}{2\vp} (  ( \bar{c}_l x +a  \bar{u} ) \vp )_x + ( \bar{c}_{\om} + \bar{u}_x ) ,\om^2 \vp \B\ra.
\]
Recall $\bar{c}_l = 1 - 2a $ \eqref{eq:profile_a0}. 
Using the explicit formula of profile \eqref{eq:profile_a0} and weight \eqref{eq:wg_a0}, we can 
evaluate the terms in $I$ that do not involve $a$ as follows
\beq\label{eq:damp_a01}
\bal
& \f{1}{2\vp} (  x \vp )_x + ( \bar{c}_{\om} + \bar{u}_x ) 
 = \f{b^2 x^4}{ 2(b^2 + x^2)^2 } \lt(  \f{ (b^2 + x^2)^2 }{b^2x^3}   \rt)_x  + \frac{b}{b^2+x^2} - 1 \\
  =& \f{b^2 x^4}{ 2(b^2 + x^2)^2 }  \lt( 4\f{ x(b^2 +x^2)}{b^2 x^3}  - 3 \f{ (b^2 + x^2)^2}{b^2 x^4}    \rt)
+ \frac{b}{b^2+x^2} - 1   =  \f{2x^2 + b}{x^2 + b^2}  - \f{5}{2}  =  -\f{1}{2},
 \eal
\eeq
where we have used $b = 1/2$. From \eqref{eq:profile_a0} and \eqref{eq:wg_a0}, we have
\beq\label{eq:damp_a02}
\bal
 &\B|\B|\f{1}{ 2 \vp} [ ( \bar{c}_l x - x + a\bar{u} )\vp ]_x\B| \B|_{L^{\infty}} = |a| 
 \B| \B| \f{1}{2\vp}(  ( -2x + \bar{u}  ) \vp )_x \B| \B|_{L^{\infty}} \\
 \leq & |a| \B| \B|  \f{ -2 + \bar{u}_x}{2}+ \f{ -2 x + \bar{u}}{x} \f{ x \vp_x}{2\vp}  \B| \B|_{L^{\infty}} 
 \leq |a|  (1 + || \bar{u}_x||_{\infty} ) \lt( 1+  \B| \B|  \f{x \vp_x}{\vp}  \B| \B|_{\infty} \rt) \les |a|.
 \eal
\eeq
Hence, we can estimate $I$ as follows
\beq\label{eq:lin_a01}
I = \B\la \f{1}{2\vp} (  ( \bar{c}_l x + a  \bar{u} ) \vp )_x + ( \bar{c}_{\om} + \bar{u}_x ) ,\om^2 \vp \B\ra
\leq - \lt(\f{1}{2} - C |a| \rt) \la \om^2 ,\vp \ra \; ,
\eeq
for some absolute constant $C$. Denote $\td{u} \teq u(x) - u_x(0) x $. \eqref{eq:normal_a0} implies that
\[
c_l x + a u = a \td{u}, \quad  \td{u}_x = u_x + c_{\om} .
\]
Using the definition of $II$ in \eqref{eq:lin_a00},\eqref{lem:vel3} and \eqref{lem:vel4}, we obtain
\beq\label{eq:lin_a02}
\bal
II =  - \B\la   ( u_x- u_x(0))  \om, \f{1}{x^3}+ \f{1}{ b^2 x}  \B\ra =   -\f{\pi}{2b^2} u^2_x(0) \leq 0.
\eal
\eeq
For $III$, we use the Cauchy-Schwarz inequality to get
\beq\label{eq:lin_a03_ref1}
\bal
III = -a \la  \td{u} \om,  \bar{\om_x}\vp \ra  \leq |a| \B| \B|   \td{u} \sqrt{x^{-6} + x^{-4}}     \B| \B|_2  \B| \B|   \bar{\om}_x \lt(x^{-6} + x^{-4}\rt)^{-1/2} \vp \om    \B| \B|_2  \; .
\eal
\eeq
For $\td{u}$, we use the Hardy inequality \eqref{eq:hd1} to obtain
\beq\label{eq:lin_a03_ref2}
\bal
&\la  \td{u}^2 , x^{-6} + x^{-4} \ra \lesssim  \la \td{u}^2_x, x^{-4} + x^{-2} \ra  \lesssim \la \om^2 , x^{-4} + x^{-2} \ra  \lesssim \la \om^2, \vp\ra \; .
\eal
\eeq
Note that \eqref{eq:profile_a0} and \eqref{eq:wg_a0} implies
\[
 \B|  \bar{\om}_x \lt(x^{-6} + x^{-4}\rt)^{-1/2} \vp \B| 
=  \B| \f{-b^2 + x^2}{  (b^2 + x^2)^2} \cdot \f{x^3}{(x^2+1)^{1/2}} \cdot \f{b^2 + x^2}{bx^2}\vp^{1/2} \B| 
 \lesssim  \vp^{1/2}.
 \]
We get
\beq\label{eq:lin_a03}
III \leq C |a|  \la \om^2, \vp\ra .
\eeq
Combining the estimates \eqref{eq:lin_a01}, \eqref{eq:lin_a02} and \eqref{eq:lin_a03}, we obtain
\beq\label{eq:lin_a0_L2}
\f{1}{2} \f{d}{dt} \la \om^2, \vp \ra \leq  -(1/2  - C|a|) \la   \om^2 , \vp \ra + N_1 + F_1 \; .
\eeq

\paragraph{\bf{Weighted $H^1$ estimate}} 
The weighted $H^1$ estimate is similar to the $L^2$ estimate. 
We use the weight $\psi$ defined in \eqref{eq:wg_a02} and perform the weighted $H^1$ estimates
\beq\label{eq:lin_a0_H1_0}
\bal
\f{1}{2} \f{d}{dt} \la \om_x^2, \psi \ra & = \la - ( ( \bar{c}_l x +a \bar{u} ) \om_{x} )_x +   ( ( \bar{c}_{\om} + \bar{u}_x ) \om )_x, \om_x \psi \ra
 + \la ( ( u_x + c_{\om} )  \bar{\om}  )_x, \om_x  \psi \ra  \\
 & - \la ( (a  u + c_l x ) \bar{\om}_x)_x ,\om_x \psi\ra  + \la N(\om)_x, \om_x \psi \ra +
\la F(\om)_x , \om_x  \psi \ra   \\
& \teq I + II + III  + N_2 + F_2 \; .
\eal
\eeq
For $I$, we obtain by using integration by parts that
\[
\bal
I& = \la - ( \bar{c}_l x +a \bar{u} ) \om_{xx} + ( -\bar{c}_l - a \bar{u}_x + \bar{c}_{\om} + \bar{u}_x) \om_x
+ \bar{u}_{xx}  \om, \om_x \psi \ra \\
&= \B\la   \f{1}{2\psi} (  ( \bar{c}_l x +a  \bar{u} ) \psi )_x + ( \bar{c}_{\om} -\bar{c}_l+ (1-a)\bar{u}_x ) ,\om_x^2 \psi \B\ra
 - \B\la \f{1}{2} ( \bar{u}_{xx} \psi)_x  , \om^2   \B\ra  .
 \eal
\]
Similar to \eqref{eq:damp_a01}, we use formula \eqref{eq:profile_a0}, \eqref{eq:wg_a02} to evaluate the terms that do not involve $a$.
\[
\bal
 \f{1}{2\psi} (  x  \psi )_x + ( \bar{c}_{\om}  - 1 + \bar{u}_x ) &= 
 \f{b^2 x^2}{2 (b^2 + x^2)^2} \lt(  \f{ (b^2 + x^2)^2}{b^2x}  \rt)_x - 2 + \f{b}{b^2 + x^2} =- \f{1}{2}, \\ 
 \lt(\bar{u}_{xx} \psi \rt)_x  &= \lt( - \f{2bx}{ (b^2 + x^2)^2} \cdot \f{(b^2+x^2)^2} {b^2x^2} \rt)_x  =\f{2}{b x^2} > 0 \;.
  \eal 
\]
Similar to \eqref{eq:damp_a02}, we use  \eqref{eq:profile_a0} and \eqref{eq:wg_a02} to show that the remaining terms in $I$ are small. We get
\[
\bal
 \B| \B|  \f{1}{2\psi} ( (\bar{c}_l x - x + a\bar{u} ) \psi)_x - (\bar{c}_l - 1) -a \bar{u}_x\B|\B|_{L^{\infty}} = |a| \B| \B|   \f{1}{2\psi}  ( ( -2 x + \bar{u}    ) \psi )_x +  2 - \bar{u}_x \B| \B|_{l^{\infty}}
 \les |a|,
 \eal
\]
where we have used $\bar{c}_l - 1 = -2a$. Therefore, we can estimate $I$ as follows 
\beq\label{eq:lin_a04_I}
I \leq  - (  \f{1}{2} - C |a| ) \la \om_x^2, \psi \ra ,
\eeq
where $C$ is some absolute constant.  For $II$, we have
\beq\label{eq:lin_a04_II}
\bal
II &=  \la ( ( u_x + c_{\om} )  \bar{\om}  )_x, \om_x  \psi \ra =  \la u_{xx} \bar{\om} ,\om_x \psi \ra  + \la  ( u_x + c_{\om} )  \bar{\om}_x, \om_x \psi \ra \\
 & =  - \B\la u_{xx} \om_x, \f{1}{x}  + \f{x}{b^2} \B\ra  - \la \td{u}_x , \om_x \bar{\om}_x \psi \ra  \teq II_1 + II_2 \; ,
\eal
\eeq
where $\td{u} = u - u_x(0) x , \td{u}_x = u_x - u_x(0)$. Note that
\[
u_{xx} = H\om_x, \quad \om_x(0) = u_{xx}(0) = 0.
\]
Applying \eqref{lem:vel3} with $(u_x, \om)$ replaced by $(u_{xx}, \om_x)$ and \eqref{lem:vel5}, we obtain
\beq\label{eq:lin_a04}
\B\la u_{xx} \om_x, \f{1}{x} \B\ra =0, \quad \la u_{xx} \om_x, x \ra = 0.
\eeq
It follows that
\beq\label{eq:lin_a04_II1}
II_1 = -\B\la u_{xx} \om_x, \f{1}{x} \B\ra - \f{1}{b^2}\la u_{xx} \om_x, x \ra  = 0.
\eeq
For $II_2$ in \eqref{eq:lin_a04_II},  we use an argument similar to \eqref{eq:lin_a03_ref1}
to obtain
\[
| II_2 | \les \la \td{u}_x^2 , x^{-4} + x^{-2} \ra^{1/2} \cdot
\la ( x^{-4} + x^{-2} )^{-1} (\bar{\om}_x \psi)^2, \om_x^2  \ra^{1/2} .
\]
\eqref{eq:lin_a03_ref2} shows that this first term in the RHS is bounded by $\la \om^2 , \vp\ra^{1/2}$. For the second term, we use the definition \eqref{eq:profile_a0} and \eqref{eq:wg_a02} to obtain 
\[
\B| ( x^{-4} + x^{-2} )^{-1} (\bar{\om}_x \psi)^2 \B| 
=\B| \f{  x^4 }{x^2 + 1}  \lt( \f{-b^2 + x^2}{  (b^2 + x^2)^2}  \rt)^2 \f{(b^2 + x^2)^2}{ b^2 x^2 } \B| \psi\les \psi.
\]
Hence, we have
\beq\label{eq:lin_a04_II2}
II_2 \les \la \om^2 , \vp \ra^{1/2}  \la  \om_x^2,  \psi \ra^{1/2}.
\eeq
For $III$ in \eqref{eq:lin_a0_H1_0}, we note that $c_l x +au = a (u - u_x(0) x )$. Similarly, we have 
\beq\label{eq:lin_a04_III}
| III | \lesssim |a|  \la \om^2, \vp \ra^{1/2}  \la \om_x^2, \psi \ra^{1/2}.
\eeq
In summary, combining \eqref{eq:lin_a04_I},\eqref{eq:lin_a04_II}, \eqref{eq:lin_a04_II1}, \eqref{eq:lin_a04_II2} and \eqref{eq:lin_a04_III}, we prove that
\beq\label{eq:lin_a0_H1}
\f{1}{2}\f{d}{dt}  \la \om_x^2, \psi \ra \leq  C \la \om^2, \vp \ra^{1/2}  \la \om_x^2, \psi\ra^{1/2} - (\f{1}{2} - C |a| )\la   \om_x^2 , \psi \ra + N_2 + F_2,
\eeq
where $C$ is some absolute constant.

\paragraph{\bf{Estimate of nonlinear and error terms}}
We use the following estimate to control $\| u_x \|_{\infty}$
\[
|| u_x||_{\infty} \leq  C || u_x||^{1/2}_2 || u_{xx}||^{1/2}_2 = C||w||^{1/2}_2 ||w_x||^{1/2}_2  \leq  C \la \om^2, \vp \ra^{1/4}  \la \om_x^{2}, \psi  \ra^{1/4} .
\]
Recall the definition of $N(\om), F(\bar{\om})$ in \eqref{eq:NF_a0}. For the nonlinear part $N_1, N_2$, we have
\beq\label{eq:a0_non}
\bal
N_1  & = \la N(\om), \om \vp \ra  \lesssim   (|a|+1) || u_x||_{\infty} \la \om^2,  \vp \ra \lesssim  || u_x||_{\infty} \la \om^2,  \vp \ra \; , \\
N_2 & = \la N(\om)_x, \om_x \psi  \ra  \lesssim   (|a|+1) || u_x||_{\infty} \la \om_x^2,  \psi \ra \lesssim  || u_x||_{\infty} \la \om_x^2,  \psi \ra, \\
\eal
\eeq
where we use that $|a|<1$ since we only consider small $|a|$ in Theorem \ref{thm:blowup_a0}.
We note that $F(\bar{\om})$ \eqref{eq:NF_a0} satisfies $F(\bar{\om} )= O(x^3) $ near $0$ and $F(\bar{\om}) = O(x^{-1})$ for large $x$. From \eqref{eq:wg_a0} and \eqref{eq:wg_a02}, we have $F(\bar{\om}) \in L^2(\vp) $ and $(F(\bar{\om}))_x \in L^2(\psi)$. Then for the error terms $F_1, F_2$, we can use the Cauchy Schwarz inequality to obtain
\beq\label{eq:a0_err}
\bal
| F_1| &=  | \la F(\bar{\om}), \om \vp \ra |  \leq  \la F^2(\bar{\om}), \vp \ra^{1/2} \la \om^2, \vp \ra^{1/2} \lesssim |a| \la \om^2, \vp \ra^{1/2} \;, \\
|F_2|  & =  | \la (F(\bar{\om}))_x, \om_x \psi \ra |  \leq  \la (F(\bar{\om}))_x^2, \psi \ra^{1/2} \la \om_x^2, \psi \ra^{1/2} \lesssim |a| \la \om_x^2, \psi \ra^{1/2}.
\eal
\eeq

\paragraph{\bf{Nonlinear Stability}}
Let $\mu < 1$ be some positive parameter to be determined. We consider the following energy norm
\[
E^2(t) \teq  \la \om^2, \vp \ra + \mu \la \om_x^2 , \psi \ra.
\]
Using the previous estimates on $u_x$ and the Cauchy Schwarz inequality, we have
\[
\la \om^2, \vp \ra^{1/2}  \la \om_x^{2}, \psi  \ra^{1/2} \leq \mu^{-1/2} E^2,   
\quad || u_x||_{\infty}  \leq C \la \om^2, \vp \ra^{1/4}  \la \om_x^{2}, \psi  \ra^{1/4} \leq  C  \mu^{-1/4} E.
\]
Combining \eqref{eq:lin_a0_L2}, \eqref{eq:lin_a0_H1}, \eqref{eq:a0_non}, \eqref{eq:a0_err} and the above estimate, we derive
\[
\bal
\f{1}{2}\f{d}{dt}  E^2(t) & \leq  - \lt( \f{1}{2}  - C |a| \rt) E^2 +  C\mu  \la \om^2, \vp \ra^{1/2}  \la \om_x^2, \psi\ra^{1/2} + C |a| E +
 C || u_x||_{\infty} E^2  \\
 & \leq - \lt( \f{1}{2}  - C |a|  - C \sqrt{\mu}\rt) E^2  + C |a| E + C \mu^{-1/4}  E^3 \;, \\
 \eal
\]
where $C$ is some absolute constant. Now we choose  $\mu$ such that $C\sqrt{\mu} < 1/4$. Note that $\mu$ is also a universal constant. It follows that
\beq\label{eq:boot_a0}
\f{1}{2}\f{d}{dt}  E^2(t) \leq  - \lt( \f{1}{4}  - C_1 |a|  \rt) E^2  + C_1 |a| E + C_1 E^3 \; ,\\
\eeq
where $C_1$ is a universal constant. For $c_{\om}(t)$ and $ c_l(t)$, they satisfy the following estimate
\[
| c_{\om}(t) |  = |u_x(t, 0) |\leq C_2 E, \quad |c_l(t) | = |a u_x(0)| \leq C_2 E \; ,
\]
for some absolute constant. Hence there exist absolute constants $a_0 , c >0$ with $C_1 a_0 < 1/8$, such that for $|a| < a_0$, if $E(0)< c |a|$, using a bootstrap argument, we obtain 
\beq\label{eq:boot_res_a0}
\bal
E(t) < c |a| , \quad |c_{\om}(t) |  , |c_l(t)|  \leq C_2E(t) < C_2 c|a|,
\eal
\eeq
 for all $t> 0$. We can further require
 \[
  a_0 <  \min (  \f{1}{8C_1}, \f{1}{2C_2c} ) \; ,
 \]
 so that we get $|c_{\om}(t)| , |c_l(t)| < C_2 c |a| < \f{1}{2}$, which implies
 \beq\label{eq:scaling_a0}
 \bar{c}_{\om}  + c_{\om}(t) < -1/2 ,  \quad  c_l(t) + \bar{c}_l  > 1/2 .
 \eeq

As a result, we can choose small initial perturbation $\om_0$ which modifies $\bar{\om}$ in the far field so that we have an initial data $\bar{\om} + \om_0$ with compact support. We can also require that $\om_{0, x}(0) =0$ and $E(0) < c |a|$. Then the bootstrap result and $\bar{c}_{\om} + c_{\om}(\tau) < -1/2 < 0$ imply the finite time blowup. We conclude the proof of Proposition \ref{prop:stab_a0}.
\end{proof}

Based on the a-priori estimate, we can further obtain the convergence result.

\subsection{Convergence to the self-similar solution}\label{sec:conv_a0}

\begin{proof}[Proof of Proposition \ref{prop:cvg_a0}]
An important observation is that the approximate self-similar profile is time-independent. Therefore, we take the time derivative in \eqref{eq:lin_a0} to obtain
\beq\label{eq:time_a0}
\om_{tt} +  ( \bar{c}_l x +a \bar{u} ) \om_{tx} = ( \bar{c}_{\om} + \bar{u}_x ) \om_t  + ( u_{x, t} + c_{\om, t} )  \bar{\om} - ( a u_t + c_{l, t} x ) \bar{\om}_x +N(\om)_{t},
\eeq
where the error term $F(\bar{\om})$ vanishes since it depends on the approximate self-similar profile only.  Note that the normalization condition also implies
\[
\f{d}{dt} w_x(t, 0) = 0.
\]
\paragraph{\bf{Exponential convergence}}
Note that the linearized operator in \eqref{eq:time_a0} is exactly the same as that in the 
weighted $L^2$ estimate \eqref{eq:lin_a0}.
Therefore, we obtain
\beq\label{eq:cvg_a0_ODE}
\f{1}{2} \f{d}{dt} \la \om_t^2, \vp \ra \leq  -(1/2  - C |a|) \la   \om_t^2 , \vp \ra + \la N(\om)_t, \om_t \vp \ra .
\eeq
The nonlinear part reads
\[
N(\om)_t = ( c_{\om, t} + u_{x, t}) \om +  ( c_{\om} + u_{x}) \om_t   -  (c_{l, t} x + a u_{t}) \om_x -  (c_{l} x + a u) \om_{x, t} \teq I + II + III + IV \; ,
\]
where $c_{\om,t} = -u_{x,t}(0), c_{l,t} = -a u_{x,t}(0)$ according to the \eqref{eq:normal_a0}.
We are going to show that 
\beq\label{eq:cvg_a0_nlin}
|  \la N(\om)_t, \om_t \vp \ra | \lesssim E(t) \la \om_t^2 ,\vp \ra .
\eeq



From previous estimates, we can control $|| \om ||_{\infty}, \ || u_x ||_{\infty},  ||\f{u}{x} ||_{L^{\infty}}, |c_{\om}|, |c_l|$ by $E(t)$. Using \eqref{eq:hd1} with $p=2,4$, $x^{-4} + x^{-2} \les \vp$ (see \eqref{eq:wg_a0}) and the $L^2$ isometry of the Hilbert transform, we have 
\[
\bal
||  (u_x - u_x(0) )  (x^{-4} + x^{-2})^{1/2} ||_2 & \lesssim || \om \vp^{1/2} ||_{L^2} \lesssim E(t), \\
  || (u_{x, t} - u_{x, t}(0) ) (x^{-4} + x^{-2})^{1/2}) ||_2  & \les || \om_t \vp^{1/2} ||_{L^2} .
\eal
\]
Moreover, we have 
\[
\B| \f{u_t(x)}{x}  \B| =   \f{1}{\pi}\B| \int_{y > 0}  \log \B| \f{x + y} {x - y} \B|   \f{1}{x}  \om_t(y)  dy \B| 
\les \la \om_t^2, \vp \ra^{1/2} \B\la \lt(  \log \B|\f{x + y} {x - y} \B| \f{1}{x}\rt)^2 , \vp^{-1} \B\ra^{1/2}
\lesssim  \la \om_t^2, \vp \ra^{1/2}.
\]
Taking $x =0$ in the above estimate, we also yield the bound for $|u_{x,t}(0)|$ and thus that for $|c_{\om,t}|, |c_{l,t}|$. 
The tail behavior of $\vp$ \eqref{eq:wg_a0} satisfies
\[
\vp = \f{b^2}{x^4} + \f{2}{x^2} + \f{1}{b^2} = O(x^{-2}) + b^{-2}  , \quad   \vp -b^{-2 }=\f{b^2}{x^4} + \f{2}{x^2}< \vp.
\]
Recall $\td{u} = u - u_x(0)x$ and \eqref{eq:normal_a0}.  We can estimate different parts of $N(\om)_t$ as follows
\[
\bal
| \la I , \om_t \vp \ra|  & \leq  |  \la  ( c_{\om, t} +  u_{x, t}) \om , \om_t (\vp - b^{-2}) \ra|  +  b^{-2}  |  \la  (c_{\om,t} + u_{x, t} ) \om , \om_t  \ra |  \\
& \les \la \td{u}^2_{x, t} , (x^{-4} + x^{-2} \ra^{1/2}  || \om||_{\infty}  \la \om^2_t, \vp \ra^{1/2}
+b^{-2} |c_{\om,t}|  || \om ||_2 || \om_t \vp^{1/2} ||_2 \\
&  \qquad +b^{-2} || u_{x, t} ||_2  || \om||_{\infty} ||\om_t ||_2   \lesssim E(t)  \la \om_t^2, \vp \ra \; ,  \\
\la  II + IV, \om_t \vp \ra& =  \B\la c_{\om } + u_x  + \f{ ( ( c_l x + au) \vp)_x}{2 \vp} , \om_t^2  \vp \ra \B\ra  \lesssim || u_x ||_{\infty} \la \om_t^2, \vp \ra
\lesssim E(t)  \la \om_t^2, \vp \ra \; , \\
\la III , \om_t \vp \ra &   = \B\la c_{l, t} + a \f{u_t}{x} , \om_x x \vp^{1/2} \om_t \vp^{1/2} \ra \lesssim \B| \B| c_{l, t} + a \f{u_t}{x}     \B| \B|_{\infty}
|| \om_x \vp^{1/2} x||_2 || \om_t\vp^{1/2}||_2  \\
& \lesssim E(t ) \la \om_t^2 ,\vp \ra ,
\eal
\]
where we have used $ | x \vp_x / \vp | \les 1 $ to estimate $II + IV$ and $ || \om_x \vp^{1/2} x||_2  = || \om_x \psi^{1/2}||_2 \les E(t)$ to obtain the last inequality. In summary, we have proved \eqref{eq:cvg_a0_nlin}. Consequently, by substituting the above estimates and \eqref{eq:boot_res_a0} into \eqref{eq:cvg_a0_ODE}, we obtain
\[
\bal
\f{1}{2} \f{d}{dt} \la \om_t^2, \vp \ra &\leq  -(1/2  - C |a|) \la   \om_t^2 , \vp \ra + C_3 E(t) \la \om_t^2, \vp \ra  \\
& \leq  -(1/2  - C |a|) \la   \om_t^2 , \vp \ra + C_3 c |a|  \la \om_t^2, \vp \ra  =   -(1/2  - C |a| - C_3 c |a| ) \la   \om_t^2 , \vp \ra
\eal
\]
for some universal constant $C_3$. Thus, there exists $0 < \d< a_0$ such that
\[
C \d + C_3 c \d < \f{1}{4}.
\]
Hence, if $|a| < \d$, we obtain
\beq\label{eq:exp_a01}
\f{d}{dt} \la \om_t^2,  \vp \ra  \leq    -(1/2  - C |a| - C_3 c |a| ) \la   \om_t^2 , \vp \ra \leq  -\f{1}{4} \la  \om_t^2 , \vp \ra.
\eeq
It follows that $ \la \om_t^2, \vp \ra $ converges to $0$  exponentially fast as $t \to \infty$ and that $\om(t)$ is a Cauchy sequence in $L^2(\vp)$ as $t \to \infty$.
It admits a limit $\om_{\infty}$ and we have
\beq\label{eq:strongl2}
||  (\om(t) - \om_{\infty}) \vp^{1/2} ||_2 \leq e^{-t/4  }.
\eeq
According to the a-priori estimate $\la \om_x(t, \cdot)^2, \psi \ra < E^2(t) < (ca)^2 $, there is a subsequence $\om(t_n)$ of $\om(t)$, such that $\om_x(t_n) \psi^{1/2}$ converges weakly in $L^2$, and the limit must be $\om_{\infty,x} \psi^{1/2}$. Therefore, we conclude that $\om_{\infty} \in L^2 (\vp)$ and $\om_{\infty, x} \in L^2 (\psi)$. Using these convergence results, we obtain
\beq\label{eq:strongcl}
 c_l(t) = -a u_x(t, 0) \to -a H \om_{\infty}(0), \quad c_{\om} = -u_x(t,0) \to - H \om_{\infty}(0),
 \eeq
 as $t \to \infty$. Using the formulas of $\bar \om$ in \eqref{eq:profile_orig}, $\vp, \psi$ in \eqref{eq:wg_a0} and the above result, we obtain $\om_{\infty} , \bar \om   \in H^1(\R)$, which implies $\om_{\infty} + \bar \om   \in H^1(\R)$.

\paragraph{\bf{Convergence to self-similar solution}}
Finally, we verify that $\om_{\infty} + \bar{\om}$ with some $c_{l,\infty}, c_{\om, \infty}$ is a 
steady state of \eqref{eq:DGdy}. 

We use $\Om, U, \kp_l, \kp_{\om}$ to denote the original solution of \eqref{eq:DGdy}
\[
\Om = \om + \bar{\om}, \ U = u + \bar{u}, \ \kp_l = c_l + \bar{c}_l , \ \kp_{\om} = c_{\om}  + \bar{\om}.
\]
In particular, we define $(\Om_{\infty}, U_{\infty})$ by
\[
\Om_{\infty} = \om_{\infty} + \bar{\om},  \quad U_{\infty, x} = H (\Om_{\infty} ).
\]
Notice that
\[
\om_t = \Om_t = (\kp_{\om} + U_x) \Om - (\kp_l x + a U) \Om_x \teq K(t).
\]
Due to the exponential convergence \eqref{eq:exp_a01}, we have
\beq\label{eq:cvg_a01}
\la K(t)^2, \vp \ra  \to 0  \quad \textrm{as } t \to +\infty.
\eeq


Suppose that $\{ \om(t_n, \cdot) \}_{n\geq 1}$ is a subsequence of $\{ \om(t, \cdot)  \}_{t\geq 0}$ such that as $n \to \infty$, $t_n \to \infty$ and $\om_x(t_n) \psi^{1/2}$ converges weakly to $\om_{\infty,x} \psi^{1/2}$ in $L^2$. From \eqref{eq:strongl2}, we obtain that $\{\om(t_n) \}_{n\geq 1}$ converges strongly to $\om_{\infty}$ in $L^2(\vp)$. Using these convergence results, we yield
\beq\label{eq:cvg_a02}
\bal
\Om(t_n, \cdot)  \vp^{1/2}  -  \Om_{\infty} \vp^{1/2} & = ( \om(t_n, \cdot) + \bar{\om} ) \vp^{1/2} - ( \om_{\infty} + \bar{\om} )  \vp^{1/2} \rightarrow  0  \quad \textrm{ in } L^2  \; ,\\
\Om(t_n, \cdot)_x  \psi^{1/2}  - \Om_{\infty,x} \psi^{1/2} & =
( \om(t_n, \cdot)_x + \bar{\om}_x ) \psi^{1/2} -   ( \om_{\infty,x} + \bar{\om}_x ) \psi^{1/2}
\rightharpoonup 0  \quad \textrm{ in } L^2 \; .
\eal
\eeq
Note that $ \psi = x^2 \vp$. It follows that
\[
  x \Om(t_n, \cdot)_x \vp^{1/2} - x \Om_{\infty, x} \vp^{1/2} \rightharpoonup 0  \textrm{ in } L^2 \; .
\]
Interpolating the convergence results in \eqref{eq:cvg_a02},
we get the pointwise convergence
\beq\label{eq:cvg_a03}
 U_x(t_n)  \to H( \Om_{\infty}  ), \quad \f{U(t_n)}{x} \to  \f{ U_{\infty}}{x}
\eeq
in $L^{\infty}$. Recall the normalization condition and the definition of $U$
\[
c_{\om}(t) = -u_x(t, 0), \quad c_{l}(t) = - au_x(t,0), \quad U = u + \bar u.
\]
We get the following convergence
\beq\label{eq:cvg_a04}
\bal
\kp_l(t_n) &= \bar{c}_l + c_l(t_n) =  \bar{c}_l-a u_x(t_n, 0) \to \bar{c}_l - a ( U_{x, \infty}(0) -\bar{u}_x(0) ) \teq c_{l,  \infty} \; , \\
\kp_{\om}(t_n) & = \bar{c}_{\om} + c_{\om}(t) = \bar{c}_{\om} - u_x(t, 0) \to \bar{c}_{\om} -  (U_{x, \infty}(0) -\bar{u}_x(0) )  \teq c_{\om, \infty} \; .
\eal
\eeq
Combining the convergence results \eqref{eq:cvg_a02}, \eqref{eq:cvg_a03} and \eqref{eq:cvg_a04}, 
we obtain that $K(t_n) \vp^{1/2} - K(\infty) \vp^{1/2}$ converges weakly to $0$ in $L^2$, i.e. 
\[
\B( ( \kp_{\om} + U_x) \Om - ( \kp_l  + a \f{U}{x} ) x \Om_x \B) \vp^{1/2}- \B(  ( c_{\om, \infty} + U_{\infty,x}) \Om_{\infty} - ( c_{l,\infty}  + a \f{ U_{\infty} }{x} ) x \Om_{\infty,x} \B) \vp^{1/2} 
 \rightharpoonup 0.
\]

Note that \eqref{eq:cvg_a01} shows that $K(t_n) \to 0$ in $L^2(\vp)$. We get
\[
( c_{\om, \infty} + U_{\infty,x}) \Om_{\infty} - ( c_{l,\infty} x + a U_{\infty} ) \Om_{\infty,x} = 0
\]
in $L^2(\vp)$. The \textit{a-priori} estimate \eqref{eq:scaling_a0} and the convergence result imply that $
c_{l, \infty} > 1/2 >0, \ c_{\om, \infty} < -1/2 < 0$. Therefore, the solution $\Om(t)$ in the dynamic rescaling equation converges to $\Om_{\infty}$ in $L^2(\vp)$ and
$(\Om_{\infty}, c_{l,\infty}, c_{\om,\infty})$ is a steady state of \eqref{eq:DGdy}, or equivalently, a solution of the self-similar equation \eqref{eq:self_similar_eqn}. Using the rescaling relations \eqref{eq:rescal1} and \eqref{eq:DGdy01}, we obtain that the singularity is asymptotically self-similar. Since $ \g = - \f{c_{l, \infty}}{c_{\om,\infty}}  > 0$, the asymptotically self-similar singularity is focusing. The regularity $\Om_{\infty} \in H^1(\R)$ follows from the result below \eqref{eq:strongcl}.
\end{proof}

\begin{remark}
An argument similar to that of proving convergence to the self-similar solutions by time-differentiation given above has been developed independently in \cite{elgindi2019finite}. There is a difference between two approaches in the sense that an artificial time variable was introduced in \cite{elgindi2019finite}, while we use the dynamic rescaling time variable.
\end{remark}

\section{Finite Time Blowup for $a=1$ with $C_c^{\infty}$ Initial Data}
\label{sec:a1}

In this section, we will prove Theorem \ref{thm:blowup_a1} regarding the finite time self-similar blowup of the original De Gregorio model with $a=1$. Compared to the De Gregorio model with small $|a|$ analyzed in the previous Section, the case of $a=1$ is much more challenging since we do not have a small parameter $a$ in the advection term $u \om_x$. The smallness of $|a|$ has played an important role both in the construction of analytic approximate self-similar profile \eqref{eq:profile_a0} and the stability analysis, where we treat the advection term as a small perturbation. We will use the same method of analysis presented in the previous section except that the approximate steady state is constructed numerically. Since our approximate steady state is constructed numerically, we also present a general strategy how to obtain rigorous error bounds for various terms using Interval arithmetic guided by numerical error analysis, see subsection \ref{Numerical-verification}.

To begin with, we consider \eqref{eq:DG} with $a=1$. The associated dynamic rescaling equation reads
 \beq\label{eq:DGdya1}
 \bal
 \om_t + (c_l x +  u ) \om_x = (c_{\om} + u_x) \om\;, \quad
          u_x = H \om \; .
 \eal
 \eeq
For odd initial datum $\om_0$ supported in $[-L, L]$, we use the following normalization conditions
  \beq\label{eq:DGnormal}
c_l =  - \f{ u(L)}{L},  \quad  c_{\om} = c_l.
 \eeq
 We fix $L = 10$. With the above conditions, we have $(c_l x + u) \B|_{x = \pm L}=0$ and
\beq\label{eq:vanish_a1}
\bal
  \pa_t  \om_x(t,0) & = \pa_x ( (u_x+ c_{\om} ) \om - (c_l x + u) \om_x )\B|_{x= 0}  \\
  &=(c_{\om} + u_x(t,0) - c_l - u_x(t,0)) \om_x(t, 0) = 0.
  \eal
 \eeq
 Thus $\om_x(t,0)$ remains constant and $x= \pm L$ is a stationary point of \eqref{eq:DGdya1} and the support of $\om$ will remain in $[-L , L]$, as long as the solution of the dynamic rescaling equation remains smooth.
 
The reader who is not interested in the numerical computation can skip the following discussion on the numerical computation and go directly to Section \ref{sec:comp} and later subsections for the description of the approximate profile and the analysis of linear stability.

\subsection{Construction of the approximate self-similar profile}\label{subsec:Computation}
We approximate the steady state of \eqref{eq:DGdya1} numerically by using the normalization conditions \eqref{eq:DGnormal}. Since $\om$ is supported on $[-L,L]$ and remains odd for all time, we restrict the computation in the finite domain $[0,L]$ and adopt a uniform discretization with grid points $x_i = i h, i =0, 1,.., n = 8000, h = L/8000$. In what follows, the subscript $i$ of $\omega_i^k$ stands for space discretization, and the superscript $k$ stands for time discretization. We solve \eqref{eq:DGdya1} numerically using the following discretization scheme:

\begin{enumerate}
	\item Initial guess is chosen as $\omega_i^0 = -\f{L - x_i}{\pi}\sin(\frac{\pi x_i}{L}), i =0, 1,.., n$.
	\item The whole function $\omega^k$ is obtained from grid point values $w_i^k$ using a standard cubic spline interpolation on $[-L,L]$, with odd extension of $w^k$ on $[-L,0]$. We approximate $w_{x,i}^k$ at the boundary using a second order extrapolation: 
	\[w_x^k(-L) = w_x^k(L) = w_{x,n}^k = \frac{3\omega_{n}^k - 4\omega_{n-1}^k+\omega_{n-2}^k }{2h}.\]
	The resulting $\omega^k$ is a piecewise cubic polynomial and $\omega^k\in C^{2,1}$. The derivative point values $w_{x,i}^k$ are evaluated to be $w_x^k(x_i)$. 
	\item Value of $u^k$ and $u_x^k$ at grid points are obtained using the kernel integrals: 
	\[u_i^k =\f{1}{\pi} \int_0^L \omega^k(y)\log\left|\frac{x_i-y}{x_i+y}\right|dy,\quad u_{x,i}^k = \f{1}{\pi}\int_0^L \frac{2y}{x_i^2-y^2}\omega^k(y)dy.\]
	In particular, for each $x_i$, the contributions to the above integrals from the neighboring intervals $[x_{i-m},x_{i+m}]$ are integrated explicitly using the piecewise cubic polynomial expressions of $\omega$; the contributions from the intervals $[0,L] \backslash [x_{i-m},x_{i+m}]$ are approximate by using a piecewise $8$-point Legendre-Gauss quadrature, in order to avoid large round-off error. We choose $m=8$. We have also computed $u^k_{xx}$ similarly and will use it later. 
	\item The integration in time is performed by the $4_{\text{th}}$ order Runge-Kutta scheme with adaptive time stepping. The discrete time step size $\Delta t_k = t_{k+1}-t_k$ is given by $\Delta t_k = \frac{1}{2}\frac{h}{\max_i | c_l x_i + u_i^k|}$, respecting the CFL stability condition $|c_l x + u^k|\frac{h}{\Delta t_k} \leq 1$. 
	\item After each time step, we apply a local smoothing on $w^k_i$ to prevent oscillation:
	\[w_i^k \longleftarrow \frac{1}{4}w_{i-1}^k + \frac{1}{2}w_i^k+\frac{1}{4}w_{i+1}^k,\quad i = 1,\dots,n-1.\]
\end{enumerate}
Our computation stops when the pointwise residual 
\[F_{\omega,i}^k = (c_{\om}^k + u_{x,i}^k) \om_i^k - (c_l^k x_i +  u_i^k ) \om_{x,i}^k\]
satisfies $\max_i|F_{\omega,i}^k|\leq 10^{-5}$. Then we use $\bar \omega = \omega^k$ as our approximate self-similar profile. The corresponding scaling is 
\[
\bar{c}_l = \bar{c}_{\om} = -0.6991
\]
by rounding up to $4$ significant digits. 

We remark that we observe second order convergence in space and fourth order convergence in time for the numerical method described above. However, we do not actually need to do convergence study (by refining the discretization) for our scheme, as we can measure the accuracy of our approximate self-similar profile {\it a posteriori}. The criterion for a good approximate self-similar profile is that it is piecewise smooth and has a small residual error
in the energy norm.

 \begin{figure}[t]
   \centering
   \includegraphics[width =0.7\textwidth ,height = 0.35\textwidth ]{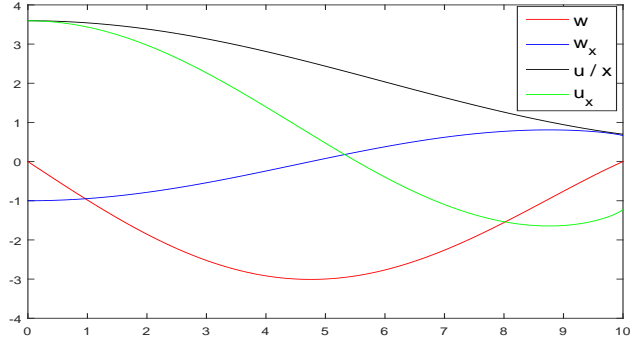}
   \caption{Approximate self-similar profile}
   \label{fig:profile}
 \end{figure}

All the numerical computations and quantitative verifications are performed by MATLAB (version 2019a) in the double-precision floating-point operation. The MATLAB codes can be found via the link \cite{Matlabcode}. To make sure that our computer-assisted proof is rigorous, we adopt the standard method of interval arithmetic (see \cite{rump2010verification,moore2009introduction}). In particular, we use the MATLAB toolbox INTLAB (version 11 \cite{Ru99a}) for the interval computations. Every single real number $p$ in MatLab is represented by an interval $[p_l,p_r]$ that contains $p$, where $p_l,p_r$ are precise floating-point numbers of 16 digits. Every computation of real number summation, multiplication or division is performed using the interval arithmetic, and the outcome is hence represented by the resulting interval $[P_l,P_r]$ that strictly contains $P$. We then obtain a rigorous upper bound on $|P|$ by rounding up $\max\{|P_l|,|P_r|\}$ to 2 significant digits (or 4 when necessary). We remark that, when encountering a non-essential ill-conditioned computation, especially a division, we will replace it by an alternative well-conditioned one. For example, for some function $f(x)$ such that $f(0)=0,f_x(0)<+\infty$, the evaluation of $\f{f(x)}{x}$ at $x=0$ will be replaced by the evaluation of $f_x(0)$.

\subsubsection{Compact support of the approximate profile}\label{sec:comp}
The approximate profile $\bar{\om}$ we obtain actually has compact support. Below we explain how we obtain a compactly supported approximate self-similar profile. First let us assume that $\om$ is a solution of the steady state equation (or equivalently self-similar equation), i.e. taking $\om_t = 0$ in \eqref{eq:DGdya1},
\[
(c_l x +  u ) \om_x = (c_{\om} + u_x) \om  ,   \quad    u_x = H \om .
\]
Differentiating both sides and then evaluating the resulting equation at $x=0$, we obtain 
\[
(c_l + u_x) \om_x |_{x=0} = (c_{\om} + u_x) \om_x |_{x=0} ,
\]
which implies $c_l = c_{\om}$, provided that $\om_x(0) \neq 0$. Suppose that we have a finite time self-similar blowup. Then the scaling factor $c_{\om}$ is negative. See the discussion in Section \ref{sec:dsform}. It follows that $c_l = c_{\om} < 0$. This also holds true for the approximate profile: $\bar c_l = \bar c_{\om} < 0$. Moreover, we have that $\bar u > 0$ for $x>0$ and grows sublinearly for large $x$. The difference between the signs of $\bar c_l x$ and $\bar u(x)$ and their different growth rates for large $|x|$ lead to the following change of sign in the approximate profile
\[
\bar c_l x_0 + \bar u(x_0) = 0,  \quad \bar c_l x +\bar u(x) > 0  \textrm{ for } 0\leq x < x_0, \quad \bar c_l x + \bar u(x) < 0  \textrm{ for } x> x_0,
\]
for some $x_0 > 0$. We expect that a similar change of sign occurs in the dynamic variable $c_l x + u$ and the solution of \eqref{eq:DGdya1} will form a shock. 
When we solve $\bar{\om}$ numerically, we can fix 
the point where the sign of $c_l x + u$ changes by imposing \eqref{eq:DGnormal}. Moreover, the approximate profile satisfies that $\bar c_{\om} + \bar u_x(x)$ is negative for $x > x_0$ (see Figure \ref{fig:profile}). For $x > x_0$, we expect that the dynamic variable $ c_{\om} +  u_x(x) $ is also negative, which implies that $( c_{\om} +  u_x(x) )  \om$ in \eqref{eq:DGdya1} is a damping term. For $x > x_0$, due to the transport term $ (c_l x + u) \om_x$ with $c_l x + u(x) < 0$ and the damping effect $(c_{\om } + u_x(x)) \om$, the solution tends to have compact support. For this reason, in our computation, we have chosen the initial data with compact support and controlled the support of the solution by imposing \eqref{eq:DGnormal}. As a result, the approximate profile also has compact support.

\subsubsection{Regularity of the approximate profile}\label{sec:regular}
In the domain $[-L, L]$, since $\bar{\om}$ is obtained from the cubic spline interpolation, it has the regularity $C^{2, 1}[-L, L]$. Moreover, since $\bar \om(x) = 0$ for $|x| \geq L$, $\bar \om$ is a Lipschitz function on the real line.
We remark that $\bar{\om}$ is in $H^1(\RR)$ but not in $H^2(\RR)$ since $\bar{\om}_x$ is discontinuous at $x = \pm L$, i.e. $\om_x( \pm L ) \neq 0$ (see Figure \ref{fig:profile}). Multiplying $(x^2 - L^2)$, we get a compactly supported and global Lipschitz function $(x^2 - L^2) \bar{\om}_x$. Hence we can define the Hilbert transform of $( (x^2 - L^2) \bar{\om}_x)_x$ which is in  $L^p$ for any $ 1 \leq p < +\infty$.

Applying \eqref{eq:commute4} in Lemma \ref{lem:commute}, we have
\[
\bar{u}_{xx} = H \bar{\om}_x , \quad \bar{u}_{xxx} (x^2 - L^2)= H ( \bar{\om}_{xx} (x^2 -L^2)) .
\]
Using the regularity of $\bar{\om}$, we have that $\bar{u}$ is at least $C^3$ in $(-L, L)$ and $\bar{u}_{xx}$ grows 
logarithmically near $x = \pm L$ since $\bar{\om}_x$ is discontinuous at $x = \pm L$.

\subsubsection{Regularity of the perturbation}
We will choose odd initial perturbation $\om_0$ such that $\om_0 + \bar{\om} \in C_c^{\infty}$ and $\om_{0,x}(0) =0$. Standard local well-posedness result shows that $\om + \bar{\om}$ remains smooth locally in time. Hence, the regularity of $\om$ and $\bar{\om}$ are the same before blowup. Since the odd symmetry of the solution $\om + \bar{\om}$ is preserved and $\bar{\om}$ is odd, this implies that $\om$ is odd. From this property and $\om_x(0) =0$ (see \eqref{eq:vanish_a1}), $\om$ is of order $O(x^3)$ near $x=0$. On the other hand, we have $\om(\pm L) = 0$ since its support lies in $[-L, L]$. In the following derivation, the boundary terms when we perform integration by parts on $\om$ terms will vanish, which can be justified by these vanishing conditions. We will use this property without explicitly mentioning it.

In \cite{martineznonlinear}, the De Gregorio model \eqref{eq:DG} with $a=1$ was solved numerically on $\R$ for $t \in [0,1]$. The author demonstrated the growth of the solution numerically and plotted the solutions at several times that have similar profiles, which share some similar structure with our $\bar{\om}$.

\subsection{Linear stability of the approximate self-similar profile}\label{sec:a1_L2}

Linear stability analysis plays a crucial role in establishing the existence and stability of the self-similar profile. We will establish the linear stability of the approximate self-similar profile in this subsection.

Linearizing \eqref{eq:DGdya1} around $\bar{\om}, \bar{u}, \bar{c}_l,\bar{c}_{\om}$ yields
\beq\label{eq:DGline}
\om_t +  ( \bar{c}_l x + \bar{u} ) \om_x = ( \bar{c}_{\om} + \bar{u}_x ) \om  + ( u_x + c_{\om} ) \bar{\om} - ( u + c_l x ) \bar{\om}_x +N(\om) + F(\bar{\om}) \;,
\eeq
where $\om, u, c_l, c_{\om}$ are the perturbations of the approximate self-similar profile, $N$ and $F$ are the nonlinear terms and the residual error, respectively
\beq\label{eq:NF_a1}
N(\om) = ( c_{\om} + u_x ) \om - (c_l x + u) \om_x, \quad F(\bar{\om}) = ( \bar{c}_{\om} + \bar{u}_x ) \bar{\om}- (\bar{c}_l x +\bar{ u}) \bar{\om}_x.
\eeq

\paragraph{\bf{Main ideas in our linear stability analysis}}
Compared to \eqref{eq:lin_a0}, \eqref{eq:DGline} does not contain a small parameter $a$ in the nonlocal term $(u+c_l x) \bar \om_x$, which makes it substantially harder to establish linear stability. There are three key observations in our linear stability estimates. First of all, we observe that the $u_x \bar \om$ term (vortex stretching) is harmless to the linear stability analysis as we have shown in Section \ref{sec:a0}. We construct the weight function carefully to fully exploit the cancellation between $u_x$ and $\om$ (see Lemma \ref{lem:vel_a1}). Secondly, we observe that there is a competition between the advection term $u \om_x$ and the vortex stretching term $u_x \om$. We expect some cancellation between their perturbation $u \bar{\om}_x$ and $u_x \bar{\om}$. By exploiting this cancellation, we obtain a sharper estimate of $u/x$ by $\om$, which improves the corresponding estimate using the Hardy inequality \eqref{eq:hd1}. Roughly speaking, for $x$ close to $0$, the term $u/x$ can be bounded by $ \om / 5$ in some appropriate norm; similarly, for $x$ close to $L$, the term $( u(x) - u(L)) / (x-L)$ can be bounded by $\om / 3$ in some appropriate norm.
The small constants, $1/5$ and $1/3$, are essential for us to obtain sharp estimates on the non-local term $u$. If we had used a rough estimate with constant $1/5$ replacing by $1/2$,
we would have failed to establish linear stability. Using the first two observations, the estimate of most interactions can be reduced to the estimate of some boundary terms. In order to obtain a sharp stability constant, we express these boundary terms as the projection of $\om$ onto some functions and exploit the cancellation between different projections to obtain the desired linear stability estimate.

Due to the odd symmetry of $u, \om$, we just need to focus on the positive real line. 
Denote
\[
\la f, g \ra  \teq \int_0^{L} f g dx ,\quad  || f ||_p = || f||_{L^{p}[0,L]}
\]
for any $1 \leq p \leq \infty$. For most integrals we consider, it is the same as the integral from $0$ to $+\infty$ since
the support of $\om$ lies in $[-L, L]$. Define a singular weight function on $[0, L]$
\beq\label{eq:wg_a1}
\vp \teq  \lt(- \f{1}{x^3} - \f{e}{x} - \f{f \cdot 2 x}{L^2 - x^2} \rt)
\cdot \lt( \chi_1  \lt(\bar{\om} - \f{x \bar{\om}_x }{5} \rt) + \chi_2 \lt( \bar{\om} -\f{ (x-L) \bar{\om}_x}{3}  \rt)  \rt)^{-1},
\eeq
where $\chi_1, \chi_2 \geq 0$ are cutoff functions such that $\chi_1 + \chi_2 =1 $ and
\[
\chi_1(x) = \begin{cases}
1  & x \in [0,4] \\
0 & x \in [6, 10] \\
\end{cases} ,\quad
\chi_1(x) = \f{  \exp\lt( \f{1}{x-4} + \f{1}{x-6}  \rt) }{ 1 +  \exp\lt( \f{1}{x-4} + \f{1}{x-6}  \rt)  } \
\forall  x \in[4, 6] .
\]
Note that the denominator in \eqref{eq:wg_a1} is negative in $(0,L)$
and that $\vp > 0 $ is a singular weight and is of order $O(x^{-4})$ near $x=0$, $O(  (x-L)^{-2})$ near $x= L$.

Performing the weighted $L^2$ estimate on \eqref{eq:DGline} yields
\beq\label{eq:lin_a11}
\bal
\f{1}{2} \f{d}{dt} \la \om^2 ,\vp \ra  &= \B\la - ( \bar{c}_l x +\bar{ u} )\om_x + (\bar{c}_w + \bar{u}_x) \om, \om \vp \B\ra
+ \B\la (u_x + c_{\om}) , \bar{\om} \om \vp \ra  - \la (  c_l x + u) , \bar{\om}_x \om \vp \B\ra  \\
& +  \la N(\om), \om \vp \ra + \la F(\bar{\om}) , \om  \vp \ra \teq D + I + N_1 +F_1.
\eal
\eeq
Note that the support of the perturbation lies in $[-L, L]$, the above integral is effectively from $0$ to $L$. For $D$, we use integration by parts to obtain 
\beq\label{eq:lin_a1_D}
D = \B\la  \f{1}{2\vp}   (  ( \bar{c}_l x +\bar{ u}  ) \vp)_x    +  (\bar{c}_w + \bar{u}_x) , \om^2 \vp \B\ra \teq \la D(\bar{\om}), \om^2 \vp \ra .
\eeq
From \eqref{eq:wg_a1}, we know that $\vp(x) = O(x^{-4})$ near $x=0$ and $\vp(x) = O( (x-L)^{-2})$ near $x =L$. Using these asymptotic properties of $\vp$, one can obtain that 
\[ D(\bar{\om})(0) =  -( \bar{c}_l + \bar{u}_x(0)) /2  < 0, \quad D(\bar{\om})(L) =   ( \bar{c}_l + \bar{u}_x(L)) /2  <0 . \]
We can verify rigorously that $D(\bar{\om})(x) $ is negative on $[0, L)$ pointwisely. In particular, we treat $ \la D(\bar{\om}), \om^2 \vp \ra$ as a damping term. See Section \ref{sec:damp} for the discussions on the derivation of the damping term.

We estimate the interaction near $x =0$ and $x=L$ differently. First we split the $I$ term into two terms as follows:
\beq\label{eq:I0}
I= \la (u_x + c_{\om})  \bar{\om}  - (  c_l x + u)  \bar{\om}_x, \om \vp \chi_1 \ra
 +  \la (u_x + c_{\om})  \bar{\om}  - (  c_l x + u)  \bar{\om}_x, \om \vp \chi_2 \ra \teq I_1 + I_2.
 \eeq
 We use different decompositions of $(u_x + c_{\om})  \bar{\om}  - (  c_l x + u)  \bar{\om}_x$ for $x$ close to $0$ and to $L$. For $x$ close to $0$ (the $\chi_1$ part), we use $c_{\om} = c_l$ to obtain
 \[
 \bal
&(u_x + c_{\om})  \bar{\om}  - (  c_l x + u)  \bar{\om}_x
= (u_x + c_{\om}) \lt(\bar{\om} - \f{1}{5} \bar{\om}_x x \rt)
+  x \bar{\om}_x \lt(  \f{1}{5}(u_x + c_{\om}) - \f{ u+ c_l x}{x}  \rt)  \\
 =& (u_x + c_{\om}) \lt(\bar{\om} - \f{1}{5} \bar{\om}_x x \rt)
+  x \bar{\om}_x \lt( \f{1}{5} (u_x -u_x(0)) - \f{ u - u_x(0)x}{x}  - \f{4}{5}(c_{\om} +u_x(0))  \rt) .
\eal
 \]
For $x$ close to $L$ (the $\chi_2$ part), using $ c_{\om} =c_l   = - u(L) /L $ \eqref{eq:DGnormal}, we have
 \[
  u + c_l x = u - u(L) + c_l (x - L).
 \]
 Therefore, we obtain
\[
\bal
&(u_x + c_{\om})  \bar{\om}  - (  c_l x + u)  \bar{\om}_x
= (u_x + c_{\om}) \bar{\om}
-   (x-L) \bar{\om}_x \cdot \f{ u - u(L) + c_l (x - L)}{x- L}   \\
 =& (u_x + c_{\om}) \lt(\bar{\om} - \f{1}{3} \bar{\om}_x (x-L) \rt)
+  (x-L) \bar{\om}_x \lt(  \f{1}{3}(u_x + c_{\om}) - \f{ u - u(L) + c_l (x - L)}{x- L}  \rt)  \\
 = & (u_x + c_{\om}) \lt(\bar{\om} - \f{1}{3} \bar{\om}_x (x-L) \rt)
+  (x-L) \bar{\om}_x  \lt(  \f{1}{3}(u_x - u_x(L)) - \f{ u - u(L) - u_x(L) (x - L)}{x- L} \rt) \\
& - \f{2}{3} (x-L) \bar{\om}_x (c_{\om} + u_x(L))  . \\
 \eal
\]
 Using \eqref{eq:I0} and the above decompositions near $x=0$, we get 
 \beq\label{eq:I1}
 \bal
I_1 &=\B\la \lt( \f{1}{5}\f{u_x - u_x(0)}{x^2} - \f{u - u_x(0) x}{x^3}  \rt) , x^3 \bar{\om}_x\om \vp \chi_1 \B\ra
+ \B\la (c_{\om} + u_x),  \lt( \bar{\om} - \f{1}{5} \bar{\om}_x x  \rt) \om\chi_1 \vp \B\ra  \\
&- \f{4}{5} (c_{\om} + u_x(0)) \la x \bar{\om}_x, \om \chi_1 \vp   \ra
\teq I_{11} + I_{12} + I_{13}  .
\eal
\eeq
Similarly, near $x = L$, we have 
\beq\label{eq:I2}
\bal
 I_2 & = \B\la \lt( \f{1}{3} \f{  u_x - u_x(L) }{x -L} - \f{ u -u(L) -u_x(L) (x-L)}{ (x-L)^2}   \rt) 
(x-L)^2 \bar{\om}_x \om \vp \chi_2  \B\ra   \\
& + \B\la (c_{\om} + u_x) ,\lt(\bar{\om} - \f{1}{3}\bar{\om}_x (x-L)   \rt)\om \vp \chi_2  \B\ra
 - \f{2}{3} (c_{\om} + u_x(L)) \la  (x-L) \bar{\om}_x , \om \vp \chi_2 \ra  \\ 
 & \teq I_{21} + I_{22}  +I_{23}. 
 \eal
 \eeq

 \subsubsection{The first part: the interior interaction}
To handle the first term on the right hand side of \eqref{eq:I1} and \eqref{eq:I2}, i.e. $I_{11}, I_{21}$, we use the Cauchy-Schwarz inequality to obtain
\beq\label{eq:I1_1}	
\bal
I_{11} &\leq  \B| \B| \lt( \f{1}{5}\f{u_x - u_x(0)}{x^2} - \f{u - u_x(0) x}{x^3}  \rt) \B| \B|_2
|| x^3 \bar{\om}_x\om \vp \chi_1 ||_2  ,\\
I_{21} &\leq   \B| \B|  \f{1}{3} \f{  u_x - u_x(L) }{x -L} - \f{ u -u(L) -u_x(L) (x-L)}{ (x-L)^2}     \B| \B|_2 ||  (x-L)^2 \bar{\om}_x \om \vp \chi_2||_2.
\eal
\eeq
Using integration by parts yields
\beq\label{eq:iden}
\bal
&\B| \B| \lt( \f{1}{5}\f{u_x - u_x(0)}{x^2} - \f{u - u_x(0) x}{x^3}  \rt) \B| \B|^2_2 \\
 =& \f{1}{25}  \B| \B| \f{u_x - u_x(0)}{x^2} \B| \B|^2_2  - \f{2}{5} \int_0^L \f{ (u_x - u_x(0))\cdot (  u - u_x(0) x) }{x^5} dx + \B| \B| \f{u - u_x(0) x}{x^3} \B| \B|^2_2 \\
 =& \f{1}{25}  \B| \B| \f{u_x - u_x(0)}{x^2} \B| \B|^2_2 -
  \f{1}{5} \f{ (u - u_x(0) x)^2}{x^5} \B|_0^L - \f{1}{5} \cdot 5 \int_0^L  \f{ (  u - u_x(0) x)^2 }{x^6} dx
+ \B| \B| \f{u - u_x(0) x}{x^3} \B| \B|^2_2 \\
 = & \f{1}{25}  \B| \B| \f{u_x - u_x(0)}{x^2} \B| \B|^2_2 -  \f{1}{5 L^5} (  u(L) - u_x(0) L  )^2 =  \f{1}{25}  \B| \B| \f{u_x - u_x(0)}{x^2} \B| \B|^2_2 -  \f{1}{5 L^3} (  c_{\om}+ u_x(0)  )^2 \\
 \leq & \f{1}{25}  \B| \B| \f{\om}{x^2}\B| \B|^2_2 -  \f{1}{5 L^3} (  c_{\om}+ u_x(0)  )^2, \\
\eal
\eeq
where we have used $c_{\om}=c_l = - u(L)/L$ in the second to the last line. To obtain the last inequality, we have used estimate \eqref{eq:hd1} with $p=4$, the facts that the integral in $|| \cdot ||_2$ is from $0$ to $L$ and that $\om$ supports in $[-L, L]$. Denote $v \teq u -u(L) - u_x(L) (x-L) .$ Obviously, we have
\[v(L) = v_x(L) = 0, \ v(0)  = -u(L) + u_x(L) L  = L (c_{\om} + u_x(L)).  \]
Using the above formula and integration by parts, we obtain 
\beq\label{eq:iden21}
\bal
 &\B| \B|  \f{1}{3} \f{  u_x - u_x(L) }{x -L} - \f{ u -u(L) - u_x(L) (x-L)}{ (x-L)^2}     \B| \B|_2^2  =
\B| \B|  \f{1}{3} \f{  v_x }{x -L} - \f{  v  }{ (x-L)^2}     \B| \B|_2^2   \\
 =& \f{1}{9}  \B| \B|  \f{  v_x }{x -L} \B| \B|_2^2
- \f{2}{3}\int_0^L \f{v v_x}{(x-L)^3} dx  +  \B| \B|  \f{  v }{ (x -L)^2} \B| \B|_2^2  \\
=& \f{1}{9}  \B| \B|  \f{  v_x }{x -L} \B| \B|_2^2
-  \f{1}{3} \f{ v^2}{(x-L)^3} \B|^L_0 - \f{1}{3} \cdot 3 \int_0^L \f{ v^2}{(x-L)^4} dx  +\B| \B|  \f{  v }{ (x -L)^2} \B| \B|_2^2  \\
= &  \f{1}{9}  \B| \B|  \f{  v_x }{x -L} \B| \B|_2^2   +  \f{1}{3} \f{v(0)^2}{(0-L)^3}  = \f{1}{9}  \B| \B|  \f{  u_x - u_x(L) }{x -L} \B| \B|_2^2 - \f{1}{3 L} ( c_{\om} + u_x(L))^2   .  \\
\eal
\eeq
Using a formula similar to \eqref{lem:vel0} yields
\[
( u_x - u_x(L)) (x-L)^{-1} = H\lt( \om (x-L)^{-1 } \rt) .
\]
We further obtain the following by using the $L^2$ isometry of the Hilbert transform
\beq\label{eq:iden22}
\bal
  \int_0^L \f{(u_x - u_x(L))^2}{(x-L)^2 } dx  &= \int_{x\in R} \f{\om^2}{(x-L)^2} dx -\int_{x \notin [0, L]}  \f{  ( u_x - u_x(L) )^2 }{ (x -L)^2} dx.
\eal
\eeq
Note that the Cauchy-Schwarz inequality implies
\[
\bal
&\int_{x \notin [0, L]}  \f{  ( u_x - u_x(L) )^2 }{ (x -L)^2} dx \geq \int_{-L}^0 \f{(u_x - u_x(L))^2}{(x-L)^2 } dx  \\
 \geq & \lt(\int_{-L}^0 (u_x - u_x(L)) dx \rt)^2  \lt(  \int_{-L}^0 (x-L)^2 dx \rt)^{-1}  \\
 =&  ( u(0) - u(-L) -u_x(L) L  )^2 \lt( \f{7}{3} L^3\rt)^{-1} = \f{3}{7} \f{ (c_{\om} + u_x(L))^2 L^2}{L^3}
 = \f{3}{7} \f{ (c_{\om} + u_x(L))^2 }{L} .
\eal
\]
Combining \eqref{eq:iden21}, \eqref{eq:iden22} and the above inequality, we get
\beq\label{eq:iden23}
\bal
 &\B| \B|  \f{1}{3} \f{  u_x - u_x(L) }{x -L} - \f{ u -u(L) - u_x(L) (x-L)}{ (x-L)^2}     \B| \B|_2^2 \\
 =& \f{1}{9} \int_{x\in R} \f{\om^2}{(x-L)^2} dx - \f{1}{9}\int_{x \notin [0, L]}  \f{  ( u_x - u_x(L) )^2 }{ (x -L)^2}  dx - \f{1}{3 L} ( c_{\om} + u_x(L))^2   \\
 \leq& \f{1}{9} \int_{x\in R} \f{\om^2}{(x-L)^2} dx - \lt(\f{1}{3 L}+ \f{1}{21L} \rt) ( c_{\om} + u_x(L))^2  . \\
 \eal
\eeq
Combining \eqref{eq:I1_1} , \eqref{eq:iden} and \eqref{eq:iden23} and using the elementary inequality
$xy \leq \lam x^2 + \f{1}{ 4\lam} y^2$, we obtain the estimate for $I_{11}, I_{21}$,
\beq\label{eq:I1_2}
\bal
I_{11} + I_{21} &\leq 25 a_1 \B| \B| \lt( \f{1}{5}\f{u_x - u_x(0)}{x^2} - \f{u - u_x(0) x}{x^3}  \rt) \B| \B|^2_2 + \f{1}{100 a_1}   || x^3 \bar{\om}_x\om \vp \chi_1 ||^2_2   \\
&+ 9 a_2  \B| \B|  \f{1}{3} \f{  u_x - u_x(L) }{x -L} - \f{ u -u(L) -u_x(L) (x-L)}{ (x-L)^2}     \B| \B|^2_2 + \f{1}{36 a_2}   ||  (x-L)^2 \bar{\om}_x \om \vp \chi_2||^2_2 \\
& \leq a_1  \B| \B| \f{\om}{x^2}\B| \B|^2_2 + \f{1}{100 a_1}   || x^3 \bar{\om}_x\om \vp \chi_1 ||^2_2  + a_2 \int_{x\in R} \f{\om^2}{(x-L)^2} dx \\
& + \f{1}{36 a_2}   ||  (x-L)^2 \bar{\om}_x \om \vp \chi_2||^2_2 - a_2 \lt( \f{3}{L} +  \f{3}{7L}\rt) ( c_{\om} + u_x(L))^2  ,\\
\eal
\eeq
where $a_1, a_2 > 0$ are some parameters to be chosen later.

\subsubsection{The second part}
Combining $I_{12}, I_{22}$ in \eqref{eq:I1}, \eqref{eq:I2} respectively, and using the definition of $\vp$ \eqref{eq:wg_a1}, we obtain
\beq\label{eq:I2_1}
\bal
& I_{12} + I_{22} =  \B\la (c_{\om} + u_x), \lt\{  \lt( \bar{\om} - \f{1}{5} \bar{\om}_x x  \rt) \chi_1
+   \lt(\bar{\om} - \f{1}{3}\bar{\om}_x (x-L)   \rt) \chi_2 \rt\}
\om  \vp \B\ra  \\
   =& \B\la  (c_{\om} + u_x) \om, \lt(- \f{1}{x^3} - \f{e}{x} - \f{f \cdot 2 x}{L^2 - x^2} \rt)   \B\ra  \\
    =& (c_{\om} + u_x(0)) \B\la \om, - \f{1}{x^3} - \f{e}{x}  \B\ra
    +\B\la  (u_x - u_x(0)) \om, - \f{1}{x^3} - \f{e}{x}  \B\ra + \B\la  (c_{\om} + u_x) \om, - \f{f \cdot 2 x}{L^2 - x^2}    \B\ra,
\eal
\eeq
where $e$ and $f$ are constants in the definition of $\vp$ \eqref{eq:wg_a1}.
Since $\om \in C^{2, 1}$ and $\om(0) = \om_x(0) = \om_{xx}(0) =0$, we have $\om \cdot x^{-3} \in L^1$ and the above integrals are well-defined.
Using  \eqref{lem:vel3} and \eqref{lem:vel4}, we obtain 
\beq\label{eq:iden31}
\bal
\B\la \f{ (u_x - u_x(0) )  w}{x^3} \B\ra & = \f{1}{2} \int_{\RR} \f{ (u_x - u_x(0) )  w}{x^3} dx = 0,
\\
  \B\la \f{ (u_x -u_x(0)) \om}{x} \B\ra &= \f{1}{2} \int_{\RR} \f{ (u_x - u_x(0) )  w}{x} dx =  \f{\pi}{4} u_x^2(0) . \\
  \eal
  \eeq
  Note that $(c_{\om} + u_x) \om$ is odd. The Tricomi identity Lemma \ref{lem:tric} implies
  \beq\label{eq:iden32}
  \bal
  &\B\la (c_{\om} + u_x) \om , -\f{2x}{L^2 - x^2}  \B\ra = -\int_{\RR^+} (c_{\om} + u_x) \om \lt( \f{1}{L- x} -\f{1}{L + x}  \rt) dx \\
  =& -\int_{\RR}  \f{  (c_{\om} + u_x) \om   }{L-x} dx = - \pi H( (c_{\om} +u_x) \om )(L) =  -\pi c_{\om} H \om (L) - \pi H(u_x \om)(L)  \\
  =& - \pi c_{\om} u_x(L) - \f{\pi}{2} ( u_x^2(L) - \om^2(L) )  =  - \pi c_{\om} u_x(L) - \f{\pi}{2}  u_x^2(L)  . \\
\eal
\eeq
Combining \eqref{eq:I2_1}, \eqref{eq:iden31} and \eqref{eq:iden32}, we obtain
\beq\label{eq:I2_2}
I_{12} + I_{22} = (c_{\om} + u_x(0)) \B\la \om,  \lt(- \f{1}{x^3} - \f{e}{x} \rt) \B\ra   - \f{\pi e}{4} u_x^2(0)   - f \pi c_{\om} u_x(L) - \f{ f \pi}{2}  u_x^2(L).
\eeq

\subsubsection{The remaining part: the boundary interaction}\label{sec:the_remaining_part}
Let $a_3 \teq a_2 ( \f{3}{L} +  \f{3}{7L} )$.
The negative term that appears in the last term of \eqref{eq:I1_2} can be written as
\beq\label{eq:I3_1}
- a_2 ( \f{3}{L} +  \f{3}{7L} ) ( c_{\om} + u_x(L))^2 = - 
a_3 ( c_{\om} + u_x(L))^2 .
\eeq
Combining \eqref{eq:I3_1}, \eqref{eq:I2_2}, $I_{13}, I_{23}$ in \eqref{eq:I1} and \eqref{eq:I2}, we obtain
\beq\label{eq:I3_2}
\bal
& I_{12} + I_{22} + I_{13} + I_{23}  - a_3 ( c_{\om} + u_x(L))^2  \\
 =& (c_{\om} + u_x(0)) \B\la \om,  \lt(- \f{1}{x^3} - \f{e}{x} \rt) \B\ra   - \f{ e \pi }{4} u_x^2(0)   - f \pi c_{\om} u_x(L) - \f{ f \pi}{2}  u_x^2(L)  \\
  & - \f{4}{5} (c_{\om} + u_x(0)) \la \om, x \bar{\om}_x \chi_1 \vp   \ra - \f{2}{3} (c_{\om} + u_x(L)) \la \om, (x-L) \bar{\om}_x \chi_2 \vp \ra - a_3 ( c_{\om} + u_x(L))^2\\
  =& u_x(0) \lt( \B\la \om,  \lt(- \f{1}{x^3} - \f{e}{x} \rt) -\f{4}{5} x \bar{\om}_x \chi_1 \vp  \B\ra   - \f{e \pi}{4} u_x(0)   \rt)  \\
 & + c_{\om} \lt(  \B\la \om,  \lt(- \f{1}{x^3} - \f{e}{x} \rt) -\f{4}{5} x \bar{\om}_x \chi_1 \vp
 - \f{2}{3} (x-L) \bar{\om}_x \chi_2 \vp\B\ra
   - f \pi u_x(L) - a_3 c_{\om} \rt)  \\
   & + u_x(L) \lt(  \B\la w , - \f{2}{3} (x-L) \bar{\om}_x \chi_2 \vp \B\ra - \f{ f\pi}{2}u_x(L) - 2 a_3c_{\om} - a_3 u_x(L) \rt). \\
\eal
\eeq
Note that
\[
\bal
u_x(0)  & = -\f{2}{\pi} \int_0^L \f{\om}{x} dx   , \quad u_x(L)  =  \f{1}{\pi} \int_0^L  \f{2 x}{L^2 - x^2}  \om dx ,\\
c_{\om} & = - \f{u(L)}{L} = \f{1}{ L\pi} \int_0^L \log \lt(  \f{ L+x}{L-x} \rt) \om(x) dx .\\
\eal
\]
All the integrals in \eqref{eq:I3_2} and $c_{\om}, u_x(0), u_x(L)$ are the projection of $\om$ onto some explicit functions. We use the cancellation of these functions to obtain a sharp estimate of the right hand side of \eqref{eq:I3_2}. 
Denote
\beq\label{eq:defg}
\bal
g_{c_{\om}} & \teq  \f{1}{L \pi }  \log \lt( \f{  L+x }{L-x}\rt)  ,\quad g_{u_x(0)} \teq - \f{2}{\pi x} , \quad g_{u_x(L)} \teq \f{2 x} { \pi (L^2 - x^2)} , \\
g_1 & \teq  \lt(- \f{1}{x^3} - \f{e}{x} \rt) -\f{4}{5} x \bar{\om}_x \chi_1  \vp - \f{ e \pi }{4} g_{u_x(0)} ,\\
g_2 &\teq    \lt(- \f{1}{x^3} - \f{e}{x} \rt) -\f{4}{5} x \bar{\om}_x \chi_1 \vp
 - \f{2}{3} (x-L) \bar{\om}_x \chi_2 \vp   - f \pi  g_{u_x(L)} -a_3 g_{c_{\om}} ,\\
 g_3 & \teq - \f{2}{3} (x-L) \bar{\om}_x \chi_2 \vp  - \lt( \f{f\pi}{2} +a_3 \rt) g_{u_x(L)} -2 a_3 g_{c_{\om}}.
\eal
\eeq
With these notations, we can rewrite \eqref{eq:I3_2} as follows
\beq\label{eq:I3_3}
\bal
&u_x(0) \la \om, g_1 \ra  + c_{\om}  \la \om, g_2 \ra + u_x(L) \la \om, g_3 \ra \\
=&  \la \om , g_{u_x(0)} \ra \la \om, g_1 \ra + \la \om , g_{c_{\om}} \ra \la \om, g_2 \ra + \la \om, g_{u_x(L)} \ra \la \om, g_3 \ra .
\eal
\eeq
For some function $R \in C([0,L]), R> 0$ to be chosen, we introduce
\beq\label{eq:opt61}
\bal
y &\teq  (R\vp)^{1/2} \om , \quad  f_1 \teq (R\vp)^{-1/2} g_{u_x(0)} , \quad
f_2 \teq  (R\vp)^{-1/2} g_1, \quad f_3 \teq  (R \vp)^{-1/2} g_{c_{\om}},\\
 f_4 &\teq  (R\vp)^{-1/2} g_2 , \quad f_5 \teq (R\vp)^{-1/2}g_{u_x(L)} , \quad f_6 \teq (R \vp)^{-1/2} g_3.
 \eal
\eeq
Our goal is to find the best constant of the following inequality for any $\om \in L^2(\vp)$ 
\beq\label{eq:opt62}
\bal
 \la f_1 , y\ra \la f_2, y \ra + \la f_3 , y\ra \la f_4 , y\ra  + \la f_5 , y\ra \la f_6 , y\ra & \leq C_{opt} ||y||_2^2 ,\\
\eal
\eeq
which is equivalent to
\[
 \la \om , g_{u_x(0)} \ra \la \om, g_1 \ra + \la \om , g_{c_{\om}} \ra \la \om, g_2 \ra + \la \om, g_{u_x(L)} \ra \la \om, g_3 \ra   \leq C_{opt} \la R, \om^2 \vp\ra,
\]
so that we can estimate \eqref{eq:I3_3} by $\la R, \om^2 \vp \ra$ with a sharp constant. From the definition of functions $g, f$, we have that $g_3 \in \mathrm{span} ( g_{c_{\om}}, g_{u_x(0)}, g_{u_x(L)},g_1, g_2 ) $ and
\beq\label{def:V}
f_6 \in \mathrm{span} ( f_1, f_2,..,f_5 ) \teq V , \quad \mathrm{dim} V = 5.
\eeq
 Without loss of generality, we assume $y \in V$ since $|| Py ||_2 \leq ||y||_2$ and $\la y , f_i \ra = \la P y , f_i \ra$, where $P$ is the orthogonal projector onto $V$. Suppose that $\{e_i\}_{i=1}^5$ is an orthonormal basis (ONB) in $V$ with respect to the $L^2$ inner product on $[0, L]$. It can be obtained via the Gram-Schmidt procedure. Then we have $z = \sum_{i=1}^5 \la z ,e_i \ra e_i$ for any $z \in V$. We consider the linear map $ T : V \to R^5$ defined by $(T z)_i = \la z, e_i \ra, \ \forall z  \ \in V$. It is obvious that $T$ is a linear isometry from $(V, \la \cdot , \cdot \ra_{L^2})$ to $R^5$ with the Euclidean inner product, i.e. $|| Tz ||_{l^2} = || z||_{L^2}$. Denote $v = Ty, v_i = T f_i \in R^5$ . Using the linear isometry, i.e. $\la f_i , y\ra = v^T v_i$ and $||y||^2_2 = v^Tv$,  we can reduce \eqref{eq:opt62} to
  \[
 \sum_{1 \leq i \leq 3} (v^T v_{2i-1} ) ( v_{2i}^T v ) =  v^T (  \sum_{1 \leq i \leq 3}  v_{2i-1} v_{2i}^T   ) v  \leq C_{opt} v^T v  .
 \]
  Denote $M \teq  \sum_{1 \leq i \leq 3}  v_{2i-1} v_{2i}^T \in R^{5\times 5}$. Then the above inequality becomes $v^T M v  \leq C_{opt} v^T v$. Using the fact that $v^T M v = v^T M^T v $, we can symmetrize it to obtain 
  \[
v^T \f{M+M^T}{2} v \leq C_{opt} v^T v.
 \]
 Since $(M^T + M)/2$ is symmetric, the optimal constant $C_{opt}$ is the maximal eigenvalue of $ (M + M^T) / 2$, i.e.
 \beq\label{eq:opt64}
C_{opt} = \lam_{\max}\lt( \f{M+M^T}{2} \rt) =   \lam_{\max} ( \f{1}{2} \sum_{1 \leq i \leq 3} ( v_{2i-1} v_{2i}^T +
v_{2i} v_{2i-1}^T)).
 \eeq
We remark that maximal eigenvalue $\lam_{\max}$ is independent of the choice of the ONB of $V$. For other ONB, the resulting $\lam_{\max}$ will be $\lam_{\max} (  Q (M+M^T) Q^T /2 )$ for some orthonormal matrix $Q \in R^{5\times 5}$, which is the same as \eqref{eq:opt64}. Using \eqref{eq:I3_2}, \eqref{eq:I3_3}, \eqref{eq:opt62} and \eqref{eq:opt64}, we have proved
\beq\label{eq:opt65}
I_{12} + I_{22} + I_{13} + I_{23}  - a_3 ( c_{\om} + u_x(L))^2 
\leq \lam_{\max} ( \f{1}{2} \sum_{1 \leq i \leq 3} ( v_{2i-1} v_{2i}^T +
v_{2i} v_{2i-1}^T)) \la R, \om^2 \vp \ra,
\eeq
where $v_{i} \in R^5$ is the coefficient of $f_i$ (see \eqref{eq:opt61}) expanded under an ONB $\{ e_i\}_{i=1}^5$ of $V = \mathrm{span}(f_1, f_2,..,f_5)$, i.e. the $j$-th component of $v_i$ satisfies $v_{ij} = \la f_i, e_j \ra$. We will choose $R$ so that $\lam_{\max} < 1$ and then the left hand side can be controlled by $\la R, \om^2 \vp \ra$.

\subsubsection{Summary of the estimates}
\label{coeff-optimized}
In summary, we collect all the estimates of $I_{ij}, i=1,2, j=1,2,3$, \eqref{eq:I1}, \eqref{eq:I2}, \eqref{eq:I1_2} and \eqref{eq:opt65}
to conclude
\beq\label{eq:I_tot}
\bal
&\la (u_x + c_{\om})  \bar{\om}  - (  c_l x + u) , \bar{\om}_x, \om \vp  \ra = I = I_1 + I_2
= \sum_{i=1,2, j=1,2,3} I_{ij} \\
\leq & a_1  \B| \B| \f{\om}{x^2}\B| \B|^2_2 + \f{1}{100 a_1}   || x^3 \bar{\om}_x\om \vp \chi_1 ||^2_2  + a_2 \int_{x\in R} \f{\om^2}{(x-L)^2} dx  \\
+& \f{1}{36 a_2}   ||  (x-L)^2 \bar{\om}_x \om \vp \chi_2||^2_2  + \lam_{\max} ( \f{1}{2} \sum_{1 \leq i \leq 3} ( v_{2i-1} v_{2i}^T + v_{2i} v_{2i-1}^T)) \la R, \om^2 \vp \ra \\
\teq & \la A(\bar{\om}) ,\om^2 \vp \ra + \lam_{\max} ( \f{1}{2} \sum_{1 \leq i \leq 3} ( v_{2i-1} v_{2i}^T +
v_{2i} v_{2i-1}^T)) \la R, \om^2 \vp \ra,
\eal
\eeq
where $A(\bar{\om})$ is the sum of the four terms in the first inequality and is given by
\[
A(\bar{\om}) =  \lt( \f{a_1}{x^4} + \f{a_2}{ (x-L)^2} + \f{a_2}{(x+L)^2} \rt) \vp^{-1}
+ \f{ (x^3\bar{\om}_x \chi_1)^2 \vp }{100 a_1} + \f{ ( (x-L)^2  \bar{\om}_x \chi_2 )^2 \vp  }{36 a_2} .
\]

\paragraph{\bf{Optimizing the parameters}} To optimize the estimate, we choose
\beq\label{eq:para_opt}
\bal
e &= 0.005, \ f = 0.004, \ a_1 =\f{1}{6}, \ a_2 = 1.4 f = 0.0056,  \ a_3  =a_2  ( 3 +  \f{3}{7}) /L = 0.00192. 
\eal
\eeq

After specifying these parameters, the coefficient of the damping term $D(\bar{\om})$ (see \eqref{eq:lin_a11}) and the coefficient of the estimate of the interior interaction $A(\bar{\om})$ are completely determined. Then we choose
\beq\label{eq:tot1R}
R(\bar{\om}) = -D(\bar{\om}) - A(\bar{\om}) - 0.3 
\eeq
in \eqref{eq:opt61}.  
The numerical values of $D(\bar{\om})$, $A(\bar{\om})$ and $R(\bar{\om})$ on the grid points are plotted in the first subfigure in Figure \ref{DG_linear}. We can verify rigorously (see the discussion below) 
that $R(\bar{\om}) =-D(\bar{\om}) - A(\bar{\om}) - 0.3 > 0$. In particular, the coefficient of the damping term satisfies $D(\bar{\om}) < -0.3 - A(\bar{\om})$ and is negative pointwisely.
The corresponding $f_i$ in \eqref{eq:opt62} are determined. 
The optimal constant in \eqref{eq:opt65} can be computed  : 
\beq\label{eq:I_tot2}
C_{opt} = \lam_{\max} ( \f{1}{2} \sum_{1 \leq i \leq 3} ( v_{2i-1} v_{2i}^T + v_{2i} v_{2i-1}^T)) < 1 .
\eeq
Combining $\la D(\bar{\om}),  \om^2 \vp \ra$ in \eqref{eq:lin_a11}, \eqref{eq:I_tot} and \eqref{eq:I_tot2},
we obtain the linear estimate
\beq\label{eq:lin_a12}
\bal
&\f{1}{2}\f{d}{dt}  \la \om^2 , \vp \ra =  \la D(\bar{\om}),  \om^2 \vp \ra + I + N_1 + F_1  \\
\leq &\la D(\bar{\om}),  \om^2 \vp \ra + \la A(\bar{\om}), \om^2 \vp \ra + \la R(\bar{\om}) ,\om^2 \vp \ra
+ N_1 + F_1  = -0.3 \la \om^2 , \vp \ra + N_1 + F_1.
\eal
\eeq

\begin{figure}[t]
      \centering
      \includegraphics[width=\textwidth]{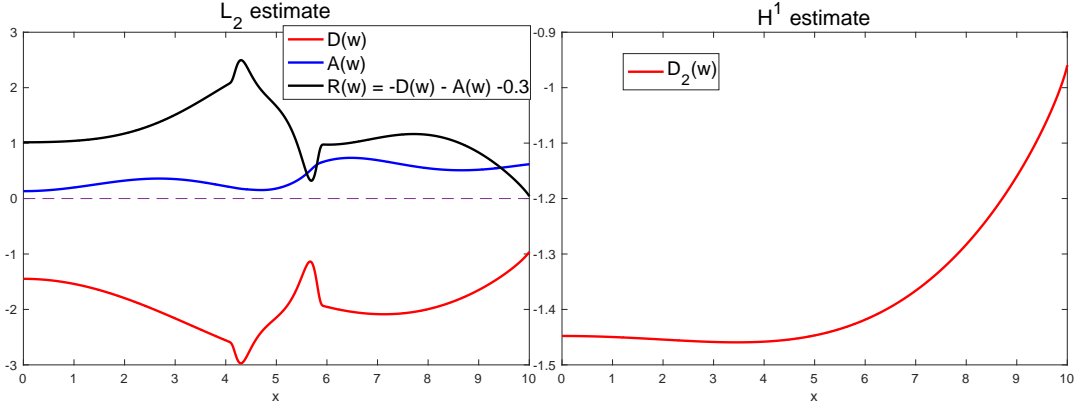}
\caption{ Left: Coefficients of the damping term $D(\bar{\om})$ in the $L^2$ estimate, the estimate of the interior interaction $A(\bar{\om})$ and the remaining terms $R(\bar{\om})$. 
Right: Coefficient of the damping term $D_2(\bar \om)$ in the $H^1$ estimate. 
 }
\label{DG_linear}
\vspace{-0.1in}
    \end{figure}

For those who are not interested in the rigorous verification of the numerical values, they can skip the following discussion and jump to Section \ref{sec:a1h1} for the weighted $H^1$ estimate.

\subsection{Rigorous verification of the numerical values}
\label{Numerical-verification}

We will use the following strategy to verify $R(\bar{\omega}) > 0$ \eqref{eq:tot1R}, $C_{opt} < 1$ \eqref{eq:I_tot2} and $D_2(\bar{\om}) < -0.95$ \eqref{eq:a1_H1_damping} to be discussed later. These quantities appear in the weighted Sobolev estimates and are determined by the profile.

(a) \textbf{Obtaining an explicit approximate self-similar profile.} As described in section \ref{subsec:Computation}, our approximate self-similar profile $\bar\omega$ is expressed in terms of a piece-wise cubic polynomial over the grid points $x_i=\frac{iL}{n},i=0,\cdots,n$. The function values, $\bar\omega(x_i),\bar \omega_x(x_i)$, which are used to construct the cubic Hermite spline, are computed accurately up to double-precision, and will be represented in the computations using the interval arithmetic with exact floating-point bounding intervals. All the following computer-assisted estimates are based on the rigorous interval arithmetic. 

(b) \textbf{Accurate point values of $\bar u,\bar u_x,\bar u_{xx}$.} 
We have described how to compute the value of 
$\bar u_x(x)$ (or $\bar u(x),\bar u_{xx}(x)$) from certain integrals involving $\bar \om$ on $[-L,L]$ in paragraph (3) in Section \ref{subsec:Computation}. For any $x\in [0,L]$, the integral contribution to $\bar u_x(x)$ from mesh intervals within $m=8$ mesh points distance from $x$ is computed exactly using analytic integration. In the outer domain that is $8h$ distance away from $x$, the integrand $\bar \om(y) / (x-y)$ is not singular and we use a composite $8$-point Legendre-Gauss quadrature. 
There are two types of errors in this computation. The first type of error is the round-off error in the computation. 
The second type of error is due to the composite Gaussian quadrature that we use to approximate the integral in the outer domain. Notice that in each interval $[ih, (i+1)h]$ away from $x$, $\bar \om$ is a cubic polynomial and the integrand $\bar \om(y) / (x-y)$ is smooth. We can estimate high order derivatives of the integrand rigorously in these intervals. 
With the estimates of the derivatives, we can further establish error estimates of the Gaussian quadrature. In particular, we prove the following error estimates of the composite Gaussian quadrature in the computation of $\bar u_x, \bar u , \bar u_{xx}$ in the Supplementary material \cite[Section 7]{CHHSupp19}
\[
Error_{GQ}(u_x) < 2 \cdot 10^{-17}, \quad 
Error_{GQ}(u) < 2 \cdot 10^{-19}, \quad 
Error_{GQ}(u_{xx}) < 5 \cdot 10^{-18}.
\]

These two types of errors will be taken into account in the interval representations of $\bar u_x$. That is, each $\bar u_x(x)$ will be represented by $[\lfloor\bar u_x(x)-\epsilon\rfloor_f,\lceil\bar u_x(x)+\epsilon\rceil_f]$ in any computation using the interval arithmetic, where $\lfloor\cdot\rfloor_f$ and $\lceil\cdot\rceil_f$ stand for the rounding down and rounding up to the nearest floating-point value, respectively. We remark that we will need the values of $\bar u_x(x)$ at finitely many points only. The same arguments apply to $\bar u(x)$ and $\bar u_{xx}(x)$ as well.

(c) \textbf{Rigorous estimates of integrals.} In many of our discussions, we need to rigorously estimate the integral of some function $g(x)$ on $[0,L]$. In particular, we want to obtain $c_1,c_2$ such that $c_1\leq \int_0^Lg(x)dx\leq c_2$. A straightforward way to do so is by constructing two sequences of values $g^{up}=\{g^{up}_i\}_{i=1}^n,g^{low}=\{g^{low}_i\}_{i=1}^n$ such that 
\[g^{up}_i \geq \max_{x\in[x_{i-1},x_i]}g(x)\quad\text{and}\quad g^{low}_i \leq \min_{x\in[x_{i-1},x_i]}g(x).\]
Then we can bound 
\[h\cdot \sum_{i=1}^ng^{low}_i \leq \int_0^Lg(x)dx\leq h\cdot\sum_{i=1}^ng^{up}_i.\]

In most cases, we will construct $g^{up}$ and $g^{low}$ from the grid point values of $g$ and an estimate of its first derivative. Let $g^{\max}=\{g^{\max}_i\}_{i=1}^N$ denote the sequence such that $g^{\max}_i = \max\{|g^{up}_i|,|g^{low}_i|\}$. Then if we already have $g_x^{\max}$, we can construct $g^{up}$ and $g^{low}$ as 
\[g^{up}_i = g(x_i) + h\cdot(g_x^{\max})_i\quad \text{and}\quad g^{low}_i = g(x_i) - h\cdot(g_x^{\max})_i.\]
We can use this method to construct the piecewise upper bounds and lower bounds for many functions we need. For example, our approximate steady state $\bar{\om}$ is constructed to be piecewise cubic polynomial using the standard cubic spline interpolation. Since $\bar{\om}_{xxx}$ is piecewise constant, we have $\bar{\om}_{xxx}^{up}$ and $\bar{\om}_{xxx}^{low}$ for free from the grid point values of $\bar{\om}$. Then we can construct $\bar{\om}_{xx}^{up/low},\bar{\om}_{x}^{up/low}$ and $\bar{\om}^{up/low}$ recursively.

Note that for some explicit functions, we can construct the associated sequences of their piecewise upper bounds and lower bounds more explicitly. For example, for a monotone function $g$, $g^{up}$ and $g^{low}$ are just the grid point values.

Moreover, we can construct the piecewise upper bounds and lower bounds for more complicated functions. For instance, if we have $f_a^{up/low}$ and $f_b^{up/low}$ for two functions, then we can construct $g^{up/low}$ for $g = f_af_b$ using standard interval arithmetic. In this way, we can estimate the integral of all the functions we need in our computer-aided arguments. 

Sometimes we need to handle the ratio between two functions, which may introduce a removable singularity. For example, in the construction of $D(\bar \om)^{up}$ and $D(\bar \om)^{low}$ for $D(\bar \om)$ in \eqref{eq:lin_a1_D}, it involves $ \f{x \vp_x}{\vp}, \f{\bar u \vp_x}{\vp}$ and $\vp$ is a singular weight of order $x^{-4}$ near $x=0$. Directly applying interval arithmetic to the ratio near a removable singularity can lead to large errors. We hence need to treat this issue carefully. For example, let us explain how to reasonably construct $g^{up/low}$ for a $g(x) = f(x)/x$ such that $f(x)$ has continuous first derivatives and $f(0)=0$. Suppose that we already have $f^{up/low}$ and $f_x^{up/low}$. Then for some small number $\varepsilon>0$, we let 
\[g^{up}_i = \max\left\{\frac{f^{up}_i}{x_{i-1}},\frac{f^{low}_i}{x_{i-1}},\frac{f^{up}_i}{x_i},\frac{f^{low}_i}{x_i}\right\}\quad \text{for each index $i$ such that $x_{i-1}\geq \varepsilon$}.\]
Otherwise, for $x\in [0,\varepsilon)$, we have
\[g(x) = \frac{f(x)}{x} = f_x(\xi(x))\quad \text{for some $\xi(x)\in[0,x)\subset[0,\varepsilon)$}.\] 
Then we choose $g^{up}_i = \max_{x\in[0,\varepsilon]} f_x^{up}$ for every index $i$ such that $x_i\leq \varepsilon$. The parameter $\varepsilon$ needs to be chosen carefully. On the one hand, $\varepsilon$ should be small enough so that the bound $f_x(\xi(x))\leq \max_{\tilde{x}\in[L-\varepsilon,L]}|f_x(\tilde{x})|$ is almost tight for $x\in[L-\varepsilon,L]$. On the other, the ratio $\varepsilon/h$ must be large enough so that $f^{up}_i/x_{i-1},f^{low}_i/x_{i-1},f^{up}_i/x_i $ and $f^{low}_i/x_i$ are close to each other for $x_{i-1}\geq \varepsilon$. Other kinds of removable singularities can be handled in a similar way. 

See more detailed discussions in the Supplementary Material \cite[Section 1.3]{CHHSupp19}.

(d) \textbf{Estimates of some (weighted) norms of $\bar{\omega},\bar u$.} Once we have used the preceding method to obtain $\bar{\om}_{xxx}^{up/low},\bar{\om}_{xx}^{up/low},\bar{\om}_{x}^{up/low}$ and $\bar{\om}^{up/low}$ from the grid point values of $\bar{\om}$, we can further estimate some (weighted) norms of $\bar{\om}$, e.g. $||\bar{\om}_x||_{L^{\infty}}, ||\bar{\om}_{xx}||_{L^{\infty}}$, rigorously. Moreover, from the discussion of the regularity of $\bar{u}, \bar{\om}$ in Section \eqref{sec:regular}, the needed norms of $\bar{u}$, e.g. $||\bar{u}_x ||_{\infty}, || \bar{u}_{xx}||_{L^2}$, can be bounded by some norms of $\bar{\om}$. See more detailed discussions in the Supplementary Material \cite[Section 1.3]{CHHSupp19}.

(e) \textbf{Rigorous and accurate estimates of certain integrals.} Our rigorous estimate for integrals in the preceding part (c) is only first order accurate. Yet this method is not accurate enough if the target integral is supposed to be a very small number. 
When we need to obtain a more accurate estimate of the integral of some function $P$, we use the composite Trapezoidal rule  
\[
\int_{ah}^{bh} P(x) dx = \sum_{a \leq i < b} ( P( x_i) + P( x_{i+1}) ) h/ 2 + \mathrm{error}(P) .
\] 
The composite Trapezoidal rule uses the values of $P$ on the grid points only, which can be obtained up to the round off error. The numerical integral error, $\mathrm{error}(P)$, can be bounded by the $L^1$ norm of its second order derivative, i.e. $C|| P^{\prime \prime} ||_{L^1} h^2$ for some absolute constant $C$. We use this approach to obtain integral estimates of some functions involving the residual $F(\bar{\om})$. For each function $P$ that we integrate, we prove in the Supplementary Material \cite[Section 3]{CHHSupp19} that $|| P^{\prime \prime} ||_{L^1}$ can be bounded by some (weighted) norms of $\bar u, \bar \om$, e.g. $|| \bar \om ||_{L^{\infty}}$, $|| \bar u_x ||_{L^{\infty}}$ and $||  \f{\bar \om_{xx} }{x}||_{L^2}$. Since these norms can be estimated by the method discussed previously, we can establish rigorous error bound for the integral.

(f) \textbf{Rigorous estimates of $C_{opt}$.} 
Denote by $M_s$ the matrix in \eqref{eq:opt64}
\[
M_s \triangleq \f{1}{2} \sum_{i=1}^3 (v_{2i-1} v_{2i}^T +  v_{2i} v_{2i-1}^T ) = \frac{1}{2}V_1V_2^T,\]
where $V_1 \triangleq [v_1,v_2,v_3,v_4,v_5,v_6]\in \mathbb{R}^{6\times 6}$ and $V_2 \triangleq [v_2,v_1,v_4,v_3,v_6,v_5]\in \mathbb{R}^{6\times 6}$, and $\{v_i\}_{i=1}^6$ are defined as in Section~\ref{sec:the_remaining_part}. Note that $M_s$ is symmetric, but not necessarily positive semidefinite. The optimal constant $C_{opt}$ is then the maximal eigenvalue of $M_s$. To rigorously estimate $C_{opt}$, we first bound it by the Schatten $p$-norm of $M_s$:
\begin{equation}\label{C*_bound}
C_{opt} \leq \|M_s\|_{p} \triangleq \mathrm{Tr}[|M_s|^p]^{1/p}\quad \text{for all $p\geq 1$}.
\end{equation}
Here $|M_s| = \sqrt{M_s^T M_s} = \sqrt{M_s^2}$. In particular, if $p$ is an even number, we have $|M_s|^p = M_s^p$. Therefore,
\[\mathrm{Tr}[|M_s|^p] = 2^{-p}\cdot \mathrm{Tr}[(V_1V_2^T)^p] = 2^{-p}\cdot \mathrm{Tr}[(V_2^TV_1)^p] \triangleq 2^{-p}\cdot \mathrm{Tr}[X^p]\]
where $X = V_2^TV_1$. Note that each entry of $X$ is the inner product between some $v_i$ and $v_j$, $i,j=1,\dots,6$. Recall from \eqref{def:V} and its following paragraph that $v_i = T f_i, i=1,2,..,6$
and that $T : V \to R^5$ is a linear isometry. We have 
\[
\la  f_i,f_j \ra =  \la Tf_i, Tf_j \ra = v_i^Tv_j.\]
Therefore, to compute the entries of $X$, we only need to compute the pairwise inner products between $f_1,\dots,f_6$ (we do not need to compute the coordinate vectors $v_i$ explicitly). This is done by interval arithmetic based on the discussion in the preceding part (c). Therefore each entry $X_{ij}$ of $X$ is represented by a pair of numbers that we can bound from above and below.
Once we have the estimate of $X$, we can compute an upper bound of $\mathrm{Tr}[X^p]$ stably and rigorously by interval arithmetic, which then gives a bound on $C_{opt}$ via \eqref{C*_bound}. In particular, we choose $p=4$ in our computation, and we can rigorously verify that $C_{opt}<1$.

\subsection{Weighted $H^1$ estimate}\label{sec:a1h1}
We choose
\beq\label{eq:wg_a1_H1}
\psi =  - \f{1}{\bar{\om}} \lt( \f{1}{x} - \f{x}{L^2} \rt) , \quad x \in [0, L],
\eeq
as the weight for the weighted $H^1$ estimate. Note that the weight $\psi$ is nonnegative for $0\leq x\leq L$, and is of order $x^{-2}$ near $x = 0$ and $O(1)$ near $x= L$.
We can perform the weighted $H^1$ estimate as follows
\beq\label{eq:lin_a1_H1}
\bal
\f{1}{2} \f{d}{dt} \la \om_x^2, \psi \ra & = \la - ( ( \bar{c}_l x + \bar{u} ) \om_{x} )_x +   ( ( \bar{c}_{\om} + \bar{u}_x ) \om )_x, \om_x \psi \ra
 + \la ( ( u_x + c_{\om} )  \bar{\om}  )_x, \om_x  \psi \ra  \\
 & - \la ( ( u + c_l x ) \bar{\om}_x)_x ,\om_x \psi\ra  + \la N(\om)_x, \om_x \psi \ra +
\la F(\bar{\om})_x , \om_x  \psi \ra   \\
& \teq I + II + III  + N_2 + F_2.
\eal
\eeq
For $I$, we use $\bar{c}_l = \bar{c}_{\om}$ and integration by parts to get
\beq\label{eq:a1_H1_D2}
\bal
I &= \B\la - (\bar{c}_l x + \bar{u}) \om_{xx} +\bar{u}_{xx} \om, \om_x \psi \B\ra \\
& = \B\la   \f{1}{2\psi} (  ( \bar{c}_l x +  \bar{u} ) \psi )_x  ,\om_x^2 \psi \B\ra + \la \bar{u}_{xx} \om, \om_x \psi \ra \teq \la D_2(\bar{\om}), \om_x^2 \psi\ra + \la \bar{u}_{xx} \om, \om_x \psi \ra .
\eal
 \eeq
The first term in $I$ is a damping term. We plot the numerical values of $D_2(\bar{\om})$ on the grid points in Figure \ref{DG_linear}. We can verify rigorously that it is bounded from above by $-0.95$. Thus we have
\beq\label{eq:a1_H1_damping}
I = \la D_2(\bar{\om}), \om_x^2 \psi\ra + \la \bar{u}_{xx} \om, \om_x \psi \ra \leq -0.95 \la \om_x^2, \psi \ra+ \la \bar{u}_{xx} \om, \om_x \psi \ra 
\teq -0.95 \la \om_x^2, \psi \ra+ I_2,
\eeq
where $I_2 = \la \bar{u}_{xx} \om, \om_x \psi \ra$. For $II, III$, we note that
\[
\bal
II + III &= \la u_{xx} \bar{\om} + (u_x + c_{\om}) \bar{\om}_x
 - (u_x + c_l) \bar{\om}_x - (c_l x + u) \bar{\om}_{xx}, \om_x \psi \ra   \\
 &= \la u_{xx} \bar{\om} ,\om_x \psi \ra  -\la  (c_l x + u) \bar{\om}_{xx}, \om_x \psi \ra
  \teq II_1 + II_2.
\eal
\]
Using the definition of $\psi$, we get
\[
II_1 = \la u_{xx} \bar{\om}, \om_x \psi \ra  =  \B\la  u_{xx} \om_x, - \f{1}{x} + \f{x}{L^2}   \B\ra .
\]
Since $\om_x(0) = 0$ by the normalization condition and $u_{xx}(0) = 0$ by the odd symmetry, we can use the same cancellation as we did in \eqref{eq:lin_a04} to get
\[
\B\la u_{xx} \om_x ,  -\f{1}{x} \B\ra = 0, \quad \la u_{xx} \om_x, x \ra =0.
\]
Therefore $II_1$ vanishes and we get
\beq\label{eq:a1_H1_rem1}
II  + III = II_2  =-\la  (c_l x + u) \bar{\om}_{xx}, \om_x \psi \ra,
\eeq
which is a cross term. In fact, after performing integration by parts, it becomes interaction among some lower order terms, i.e. of the order lower than $\om_x$ (e.g. $u, u_x, \om$).

\begin{remark}\label{rem:com_ass1}
So far, we have established all the delicate estimates of the linearized operator that exploit cancellations of various nonlocal terms. We have obtained the linear stability at the $L^2$ level and the linear stability estimates for the terms of the same order as $\om_x$, e.g. $u_{xx}$, in the weighted $H^1$ estimates after performing integration by parts. The remaining estimates \textit{do not} require specific structure of the equation. Suppose that we have a sequence of approximate steady states $\om_{h_i}$ with $h_i$ converging to $0$ that enjoy similar estimates and have approximation error $\la F(\om_{h_i})^2, \vp\ra + \la F(\om_{h_i})_x^2, \psi \ra$ of order $h_i^{\b}$ for some constant $\b>0$ independent of $h_i$, where $F(\om_{h_i})$ is defined similarly as that in \eqref{eq:NF_a1}. Then we can apply the above stability analysis to the profile $\om_{h}$ and the argument in Sections \ref{sec:nonsta_a},\ref{sec:conv_a0} to finish the remaining steps of the proof by choosing a sufficiently small $h = h_n$. Here $h$ plays a role similar to the small parameter $a$ in these sections. 
An important observation is that $h_n$ and the required approximation error to close the whole argument can be estimated effectively. Once we have determined $h_n$, we can construct the approximate steady state $\om_{h_n}$ numerically and verify whether $\om_{h_n}$ enjoys similar estimates and has the desired approximation error {\it a posteriori}.
\end{remark}

In the following discussion, we first give some rough bounds and show that the remaining terms can be bounded by the weighted $L^2$ or $H^1$ norm of $\om$ with constants depending continuously on $\bar{\om}$. This property implies that similar bounds will also hold true if we replace the approximate steady state $\bar{\om}$ by another profile $\hat{\om}$, if $\bar \om - \hat \om$ is sufficiently small in some energy norm.
We will provide other steps in the computer-assisted part of the paper later in this section.



 The remaining linear terms in the weighted $H^1$ estimate are $I_2 = \la \bar{u}_{xx} \om, \om_x \psi \ra $ in \eqref{eq:a1_H1_damping} and $II_2$ in \eqref{eq:a1_H1_rem1}. Denote $\rho = x^{-2} + (x-L)^{-2}$. Note that $u + c_l x |_{x=0, L} =0, c_{\om} = c_l$.  Applying integration by parts to the integral $ || u_x + c_l - \f{1}{2} ( u+ c_lx) / x ||_2^2 , || u_x + c_l - \f{1}{2} ( u + c_lx ) / (x - L) ||_2^2$ and using an argument similar to those in \eqref{eq:iden}, \eqref{eq:iden21}, we get
\[
|| (u+ c_l x )\rho^{1/2} ||_2^2 = \int_0^L  (u+ c_l x)^2 (\f{1}{x^2} + \f{1}{ (x-L)^2} ) dx \leq 8 \int_0^L (u_x + c_l)^2 dx .
\]
Using the $L^2$ isometry of the Hilbert transform, the identity $ \int_0^L u_x dx = u(L) = - L  \cdot c_l$ and expanding the square, we further obtain 
\[
 || (u+ c_l x )\rho^{1/2} ||_2^2 \leq 8 || \om||_2^2 + 8( 2 c_l \cdot u(L) + L c_l^2) 
 \leq  8 || \om||_2^2  \leq  8 || \vp^{-1} ||_{L^{\infty}}  \la \om^2, \vp \ra.
\]
Applying the Cauchy-Schwarz inequality, we can estimate $I_2, II_2$ as follows 
\[
\bal
|I_2| &= | \la  \bar{u}_{xx} \om, \om_x \psi  \ra |    \leq || \bar{u}_{xx} \psi^{1/2} \vp^{-1/2} ||_{L^{\infty}[0,L]} \la \om^2, \vp\ra^{1/2}  \la \om_x^2, \psi \ra^{1/2},  \\
|II_2|  &= |\la  (c_l x + u) \bar{\om}_{xx}, \om_x \psi |\ra  \leq 
  || \rho^{-1/2} \bar \om_{xx} \psi^{1/2} ||_{L^{\infty}[0,L]} 
  \la (c_l x + u)^2, \rho \ra^{1/2} \la \om_x^2, \psi \ra^{1/2} . 
\eal
\]
Hence, combining the above estimates, we yield
\beq\label{eq:a1_H1_rough}
|I_2| + |II_2| \leq C_1(\bar{\om}) \la \om^2, \vp\ra^{1/2} 
\la \om_x^2, \psi \ra^{1/2},  \\
\eeq
where 
\beq\label{eq:lin_a1_C1}
\bal
C_1(\bar{\om}) & \teq  || \bar{u}_{xx} \psi^{1/2} \vp^{-1/2} ||_{L^{\infty}[0,L]}
+  \sqrt{8} || \rho^{-1/2} \bar \om_{xx} \psi^{1/2} ||_{L^{\infty}[0,L]} || \vp^{-1}||^{1/2}_{L^{\infty}}
\eal
\eeq
and $\rho = x^{-2} + (x-L)^{-2}$.
From the definitions of $\vp, \psi$ \eqref{eq:wg_a1}, \eqref{eq:wg_a1_H1},  the quantities appeared in $C_1(\bar{\om})$ satisfy that 
\[
\bal
  &\vp^{-1} = O( ( x^{-4} + (x-L)^{-2} )^{-1} ), \quad  |\bar{u}_{xx} \psi^{1/2} \vp^{-1/2} | = O( | \bar{u}_{xx}  ( x^{-1} + (L-x)^{-1})^{-1} | ),  \\
 & | (x^{-2} + (x-L)^{-2} )^{-1/2} \bar{\om}_{xx} \psi^{1/2} | = O( |(1 + (x-L)^{-2}  )^{-1/2}  \bar{\om}_{xx} | ) .
 \eal
\]
In particular, these quantities are bounded for any $x \in [0, L]$ and thus $C_1(\bar{\om})$ is finite. 

Therefore, combining \eqref{eq:lin_a1_H1}, \eqref{eq:a1_H1_damping}, \eqref{eq:a1_H1_rem1} and \eqref{eq:a1_H1_rough}, we prove for any $\e >0$, 
\beq\label{eq:lin_a13}
\f{1}{2} \f{d}{dt}  \la \om_x^2, \psi \ra
\leq - 0.95 \la \om_x^2 ,\psi \ra + \e \la \om_x^2 ,\psi \ra + (4\e)^{-1}C_1(\bar{\om})^2 \la \om^2, \vp \ra
+  N_2 + F_2 ,
\eeq
From \eqref{eq:lin_a12} and \eqref{eq:lin_a13}, we can choose $\e , \mu > 0 $ and construct the energy $E(t)^2 = \la \om^2,\vp \ra + \mu \la \om_x^2, \psi \ra$ such that
\beq\label{eq:lin_a1_tot}
\f{d}{dt} E(t)^2 \leq - C(\mu, \e) E(t)^2 + N_1 + F_1 + \mu (N_2 + F_2),
\eeq
where $C(\mu ,\e) > 0$ depends on $\mu, \e$. For example, one can choose $\e = 0.65, \; \mu = 0.4 \e C_1(\bar{\om})^{-2}$ to obtain $C(\mu, \e) = 0.2$. We have now completed the weighted $L^2$ and $H^1$ estimates at the linear level.

\subsubsection{Nonlinear stability}\label{sec:non_a1}
Recall that $N, F$ are defined in \eqref{eq:NF_a1}, $N_1, F_1$ in \eqref{eq:lin_a11}, and $N_2, F_2$ in \eqref{eq:lin_a1_H1}. Since $c_l = c_{\om}$, a direct calculation yields
$ \pa_x N(\om) =  u_{xx} \om - (c_l x + u) \om_{xx} $.

Using integration by parts similar to that in \eqref{eq:lin_a1_D} and \eqref{eq:a1_H1_D2}, we obtain 
\[
\bal
& N_1  + \mu N_2 = \B\la \f{1}{2\vp}   (  ( {c}_l x + u  ) \vp)_x    +  (c_{\om} + u_x) , \om^2 \vp \B\ra
+ \mu \B\la \f{1}{2\psi} (  ( c_l x +  u ) \psi )_x , \om_x^2 \psi \B\ra 
+ \mu \la u_{xx}  \om, \om_x \psi \ra .\\
\eal
\]
Recall $E(t) = ( \la \om^2 ,\vp \ra + \mu \la \om_x^2 ,\psi \ra)^{1/2}$. We can estimate $ u_x, \om, u_{xx}$ as follows 
\[
\bal
||u_x||_{L^{\infty}} &\leq  2||u_x||_{L^2(\R^+)}^{1/2} ||u_{xx}||_{L^2(\R^+)}^{1/2} 
\leq 2 || \om ||_2^{1/2} || \om_{x}||_2^{1/2} \\
&\leq  2 \mu^{-1/4} || \vp^{-1} ||^{1/4}_{L^{\infty}}  || \psi^{-1} ||^{1/4}_{L^{\infty}} E(t), \\
  || \om||_{L^{\infty}} & \leq || \om_x||_{L^1}
 \leq \la \om_x^2, \psi \ra^{1/2}  || \psi^{-1} ||^{1/2}_{L^1[0,L]} 
 \leq  \mu^{-1/2}|| \psi^{-1} ||^{1/2}_{L^1[0,L]}  E(t) , \\
 || u_{xx} x^{-1}||_2 &\leq || \om_x x^{-1} ||_2 
 \leq \mu^{-1/2} || \psi^{-1/2} x^{-1} ||_{L^{\infty}}  E(t),
  \eal  
\]
where we have used \eqref{eq:commute3}, $\om_x(0) = 0$ and the $L^2$ isometry of the Hilbert transform to obtain the last estimate. Recall $c_l = c_{\om} = -u(L) / L$ \eqref{eq:DGnormal}. We have $c_l x + u |_{x= 0, L}= 0, |c_l |  = |c_{\om} | \leq || u_x||_{L^{\infty}}$ and
\[
\bal
 | c_l x + u | &\leq \min(|x|, |L-x|) \cdot || c_{\om} + u_x||_{L^{\infty}[0, L]}  \leq 
 2\min(|x|, |L-x|)  || u_x||_{\infty} .
 \eal
\]
For any $x \in [0, L]$, using the Leibniz rule, we derive
\[
\B|\f{((c_l x + u ) \vp  )_x }{\vp}\B|  
+ \B|\f{ ((c_l x + u ) \psi  )_x}{\psi} \B|  
\leq 2  ( 2+ \B|\B| ( |x| \wedge (L-x) ) ( \f{ |\vp_x|}{\vp} + \f{ |\psi_x|}{\psi} ) \B| \B|_{L^{\infty}} ) || u_x||_{L^{\infty} } \teq C_2(\bar{\om}) ||u_x||_{L^{\infty}}.
\]

Combining the above estimates, we prove
\beq\label{eq:lin_a142}
\bal
 N_1  + \mu N_2  &\leq  ( C_2(\bar \om) +2 ) || u_x ||_{L^{\infty}}
 (  \la \om^2,  \vp \ra + \mu \la \om_x^2, \psi \ra  )  \\
&+ \mu || x \psi^{1/2} ||_{L^{\infty}} || u_{xx} x^{-1}||_2 || \om ||_{L^{\infty}}  \la \om_x^2, \psi \ra^{1/2}  \leq  C_3(\bar \om ,\mu) E(t)^3, 
\eal
\eeq
where
\[
 C_3( \bar \om, \mu) =  2 \mu^{-1/4}  ( C_2(\bar \om) +2 ) || \vp^{-1} ||^{1/4}_{L^{\infty}}   
 || \psi^{-1} ||^{1/4}_{L^{\infty}}
 + \mu^{-1/2} || x \psi^{1/2}||_{L^{\infty}} || x^{-1} \psi^{-1/2}||_{L^{\infty}}  
|| \psi^{-1}||^{1/2}_{L^1[0,L]}.
 \]
 We remark that the above $L^{\infty}$ norms are taken over $[0, L]$. From the definition of $\vp, \psi$, it is not difficult to verify that $C_3(\bar \om, \mu) < +\infty$.

To estimate the error term, we use the Cauchy-Schwarz inequality 
\beq\label{eq:lin_a15}
\bal
  F_1 + \mu F_2 &= \la F(\bar{\om}), \om \vp \ra + 
\mu \la F(\bar{\om})_x, \om_x \psi\ra  \\
&\leq   ( \la F(\bar{\om})^2, \vp \ra + \mu \la F(\bar{\om})_x^2, \psi\ra  )^{1/2} E(t) \teq \mathrm{error}(\bar{\om}) E(t),
\eal
\eeq


\paragraph{\bf{Guideline for the remaining computer assisted steps}}
Recall the definition of $\vp,\psi$ in \eqref{eq:wg_a1} and \eqref{eq:wg_a1_H1}. From the weighted $L^2$ and $H^1$ estimates, and the estimates of the nonlinear terms, we see that the 
coefficients and constants, e.g. $D(\bar{\om})$ in \eqref{eq:lin_a1_D}, $C_1(\bar{\om})$ in \eqref{eq:lin_a1_C1} and $C_3(\bar{\om},\mu)$ in \eqref{eq:lin_a142}, depend continuously on $\bar{\om}$. Hence, for two different approximate steady states $\om_{h_1}, \om_{h_2}$ computed using different mesh $h_2 < h_1$, 
if $\om_{h_1} - \om_{h_2}$ is small in some norm, e.g. some weighted $L^2$ or $H^1$ norm, we expect that all of these estimates hold true for these two profiles with very similar coefficients and constants. At the same time, the residual error of the profile computed using the finer mesh $\mathrm{error}(\om_{h_2})$ can be much smaller than that of the coarse mesh 
$\mathrm{error}(\om_{h_1})$. In particular, if the numerical solution $\om_h$ exhibits convergence in a suitable norm as we refine the mesh size $h$, then we can obtain a sequence of approximate steady states that enjoy similar estimates with decreasing residual $\mathrm{error(\om_h)}$. See also the Remark \ref{rem:com_ass1}. From our numerical computation, we did observe such convergence 
of $\om_h$ computed using several meshes with decreasing mesh size $h$.
Using the estimates that we have established, we can obtain nonlinear estimate for each profile $\bar{\om}$ similar to \eqref{eq:boot_a0} 
\[
\f{1}{2} \f{d}{dt} E^2(t) \leq  -K_1(\bar{\om}) E^2(t) + K_2(\bar{\om}) E^3(t)  + \mathrm{error}(\bar{\om}) E(t),
\]
where $E(t)^2 = \la \om^2 , \vp \ra + \mu(\bar{\om}) \la \om_x^2, \psi \ra$ and the positive constants $K_1(\bar{\om}), K_2(\bar{\om}), \mu(\bar{\om}) $ depend continuously on $\bar{\om}$. From this inequality, we can estimate the size of $\mathrm{error}(\bar{\om})$ that is required to close the bootstrap argument. A sufficient condition is that there exists $ y > 0$ such that $-K_1(\bar{\om}) y^2+ K_2(\bar{\om}) y^3 + \mathrm{error}(\bar{\om}) y < 0$, which is equivalent to 
\beq\label{eq:boot_error}
4\cdot  \mathrm{error}(\bar{\om}) \cdot K_2(\bar{\om}) < K_1(\bar{\om})^2.
\eeq
Hence, we obtain a good estimate on $\mathrm{error}(\bar{\om})$ that is required to close the whole estimate.

In practice, we first compute an approximate steady state $\bar{\om}_h$ using a relatively coarse mesh, e.g. mesh size $h= L / 1000$ or $L / 2000$ (correspond to 1000 or 2000 grid points). Then we can perform all the weighted $L^2$, $H^1$ estimates and determine the weights $\vp, \psi$, the decomposition in the estimates and all the parameters in \eqref{eq:para_opt} to obtain the linear stability, and perform the nonlinear estimates. After we obtain these estimates, we can determine a upper bound of $\mathrm{error}(\bar{\om})$ using \eqref{eq:boot_error} and choose a finer mesh with mesh size $h_2$ to construct a profile $\bar \om_{h_2}$ with a residual error less than this upper bound. After we extend all the corresponding estimates to the profile $\bar{\om}_{h_2}$, we found that the corresponding constants and coefficients in the estimates are almost the same as those that we have obtained using $\bar \om_h$ constructed by a coarser mesh. Therefore, we can perform analysis on $\bar{\om}_{h_2}$ and close the whole argument.

In the Supplementary material \cite[Sections 2,4]{CHHSupp19}, we will provide much sharper estimates of the cross terms \eqref{eq:a1_H1_rough}, \eqref{eq:lin_a13} and the nonlinear terms \eqref{eq:lin_a142}. These sharper estimates provide an estimate of the upper bound of $\mathrm{error}(\bar \om)$ in \eqref{eq:boot_error} that is not too small. This enables us to choose a modest mesh to construct an approximate profile with a residual error less than this upper bound. In particular, we choose $h = 2.5 \cdot 10^{-5}$ and the computational cost of $\bar{\om}_h$ is affordable even for a personal laptop computer. The rigorous estimate for the residual error of this profile in the energy norm is established in the Supplementary material \cite[Section 3]{CHHSupp19}. More specifically, we can prove the following estimate, which improves the estimate given by \eqref{eq:lin_a13} significantly.

\begin{lem}\label{lem:cross}
The weighted $H^1$ estimate satisfies
\[
\f{1}{2} \f{d}{dt}  \la \om_x^2, \psi \ra  = I + II_2 + N_2 + F_2
\leq - 0.25 \la \om_x^2 ,\psi \ra + 7.5 \la \om^2, \vp \ra +  N_2 + F_2,
\]
where $I, II_2$ combine the damping and the cross terms and are defined in \eqref{eq:a1_H1_damping}, \eqref{eq:a1_H1_rem1}, respectively.
\end{lem}

These refinements are not necessary if one can construct an approximate profile with a much smaller residual error using a more powerful computer with probably $10-100$ times more grid points. With these refined estimates and the rigorous estimate of the residual error of $\bar \om_h$,
we choose $\mu = 0.02$ and bootstrap assumption $E(t) = \la \om^2, \vp \ra + \mu \la \om_x^2, \psi \ra < 5 \cdot 10^{-4}$ to complete the final bootstrap argument. We refer the reader to the Supplementary material \cite[Section 5]{CHHSupp19} for the detailed estimates in the bootstrap argument.

The remaining steps are the same as those in the proof of Theorem \ref{thm:blowup_a0}. Recall the weights $\vp$ \eqref{eq:wg_a1} and $\psi$ \eqref{eq:wg_a1_H1} in the weighted $L^2$ and $H^1$ estimates and the regularity of the approximate profile $\bar \om$ in Section \ref{sec:regular}. Note that $\vp$ is of order $O(x^{-4})$ near $x=0$ and $ O(( x-L)^{-2})$ near $x =  L$, and $\psi$ is of order $O(x^{-2})$ near $x=0$ and $ O(1)$ near $x =  L$. We can choose a small and odd initial perturbation $ \om$ supported in $[-L, L]$ with vanishing $\om_x(0) = 0$ such that $ \om$ restricted to $[0, L]$ satisfies $\om \in L^2(\vp), \om_{x} \in L^2(\psi) $ and $\om  + \bar \om \in C_c^{\infty} $. The bootstrap result implies that for all time $t>0$, the solution $\om(t)+\bar{\om}$, $c_l +\bar{c}_l = c_{\om}(t) + \bar{c}_{\om}$ remain close to $\bar{\om}, \bar{c}_{\om}  ( \bar c_{\om}< -0.69 )$, respectively. Moreover, in the Supplementary Material \cite[Section 6]{CHHSupp19}, we have established the following estimate
\[
\f{1}{2} \f{d}{dt} \la \om_t^2, \vp \ra \leq -0.15  \la \om_t^2, \vp \ra .
\]
Using this estimate and a convergence argument similar to that in Section \ref{sec:conv_a0}, we prove that the solution eventually converges to the self-similar profile $\om_{\infty}$ with scaling factors $c_{l, \infty} = c_{\om, \infty} < 0$. Since $ \g=-\f{c_{l,\infty}}{c_{\om,\infty}} = -1 < 0 $, the asymptotically self-similar singularity is expanding. Thus we obtain an expanding and asymptotically self-similar blowup of the original De Gregorio model with scaling exponent $ \g= -1$ in finite time.

\section{Finite Time Blowup for $C^{\alpha}$ Initial Data}\label{sec:alpha} 


In \cite{Elg17}, Elgindi and Jeong obtained the $C^{\alpha}$ self-similar solution $\om_{\al}$ of the Constantin-Lax-Majda equation
\[
c_l x \om_x = (c_{\om} + u_x) \om 
\]
for all $\al \in (0, 1]$, which reads 
\beq\label{eq:CLM_self}
\bal
w_{\al} & = - \f{ 2\sin\lt( \f{\al \pi}{2}   \rt)  \textrm{sgn} (x) |x|^{\al}   }{ 1 + 2 \cos\lt( \f{\al \pi}{2}   \rt) |x|^{\al}  + |x|^{2\al}   } , \quad u_{\al,x} =   \f{ 2(1 + \cos\lt( \f{\al \pi}{2}   \rt) |x|^{\al}   ) }{ 1 + 2 \cos\lt( \f{\al \pi}{2}   \rt) |x|^{\al} + |x|^{2\al}  } \; , \\
c_l &= \f{1}{\al},  \quad  c_{\om} = -1  , \\
\eal
\eeq
where $c_l, c_{\om}$ are the scaling parameters. 

In this section, we will use the above solutions to construct approximate self-similar solutions analytically and 
use the same method of analysis presented in Section \ref{sec:a0} to prove finite time asymptotically self-similar singularity for $C^{\alpha}$ initial data with small $\alpha$ on both the real line and on the circle. We will focus on solution of \eqref{eq:DG} with odd symmetry that is preserved during the evolution. In particular, we will construct odd approximate steady state and analyze the stability of odd perturbation around the approximate steady state.

\subsection{  Finite time blowup on $\R$ with $C_c^{\al}$ initial data} \label{sec:unif} 

 In this section, we prove Theorem \ref{thm0:unif}. Throughout the proof, we impose $|a \al| < 1$ and $\al < \f{1}{4}$. We choose the following weights in the stability analysis 
\beq\label{eq:wg_alpha}
 \vp_{\al} = - \f{1}{ \sgn(x) \om_{\al}} \f{ 1 + 2 \cos\lt( \f{\al \pi}{2}   \rt) |x|^{\al}  + |x|^{2\al}    }{ |x|^{1 + 2 \al}} , \quad   \psi_{\al} =  \f{1}{\al^2} \vp_{\al} x^2.
\eeq
We choose these weights so that the estimates of $\la\om^2, \vp_{\al}\ra$ and $\la \om_x^2, \psi_{\al} \ra$ are comparable in the energy estimates.

\subsubsection{Normalization Conditions and Approximate  Steady State}
The self-similar equation of DG model with parameter $a$ reads
\beq\label{eq:DG_selfeq}
(c_l x + a u) \om_x = (c_{\om} + u_x) \om \; .
\eeq
For any $a > 0, \al \in (0,1)$, we  construct $C^{\al}$ approximate self-similar profile of  \eqref{eq:DG_selfeq} below
\beq\label{eq:solu_calpha}
\bal
\om_{\al}, \quad u_{\al} , \quad  \bar{c}_{l, \al} &= \f{1}{\al} -  a u_{\al, x}(0) = \f{1}{\al} - 2a,  \quad  \bar{c}_{\om} = -1  . \\
\eal
\eeq
The only difference between the above solution and the $C^{\al}$ self similar solutions of CLM \eqref{eq:CLM_self}  is the $c_l$ term. The above solution satisfies \eqref{eq:DG_selfeq} up to an error
\beq\label{eq:error_alpha}
F_{\al}(\om_{\al})   = -(\bar{c}_l x - \f{1}{\al} x + au_{\al}) \om_{\al, x} =   - a ( u_{\al} - u_{\al,x}(0) x )  \om_{\al, x}  .
\eeq
Linearizing the dynamic rescaling equation \eqref{eqn-dyn-rescale} around the approximate self-similar profile in \eqref{eq:solu_calpha}, we obtain the following equation for the perturbation $\om, u, c_l, c_{\om}$
\beq\label{eq:lin_alpha}
\om_t +  ( \bar{c}_{l, \al} x +au_{\al} ) \om_x = ( \bar{c}_{\om} + u_{\al, x} ) \om  + ( u_x + c_{\om} )  \om_{\al} - ( a u + c_l x )  \om_{\al,x} +N(\om) + F_{\al}(\om_{\al})\; ,
\eeq
where the error term $F_{\al}(\om_{\al})$ is given in \eqref{eq:error_alpha} and the nonlinear part is given by
\[
N(\om) = ( c_{\om} + u_x ) \om - (c_l x + a u) \om_x.
\]

We choose the following normalization conditions for $c_l(t) , c_{\om}(t) $
\beq\label{eq:normal_al}
\bal
& c_l(t) =  - a u_x(t, 0) , \quad c_{\om}(t) = - u_x(t, 0) .
\eal
\eeq
Using  \eqref{eq:solu_calpha} and $u_{\al,x}(0)= 2$, we can rewrite the above conditions as
\beq\label{eq:normal_al2}
 c_l(t) + \bar{c}_l  = \f{1}{\al} - a (u_x(t, 0) + u_{\al,x}( 0) ),\quad c_{\om} + \bar{c}_{\om} = 1 - (u_x(t,0) + u_{\al,x}(0) ).
\eeq

\subsubsection{Estimate of the velocity, the self-similar solution and the error}
We introduce the notation
\beq\label{eq:tildeu}
\td{u} \teq u - u_x(0) x, \quad \td{u}_x = u_x - u_x(0) \;,
\eeq
and use the weighs defined in  \eqref{eq:wg_alpha}  to perform the $L^2, H^1$ estimates.

We first state some useful properties of the $C^{\al}$ approximate self-similar solution that we will use in our stability analysis.
\begin{lem}\label{lem:est_profile}
For $\al \in (0, 1]$, we have the following estimates for the self-similar solutions defined in \eqref{eq:CLM_self}.
(a)  Uniform estimates on the damping effect
\beq\label{eq:est_dp}
\bal
&\f{1}{2\vp_{\al}} \lt( \f{1}{\al} x   \vp_{\alpha} \rt)_x + ( \bar{c}_{\om} + u_{\alpha,x} )   =  -1/2\; , \\
& \f{1}{2\psi_{\al}} \lt( \f{1}{\al} x   \psi_{\al} \rt)_x + ( \bar{c}_{\om} + u_{\alpha,x} ) - \f{1}{\al}   =  -1/2\; , \\
&  \f{  ( u_{\al, xx} \psi_{\al} )_x }{2 \psi_{\al}}  x^2 = \f{ 4 \al^2 |x|^{\al} ( |x|^{\al} + \cos\lt( \f{\al \pi}{2}   \rt)     ) }  { ( 1 + 2 \cos\lt( \f{\al \pi}{2}   \rt) |x|^{\al}  + |x|^{2\al} )^2} \geq 0 \; .
\eal
\eeq

(b) Vorticity and velocity estimates:
\begin{align}
&\B|\B| \f{x w_{\al,x}}{w_{\al}}  \B| \B|_{\infty}  \les   \al  , \quad \B| \B| \f{x^2 w_{\al,xx}}{w_{\al}}  \B| \B|_{\infty} \lesssim \al , \quad  \B| \B| \f{ x^2 \om_{\al, xx} +x \om_{\al, x}  }{ \om_{\al}} \B| \B|_{\infty}  \lesssim \al^2  \; , \label{eq:est_w}\\
&\B|  \f{u_{\al}}{x} - u_{\al, x}(0) \B|   \lesssim  |x|^{\al} \wedge 1 ,\quad  \B|  \f{u_{\al}}{x} - u_{\al, x} \B|   \lesssim  \al ( |x|^{\al} \wedge 1 ) \; . \label{eq:est_v}
\end{align}

(c) Asymptotic estimates of $\vp_{\al} , \psi_{\al}$:
\beq\label{eq:est_asy}
\bal
&\vp_{\al}  \asymp \f{1}{\al}   ( |x|^{-1 - 3\al}  + |x |^{-1 + \al}) \; ,\\
&\psi_{\al}   = \f{1}{\al^2} x^2 \vp_{\al}  \asymp \f{1}{\al^3}   ( |x|^{1 - 3\al}  + |x |^{1 + \al})\; , \\
& \B| \B| \f{ x \psi_{\al, x}  }{\psi_{\al}} - 1  \B|\B|_{\infty}  \lesssim \al, \quad \B|\B| \f{ x \vp_{\al, x}  }{\vp_{\al}} + 1  \B| \B|_{\infty} \lesssim \al \; ,
\eal
\eeq
where $A \asymp B$ means that $A \leq CB$ and $B \leq CA $ for some universal constant $C$.

(d) The smallness of the weighted $L^2$ and $H^1$ errors:
\begin{align}
&\la F_{\al}(\om_{\al})^2, \vp_{\al} \ra \lesssim  a^2 \al^2, \quad \la(  F_{\al}(\om_{\al}))_x^2, \psi_{\al} \ra \lesssim  a^2 \al^2,  \label{eq:est_err}  \\
& \la ( |x|^{\al} \wedge 1 )^2 \om^2_{\al, x}, \psi_{\al} \ra \les 1 \label{eq:est_err2}.
\end{align}
\end{lem}


These estimates are elementary and we defer the proof to the Appendix \ref{app:profile}. 

\begin{remark}
We will use \eqref{eq:est_dp} to derive the damping terms in the weighted $L^2$ and $H^1$ estimates. Using \eqref{eq:est_w}, we gain a small factor $\al$ from the derivatives of $\om_{\al}$.
This enables us to show that the perturbation term $u\om_{\al, x}$ is small. Estimates \eqref{eq:est_asy} shows that $x\psi_{\al,x} / \psi_{\al}, \  x \vp_{\al, x} / \vp_{\al}$  are close to $1$ and $ -1$, respectively, which allows us to estimate $\vp_{\al,x}, \psi_{\al,x}$ effectively.
\end{remark}



\begin{lem}[$L^{\infty}$ estimate]\label{lem:vel_inf}
\begin{align}
 || u_x||_{\infty} &\lesssim  \la \om^2 ,\vp_{\al} \ra^{1/4}    \la \om_x^2 ,\psi_{\al} \ra^{1/4}  ,  \label{lem:vel_inf1} \\
 \B| \td{u}_x - \f{\td{u}}{x} \B|  &\lesssim  \al  \la \om_x^2, \psi_{\al} \ra^{1/2} |x^{\al}|  \wedge 1
\lesssim  \al  \la \om_x^2, \psi_{\al} \ra^{1/2} ,
 \label{lem:vel_inf2} \\
 | \om(x)| &\lesssim  \al  \la \om_x^2, \psi_{\al} \ra^{1/2} |x^{\al}|  \wedge 1,
 \label{lem:w_inf}
  \end{align}
where $\td{u} = u - u_x(0) x$.
\end{lem}


We also defer the proof to the Appendix \ref{app:profile}. \eqref{lem:vel_inf2} shows that we can gain a small factor $\al$ from $\td{u}_x - \f{\td{u}}{x}  = u_x - u / x$.

We use a strategy similar to that in the proof of Theorem \ref{thm:blowup_a0} to prove Theorem \ref{thm0:unif}.
The key step is establishing linear stability by taking advantage of the following:

(a) the stretching effect $\bar{c}_{l, \alpha} x \om_x  $ and the damping term $( \bar{c}_{\om} + u_{x, \al}) \om$ ;

(b) the cancellation \eqref{lem:vel2}, \eqref{lem:vel3} involving the vortex stretching term $ u_x \om_{\al}$;

(c) the smallness of the advection term $a u \om_{\al ,x }$ (see \eqref{eq:est_w}) by choosing $|a \al|$ to be sufficiently small .

To control the velocity $u$,  we need to use Lemma \ref{lem:vel} in the Appendix, which states some nice properties of the Hilbert transform for a H\"older continuous function.

\subsubsection{Linear Estimate}
We first perform the weighted $L^2$ estimate with respect to \eqref{eq:lin_alpha}. We proceed as follows
\beq\label{eq:lin_L2}
\bal
\f{1}{2} \f{d}{dt} \la  \om^2,  \vp_{\al} \ra  &= \la -   ( \bar{c}_{l, \al} x +au_{\al} ) \om_x  +  ( \bar{c}_{\om} + u_{\al, x} ) \om , \om \vp_{\al} \ra
+  \la ( u_x + c_{\om} )  \om_{\al}  ,  \om \vp_{\al} \ra \\
& -   \la ( a u + c_l x )  \om_{\al,x} ,\om \vp_{\al} \ra   + \la N(\om), \om \vp_{\al} \ra +  \la F_{\al}(\om_{\al}), \om \vp_{\al} \ra   \\
&  \teq I + II + III + N + F .
\eal
\eeq
For $I$, we use integration by parts,  \eqref{eq:est_dp} and $\bar{c}_{l, \al} = \f{1}{\al} - a u_{\al, x} (0)$  to get
\beq\label{eq:lin_alpha1_1}
\bal
I &= \B\la \f{1}{2\vp_{\al}} (  ( \bar{c}_{l,\alpha} x +a  u_{\alpha} ) \vp_{\alpha} )_x + ( \bar{c}_{\om} + u_{\alpha,x} )   ,\om^2 \vp_{\alpha} \B\ra  \\
 & =  -\f{1}{2}  \la \om^2, \vp_{\al} \ra + a \B\la  \f{1}{2\vp_{\al} }   ( (u_{\al} - u_{\al, x}(0) x)  \vp_{\al})_x , \om^2 \vp_{\al} \B\ra   .\\
 \eal
\eeq
For the second term, we use \eqref{eq:est_v} and \eqref{eq:est_asy} to yield
\[
\bal
&\B| \f{1}{2\vp_{\al} }   ( (u_{\al} - u_{\al, x}(0) x)  \vp_{\al})_x \B|
=\B| \f{1}{2} (u_{\al, x} - u_{\al,x}(0)) + \f{ u_{\al} - u_{\al, x}(0) x  }{x}  \f{ x\vp_{\al, x} }{2 \vp_{\al}
} \B| \\
=&\B| \f{1}{2} ( u_{\al, x} - \f{u_{\al}}{x}) + \f{ u_{\al} - u_{\al, x}(0) x  }{x}
 \lt( \f{ x\vp_{\al, x} }{2 \vp_{\al}} + \f{1}{2} \rt) \B| \les \al + 1 \cdot \al \les \al \; .
\eal
\]
Combining \eqref{eq:lin_alpha1_1} with the above estimate, we derive 
\beq\label{eq:lin_alpha1}
I \leq  -\f{1}{2}  \la \om^2 , \vp_{\al} \ra + C |a| \al  \la \om^2 , \vp_{\al} \ra  =  -  \lt( \f{1}{2}  - C |a| \al \rt)  \la \om^2 , \vp_{\al} \ra ,
\eeq
where $C >0$ is some universal constant.

Recall the definitions of $\vp_{\alpha}$ in \eqref{eq:wg_alpha}, $c_l = -a u_x(0), c_{\om} =-u_x(0)$ in \eqref{eq:normal_al} and $\td{u}, \td{u}_x$ in \eqref{eq:tildeu}. We have $c_l x + au = a\td{u}, c_{\om} + u_x = \td{u}_x$. For $II$, we use the cancellation \eqref{lem:vel2} and \eqref{lem:vel3} to get
\beq\label{eq:lin_alpha2}
\bal
II&  =  \la \td{u}_x  \om_{\al} ,\om  \vp_{\al}  \ra = -  \B\la \td{u}_x \om \cdot \sgn(x),  |x|^{-1-2\al}+  2 \cos\lt( \f{\al \pi}{2}   \rt)  ||x|^{-1 - \al} + |x|^{-1} \B\ra \\
 & \leq  -  \la \td{u}_x \om \cdot \sgn(x),  |x|^{-1} \ra = -\f{\pi}{2} u_x^2(0) \leq 0.
\eal
\eeq
For $III$, we have
\[
\bal
|III|& =  |    \la ( a u + c_l x )  \om_{\al,x} ,\om \vp_{\al} \ra |   = \B| a  \B\la \f{\td{u}}{x}  \f{\om_{\al,x} x} {\om_{\al}} ,  \om  \f{ 1 + 2 \cos\lt( \f{\al \pi}{2}   \rt) |x|^{\al}  + |x|^{2\al}    }{ |x|^{1 + 2 \al}}   \B\ra \B| \\
& \lesssim   |a|  \B\la \B|  \f{\td{u}}{x} \B|   \B|  \f{\om_{\al,x} x} {\om_{\al}} \B| , | \om|  ( |x|^{-1 -2 \al} + |x|^{-1} ) \B\ra.
\eal
\]
Using the estimate for $\om_{\al}$ \eqref{eq:est_w} and the Hardy inequality \eqref{eq:hd2}, we obtain
\beq\label{eq:lin_alpha3}
\bal
|  III |& \lesssim |a|\al   \B\la \B|  \f{\td{u}}{x} \B|  , | \om|  ( |x|^{-1 -2 \al} + |x|^{-1} ) \ra \\
&\lesssim  |a|  \al   \la \td{u}^{2}, |x|^{-3 - 3 \al }\ra^{1/2}   \la \om^2, |x|^{-1 -\al} \ra^{1/2} +  |a| \al  \la \td{u}^{2}, |x|^{-3 -  \al }\ra^{1/2}   \la \om^2, |x|^{-1 + \al} \ra^{1/2}  \\
& \lesssim |a| \al  \al^{-1}    \la \om^{2}, |x|^{-1 - 3 \al }\ra^{1/2}   \la \om^2, |x|^{-1 -\al}  \ra^{1/2}  + |a| \al \al^{-1} \la  \om^{2}, |x|^{-1 -  \al }\ra^{1/2}   \la \om^2, |x|^{-1 + \al} \ra^{1/2} \\
& \lesssim |a| \al  \la \om^2, \vp_{\al}\ra,
\eal
\eeq
where we have used \eqref{eq:est_asy} to obtain the last inequality.

Plugging \eqref{eq:lin_alpha1}, \eqref{eq:lin_alpha2} and \eqref{eq:lin_alpha3} in \eqref{eq:lin_L2}, we arrive at
\beq\label{eq:lin_alphaL2}
\f{1}{2} \f{d}{dt} \la \om^2 ,\vp_{\al} \ra  \leq -\lt( \f{1}{2} - C |a| \al  \rt) \la \om^2, \vp_{\al } \ra + 
\la N(\om) , \om \vp_{\al} \ra + \la F_{\al}(\om_{\al}) ,  \om  \vp_{\al} \ra .
\eeq

\subsubsection{Weighted $H^1$ Estimate}
Recall the definition of the weight $\psi_{\al}$ in  \eqref{eq:wg_alpha}. We now perform the weighted $H^1$ estimate with respect to \eqref{eq:lin_alpha}
\beq\label{eq:alpha_H1}
\bal
\f{1}{2} \f{d}{dt} \la \om_x^2, \psi_{\al} \ra & = \la - ( ( \bar{c}_{l, \al} x +au_{\al} ) \om_{x} )_x +   ( ( \bar{c}_{\om} + u_{\al, x} ) \om )_x, \ \om_x \psi_{\al} \ra \\
 &+ \la ( ( u_x + c_{\om} ) \om_{\al}  )_x,  \ \om_x  \psi_{\al} \ra  
  - \la ( (a  u + c_l x )  \om_{\al, x})_x ,\ \om_x \psi_{\al}\ra \\
  & + \la N(\om)_x. \om_x \psi_{\al} \ra +
\la F_{\al}(\om_{\al})_x , \  \om_x  \psi_{\al} \ra    \teq I + II + III  + N_2 + F_2.
\eal
\eeq

For $I$, we again use integration by parts to obtain
\[
\bal
I  &= \la - ( \bar{c}_{l, \al} x +au_{\al} ) \om_{xx}   - ( \bar{c}_{l, \al}  + a u_{\al, x}) \om_x  + ( \bar{c}_{\om} + u_{\al, x}  ) \om_x,  \om_x \psi_{\al } \ra
+ \la u_{\al, xx}  \om, \om_x \psi \ra  \\
& = \B\la  \f{    (  ( \bar{c}_{l, \al} x +au_{\al} ) \psi_{\al}     )_x      } {2 \psi_{\al}}    - ( \bar{c}_{l, \al}  + a u_{\al, x})  +   ( \bar{c}_{\om} + u_{\al, x}  ), \om_x^2 \psi_{\al}   \B\ra   - \B\la   \f{  ( u_{\al, xx} \psi_{\al} )_x }{2 \psi_{\al}} ,\om^2 \psi  \B\ra \teq I_1 + I_2.
\eal
\]
From \eqref{eq:est_dp}, we get
\[
I_2 \leq 0 .
\]
Recall that $\bar{c}_{l, \al} = \f{1}{\al}  - a  u_{\al, x}$ and
\[
\bar{c}_{l, \al} x + a u_{\al, x} = \f{1}{\al} + a (  u_{\al} -  u_{\al, x} x) = \f{1}{\al}  + a \td{u}_{\al},  \quad  \bar{c}_{l, \al} + a u_{\al; x} = \f{1}{\al} + \td{u}_{\al, x}.
\]
We can use \eqref{eq:est_dp} to obtain
\[
\bal
I_1 & =  - \f{1}{2} \la \om_x^2, \psi_{\al} \ra  + a \B\la  \f{  ( \td{u}_{\al} \psi_{\al}     )_x      } {2 \psi_{\al}}  - \td{u}_{\al, x}
, \om_x^2 \psi_{\al}  \B\ra  \\
 &=  - \f{1}{2} \la \om_x^2, \psi_{\al} \ra   + a \B\la  \f{ \td{u}_{\al}}{x}  \f{ x \psi_{\al,x}}{ 2\psi_{\al}}   - \f{ \td{u}_{\al, x} }{2}, ,\om_x^2\psi_{\al} \B\ra.
\eal
\]
Using \eqref{eq:est_v} and \eqref{eq:est_asy}, we get
\[
\B|   \f{ \td{u}_{\al}}{x}  \f{ x \psi_{\al,x}}{ 2\psi_{\al}}   - \f{ \td{u}_{\al, x} }{2}  \B|
\leq  \B|  \f{\td{u}_{\al} }{x}   \B|   \cdot  \B|     \f{ x \psi_{\al,x}}{ 2\psi_{\al}} - \f{1}{2} \B|  +   \f{1}{2}  \B|  \f{\td{u}_{\al}}{x}  - \td{u}_{\al, x}  \B|
\lesssim || u_{\al, x}||_{\infty} \al + \al \lesssim \al .
\]
It follows that
\begin{align}
I_1 & \leq    -\f{1}{2} \la \om_x^2, \psi_{\al} \ra + C  |a| \al  \la \om_x^2,  \psi_{\al} \ra  =  -\lt(  \f{1}{2} - C |a| \al  \rt)  \la \om_x^2 ,\psi_{\al} \ra , \notag  \\
\Rightarrow  I & = I_1 + I_2 \leq I_1  \leq  -\lt(  \f{1}{2} - C |a| \al  \rt)  \la \om_x^2 ,\psi_{\al} \ra \label{eq:lin_alpha4} .
\end{align}

For $II$, we have
\beq\label{eq:alpha_H1_II}
II = - \la u_{xx} \om_{\al},  \om_x  \psi_{\al} \ra  + \la \td{u}_x \om_{\al, x}, \om_x  \psi_{\al} \ra  \teq II_1 + II_2.
\eeq
Note that
\[
H(\om_x ) = u_{xx}, \  H(\om_x x)(0) = 0 \Rightarrow  H(\om_x x) = u_{xx } x.
\]
Moreover, we have that $\om_x x$ is odd, $(xu_{xx})(0) = (x \om_x)(0) = 0$. Applying the cancellation \eqref{lem:vel2}, \eqref{lem:vel3} with $(u_x, \om)$ replaced by $(x u_{xx}, x\om_x)$, we can estimate $II_1$ as follows
\beq\label{eq:lin_alpha51}
\bal
II_1  & = - \B\la (xu_{xx} ) (x  \om_x) ,  \f{1}{\al^2} \f{ (1 + 2 \cos\lt( \f{\al \pi}{2}   \rt) |x|^{\al}  + |x|^{2\al}  ) } { \sgn(x)  |x|^{1 + 2\al }  } \B\ra \\
& \leq - \B\la  (xu_{xx} ) (x  \om_x),  \f{1}{\sgn(x) \al^2 |x|} \B\ra  = - \f{\pi}{2\al^2} (x u_{xx})(0)^2  = 0  .
\eal
\eeq

The remaining terms in the weighted $H^1$ estimate are $II_2$ in \eqref{eq:alpha_H1_II} and $III$ in \eqref{eq:alpha_H1}, which can be decomposed as follows
\[
\bal
 III & =   - \la ( (a  u + c_l x )  \om_{\al, x})_x ,\om_x \psi_{\al}\ra = - a \la   \td{u}_{x} \om_{\al, x} +  \td{u} \om_{\al,xx}, \om_x \psi_{\al} \ra \\
& = - a \B\la     ( \td{u}_x - \f{\td{u}}{x})    \om_{\al, x}   , \om_x \psi_{\al}\B\ra -  a  \B\la   \f{\td{u}}{x}  (  \om_{\al, x} +  x\om_{\al, xx} ), \om_x \psi_{\al} \B\ra
\teq  III_1 + III_2.
\eal
\]

We perform the estimate of $III_1, III_2$ and the estimate of $II_2$ can be done similarly.  Using  the pointwise estimate \eqref{lem:vel_inf2} and the Cauchy Schwarz inequality, we can estimate $III_1$ as follows
\beq\label{eq:lin_alpha61}
\bal
III_1 & \leq   |a| \al \la  \om_x^2, \psi_{\al} \ra^{1/2}  \cdot \la ( |x|^{\al} \wedge 1 )  | \om_{\al, x}|,    |  \om_x | \psi_{\al} \ra  \\
& \lesssim |a| \al \la  \om_x^2, \psi_{\al} \ra   \cdot  \la ( |x|^{\al} \wedge 1 )^2 \om^2_{\al, x}, \psi_{\al} \ra^{1/2}  \lesssim |a| \al \la  \om_x^2, \psi_{\al} \ra,
\eal
\eeq
where we have used \eqref{eq:est_err2} to obtain the last inequality.

For $III_2$, we first use \eqref{eq:est_w} to obtain
\[
\bal
| III_2 | &= |a|  \B|    \B\la   \f{\td{u}}{x}  (  \om_{\al, x} +  x\om_{\al, xx} ), \om_x  |x|^{1- 2 \al}   \f{ (1 + 2 \cos\lt( \f{\al \pi}{2}   \rt) |x|^{\al}  + |x|^{2\al}  ) }{\al^2\sgn(x) \om_{\al}} \B\ra \B|  \\
& = |a|  \B|   \B\la   \f{\td{u}}{x^2}  \f{ x(  \om_{\al, x} +  x\om_{\al, xx} )}{ \sgn(x) \om_{\al}}, \om_x    |x|^{1- 2 \al}      \f{ (1 + 2 \cos\lt( \f{\al \pi}{2}   \rt) |x|^{\al}  + |x|^{2\al}  ) }{\al^2} \B\ra \B| \\
& \lesssim  |a|    \B\la \B|  \f{\td{u}}{x^2} \B|   \al^2  , | \om_x  |    \   |x|^{1- 2 \al}      \f{ (1 + 2 \cos\lt( \f{\al \pi}{2}   \rt) |x|^{\al}  + |x|^{2\al}  ) }{\al^2} \B\ra \\
&= |a|    \B\la \B|  \f{\td{u}}{ x}   \B|  , | \om_x   |\lt(1 +2 \cos\lt( \f{\al \pi}{2}   \rt)  |x|^{-\al} +  |x|^{-2 \al}\rt)  \B\ra  \lesssim
|a|    \B\la \B|  \f{\td{u}}{ x}   \B|  , | \om_x  | (1 + |x|^{-2 \al})  \B\ra.
\eal
\]
Then we use the Hardy inequality \eqref{eq:hd2} to estimate $\td{u}$
\[
\bal
|III_2|  &\lesssim |a|     \B\la \B|  \f{\td{u}}{ x}   \B|  , | \om_x  | (1 + |x|^{-2 \al})  \B\ra  \\
&\lesssim |a|   \la \td{u}^{2}, |x|^{-3 -3 \al} \ra^{1/2} \la \om_x^2, |x|^{1 - \al} \ra^{1/2}
+ |a|   \la \td{u}^{2}, |x|^{-3 - \al} \ra^{1/2} \la \om_x^2, |x|^{1 + \al} \ra^{1/2}  \\
& \lesssim |a|\al^{-1}   \la \om^{2}, |x|^{-1 -3 \al} \ra^{1/2} \la \om_x^2, |x|^{1 - \al} \ra^{1/2}
+ |a| \al^{-1}   \la \om^2, |x|^{-1 - \al} \ra^{1/2} \la \om_x^2, |x|^{1 + \al} \ra^{1/2} . \\
\eal
\]
Using \eqref{eq:est_asy}, we derive
\beq\label{eq:lin_alpha62}
|III_2|  \lesssim |a| \al^{-1} \al^{1/2} \la \om^2, \vp_{\al} \ra^{1/2} \al^{3/2}   \la \om_x^2, \psi_{\al}\ra^{1/2} = |a|\al  \la \om^2, \vp_{\al} \ra^{1/2}    \la \om_x^2, \psi_{\al}\ra^{1/2}.
\eeq
Similarly, for $II_2$,  using the smallness $ \B|\B| \f{x w_{\al,x}}{w_{\al}}  \B| \B|_{\infty}  \les   \al $ in \eqref{eq:est_w}, the weighted estimate \eqref{lem:vel1} and estimate of $\psi_{\al}, \vp_{\al}$ \eqref{eq:est_asy}, one can obtain 
\beq\label{eq:lin_alpha5}
| II_2| \les \la \om^2 ,\vp_{\al} \ra^{1/2}  \la \om_x^2, \psi_{\al} \ra^{1/2}.
\eeq

Plugging \eqref{eq:lin_alpha4}, \eqref{eq:alpha_H1_II}, \eqref{eq:lin_alpha51}, \eqref{eq:lin_alpha61}, \eqref{eq:lin_alpha62} and \eqref{eq:lin_alpha5} in \eqref{eq:alpha_H1}, we obtain 
\beq\label{eq:lin_alphaH1}
\bal
\f{1}{2} \f{d}{dt} \la \om_x^2, \psi_{\al} \ra & \leq -\lt( \f{1}{2} - C  |a| \al \rt) \la \om_x^2, \psi_{\al} \ra
+ C   \la \om^2, \vp_{\al} \ra^{1/2}    \la \om_x^2, \psi_{\al}\ra^{1/2}  \\
& +\la N(\om)_x , \om_x \psi_{\al} \ra + \la F_{\al}(\om_{\al})_x , \  \om_x  \psi_{\al} \ra ,  \\
\eal
\eeq
for some universal constant $C$, where we have used $|a \al| < 1$.

In the following two subsections, we aim to control the nonlinear and error terms
\[
 \la N(\om) , \om \vp_{\al} \ra ,  \ \la F_{\al}(\om_{\al}) ,  \om  \vp_{\al} \ra ,\ \la N(\om)_x, \om_x \psi_{\al} \ra ,  \la F_{\al}(\om_{\al})_x , \  \om_x  \psi_{\al} \ra   \\
\]
in \eqref{eq:lin_alphaL2} and \eqref{eq:lin_alphaH1}.

\subsubsection{Estimates of nonlinear terms}
Recall from \eqref{eq:normal_al} and \eqref{eq:tildeu} that
\[
c_l x + a u = a ( u - u_x(0) x) = a \td{u}, \quad c_{\om} + u_x =u_x - u_x(0) = \td{u}_x .
\]
For the nonlinear terms in \eqref{eq:lin_alphaL2} and \eqref{eq:lin_alphaH1}, we use integration by parts to obtain 
\[
\bal
\la N(\om), \om \vp_{\al} \ra  &= \la (c_{\om} + u_x  ) \om - (c_l x + a u) \om_x,  \om \vp_{\al} \ra   = \B\la   \td{u}_x  + \f{ (a \td{u} \vp_{\al}  )_{x} }{ 2 \vp_{\al}}, \om^2 \vp_{\al} \B\ra \\
& = \la \td{u}_x,\om^2 \vp_{\al} \ra  + \f{a}{2}\B\la  \lt(\td{u}_x  + \f{\td{u}}{x}  \f{ x \vp_{\al, x}}{\vp_{\al}} \rt), \om^2 \vp_{\al} \B\ra \teq I_1 + I_2 , \\
\la N(\om)_x, \om_x \psi_{\al} \ra & =
\la ( (c_{\om} + u_x  ) \om - (c_l x + a u) \om_x)_x ,\om_x \psi_{\al} \ra  \\
& = \la u_{xx} \om + \td{u}_x \om_x , \om_x \psi_{\al} \ra
-a \B\la  \td{u}_x \om_x + \td{u} \om_{xx}, \om_x \psi_{\al}   \B\ra \\
& =   \la \td{u}_x \om_x ,\om_x \psi_{\al} \ra +\la u_{xx} \om , \om_x \psi_{\al} \ra
+ a \B\la  - \td{u}_x + \f{ (\td{u} \psi_{\al})_x}{2\psi_{\al}} ,\om^2_x \psi_{\al}   \B\ra   \teq II_1 + II_2 + II_3 .\\
\eal
\]
For each term $I_i, II_j$, 
we use Lemma \ref{lem:vel_inf} to control the $L^{\infty}$ norm of $\om, \td{u}/x, \td{u}_x $ or $\td{u}_x - \td{u}/x$, and use $\la \om^2, \vp_{\al} \ra$, $\la \om_x^2, \psi_{\al} \ra$ to control other terms. We present the estimate of $II_3$ that has a large coefficient $a$ and is more complicated. Other terms can be estimated similarly. For $II_3$, we notice that
\[
\bal
 - \td{u}_x + \f{ (\td{u} \psi_{\al})_x}{2\psi_{\al}}  & = -\f{1}{2}  \td{u}_x + \f{1}{2} \f{\td{u}}{x} \f{ \psi_{\al,x} x}{\psi_{\al}}   = -\f{1}{2} \lt(  \td{u}_x - \f{\td{u}}{x}  \rt) + \f{1}{2} \f{\td{u}}{x} \lt(  \f{ \psi_{\al,x} x}{\psi_{\al}}  -1 \rt).  \\
\eal
\]

Then we use the $L^{\infty}$ estimate \eqref{lem:vel_inf2} to control $\td{u}_x - \td{u}/x$, \eqref{lem:vel_inf1} to control $\td{u} / x = u/x - u_x(0)$  and \eqref{eq:est_asy} to estimate the terms involving $\psi_{\al}$. This gives
\beq\label{eq:nlin_al32}
\bal
&II_3 
 =  \f{a}{2} \B\la - \lt(  \td{u}_x - \f{\td{u}}{x}  \rt) +  \f{\td{u}}{x} \lt(  \f{ \psi_{\al,x} x}{\psi_{\al}}  -1 \rt), \om_x^2 \psi_{\al} \B\ra \\
& \les |a| \lt( \B| \B|  \td{u}_x - \f{\td{u}}{x} \B|\B|_{L^{\infty}} + || u_x||_{\infty} \B| \B|   \f{ \psi_{\al,x} x}{\psi_{\al}}  -1\B|\B|_{L^{\infty}}  \rt)\la \om_x^2 ,\psi_{\al} \ra \\
& \lesssim ( |a|\al   \la \om_x^2, \psi_{\al}\ra^{1/2} + |a|\al \la \om^2, \vp_{\al}\ra^{1/4} \la \om_x^2, \psi_{\al} \ra^{1/4} )  \la \om_x^2 ,\psi_{\al} \ra \les (  \la \om^2, \vp_{\al} \ra + \la \om_x^2 ,\psi_{\al}\ra )^{3/2},
\eal
\eeq
where we have used $|a \al| < 1$. Similarly, we have 
\beq\label{eq:nlin_alother}
I_1, I_2, II_1, II_2 \les (  \la \om^2, \vp_{\al} \ra + \la \om_x^2 ,\psi_{\al}\ra )^{3/2}.
\eeq
Combining \eqref{eq:nlin_al32} and \eqref{eq:nlin_alother}, we obtain the following estimates for the nonlinear terms
\beq\label{eq:nlin_al}
\bal
\la N(\om), \om \vp_{\al} \ra  &= I_1 + I_2 \lesssim   (  \la \om^2, \vp_{\al} \ra + \la \om_x^2 ,\psi_{\al}\ra )^{3/2} ,\\
\la N(\om)_x, \om_x \psi_{\al}\ra & = II_1 + II_2 +II_3  \les (  \la \om^2, \vp_{\al} \ra + \la \om_x^2 ,\psi_{\al}\ra )^{3/2} .
\eal
\eeq

\subsubsection{Estimates of the error terms}
Recall the error terms in the weighted $L^2$, $H^1$ estimates in \eqref{eq:lin_alphaL2} and \eqref{eq:lin_alphaH1} are given by
\[
\la F_{\al}(\om_{\al}), \om \vp_{\al} \ra, \quad \la (F_{\al}(\om_{\al}))_x, \om_x \psi_{\al} \ra.
\]
Using the Cauchy-Schwarz inequality and the error estimate \eqref{eq:est_err}, we obtain
\beq\label{eq:err_al}
\bal
\la F_{\al}(\om_{\al}), \om \vp_{\al} \ra & \leq \la F_{\al}(\om_{\al})^2, \vp_{\al} \ra^{1/2}
\la \om^2, \vp_{\al} \ra^{1/2} \lesssim |a| \al \la \om^2, \vp_{\al} \ra^{1/2} , \\
\la (F_{\al}(\om_{\al}) )_x, \om_x \psi_{\al} \ra & \leq \la ( F_{\al}(\om_{\al}))_x^2, \psi_{\al} \ra^{1/2}
\la \om_x^2, \psi_{\al} \ra^{1/2} \lesssim   |a| \al \la \om_x^2, \psi_{\al} \ra^{1/2} .\\
\eal
\eeq

\subsubsection{ Nonlinear stability and convergence to self-similar solution }

Now, we combine the weighted $L^2$, $H^1$ estimates \eqref{eq:lin_alphaL2}, \eqref{eq:lin_alphaH1}, the estimates of nonlinear terms \eqref{eq:nlin_al} and error terms \eqref{eq:err_al}. Using these estimates and an argument similar to that in the analysis of nonlinear stability in Section \ref{sec:nonsta_a}
, we can choose an absolute constant $ 0 < \mu$ such that the following energy 
\[
E^2(t) \teq \la \om^2, \vp_{\al} \ra + \mu \la \om_x^2, \psi_{\al} \ra
\]
satisfies the differential inequality 
\beq\label{eq:boot_al2}
\f{1}{2} \f{d}{dt} E^2(t) \leq  - \lt( \f{3}{8} - C |a| \al \rt) E^2(t) + C |a| \al  E(t) +  C E^3(t),
\eeq
where $C > 0$ is an absolute constant. From \eqref{eq:est_asy}, we have 
\[
\bal
&|u_x(0)| \les \int | \om_x(y)|  |\log(y)| dy \les (\int \om_x^2 \psi_{\al})^{1/2} ( \int \psi_{\al}^{-1} |\log y|^2)^{1/2}  \\
 \les &E(t)  \lt(\al^3 \int | \log y |^2 (|y|^{1+\al} + |y|^{1- 3\al})^{-1}  dy \rt)^{1/2} 
\les E(t) (\al^3 \al^{-1})^{1/2} \les \al E(t).
\eal
\]
The normalization condition \eqref{eq:normal_al} implies 
\beq\label{eq:boot_cw}
| c_{\om} (t)| = |u_x(t, 0 )| \leq C_3\al E(t) , \quad |c_l(t)| = |a u_x(0)|\leq  C_3 |a| \al E(t) ,
\eeq
for some absolute constant $C_3 > 0$.

Finally, we perform the bootstrap argument. We first choose $C_2 = 16 C$, where $C$ is defined in \eqref{eq:boot_al2}, and then choose $C_1$ small such that
\beq\label{eq:boot_al3}
\bal
C_2  = 16 C ,  \quad   CC_1 + C C_2C_1 + C_3 C_2 C_1 + C_1 <1/16 . 
\eal
\eeq
Using the bootstrap argument and an argument similar to that in \eqref{eq:boot_a0}, we obtain that if
\[
|a| \al < C_1  , \quad E(0) < C_2 |a| \al ,
\]
then  we must have $E(t) < C_2 |a| \al < C_2 C_1$. This is because the right hand side of \eqref{eq:boot_al2} at $E(t) = C_2 |a| \al$ is given by
\[
\bal
&E(t)^2 \lt( -  \f{3}{8} + C |a| \al + \f{C |a|\al}{E(t)} + C  E(t)   \rt)  = E(t)^2 \lt( - \f{3}{8} + (C+CC_2) |a| \al + \f{C }{C_2}   \rt)  
<  -\f{1}{4} E^2(t)< 0.
\eal
\]
Finally, we show that $c_{\om} + \bar{c}_{\om} < 0$. The bootstrap results, \eqref{eq:boot_cw} and \eqref{eq:boot_al3} imply that
\[
|c_{\om}(t)| , |c_l(t)| < C_3 E(t) < C_3 C_2 C_1 < \f{1}{16} \; .
\]
It follows 
\beq\label{eq:scaling_al}
\bal
\bar{c}_{\om} + c_{\om}(t)  =-1 + c_{\om}(t) < - \f{1}{2},   \quad
\bar{c}_l + c_l(t) = \f{1}{\al} - 2 a + c_l(t) > \f{1}{2\al} - \f{1}{16} \geq \f{1}{4 \al} .
\eal
\eeq
As a result, we can choose small initial perturbation $\om_0$ with $E(0) < C_2 |a| \al$.
Moreover, we choose $\om_0$ in such a way that it modifies $\om_{\al}$ at the far field to make the initial data $\bar{\om} + \om_0$ have compact support. The bootstrap result and $\bar{c}_{\om} + c_{\om}(\tau) < -1/2 < 0$ imply that the physical solution blows up in finite time.

\subsubsection{Convergence to the steady state}
Using the same argument as that in the analysis of the case of small $|a|$ in Section \ref{sec:conv_a0} and the a-priori estimate \eqref{eq:scaling_al}, we can prove that there exists
\[
c_{l, \infty} \geq (4\al)^{-1}, \ c_{\om, \infty} <-1/2, \ \om_{ \infty} \in H^1(\psi_{\al}) \cap L^2(\vp_{\al}), \ u_{\infty, x} = H\om_{\infty},
\]
which satisfy the self-similar equation \eqref{eq:DG_selfeq} in $L^2(\vp_{\al})$ and in the dynamic rescaling equation, $\om(t) + \om_{\al}$ converges to $\om_{\infty}$ exponentially fast in $L^2(\vp_{\al})$.  Therefore, the solution of \eqref{eq:DG} develops a focusing ($c_{l,\infty} > 0$) and asymptotically self-similar singularity in finite time.

\subsection{Finite Time Blowup on Circle }\label{sec:circle}
In this subsection, we consider the De Gregorio model on $S^1$
\beq\label{eq:DG_S1}
\bal
\om_t + a u \om_x& = u_x \om \quad  x \in  [-\pi /2, \pi /2] \; , \quad
       u_x =  H_c \om \; ,
\eal
\eeq
where $\om, u $ are $\pi$-periodic and $H_c$ is the Hilbert transform on the circle
\beq\label{eq:hil_circ}
u_x = H_c \om = \f{1}{\pi} \int_{-\pi/2}^{\pi/2} \om(y) \cot( x- y) dy.
\eeq



Our goal is to prove Theorem \ref{thm:circle}. The proof is based on the comparison of the Hilbert transform on the real line and on $S^1$,  and on the control of the support of the vorticity $\om$. If the asymptotically self-similar blowup on $\R$ from compactly supported initial data is focusing, we can show that the support of the solution at the blow-up time remains finite.
Moreover, we show that the difference between the velocities generated by different Hilbert transforms in the support of $\om$ can be arbitrarily small by choosing initial data with small support. Therefore, the blowup mechanism of \eqref{eq:DG} on the real line applies to \eqref{eq:DG} on the circle.

We focus on the $C^{\al}$ case, i.e. case (2) in Theorem \ref{thm:circle}. The proof of the other case for small $|a|$ is similar and simpler.

\subsubsection{Dynamical Rescaling}
We consider the following dynamic rescaling of \eqref{eq:DG_S1}
\[
\bal
\Om(x, \tau) &= C_{\om}(\tau) \om( C_l(\tau) x, t(\tau)) , \quad U_x(x, \tau) = C_{\om}(\tau) u_x(C_l(\tau) x, t(\tau)) . \\
 \eal
\]
Denote by $S(\tau)$ the size of support of $\Om(\cdot , t(\tau))$, i.e. $\supp(\Om) = [-S(\tau), S(\tau)]$. This is equivalent to assuming that the size of $supp(\omega)$ is $C_l(\tau)S(\tau)$. We will choose $C_l(0)S(0) $ to be small
and show that $C_l(\tau) S(\tau)$ remains small up to the blowup time.
We have
\beq\label{eq:dyn_hil}
\bal
U_x( x, \tau) &=C_{\om}(\tau) u_x(C_l x, t(\tau)) = \f{1}{\pi} C_{\om}(\tau) \int_{-\pi/2}^{\pi/2} \om(y, t(\tau)) \cot( C_l(\tau)x - y) dy   \\
& = \f{1}{\pi} C_{\om}(\tau) \int_{- C_l(\tau) S(\tau)}^{ C_l(\tau) S(\tau)} \om(y, t(\tau)) \cot( C_l(\tau)x - y) dy   \\
& = \f{C_{\om}(\tau)}{\pi} \int_{ -  S(\tau) }^{   S(\tau)} \om(C_l y ,t(\tau) )  \cot(C_l(\tau)x - C_l(\tau) y) C_l(\tau) dy \\
& = \f{1}{\pi} \int_{ - S(\tau)}^{ S(\tau)} \Om(y ,\tau) \cot(C_l(\tau)x - C_l(\tau) y) C_l(\tau) dy
\teq H_{\tau} \Om (x).
\eal
\eeq
We introduce the time-dependent Hilbert transform $H_{\tau}$. The corresponding $U$ is given by
\beq\label{eq:hilu}
\bal
U(x, \tau) &= \int_0^x U_x( y, \tau) dy = \f{1}{\pi} \int_{-S(\tau)}^{S(\tau)}
\Om(y) \log| \sin( C_l(\tau) x - C_l(\tau) y)| dy  \\
& =  \f{1}{\pi} \int_{0}^{S(\tau)}
\Om(y) \log\B| \f{ \sin( C_l(\tau) x - C_l(\tau) y) }{  \sin( C_l(\tau) x + C_l(\tau) y) }\B| dy . \\
\eal
\eeq

With this notation, we can formulate the dynamic rescaling equation below
\beq\label{eq:DGdy_S1}
\bal
\Om_{\tau} + (c_l x + a U ) \Om_x  &= (c_{\om} + U_x) \Om , \\
 U_x &= H_{\tau} \Om.
\eal
\eeq

To simplify our notations, we still denote $\Om, U, \tau$ in the dynamic rescaling space by $\om, u, t$ i.e.
\[
 (\Om, U, \tau) \to (\om, u, \ t).
\]

\subsubsection{The bootstrap assumption}
We make the following bootstrap assumption.

(a) Support of $\om$ in the physical space : For all $t > 0$ we have
\beq\label{eq:boot_asm1}
C_l(t) S(t) < \f{\pi}{4}.
\eeq

(b) 
Boundedness of the solution: 
Let $\vp_{\al}, \psi_{\al}$ be the weights in \eqref{eq:wg_alpha}. We assume
\beq\label{eq:boot_asm2}
\bal
&\la \om^2, |y|^{-1 - \al} +  |y|^{-1 + \al} \ra +  \la \om_x^2, |y|^{1 - \al} + |y|^{1 + \al} \ra \\
< & 10 \la \om_{\al}^2, |y|^{-1 - \al} +  |y|^{-1 + \al} \ra +  10 \la \om_{\al,x}^2, |y|^{1 - \al} + |y|^{1 + \al} \ra + \al^{-2}  + 1 ,\\
 &| c_{\om}(t) + 1 | <   \f{1}{2} ,\quad  |c_l(t) -\f{1}{\al}| < \f{1}{2\al} ,
 \eal
\eeq
where $c_l, c_{\om}, \om$ is the solution of \eqref{eq:DGdy_S1}. We remark that we do not require smallness of $\om, \om_x$ in the assumption.

\subsubsection{Control of the support}

We choose the same weights $\vp_{\al}, \psi_{\al}$ as in \eqref{eq:wg_alpha} for later energy estimate. 
The evolution of the support of $\Om$ in \eqref{eq:DGdy_S1}, i.e. $S(t)$, is given by 
\beq\label{eq:S1_gw0}
\f{d}{d t} S(t)  = c_l(t) S(t) +  a U(S(t), t).
\eeq
Firstly, we show that $U$ has a sublinear growth if $\la \om^2 , \vp_{\al}\ra$ is bounded. Using \eqref{eq:hilu} and the Cauchy-Schwarz inequality, we get
\beq\label{eq:S1_gwu1}
\bal
| U(S(t)|&\lesssim  \la \om^2, |y|^{-1+\al} \ra^{1/2}
\lt(\int_0^{S(t)}  |y|^{1 - \al} \lt(\log\B| \f{ \sin( C_l(t) S(t) - C_l(t) y) }{  \sin( C_l(t) S(t) + C_l(t) y) }\B| \rt)^2 dy  \rt)^{1/2} .
\eal
\eeq
Since $ 0 <  y  < S(t) $ and $(|y| + S(t)) C_l(t) < \pi/2$ \eqref{eq:boot_asm1}, we can use
\[
\f{2}{\pi} x \leq \sin(x) \leq x , \quad x \in [0, \pi/2]
\]
to obtain for any $y \in [0, S(t)]$
\[
\bal
\lt(\log\B| \f{ \sin( C_l(t) S(t) - C_l(t) y) }{  \sin( C_l(t) S(t) + C_l(t) y) }\B| \rt)^2 \lesssim
1 +  \lt( \log \B| \f{ C_l(t) (S(t) - y)}{ C_l(t) (S(t) + y )} \B| \rt)^2  =1 +  \lt( \log \B| \f{ S(t) - y}{S(t) + y} \B| \rt)^2   .
\eal
\]
Substituting the above estimate in the integral in \eqref{eq:S1_gwu1}, we obtain
\[
\bal
| U(S(t)|  &\les  \la \om^2, |y|^{-1+\al} \ra^{1/2}
\lt( \int_0^{S(t)}  |y|^{1 - \al}   \lt( 1 +  \lt( \log \B| \f{ S(t) - y}{S(t) + y} \B| \rt)^2 \rt) dy \rt)^{1/2} \\ 
& \les \la \om^2, |y|^{-1+\al} \ra^{1/2}
\lt( |S(t)|^{2 - \al} \int_0^1 |z|^{1- \al} \lt( 1 + \lt(\log \B| \f{1 - z}{1 + z} \B| \rt)^2\rt) dz
 \rt)^{1/2}  ,
\eal
\]
where we have used the change of variable $y = S(t) z$ to get the second inequality. 
Using the above estimate and \eqref{eq:boot_asm2}, we obtain 
\[
| U(S(t)| \les \la \om^2, |y|^{-1+\al} \ra^{1/2} S(t)^{1-\al  /2}
\lesssim_{\al} S(t)^{1 - \al/2} .
\]
Substituting the above estimate in \eqref{eq:S1_gw0}, we yield
\beq\label{eq:S1_gw2}
\f{d}{dt} S(t) \leq c_l(t) S(t) + C(a, \al) S(t)^{1- \al/2},
\eeq
where the constant $C(a,\al)$ only depends on $a, \al$. Recall
\[
C_{l}(t) = C_l(0)   \exp\lt( -\int_0^t c_l(s) ds\rt).
\]
Denote $P(t) \teq C_l(t )S(t) $. \eqref{eq:S1_gw2} implies the following differential inequality
\beq\label{eq:circle_P}
\bal
  \f{d}{dt} P(t) &= \f{d}{dt} (C_l(t )S(t) ) \leq C(a, \al) C_l(t)^{ \al / 2}  (C_l(t )S(t) )^{1 - \al/2}  \\
&= C(a,\al) C_l(t)^{ \al / 2} P(t)^{1 -\al/2}  .
\eal
\eeq
Using the bootstrap assumption $c_l(t) > \f{1}{2\al}$ \eqref{eq:boot_asm2}, we have
$C_l(t) \leq C_l(0) e^{ -\f{t}{ 2\al}}.$ From this estimate and \eqref{eq:circle_P}, we further obtain 
\[
 \f{d}{dt} P(t)^{\al/2} \leq C(a,\al) C_l(t)^{ \al / 2} \leq C(a,\al) C_l(0)^{\al/2}\exp\lt( -\f{t}{4}\rt) 
 \; ,
 \]
 which implies
 \[
  P(t)^{\al / 2}  \leq P(0)^{\al/2} + C(a,\al) C_l(0)^{\al/2} \int_0^t  \exp\lt(-\f{s}{4}\rt) ds
< P(0)^{\al/2} + C(a,\al) C_l(0)^{\al/2} ,
\]
where the $C(a,\al)$ only depends on $a,\al$ and may vary from line to line. Recall $P(0) = C_l(0)S(0)$. As a result of the above estimate, we obtain 
\beq\label{eq:S1_gw3}
P(t)^{\al/2} \leq ( 1 + C(a,\al) S(0)^{-\al/2} )P(0)^{\al/2} \Rightarrow P(t) \leq C(a,\al, S(0)) P(0)  ,
\eeq
where the constant $C(a,\al, S(0))$ depends on $a,\al$ and $S(0)$. 

\subsubsection{Comparison between different Hilbert transforms}
\begin{lem}[Comparison of Hilbert transforms]\label{lem:hil}
With the bootstrap assumptions \eqref{eq:boot_asm1} and \eqref{eq:boot_asm2}, for $|x| \leq S(t)$, the difference between $H_{t}$ \eqref{eq:dyn_hil} on the circle and the Hilbert transform on the real line $H$ can be controlled by
\beq\label{eq:dhil}
\bal
 | ( H_{t} \om )(x) - H \om(x) | & \lesssim_{\al}  C_l(t) S(t) \; , \\
  | x ( H_{t} \om_x)(x) - x (H \om_x)(x) |  &\lesssim_{\al}  C_l(t) S(t) \; .
 \eal
\eeq
\end{lem}

\begin{remark}
We only care about $x$ in the support of $\Om$ since for $x$ outside the support of $\Om$, $U(x)$ does not enter the equation \eqref{eq:DGdy_S1}.
\end{remark}

\begin{proof}
 It suffices to consider $x \in [0, S(t)]$ due to the symmetry. We only prove the second inequality in \eqref{eq:dhil} and the first one can be proved similarly. Firstly, from \eqref{eq:dyn_hil}, we have
\beq\label{eq:dhil_pf}
|x (H_{t} \om_x ) (x) -x (H \om_x)(x) | =\B|\f{ x }{\pi} \int_{ - S(t)}^{ S(t)} \om_x(y ,t)\lt( \cot(C_l(t)x - C_l(t) y) C_l(t)   -  \f{1}{x - y} \rt) dy \B|. 
\eeq
The bootstrap assumption \eqref{eq:boot_asm1} shows that $ | C_l (x- y)| \leq \f{\pi}{2}$ for $|x|, |y| \leq S(t)$. Using the elementary inequality
$\B| \f{1}{z} - \cot z \B| \lesssim  \min( |z| , 1) \les 1 , \  \forall |z| \lesssim \f{\pi}{2}$, we obtain 
\[
\B|  \cot(C_l(t)x - C_l(t) y) C_l(t)   -  \f{1}{x - y} \B| 
= C_l(t) \B| \cot(C_l(t)x - C_l(t) y) -\f{1}{C_l(t)(x-y)} \B|\les C_l(t).
\]
Using the Cauchy-Schwarz inequality, we can estimate \eqref{eq:dhil_pf} as follows
\[
\bal
&|x (H_{t} \om_x ) (x) -x (H \om_x)(x) | \les 
 C_l(t)|x| \int_{-S(t)}^{S(t)}  | \om_x (y, t)|  dy \\
 \leq & C_l(t) |x| \la \om_x^2, |y|^{1-\al}+ |y|^{1+\al} \ra^{1/2} \lt( \int_R  \f{1}{ |y|^{1+\al} + |y|^{1-\al} } dy \rt)^{1/2} \; .\\
\eal
\]
Using $|x|\leq S(t)$ and \eqref{eq:boot_asm2}, we yield
\[
|x (H_{t} \om_x ) (x) -x (H \om_x)(x) | \les_{\al}  C_l(t) S(t).
\]
 \end{proof}

\subsubsection{Finite time blowup}
Recall that for compactly supported solution $\om(x, \tau)$ with support size $S(\tau)<+\infty$ in the dynamic rescaling equation \eqref{eq:DGdy_S1}, it corresponds to a solution $\om_{phy}$ at time $t(\tau)$ in the physical space \eqref{eq:DG_S1} via
\[
\bal
&\om_{phy}(x, t(\tau)) = C_{\om}(\tau)^{-1} \om( C_l(\tau)^{-1} x, \tau) ,  \\
& C_l(\tau) = C_l(0) \exp\lt( -\int_0^{\tau} c_l(s) ds  \rt) ,\quad t(\tau) = \int_0^{\tau} \exp \lt( \int_0^s c_{\om}(r) dr \rt) ds .
\eal
\]
See the discussion in Section \ref{sec:dsform}.
By abusing the notation, we still use $t$ as the time variable in the dynamic rescaling equation. 
We can rewrite \eqref{eq:DGdy_S1} as follows
\beq\label{eq:DGdy_S1m1}
\bal
\om_{t} + (c_l x + a u ) \om_x  &= (c_{\om} + u_x) \om  +  ( ( H_{t} \om)(x) - (H\om)(x) ) \om + a ( (I\om)(x) -  (I_{t}\om)(x) ) x\om_x
 \\
 u_x &= H \om
\eal
\eeq
where $u = x (I\om)(x)$ and the operator $I_{t} \om, I\om$ are
\[
\bal
(I_{t}\om)(x) & =  \f{1}{x} \int_0^x  (H_{t} \om)(y) dy  , \quad (I\om)(x)& = \f{1}{x} \int_0^x (H\om)(y) dy ,\\
\eal
\]
i.e. $1/x$ times the velocity generated by different Hilbert transforms. 
We choose the following normalization condition
\beq\label{eq:normal_circle}
c_l(t) = \f{1}{\al} -  a ( H_{t} \om(t, \cdot))(0) , \quad c_{\om}(t) =  1-  ( H_{t} \om(t, \cdot))(0) .%
\eeq
The difference between \eqref{eq:normal_al} and the above condition is the Hilbert transform, which can be bounded by \eqref{eq:dhil}. 

For the difference of the Hilbert transform in \eqref{eq:DGdy_S1m1}, we use \eqref{eq:dhil} to obtain the pointwise estimate of $H_t \om - H\om$ and $x (H_t \om - H \om)_x$. Similarly, we have the pointwise estimate of $I\om(x) - I_t \om(x)$ for all  $ |x| \leq S(t)$
\beq\label{eq:hil2}
\bal
| (I\om)(x) -  (I_{t}\om)(x)| & \leq \sup_{ |y| \leq |x|} | ( H_{t} \om )(x) - H \om(x) | \les_{\al}   C_l(t) S(t) \; ,  \\
 | x (I \om - I_t \om )_x(x)  |  &\leq  |  (x (I \om - I_t \om ))_x| + | (I \om - I_t \om) (x) |  \\
  &= | (( H_{t} \om )(x) - H \om(x)   |  +|(I \om - I_t \om) (x)|  \les_{\al} C_l(t) S(t).
  \eal
\eeq

The proof of Theorem \ref{thm:circle} for the $C^{\alpha}$ case is essentially the same as that of Theorem \ref{thm:blowup_a0} and Theorem \ref{thm0:unif} so we only give a sketch. We construct compactly supported approximate steady state $\bar{\om}_c$ by truncating the approximate steady state $\om_{\al}$ in \eqref{eq:solu_calpha}. This truncation allows us to have compactly supported
perturbation $ \om -\bar{\om}_c$ if the initial data $\om$ has compact support and then apply the comparison Lemma \ref{lem:hil}. 
 The associated profiles for the velocity and $\bar{c}_l, \bar{c}_{\om}$ are
\beq\label{eq:circle_cw1}
\bar{u}_x = H_t \bar{\om}_c, \quad  \bar{c}_l = \f{1}{\al} - a \bar{u}_x(0), \quad \bar{c}_{\om} = 1 - \bar{u}_x(0) .
\eeq
We remark that the above profiles are time dependent due to the transform $H_t$. Yet, they are close to the counterparts with $H_t$ replaced by the Hilbert transform $H$ on $\R$ and we can treat them as almost time independent. 
The above choices of $\bar{c}_{\om}, \bar{c}_l$ are consistent with those in \eqref{eq:normal_circle} and \eqref{eq:normal_al} for the $C^{\al}$ case on $\R$. We can truncate $\om_{\al}$ in the far field so that $\bar{\om}_c$ is sufficiently close to $\om_{\al}$ in the sense that
\beq\label{eq:circle_cw2}
| H \bar{\om}_c (0)-H \om_{\al}(0)| < \al^{10} 10^{-10}, \quad 
 \la (\bar{\om}_c - \om_{\al})^2, \vp_{\al} \ra + \la (\bar{\om}_{c,x} - \om_{\al,x})^2, \psi_{\al} \ra
 < \al^{10} 10^{-10},
\eeq
 where $\vp_{\al}, \psi_{\al}$ are the weights used in the analysis of the $C^{\al}$ case in \eqref{eq:wg_alpha}. 

Denote by $\bar{S}$ the size of support of $\bar{\om}_c$.  We remark that $\bar{u}_x(0)$ and $\bar{c}_l, \bar{c}_{\om}$ given above are close to $2, \f{1}{\al} - 2 a, -1$, respectively, since we have 
\beq\label{eq:circle_cw3}
\bar{u}_x(0) = H_t \bar{\om} (0)= (  H_t(\bar{\om}_c(0)) -H \bar{\om}_c(0) ) +  ( H \bar{\om}_c (0)-H \om_{\al}(0) )  + 2,
\eeq
where we have used $H \om_{\al}(0) = 2$ (see \eqref{eq:CLM_self}). The second term is small according to \eqref{eq:circle_cw2} and the first term can be made arbitrarily small by choosing $C_l(0)$ to be sufficiently small later.

For compactly supported initial data $\om \in C^{\al}$ with support size $S(0) > \bar{S}$, all the nonlinear stability analysis in the proof of 
Theorem \ref{thm0:unif} can be derived for the perturbation $\om - \bar{\om}_c$ in almost the same way with two minor differences.
Firstly, the resulting estimates have slightly larger constants due to the small difference between $\om_{\al}$ and $\bar{\om}_c$, which is of order $\al$ due to \eqref{eq:circle_cw2}. Secondly, they contain additional terms depending on the difference between two Hilbert transforms, which can 
be bounded using \eqref{eq:hil2}. 

Therefore, under the bootstrap assumption \eqref{eq:boot_asm1} and \eqref{eq:boot_asm2}, we can derive the following estimates similar to \eqref{eq:boot_al2}, \eqref{eq:boot_cw}
\beq\label{eq:boot_S1}
\bal
 \f{1}{2} \f{d}{dt} E^2(t) & \leq  - \lt( \f{3}{8} - C |a| \al \rt) E^2(t) + C |a| \al  E(t) +  C E^3(t)
+  C_4(a, \al  ) C_l(t) S(t) E^2(t), \\
| c_{\om}(t) - \bar{c}_{\om} | & \leq  C_3\al  E(t) + C_4(a, \al  ) C_l(t) S(t)   , \\
 |c_l(t) - \bar{c}_l|  &\leq  C_3 |a| \al E(t) + C_4(a, \al  ) C_l(t) S(t)  ,  \\
 |  H_t(\bar{\om}_c(0)) - & H \bar{\om}_c(0) | < C_4(a, \al  ) C_l(t) S(t) ,
\eal
\eeq
where 
\beq\label{eq:energy_alpha}
E^2(t)=  \la (\om(t) - \bar{\om}_c )^2, \vp_{\al} \ra + \mu  \la (\om_x(t) - \bar{\om}_{c,x} )^2, \psi_{\al} \ra 
\eeq
 for some absolute constant $0<\mu < 1$ and the constant $C_4(a, \al)$ depends on $a, \al$. Using the control of the support \eqref{eq:S1_gw3}, we have
\beq\label{eq:boot_S10}
C_l(t) S(t) = P(t)   \leq  C(a,\al, S(0)) P(0) = C(a,\al, S(0)) C_l(0) S(0) .
\eeq

Consider a function 
\beq\label{eq:boot_func}
f(x) = - ( \f{3}{8}  - C |a| \al ) x^2 + \f{1}{8} x^2+ C |a| \al x + C x^3.
\eeq

Clearly, if $|a| \al < C_1$ for some sufficiently small constant $C_1$, there exists an absolute constant $C_2$ such that $ f( C_2 |a| \al) < 0$. We can further require that $C_1$ be so small that 
\beq\label{eq:circle_para}
C_2 |a| \al < C_2 C_1< \f{\mu}{100}, \quad (C_3 +  1) C_2 |a| \al 
< (C_3 +  1) C_2 C_1 < \f{1}{100 } .
\eeq

Note that $C_l(0)$ is independent of the initial data $\om(0, \cdot)$ in the dynamic rescaling space and only depends on how we rescale $\om(0,\cdot)$ to get the data in the physical space. 
We first choose compactly supported  $\om(0,\cdot)$ with $\bar{S} \leq S(0) < +\infty$ that satisfies $E(0) < C_2 |a| \al$, where $\bar{S}$ is the size of support of $\bar{\om}_c$. Then we choose $C_l(0)$ sufficiently small such that
\[
( C(a, \al, S(0)) + C_4(a, \al  ) C(a, \al, S(0)) + 1)  C_l(0) S(0) < \f{1}{16}.
\]

Under the bootstrap assumption \eqref{eq:boot_asm1} and \eqref{eq:boot_asm2}, we plug the above inequality in \eqref{eq:boot_S10} and \eqref{eq:boot_S1} to get
\[
\bal
&(C_4(a,\al) + 1)C_l(t) S(t) \leq (C_4(a,\al) + 1) C(a, \al, S(0)) C_l(0) S(0)  < \f{1}{16} < \f{\pi}{4} ,\\
&  \f{1}{2} \f{d}{dt} E^2(t) \leq  - \lt( \f{3}{8} - C |a| \al \rt ) E^2(t) + C |a| \al  E(t) +  C E^3(t)+ \f{1}{16} E^2(t) .
 \eal
\]

From the definitions of $f$ in \eqref{eq:boot_func}, $C_2$ and $f(C_2 |a|\al) < 0$, we know that the additional bootstrap assumption $E(t) < C_2 |a| \al$ can be continued. Finally, we verify that $E(t) < C_2 |a|\al$ implies the bootstrap assumptions \eqref{eq:boot_asm1}, \eqref{eq:boot_asm2} so that all of these assumptions can be continued. From \eqref{eq:circle_para}, we have $E(t) < C_2 |a| \al < \min(\f{\mu}{100} , \f{1}{100})$. Denote 
\[
\rho_1(x) \teq |x|^{-1 - \al} +  |x|^{-1 + \al}, \quad \rho_2(x) \teq |x|^{1 - \al} + |x|^{1 + \al} .
\]

Using the triangle inequality and $\rho_1(x)  \leq \vp_{\al}, \rho_2(x) \leq  \psi_{\al} $ (see \eqref{eq:wg_alpha}),  
we get
\[
\bal
\la \om^2, \rho_1 \ra +  \la \om_x^2, \rho_2 \ra & \leq   5 ( \la \om_{\al}^2, \rho_1 \ra +  \la \om_{\al,x}^2, \rho_2 \ra ) + 5 (\la (\om_{\al} - \bar{\om}_c)^2, \vp_{\al} \ra +  \la (\om_{\al,x} - \bar{\om}_{c,x})^2, \psi_{\al} \ra)  \\
& +  5( \la (\om  - \bar{\om}_c)^2, \vp_{\al} \ra +  \la (\om_x - \bar{\om}_{c,x})^2, \psi_{\al} \ra )  \teq J_1 + J_2 + J_3.
\eal
\]
Using \eqref{eq:circle_cw2}, \eqref{eq:energy_alpha} and $E(t) < \f{\min(\mu,1)}{100}$, we have $J_2 < \f{1}{10}, J_3 < \f{1}{10}$. Hence, we prove the first inequality in \eqref{eq:boot_asm2}.
From \eqref{eq:circle_cw1} and \eqref{eq:circle_cw3}, we have 
\[
\bar{c}_{\om} + 1 = 2 - \bar{u}_x(0) =  -(  H_t(\bar{\om}_c(0)) -H \bar{\om}_c(0) ) -  ( H \bar{\om}_c (0)-H \om_{\al}(0) ) .
\] 
Using the triangle inequality and \eqref{eq:circle_cw3}, we obtain 
\[
| c_{\om} +1 | \leq  |\bar{c}_{\om} - c_{\om}| 
+ |   H_t(\bar{\om}_c(0)) -H \bar{\om}_c(0)  |  +   | H \bar{\om}_c (0)-H \om_{\al}(0) |.
\] 
The estimate of each term on the right hand side follows from \eqref{eq:circle_cw2}, \eqref{eq:boot_S1} and the estimates of $E(t), S(t) C_l(t)$ established above. Similarly, we can estimate $c_l - \f{1}{\al}$. These estimates imply the second inequality in \eqref{eq:boot_asm2}.

The remaining steps to obtain finite time blowup are exactly the same as those in the proof of Theorem \ref{thm:blowup_a0}  and we conclude the proof of Theorem \ref{thm:circle} for the $C^{\al}$ case. For the case of small $|a|$, the proof is completely similar and we omit the proof here.

 \subsection{Criticality of the $C^{\alpha}$ Regularity }

Note that the self-similar solutions in \eqref{eq:CLM_self} all satisfy that $w_{\al}$ is odd and $w_{\al}$ is negative for $x > 0 $, we shall consider general initial datum within this class. We remark that the results in this Section only hold true for positive $a$.

\subsubsection{The DG equation on the real line}
In this Section, we prove Theorem \ref{thm:crit}, which implies that for large positive $a$, the H\"older regularity in the initial datum is crucial for the focusing self-similar blow-up.

\begin{remark}
In the later proof, we choose $C_1 = (1 + 0.015) / (1 -0.015) \approx 1.03$ in Theorem \ref{thm:crit}.
The compact support assumption in Theorem can be relaxed easily by imposing a growth condition on $\om_0$, e.g. $\om_0$ is bounded. 
\end{remark}

To prove Theorem \ref{thm:crit}, we need the following crucial Lemma, which indicates that the advection term can be stronger than the nonlinear term.

\begin{lem}\label{lem:compare}
Let $\e = 0.015$. Suppose that $\b \in [1,2)$ and $a >0$ satisfies
\beq\label{eq:comp_cond}
a > \f{ \e (\b -1) + 1 }{ (1-\e) (\b-1)}.
\eeq
The following inequality
  \beq\label{eq:compare}
  \int_0^{ \infty}  \f{u_x \om - a u \om_x}  {y^{\b}} dy \geq 0 
  \eeq
  holds as long as the left hand side is well-defined and that $\om$ is odd and non-positive for $x>0$.
  \end{lem}

\begin{proof}
From the assumption of $\om$, we know
\[
\om(x) = O(|x|^{\al}) ,  u(x) = O(|x|) \; ,
\]
for small $|x|$. Denote by $I$ the integral in \eqref{eq:compare}. Using integration by parts and expanding the kernel, we get 
\[
\bal
I &= \int_0^{\infty} (1+a) \f{u_x \om}{x^{\b}} - a \b \f{u\om}{x^{1+\b}} dx = \f{1}{\pi} \int_0^{\infty}  (1+a) \f{\om(x)}{x^{\b}} \int_0^{\infty} \om(y) \lt( \f{1}{x-y} - \f{1}{x+y}\rt) dy dx  \\
& \quad \quad + \f{1}{\pi}\int_0^{\infty} a \b \f{\om(x)}{ x^{1 +\b}} \int_0^{\infty} \log \B|\f{ y+x}{y-x}\B| \om(y) dy dx \\
&= \f{1 + a}{\pi} \int_0^{\infty}\int_0^{\infty} \lt[ \f{1}{x^{\b}} \lt( \f{1}{x-y} - \f{1}{x+y}   \rt)
+ \f{a\b}{1+a} \f{1}{x^{\b+1}}  \log \B|\f{ y+x}{y-x}\B|  \rt] \om(x) \om(y)  dx dy.
\eal
\]
Since $\om$ is odd, we can symmetrize the integral kernel
\[
\bal
I = \f{1+a}{2\pi} \int_0^{\infty}\int_0^{\infty} &\lt[
\f{1}{x^{\b}} \lt( \f{1}{x-y} - \f{1}{x+y}   \rt)
+ \f{a\b}{1+a} \f{1}{x^{\b+1}}  \log \B|\f{ y+x}{y-x}\B|
\rt. \\
 + &\lt.  \f{1}{y^{\b}} \lt( \f{1}{y-x} - \f{1}{x+y}   \rt)
+ \f{a\b}{1+a} \f{1}{y^{\b+1}}  \log \B|\f{ y+x}{y-x}\B|  \rt] \om(x) \om(y) dx dy.
\eal
\]
Denote
\[
\tau = \f{a\b}{1 + a}, \quad  s  = \f{y}{x}.
\]
We can simplify the integrand as follows
\[
\bal
&\f{1}{x^{\b}} \lt( \f{1}{x-y} - \f{1}{x+y}   \rt)  +  \f{1}{y^{\b}} \lt( \f{1}{y-x} - \f{1}{x+y}   \rt) \\
= & \f{1}{y^{1  + \b}} \f{y^{1+\b}}{x^{1+\b}} \lt( \f{x}{x-y} - \f{x}{x+y} \rt)
+ \f{1}{y^{1+\b}}  \lt( \f{y}{y-x} - \f{y}{x+y}   \rt) \\
=& \f{1}{y^{1+\b}} s^{1 + \b} \lt( \f{1}{1 - s} - \f{1}{1+s} \rt) + \f{1}{y^{1+\b}} \lt( \f{s}{s-1} - \f{s}{s+1}   \rt) =  - \f{1}{y^{1+\b}} (s^{1+\b} -1) \f{2s}{s^2 - 1}  ,\\
& \f{1}{x^{\b+1}}  \log \B|\f{ y+x}{y-x}\B| + \f{1}{y^{\b+1}}  \log \B|\f{ y+x}{y-x}\B|   =  \f{1}{y^{1+\b}}  (s^{1 + \b} +1)  \log \B|\f{ s+1}{s-1}\B| .
\eal
\]
Then $I$ becomes
\beq\label{eq:comp_ineq}
I = \f{1+a}{2\pi} \int_0^{\infty}\int_0^{\infty}
\f{1}{y^{1+\b}} \lt( \tau ( s^{\b+1}  + 1)  \log \B|\f{ s+1}{s-1}\B| - (s^{1 + \b} - 1) \f{2s}{s^2 - 1}  \rt) \om(x)\om(y) dx dy.
\eeq
Denote
\[
F(s, \b) \teq \f{ 1 - s^{1+\b} }{1 + s^{1+\b}} \f{2s}{1-s^2} \lt(\log \B|\f{ s+1}{s-1}\B| \rt)^{-1}, \quad s \in [0, 1], \ \b \in [1,2] .
\]
We have the following basic property for $F(s, \b)$.

\begin{lem}\label{lem:comp_F}
Assume that $s\in[0,1], \b \in [1,2]$. Then
(a) $F(s, \b)$ is monotonically increasing with respect to $\b$. (b) For any $s \in[0,1]$, we have
\beq\label{eq:comp_F}
\bal
F(s, 1) \leq 1 , 
\quad F(s, \b) < 1 + 0.015 (\b-1) \quad \forall  \b \in (1, 2] .\\
\eal
\eeq
\end{lem}
We defer the proof of the above Lemma to the Appendix. Let $\e = 0.015$. As a result, if
\beq\label{eq:comp_ineq2}
 \tau = \f{a \b}{1+a} >  1 + \e (\b-1) \iff  a > \f{ \e (\b -1) + 1 }{ (1-\e) (\b-1)},
\eeq
then $I$ in \eqref{eq:comp_ineq} is non-negative
\[
I = \f{1+a}{2\pi} \int_0^{\infty}\int_0^{\infty}
\f{1}{y^{1+\b}} \lt( ( s^{\b+1}  + 1)  \log \B|\f{ s+1}{s-1}\B| (\tau - F(s, \b))  \rt) \om(x)\om(y) dx dy \geq0,
\]
 where we have used $\om(x) \om(y) \geq 0 \ \forall x , y \geq0 $.
\end{proof}
Now, we are in a position to prove Theorem \ref{thm:crit}.

\begin{proof}[Proof of Theorem \ref{thm:crit}]
Denote $a_0  = \f{1 + \e}{1-\e}$.   
If $1 \geq \al > a_0/ a$, we get
\[
a \al > a_0 = \f{1 + \e}{1 - \e} \geq \f{ \e \al + 1 }{1 -\e}  \Rightarrow a > \f{ 1 + \e \al}{ (1 -\e) \al}.
\]
Therefore, we can choose $1 \leq \b < \al + 1$, e.g. $\b = 1+\al - \d$ for some sufficiently small $\d$, such that
\[
a >  \f{ 1 + \e (\b -1)}{ (1 -\e) (\b-1)},
\]
i.e. $(a, \b)$ satisfies the assumption \eqref{eq:comp_cond} in Lemma \ref{lem:compare}. For $\om \in C^{\al}$ 
with compact support (or some growth condition at the far field), we have $\om(y) |y|^{-\b} \in L^1$. 
Using \eqref{eq:DG} and Lemma \ref{lem:compare}, we get
\[
\f{d}{dt} \int_0^{\infty} \f{\om(t, y)}{y^{\b}} dy = \int_0^{\infty} \f{ u_x \om - a u \om_x }{y^{\b} } dy \geq 0.
\]
Note that $\om$ is odd and non-positive for $ x >0$. We yield
\[
 0 \geq \int_0^{\infty} \f{\om(t, y)}{y^{\b}} dy \geq \int_0^{\infty} \f{\om(0, y)}{y^{\b}} dy
 \ \Rightarrow   \ \B| \B|  \f{\om(t, \cdot)}{|y|^{\b}}\B|\B|_1  \leq \B| \B|  \f{\om(0, \cdot)}{|y|^{\b}}\B|\B|_1< +\infty
\]
for all $t > 0$. If $\om$ blows up in a self-similar fashion, i.e.
\[
\om(t) \to (T-t)^{-1} \Om\lt( \f{x}{(T-t)^{c_l}}\rt)\; ,
\]
in some suitable functional space, (the convergence can be measured in the dynamic rescaling space), we can plug the self-similar blowup ansatz in $I(t)$ to yield
\[
\B| \B|  \f{\om(t, \cdot)}{|y|^{\b}}\B|\B|_1 \to (T-t)^{-1} \int_{\RR}\B| \Om\lt( \f{x}{(T-t)^{c_l}}\rt) \B| |x|^{-\b} dx
= (T-t)^{-1+ c_l - \b c_l } \int_{\RR} \B| \f{\Om(x)}{x^{\b}} \B| dx
\]
as $t \to T$. Since $\B| \B|  \f{\om(t, \cdot)}{|y|^{\b}}\B|\B|_1$ is bounded uniformly in $t$, we get
\[
-1 + c_l - \b c_l \geq 0 \Rightarrow  c_l \leq -\f{1}{\b-1} \; ,
\]
for any $\b < 1+ \al$. Letting $\b \to \al + 1$, we get $ c_l \leq - \al^{-1}$.
\end{proof}

\subsubsection{DG equation on the circle}
\label{sec:crit_S1}
For the DG equation the circle, we can prove a stronger result.

\begin{thm}\label{thm:nonblow_S1}
Suppose that $\om \in C^{\al}$ is odd, $\pi$ periodic, $\om \leq 0$ for $x \in (0, \pi /2)$ and $1 \geq \al > a_0/a$.
Then the result in Theorem \ref{thm:crit} holds true. Furthermore, 
$u_x(0,t)$ and $|| \om||_1$ do not blow up at the first singularity time $T$, if it exists, and grow at most exponentially fast up to $T$.
\end{thm}

 \begin{remark}
  From our numerical experiments, we found that for various initial data, $u_x(0,t ) = \max u_x(x,t)$ for some finite time. Theorem \ref{thm:nonblow_S1} gives a strong indication that $u_x$ is bounded from above if $u_x(0)$ is bounded. From
  \[
  \f{d}{dt}\max \om \leq  ( \max u_x)  \max \om,
  \]
  ( $\max \om >0$) and the assumption $u_x(0,t ) = \max u_x(x,t)$, we obtain the boundedness of $\max \om$ for all time. Since $\om$ is odd, we get
  $|| \om||_{\infty} =   \max \om$
  is bounded globally. Applying the BKM-type blowup criterion yields the global well-posedness.
  \end{remark}

The following Lemma is an analogy of Lemma \ref{lem:compare} on the circle.
  \begin{lem}\label{lem:compare_S1}
With the assumption as Lemma \ref{lem:compare}, 
  \[
  \int_0^{\pi/2}  ( u_x \om - a u \om_x)   \lt( \cot y\rt)^{\b}  dy \geq 0 .
  \]
  \end{lem}
The proof is similar to that of Lemma \ref{lem:compare} and we defer it to the Appendix.

 \begin{proof}[Proof of Theorem \ref{lem:compare_S1}]
Using Lemma \ref{lem:compare_S1} and an argument similar to that in the proof of Theorem \ref{thm:crit}, we derive
\[
\int_0^{\pi/2} \om_0 \lt( \cot y\rt)^{\b}  \leq  \int_0^{\pi/2} \om \lt( \cot y\rt)^{\b}  \leq 0
\Rightarrow  ||  \om(t, \cdot) |\cot y|^{\b} ||_{L^1} \leq || \om_0 | \cot y|^{\b} ||_1 < +\infty \; ,
\]
for some $1 < \b < 1 + \al$. Next, we estimate the $L^1$ norm of $\om$. 
Integrating \eqref{eq:DG_S1} from $\pi / 2$ to $\pi$ yields
  \beq\label{eq:l11}
  \bal
  &\f{d}{dt }\int_{0}^{  \pi/2} \om\lt(y + \f{\pi}{2}\rt) dy  = -(1 + a) \int_0^{\pi/2} \int_0^{\pi/2}    \om\lt(x + \f{\pi}{2}\rt)  \om\lt(y + \f{\pi}{2} \rt)  \cot \lt( x+y  \rt) dx dy \\
  =&  -2(1 + a) \int_{0 \leq x \leq y \leq \pi/2  }  \om\lt(x + \f{\pi}{2}\rt)  \om\lt(y + \f{\pi}{2}\rt)  \cot \lt(x+y  \rt) dx dy . \\
    =&  -2(1 + a) \int_0^{\pi/2}   \om\lt(x + \f{\pi}{2}\rt)  \lt( \int_x^{\pi/2}  \om\lt(y + \f{\pi}{2}\rt)  \cot \lt(x + y  \rt) dy \rt) dx.  \\
\eal
  \eeq
  Note that $\om \geq 0 $ on $(\pi/2, \pi)$ and $\cot( x+ y) \geq 0$ if $x+y \leq \pi/2$. We yield
  \[
  -  \int_x^{\pi/2}  \om(y + \pi/2)  \cot \lt( x+ y \rt) dy
  \leq  - \int_{x \vee (\pi/2 - x)}^{\pi/2} \om(y + \pi/2) \cot \lt( x+ y \rt) dy .
  \]
  For $x \in [0, \pi/2],  x \vee (\pi/2 - x) \leq y \leq \pi /2$, we have
\[
0 \leq x + y  - \pi / 2  \leq y \leq \pi / 2 , \  \pi /4 \leq y \ \Rightarrow  \ -\cot(x + y)   = \tan(  x + y - \pi / 2 ) \leq \tan y \leq (\tan y)^{\b} \; ,
\]
where $\b \geq 1$ satisfies the assumption in Lemma \ref{lem:compare}. It follows that
\[
\bal
 & -  \int_x^{\pi/2}  \om(y + \pi/2)  \cot \lt( x+ y \rt) dy
\leq \int_{x \vee (\pi/2 - x)}^{\pi/2} \om(y + \pi/2) \tan y dy  \\
 \leq & \int_0^{\pi/2} \om(y + \pi/2) (\tan y)^{\b} dy \leq || \om(t,\cdot) |\cot y|^{\b} ||_{L^1} \leq || \om_0 |\cot y|^{\b} ||_{L^1}.
\eal
\]

Plugging the above estimate in \eqref{eq:l11} implies
\[
\f{d}{dt }\int_{0}^{  \pi/2} \om(y + \pi/2 ) dy  \les (1 + a)|| \om_0 |\cot y|^{\b} ||_{L^1} \int_0^{\pi/2} \om( y + \pi/2) dy.
\]
Note that $\om \geq 0 $ on $[\pi/2 , \pi]$ and $|| \om||_{L^1} = 2 \int_{0}^{  \pi/2} \om(y + \pi/2 )  $.
Using the Gronwall inequality, we obtain
  \[
  || \om(t, \cdot)||_{L^1} \leq \exp \lt( C (1+a) \B| \B|   \om_0  | \cot y |^{\b}  \B| \B|_{L^1} t  \rt)   || \om_0||_{L^1},
  \]
  where $C$ is some universal constant. Interpolating $|| \om||_{L^1}$ and $\B| \B|   \om  | \cot y\|^{\b}  \B| \B|_{L^1}$ gives
  \[
 | u_x(0)  |   = \B| \f{2}{\pi} \int_{0}^{\pi/2} \om   \cot y  dy\B| \lesssim || \om||_{L^1}^{ 1- 1/\b} \B| \B|   \om  |\cot y |^{\b}  \B| \B|^{1/\b}_{L^1}
 \lesssim K_1 \exp \lt( K_2 t  \rt),
   \]
  where $K_1, K_2$ depend on the initial datum and $a, \al$ only.
\end{proof}

Finally, we state a result for $a = 1$. 
\begin{prop}\label{prop:crit_a1}
For $a=1$, suppose that $\om_0 \in H^s(S^1), s > 5/2$,  $\om_{0, x}(0) = 0$ and $\om_0 \leq 0$ for $x \in (0, \pi/2)$, then $u_x(t,0),\; ||\om||_1$ do not blow up at the blowup time $T<+\infty$, if it exists. 
\end{prop}

\begin{proof}
Since $\om \in H^s, s > 5/2$, we have local well-posedness and that $\om(t,\cdot) \in C^2$ by the Sobolev embedding. Note that for $a = 1$, $\om_x(t, 0) \equiv \om_x(0, 0) = 0$. Since $ \om(t, 0) = \om_x(t, 0) = 0$,
we have $\om(t, x) = O(x^2)$ near $x=0$. Define 
\[
I \teq \int_0^{\pi/2}  ( u_x \om -  u \om_x)   \lt( \cot y\rt)^{\b}  dy  .
\]
for any $\b \in (1, 3)$. In particular, for $\b =2.2$, using an argument similar to that in the proof of Lemma \ref{lem:compare_S1}, one can show that $I > 0$. The boundedness of $u_x(0) , || \om||_1$ follows by using an 
argument similar to that in the proof of Theorem \ref{thm:nonblow_S1}.
\end{proof}
\begin{remark}
The regularity of $\om_0$ can be relaxed easily and we do not explore it. 
\end{remark}

\section{Finite Time Blowup for Negative $a$ with $C^{\infty}$ initial data}\label{sec:DG_neg}

For the sake of completeness, we state the finite time blowup result of \eqref{eq:DG} for negative $a$ with smooth initial data.
\begin{thm}\label{thm:neg}
Let $\om \in C_c^{\infty}(\RR)$ or $\om \in C^{\infty}(S^1)$ be an odd function such that $u_x(0) =H\om(0)>0$. Then \eqref{eq:DG} with $a<0$ develops a singularity in finite time.
 \end{thm}

The real line case was proved in the work of Castro and C\'ordoba \cite{Cor10}. We will present a proof for $S^1$. We consider $\pi$ periodic and use the Hilbert transform given in \eqref{eq:hil_circ}.

\begin{proof}
Taking the Hilbert transform on \eqref{eq:DG} yields
\[
(u_x)_t  = \f{1}{2} (u_x^2 - w^2) - a H(  u \om_x) .
\]
Note that $\om(0) =0$.  Choosing $x= 0$ gives
 \beq\label{eq:ux0}
\f{d}{dt} u_x(t, 0)  = \f{1}{2} u_x(t,0)^2  - a H( u \om_x)(t,0).
 \eeq
Next we show that $H( u \om_x)(t, 0) \leq 0$. Since $\om$ is odd, $\pi$-periodic and smooth locally in time 
, it admits a decomposition 
\[
\om(t, x) = \sum_{n \geq 1} a_n(t) \sin (2nx) , \quad \om_x = \sum_{n \geq 1} 2n a_n(t) \cos (2nx),
\]
for some $a_n(t)$ decays sufficiently fast as $n \to +\infty$. It is easy to show that 
\[
u(t, x) =  -\sum_{ n\geq 1} \f{a_n}{2n} \sin(2nx).
\]
Next, we compute $ u / \sin(x), \om_x \cos x$. Using telescoping, we get 
\[
\f{\sin(2nx)}{\sin(x)} = \sum_{1 \leq k \leq n} 2\cos( (2k-1)x) , \quad  \cos(2nx ) \cos x =  \f{\cos(2n-1) x + \cos(2n+1)x}{2} .
\]

It follows that
\[
\bal
\f{u}{\sin x} &=-\sum_{ n\geq 1} \f{a_n}{2n} \sum_{ 1 \leq k\leq n} 2 \cos((2k-1) x) 
 = - \sum_{k \geq 1} \cos((2k-1) x) \sum_{n \geq k} \f{a_n}{n}, \\
 \om_x \cos x &= \sum_{n \geq 1} 2n a_n \f{\cos(2n-1) x + \cos(2n+1)x}{2} = 
 \sum_{n\geq 1 } \cos( (2n-1)x)  ( n a_n +(n-1 ) a_{n-1} ),
 \eal
\]
where $a_0 = 0$ and we have used summation by parts to get the last two identities, which are valid since $a_n$ decays sufficiently fast. Using the orthogonality of $\{ \cos( (2n-1) x) \}_{n \geq 1}$ on $L^2( -\pi/2, \pi/2)$, we derive
\[
H( u\om_x)(t, 0 ) = - \f{1}{\pi} \int_{-\pi/2}^{\pi/2} \f{u}{\sin x} \om_x \cos(x) dx =  \f{1}{2} \sum_{k \geq 1} (\sum_{n \geq k} \f{a_n}{n}) ( k a_k +(k-1 ) a_{k-1} ).
\]
Denote $S_k \teq \sum_{n \geq k} \f{a_n}{n}$ for   $k \geq 1$ and  $S_0 = 0$.
Since $a_n$ decays sufficiently fast, so does $S_n$. We then have $a_k = k (S_k - S_{k+1})$ and
\[
 k a_k +(k-1 ) a_{k-1} = k^2( S_k - S_{k+1}) + (k-1)^2(S_{k-1}- S_k). 
\]
We can reduce $H(u\om_x)(t,0)$ to 
\[
\bal
 &H( u\om_x)(t, 0 )  =  \f{1}{2} \sum_{k \geq1} S_k ( k^2( S_k - S_{k+1}) + (k-1)^2(S_{k-1}- S_k) ) \\
 =& \f{1}{2} \sum_{ k \geq 1} S_k^2 (2k-1) - \f{1}{2}  \sum_{ k\geq 1} S_k S_{k+1} k^2   + \f{1}{2} \sum_{ k \geq 1} S_k S_{k-1}(k-1)^2    
 = \f{1}{2} \sum_{ k \geq 1} S_k^2 ( 2k-1)  \geq 0 .
\eal
\]
Consequently, for $a< 0$,  \eqref{eq:ux0} implies
\[
\f{d}{dt}u_x(t,0)  \geq \f{1}{2} u_x^2(t,0) .
\]
Since $u_x(0,0)>0$, it follows that the solution must develop a finite time singularity. 
  \end{proof}

\appendix
\section{}
\subsection{Properties of the Hilbert transform}\label{app:hil}
Throughout this section, without specification, we assume that $\om$ is smooth and decays sufficiently fast. The general case can be obtained easily by approximation. The following identity is very well known whose proof can be found in, e.g. \cite{Elg17}.
\begin{lem}[The Tricomi identity]\label{lem:tric}
 We have
\[
H( \om H\om)  =  \f{1}{2} (  (H \om)^2 - \om^2  ).
\]
\end{lem}

The Hilbert transform has a nice property that it almost commutes with the power $x^{-1}, x$.
\begin{lem}\label{lem:commute} Suppose that $u_x = H\om$. Then we have
\beq\label{lem:vel0}
\f{u_x - u_x(0)}{x} = H\lt( \f{\om}{x}\rt)  , \textrm{ or equivalently } \ (H\om)(x) = (H\om)(0) + x H\lt( \f{\om}{x}\rt) .
\eeq
Similarly, we have
\beq\label{eq:commute2}
u_{xx} = H\om_x, \quad x u_{xx} = H(x \om_x).
\eeq
Suppose that in addition $\om$ is odd. Then we further have
\beq\label{eq:commute3}
  x^2 u_{xx}  =  H (x^2 \om_x),  \quad x u_x = H(x\om), \quad \f{u_{xx}}{x} = H \lt(  \f{\om_x - \om_x(0)}{x} \rt) .
\eeq
If $\om$ is odd and a piecewise cubic polynomial supported on $[-L,L]$ with $\om(L) = \om(-L) = 0$
($\om^{\prime}, \om^{\prime \prime}$ may not be continuous at $x = \pm L$), then we have
\beq\label{eq:commute4}
 u_{xxx} (x^2 - L^2) = H(  \om_{xx}(x^2 -L^2) )  .
\eeq
\end{lem}

\begin{proof}
The identity \eqref{lem:vel0} is very well known. We have 
\[
\f{u_x - u_x(0)}{x} = \f{1}{\pi x} P.V. \int   \om(y) \lt( \f{1}{x-y} + \f{1}{y} \rt) dy
 = \f{1}{\pi} P.V. \int \f{\om(y)}{(x-y) y} dy= H\lt( \f{\om}{y}\rt)(x).
 \]

For \eqref{eq:commute2}, note that
\[
H \om_x = u_{xx} , \quad H(x \om_x)(0) = - \f{1}{\pi}\int \om_x  dx= 0 .
\]
From \eqref{lem:vel0}, we get
\[
H(x \om_x)(x) = H(x\om_x)(0) + x (H\om_x)(x) =  x u_{xx}(x).
\]

For \eqref{eq:commute3}, if $\om$ is odd, then we obtain
\[
H(x^2 \om_x)(0) = - \f{1}{\pi} \int  x \om_x dx  = \f{1}{\pi} \int \om dx = 0 .	
\]
Applying \eqref{lem:vel0} again yields
\[
H(x^2 \om_x) = H(x^2 \om_x)(0) + x H(x\om_x) =x H(x\om_x) = x^2 u_{xx}.
\]

For the second identity, since $\om$ is odd, we can apply a similar argument to yield $H(x \om)(0) = - \f{1}{\pi} \int \om dx  = 0 $ and 
\[
H(x \om)(x) = H(x\om)(0) + x H\om  = x  H \om = x u_x .
\]

For the third identity in \eqref{eq:commute3}, first of all, we have
\[
\om_x  = -H u_{xx}.
\]
If $\om$ is odd, then $u, u_{xx}$ are also odd. $\f{ \om_x - \om_x(0)}{x}$ and $\f{u_{xx}}{x}$ are $L^2$
for $\om$ smooth with suitable decay at infinity. Using an argument similar to that in the proof of \eqref{lem:vel0} implies
\[
\f{\om_x - \om_x(0)}{x} =  -H \lt(  \f{u_{xx}}{x} \rt).
\]
Applying the Hilbert transform on both sides proves the third identity.

Next, we consider \eqref{eq:commute4}. From the assumption of $\om$, we know $\om \in H^1(\RR)$. We can apply \eqref{eq:commute3} to yield
\[
x^2 u_{xx} = H(x^2 \om_x) , \quad L^2 u_{xx} = L^2 H(\om_x) ,  
\]
which implies $(x^2 -L^2) u_{xx}= H( \om_x(x^2 - L^2))$. Since $\om$ is a piecewise cubic polynomial on $[-L,L]$ and is continuous globally, we further have that 
$\om_{x}(x^2 -L^2)$ is globally Lipschitz and it is in $H^1(\RR)$. 
By the $L^2$ isometry of the Hilbert transform, we get $u_{xx}(x^2-L^2) \in H^1(\RR)$. Using the fact that the derivative commutes with the Hilbert transform, we yield 
\[
\pa_x H( \om_x(x^2-L^2)) = H(  \pa_x ( \om_x(x^2 - L^2) ) ),
\]
which implies 
\[  u_{xxx}(x^2 -L^2) + 2 u_{xx} x = H( \om_{xx}(x^2 -L^2)  + 2x \om_x) .\]
Using the linearity of the Hilbert transform and $u_{xx} x = H(x \om_x)$\eqref{eq:commute2}, we conclude the proof of \eqref{eq:commute4}.
\end{proof}

The cancellation in the following Lemma is crucial in our linear stability analysis.

\begin{lem}\label{lem:vel_a1} Suppose $u_x = H\om$.
(a)  We have
\begin{align}
\int_{\RR} \f{  ( u_x - u_x(0) ) \om    }{  x  } dx & = \f{\pi}{2} ( u_x^2(0) + \om^2(0)) \geq 0 .\label{lem:vel3} 
\end{align}
Furthermore, if $\om$ is odd (so is $u_{xx}$ due to the symmetry of Hilbert transform), we have
\begin{align}
\int_{\RR} \f{  ( u_x - u_x(0) ) \om    }{  x^3  } dx & =  \f{\pi}{2} ( \om_x^2(0) -u^2_{xx}(0)) =\f{\pi}{2}  \om_x^2(0) \geq 0 \label{lem:vel4}.
\end{align}
In particular, the right hand side of \eqref{lem:vel3} vanishes if $ u_x(0) = \om(x) =0$.

(b)  We have
\beq\label{lem:vel5}
   \int_{\RR} u_{xx} \om_x x   dx = 0  .
\eeq

(c) The Hardy inequality: Suppose that $\om$ is odd and $\om_x(0)=0$. For $p = 2, 4$, we have
\beq\label{eq:hd1}
\int \f{  (u  - u_x(0) x)^2} { |x|^{p+2} } dx \leq  \lt(\f{2}{p +1 } \rt)^2   \int \f{  (u_x - u_x(0))^2}{ |x|^{ p } }  dx = \lt(\f{2}{p +1 } \rt)^2  \int   \f{  \om^2}{ |x|^{ p } } dx.
\eeq
\end{lem}

\begin{proof}[Proof of \eqref{lem:vel3}]
Note that $u_x = H \om, \  u_x(0)  = -\f{1}{\pi} \int \f{\om}{x} dx.$
Using Lemma \ref{lem:tric}, we get
\[
\bal
\int \f{ (u_x -u_x(0)) \om}{x} dx & = \int \f{\om \cdot H \om}{x} dx - u_x(0) \int \f{\om}{x} dx
= -\pi H ( \om \cdot H \om)(0) +  \pi u_x(0) \cdot u_x(0) \\
& = \f{\pi}{2} (\om^2(0) - u^2_x(0)) + \pi u_x^2(0) = \f{\pi}{2} ( \om^2(0) + u^2_x(0) ) .
\eal
\]
If $\om(0) =0 $, the above estimates are reduced to $\f{\pi}{2} u^2_x(0)$.
\end{proof}


\begin{proof}[Proof of \eqref{lem:vel4}]
If $\om$ is odd and smooth, then $\om / x$ is even and smooth and $H(\om / x)$ is odd. Using \eqref{lem:vel0} and Lemma \ref{lem:tric}, we have
\[
\bal
\int \f{  ( u_x - u_x(0) ) \om    }{  x^3  } dx &= \int \f{1}{x} \f{\om}{x} H\lt( \f{\om}{x}\rt) dx
= -\pi H\lt( \f{\om}{x} H\lt( \f{\om}{x}\rt)    \rt)(0)   \\
&= \f{\pi}{2} \lt\{   \lt( \f{\om}{x}(0)  \rt)^2 - H\lt( \f{\om}{x}  \rt)(0)^2   \rt\}
 = \f{\pi}{2} ( \om_x^2(0) - u^2_{xx}(0)) .
\eal
\]
If $u_{xx}(0) = 0$, the above equality is reduced to $\f{\pi}{2} \om^2_x(0)$.
\end{proof}

\begin{proof}[Proof of \eqref{lem:vel5}]
Applying \eqref{lem:vel3} with $(u_x, \om)$ replaced by $(u_{xx}, \om_x)$ yields
\[
\bal
\la u_{xx} \om_x, x\ra &= \int \f{  (x\om_x)  H(x\om_x) }{x}  dx= \int \f{  (x\om_x)  ( H(x\om_x )-H(x\om_x )(0))  }{x}  dx \\
&= \f{\pi}{2} (  (x\om_x)^2(0) + ( x u_{xx})^2(0)    ) =0 ,
\eal
\]
where we have used $(x u_{xx})(0) = (x\om_x)(0) = 0$ to obtain the last equality.
\end{proof}

\begin{proof}[Proof of \eqref{eq:hd1}]
The first inequality in \eqref{eq:hd1} is the standard Hardy inequality \cite{Har52}. Since $\om$ is odd and $\om_x(0) =0$, $\om/x, \om/x^2 \in L^2(\RR)$.  From \eqref{lem:vel_a1}, we have
\[
\f{u_x - u_x(0)}{x} = H\lt( \f{\om}{x}\rt), \quad  H\lt( \f{\om}{x^2} \rt) =\f{1}{x} \lt( H\lt( \f{\om}{x}\rt) -H\lt(\f{\om}{x}\rt)(0) \rt) .
\]
Since $\om$ is odd, we obtain $ H( \f{\om}{x}) = 0$. Hence, we can simplify the second equality as follows
\[
\bal
 H\lt( \f{\om}{x^2} \rt) = \f{1}{x} H\lt( \f{\om}{x}\rt) = \f{1}{x} \f{ u_x - u_x(0)}{x} = \f{u_x - u_x(0)}{x^2}.
 \eal
\]

Applying the $L^2$ isometry property of the Hilbert transform $H$ to $H( \f{\om}{x}), H(\f{\om}{x^2})$, we establish the equality in \eqref{eq:hd1}.
\end{proof}

The following Lemma is an analogy of Lemma \ref{lem:vel_a1} for H\"older continuous functions. \eqref{lem:vel1},\eqref{lem:vel15} and \eqref{lem:vel2} are from C\'ordoba \& C\'ordoba \cite{Cor06}.

\begin{lem}[Weighted estimate for $C^{\al}$ functions]\label{lem:vel}
Suppose that $u_x = H \om$ and $\om$ is odd in \eqref{lem:vel1}, \eqref{lem:vel2} and \eqref{eq:hd2}.
(a)
For $\b \in (0, 2) $, we have
\begin{align}
\int \f{  (u_x - u_x(0))^2}{ |x|^{1+\b} }  dx \leq \f{ 1  }{  \tan^2 \f{\b \pi}{4} \wedge \cot^2 \f{\b \pi}{4}   }   \int \f{w^2}{|x|^{1+\b}} dx  \lesssim\f{1}{ ( \b \wedge (2-\b) )^2}     \int \f{w^2}{|x|^{1+\b}} dx , \label{lem:vel1} \\
\int \f{  u_x^2}{ |x|^{1-\b} }  dx\leq \f{ 1  }{ \tan^2 \f{\b \pi}{4} \wedge \cot^2 \f{\b \pi}{4}   }   \int \f{w^2}{|x|^{1 -\b}} dx  \lesssim\f{1}{ ( \b \wedge (2-\b) )^2}     \int \f{w^2}{|x|^{1 - \b}} dx  , \label{lem:vel15}
\end{align}
provided that the right hand side is finite, where $ a \wedge b = \min(a,b)$. Note that we do not need to assume that $\om$ is odd in \eqref{lem:vel15}.

(b) For $ \b \in (0,2)$, we have
\beq\label{lem:vel2}
\int \f{  ( u_x - u_x(0) ) \om    }{ \sgn(x) |x|^{1 + \b}  } dx \geq 0  .
\eeq

(c) 1D Hardy inequality \cite{Har52}:  For $ \b  \in (0,1)$, we have
\beq\label{eq:hd2}
\int \f{  (u  - u_x(0) x)^2}{  |x|^{3 + \b} }  dx \leq  \lt(\f{2}{ \b+2 } \rt)^2   \int \f{  (u_x - u_x(0))^2}{ |x|^{ \b + 1 } } dx
 \lesssim \f{1}{\b^2}   \int   \f{  \om^2}{ |x|^{  \b + 1 } }.
\eeq
\end{lem}
The first inequality in \eqref{eq:hd2} is the Hardy inequality \cite{Har52} and the second inequality in \eqref{eq:hd2} follows from \eqref{lem:vel1}.

\subsection{Estimate of the $C^{\alpha}$ approximate self-similar solution}\label{app:profile}
We establish the estimates of the approximate self-similar solution in Lemma \ref{lem:est_profile} in this section.

\subsubsection*{Proof of Lemma \ref{lem:est_profile}}

\subsubsection*{Proof of \eqref{eq:est_dp}}
Recall the explicit formula for $\om_{\al}, u_{\al,x}$ in \eqref{eq:CLM_self} and the weight $\vp_{\al}, \psi_{\al}$ in \eqref{eq:wg_alpha}. Denote $c_{\al} = \cos (\al \pi / 2), s_{\al} = \sin(\al \pi / 2)$. Without loss of generality, we consider $x>0$. We have
\[
 \vp_{\al} = \f{1}{2s_{\al}} \f{  (1 + 2 c_{\al} x^{\al} + x^{2\al} )^2}{ x^{1 + 3\al}  }, \quad 
 \f{x \vp_{\al,x}}{\vp_{\al}} = x (\log(\vp_{\al})_x = \f{ 2( 2\al x^{2\al} +2 \al c_{\al} x^{\al}  ) }{  1 + 2 c_{\al} x^{\al} + x^{2\al}  } - (1 + 3\al),
\]
which implies 
\[
\bal
&\f{1}{2\vp_{\al}} \lt( \f{1}{\al} x   \vp_{\alpha} \rt)_x + ( \bar{c}_{\om} + u_{\alpha,x} )  
=  \f{1}{2\al} + \f{x \vp_{\al, x}}{2\al \vp_{\al}} -1 +  \f{ 2(1 + c_{\al}  x^{\al}   ) }{ 1 + 2 c_{\al}x^{\al} + x^{2\al}  }  \\
= &  - \f{1+3\al}{2\al} + \f{ 2( x^{2\al} + c_{\al} x^{\al}) }{  
1 + 2 c_{\al}x^{\al} + x^{2\al}} +\f{1}{2\al}-1 +  \f{ 2(1 + c_{\al}  x^{\al}   ) }{ 1 + 2 c_{\al}x^{\al} + x^{2\al}  }  = -\f{3}{2} -1 + 2 = -\f{1}{2},
\\
\eal
\]
which is the first identity in \eqref{eq:est_dp}. The second identity can be proved similarly. For the third inequality, we have
\[
\bal
&u_{\al,xx} \psi_{\al} = \pa_x\lt( \f{2 (1 + c_{\al} x^{\al})}{ 1 + 2c_{\al} x^{\al} +x^{2\al}} \rt) \f{(1 + 2c_{\al}x^{\al} + x^{2\al})^2  }{ 2 s_{\al} \al^2 }x^{1- 3\al}  \\
 =& \lt( 2\al c_{\al}x^{\al-1} (1 + 2c_{\al}x^{\al} + x^{2\al} )
 - 2(1 + c_{\al}x^{\al}) 2 \al ( c_{\al} x^{\al-1} + x^{2\al -1} )   \rt)  x^{1-3\al}  (2 s_{\al} \al^2)^{-1} \\ 
 = & 2\al (  -  c_{\al} x^{3\al -1}  - 2 x^{2\al - 1}  - c_{\al}x^{\al-1} )   x^{1-3\al}  (2 s_{\al} \al^2)^{-1}  = - ( 2 x^{ -\al }  + c_{\al} (x^{-2\al} + 1) )(s_{\al} \al)^{-1}
 ,
 \eal
\]
which is monotone increasing with respect to $x$ and the desired inequality follows.

\subsubsection*{Proof of \eqref{eq:est_w}}
Since all quantities are symmetric, we only consider $x \geq 0$.  For \eqref{eq:est_w}, we have
\[
\bal
\B| \f{ x \om_{\al, x}} {\om} \B| &=\al \B| \f{ 1 - x^{\al}   }{  1 +2  \cos\lt( \f{\al \pi}{2} \rt) |x|^{\al}+  |x|^{2\al}  } \B|  \leq \al \; , \\
 \B| \f{ x^2 \om_{\al, xx} +x \om_{\al, x}  }{ \om_{\al}} \B|
 &= \al^2\B| \f{ 1 - 6 x^{2\al} + x^{4\al} -2 \cos\lt( \f{\al \pi}{2}   \rt) x^{\al} (1 +   x^{2\al} )  }{ (1 + 2 \cos\lt( \f{\al \pi}{2}   \rt) x^{\al} + x^{2\al} )^2 } \B| \les \al^2,
\eal
\]
uniformly for all $x \geq 0$. Using the triangle inequality, we get
\[
 \B|\B| \f{ x^2 \om_{\al, xx}  }{ \om_{\al}} \B| \B|_{\infty} \leq  \B| \B| \f{ x \om_{\al, x}} {\om} \B|\B|_{\infty} +  \B|\B| \f{ x^2 \om_{\al, xx} +x \om_{\al, x}  }{ \om_{\al}} \B| \B|_{\infty}\les \al.
\]

\subsubsection*{Proof of \eqref{eq:est_v}}
Assume $y \geq 0 $. For \eqref{eq:est_v}, we have
\[
\bal
\B|  \f{u_{\al}(y)}{y} - u_{\al, x}(0) \B|
 &= \B| \f{1}{y} \int_0^y  \lt(  \f{ 2(1 + \cos\lt( \f{\al \pi}{2}   \rt) |x|^{\al}   ) }{ 1 + 2 \cos\lt( \f{\al \pi}{2}   \rt) |x|^{\al} + |x|^{2\al}  }  - 2 \rt)  dx \B|
  =  \B| \f{1}{y}  \int_0^y \f{ 2 |x|^{2 \al} + 2 \cos\lt( \f{\al \pi}{2}   \rt) |x|^{\al}  }{  1 + \cos\lt( \f{\al \pi}{2}   \rt) |x|^{\al} +|x|^{2 \al}    } dx  \B|  \\
  &\les \B| \f{1}{y} \int_0^y |x|^{\al } \wedge 1 dx \B| \leq \min\lt( \B| \f{1}{y} \f{|y|^{1+\al}}{1+\al} \B|,  \ 1   \rt) \les |y|^{\al} \wedge 1 \;,  \\
  \B|  \f{u_{\al}(y) }{y} - u_{\al, x}(y) \B| &  = \B| \f{1}{y }\int_0^y (u_{\al, x}(x) - u_{\al,x}(y)) dx
\B|  = \B| \f{1}{y} \int_0^y x u_{\al, xx}(x) dx \B|  \\
& = \B| \f{2\al}{y} \int_0^y \f{x^{\al} ( 2x^{\al} + \cos\lt( \f{\al \pi}{2}   \rt)(1 + x^{2\al}) )  }{ ( 1 + \cos\lt( \f{\al \pi}{2}   \rt) |x|^{\al} +|x|^{2 \al}  )^2  }  dx\B|
\les \B| \f{2\al}{y} \int_0^y |x|^{\al} \wedge 1 dx \B| \les \al |y|^{\al} \wedge 1 .
\eal
\]

\subsubsection*{Proof of \eqref{eq:est_asy}}
For \eqref{eq:est_asy}, the first two inequalities follow from the definition of
$\vp_{\al}, \psi_{\al}$ in \eqref{eq:wg_alpha} and
\[
- \f{ 1}{ \om_{\al} } \asymp \f{1}{\al} (|x|^{\al} + |x|^{-\al} ).
\]
From the definition $\psi_{\al} = x^2 \vp_{\al}/\al^2$, we know
\[
\f{x \psi_{\al, x}}{\psi_{\al}} -1 =  \f{x (x^2 \vp_{\al})_x}{ x^2 \vp_{\al}} -1= \f{ x^3 \vp_{\al,x} + 2x^2 \vp_{\al}   }{x^2 \vp_{\al}} -1= \f{ x \vp_{\al, x}  }{\vp_{\al}} + 1.
\]
Hence, for the third inequality in \eqref{eq:est_asy}, we get
\[
\bal
 \B| \f{ x \psi_{\al, x}  }{\psi_{\al}} - 1  \B|   = \B| \f{ x \vp_{\al, x}  }{\vp_{\al}} + 1  \B|  &= \al \B|  \f{-3 + x^{2\al} - 2 \cos\lt( \f{\al \pi}{2}   \rt)x^{\al}} { 1 + \cos\lt( \f{\al \pi}{2}   \rt) x^{\al} +x^{2 \al}  } \B| \les \al. \\
 \eal
\]

\subsubsection*{Proof of \eqref{eq:est_err}}
Recall the definition of error term in \eqref{eq:error_alpha}
\[
F_{\al}(\om_{\al})  =   - a ( u_{\al} - u_{\al,x}(0) x )  \om_{\al, x}.
\]
For the first inequality in \eqref{eq:est_err}, we use the result \eqref{eq:est_w} and \eqref{eq:est_v} 
that we just proved to yield
\[
\bal
\la F_{\al}(\om_{\al})^2 ,\vp_{\al} \ra & = a^2 \B\la ( u_{\al} - u_{\al,x}(0) x )^2  \om^2_{\al, x},
- \f{1}{ \sgn(x) \om_{\al}} \f{ 1 + 2 \cos\lt( \f{\al \pi}{2}   \rt) |x|^{\al}  + |x|^{2\al}    }{ |x|^{1 + 2 \al}}  \B\ra  \\
& \les a^2 \B\la ( u_{\al} - u_{\al,x}(0) x )^2  \om^2_{\al, x},
- \f{1}{ \sgn(x) \om_{\al}}  ( |x|^{-1} + |x|^{-1 - 2\al} )  \B\ra \\
& = a^2 \B\la \lt(\f{u_{\al}}{x} - u_{\al, x}(0) \rt)^2 \f{x^2  \om^2_{\al, x}}{\om^2_{\al}} ,
| \om_{\al}|  ( |x|^{-1} + |x|^{-1 - 2\al} )  \B\ra \\
&\les a^2 \B\la (|x|^{\al} \wedge 1)^2 \al^2, |\om_{\al}| ( |x|^{-1} + |x|^{-1 -2\al} ) \B\ra\\
& \les a^2 \al^2 \B\la |\om_{\al}|, |x|^{-1} \ra = a^2 \al^2 \B| \int_{\RR} \f{\om_{\al}}{x} dx \B|
= a^2 \al^2 \pi | u_{\al, x}(0)| \les a^2 \al^2 ,
\eal
\]
where we have used $\om_{\al} /x \leq 0 $ for all $x \in \RR$ to obtain the last line.

For the second inequality in \eqref{eq:est_err}, we first rewrite $(F_{\al}(\om_{\al}) )_x$ as follows
\[
\bal
(F_{\al}(\om_{\al}) )_x &=  - a ( ( u_{\al} - u_{\al,x}(0) x )  \om_{\al, x})_x
= - a \lt\{  \lt(\f{u_{\al}}{x} - u_{\al,x}(0) \rt) (x \om_{\al, x})  \rt\}_x \\
& = -a \lt(\f{u_{\al}}{x} - u_{\al,x}(0) \rt)_x (x \om_{\al, x})
- a \lt(\f{u_{\al}}{x} - u_{\al,x}(0) \rt) (x \om_{\al, x})_x  \\
& = a \lt( \f{u_{\al}(x)  }{x}  - u_{\al, x}(x) \rt) \om_{\al, x}  - a\lt(\f{u_{\al}}{x} - u_{\al,x}(0) \rt) \f{ x^2 \om_{\al, xx} + x\om_{\al, x} }{x} .
\eal
\]
We can use \eqref{eq:est_v} to estimate $u_{\al}$ and $u_{\al, x}$ as follows:
\[
|(F_{\al}(\om_{\al}) )_x | \les a \al \lt( |x|^{\al} \wedge 1 \rt) | \om_{\al, x} | +
a ( |x|^{\al} \wedge 1 )  \B| \f{ x^2 \om_{\al, xx} + x\om_{\al, x} }{x}\B| .
\]
Then we use \eqref{eq:est_w} to estimate $\om_{\al,x}, x^2\om_{\al, xx} + x\om_{\al,x} $
\beq\label{eq:pf_err11}
|(F_{\al}(\om_{\al}) )_x |  \les a \al \lt( |x|^{\al} \wedge 1 \rt) \al \B| \f{\om_{\al}}{x} \B|
+ a ( |x|^{\al} \wedge 1 )  \al^2 \B| \f{\om_{\al}}{x} \B|
\les a\al^2 (|x|^{\al} \wedge 1) \B| \f{\om_{\al}}{x} \B|.
\eeq
Hence, we can estimate the weighted $H^1$ error as follows
\beq\label{eq:pf_err12}
\bal
&\la (F_{\al}(\om_{\al}) )^2_x, \psi \ra = \B\la  (F_{\al}(\om_{\al}) )^2_x,
\f{x^2}{ \al^2 |\om_{\al}| } \f{ 1 + 2 \cos\lt( \f{\al \pi}{2}   \rt) |x|^{\al}  + |x|^{2\al}    }{ |x|^{1 + 2 \al}}  \B\ra \\
\les&  \B\la  (F_{\al}(\om_{\al}) )^2_x,
\f{x^2}{ \al^2 |\om_{\al}| } ( |x|^{-1} + |x|^{-1-2\al} ) \B\ra  \\
 \les& a^2 \al^4 \B\la (|x|^{\al} \wedge 1)^2 \B| \f{\om_{\al}}{x} \B|^2, \f{x^2}{ \al^2 |\om_{\al}| } ( |x|^{-1} + |x|^{-1-2\al} )\B\ra \\
 \les& a^2 \al^4 \B\la \B| \f{\om_{\al}}{x} \B|^2, \f{x^2}{ \al^2 |\om_{\al}| } |x|^{-1}\B\ra
 = a^2 \al^2 \la |\om_{\al}|, |x|^{-1} \ra = a^2 \al^2 \pi | u_{\al, x}(0)| \les a^2 \al^2.
\\
\eal
\eeq

\subsubsection*{Proof of \eqref{eq:est_err2}}
Using \eqref{eq:est_w}, we get
\[
 ( |x|^{\al} \wedge 1 ) | \om_{\al, x}| \les \al  ( |x|^{\al} \wedge 1 )  \B| \f{\om_{\al}}{x} \B|.
\]
Note that the above bound is the same as \eqref{eq:pf_err11} up to a factor $a\al$. Hence, we can use the same estimate as in \eqref{eq:pf_err12} to yield
\[
\la ( |x|^{\al} \wedge 1 )^2 \om^2_{\al, x}, \psi_{\al} \ra  \les 1 . \qedhere
\]

\subsubsection*{Proof of Lemma \ref{lem:vel_inf}}
\subsubsection*{Proof of \eqref{lem:vel_inf1}}
Firstly, we apply the weighted inequality \eqref{lem:vel15} with $\b = \al, 2 - \al$ to yield
\[
\bal
&\int \f{u_x^2}{|x|^{1 - \al} } dx \les \f{1}{\al^2} \int \f{\om^2} {|x|^{1 - \al}} dx , \quad \int  u_{xx}^2 |x|^{1 - \al}  dx \les \f{1}{\al^2} \int \om_x^2 |x|^{1-\al} dx . \\
\eal
\]
It follows that
\[
\bal
|| u_x||^2_{\infty} &\leq  2\int | u_x u_{xx} | dx
\leq 2\lt(  \int \f{u_x^2}{|x|^{1 - \al} } dx  \rt)^{1/2} \lt( \int  u_{xx}^2 |x|^{1 - \al}  dx \rt)^{1/2} \\
&\les \f{1}{\al^2}\lt(  \int \f{\om^2}{|x|^{1 - \al} } dx  \rt)^{1/2} \lt( \int  \om_{x}^2 |x|^{1 - \al}  dx \rt)^{1/2}.
\eal
\]
Using the asymptotic properties of $\vp_{\al}, \psi_{\al}$ in \eqref{eq:est_asy} and
$|x|^{1 - \al} \leq |x|^{1- 3\al} + |x|^{1+\al}$, 
we conclude
\[
|| u_x||_{\infty} \les   \la \om^2 ,\vp_{\al} \ra^{1/4}    \la \om_x^2 ,\psi_{\al} \ra^{1/4} .
\]

\subsubsection*{Proof of \eqref{lem:vel_inf2}}
Firstly, we use $\om(0)= 0$ and integration by parts to get
\[
\bal
J \teq&\int_0^{\infty} \om_x(y) \lt( \f{y}{x} \log\B|\f{ x +y }{ x -y}  \B| - 2  \rt) dy 
= - P.V. \int_0^{+\infty} \om(y)\lt( \f{1}{x} \log\B|\f{ x +y }{ x -y}  \B| 
+ \f{y}{x} (\f{1}{y+x} - \f{1}{y-x}) \rt) dy  \; .
\eal
\]
Using 
\[
\f{y}{x} (\f{1}{y+x} - \f{1}{y-x}) = \f{1}{x-y}  -\f{1}{x+y},
\]
 $\td{u} = u - u_x(0)x$ and that $\om$ is odd, we derive 
\[
J = - P.V. \int_0^{+\infty} \om(y)\lt( \f{1}{x} \log\B|\f{ x +y }{ x -y}  \B| 
 + (  \f{1}{x-y}  -\f{1}{x+y} )\rt) dy   =\pi \lt(  \f{u}{x} - u_x \rt)
 = \pi \lt(  \f{\td{u} }{x} - \td{u}_x \rt) .
\]
Using this integral formula of $\td{u}/x - \td{u}_x = J$, the asymptotic property \eqref{eq:est_asy} and the Cauchy-Schwartz inequality, we have
\[
\bal
\B| \f{\td{u}}{x} - \td{u}_x  \B|
&\les \la \om_x^2 , |y|^{1-3\al} + |y|^{1 + \al}  \ra^{1/2}
\lt(\int_0^{\infty} \lt( \f{y}{x} \log\B|\f{ x +y }{ x -y}  \B| - 2  \rt)^2 \f{1}{|y|^{1-3\al} + |y|^{1+\al}} dy \rt)^{1/2} \\
& \les \al^{3/2} \la \om_x^2, \psi_{\al} \ra \lt(\int_0^{\infty} \lt( \f{y}{x} \log\B|\f{ 1 +y /x}{ 1 -y/x}  \B| - 2  \rt)^2 \f{1}{|y|^{1-3\al} + |y|^{1+\al}} dy \rt)^{1/2}.
\eal
\]
Next, we estimate the integral
\[
I(x ) \teq \int_0^{\infty} \lt( \f{y}{x} \log\B|\f{ 1 +y /x}{ 1 -y/x}  \B| - 2  \rt)^2 \f{1}{|y|^{1-3\al} + |y|^{1+\al}} dy .
\]
To conclude \eqref{lem:vel_inf2}, it suffices to prove
\[
|I(x)| \les \f{1}{\al} ( |x|^{2\al} \wedge 1).
\]
Without loss of generality, we assume $x \geq 0$. Using change of variable $s = y /x$ yields
\[
\bal
I(x) &= \int_0^{\infty} \lt(  s \log \B|  \f{ s+1 }{s-1} \B| - 2\rt)^2 \f{x}{  (xs)^{1-3\al}
+ (xs)^{1+\al}  } ds\\
 &= \lt( \int_0^{x^{-1} \wedge 1/2} + \int_{x^{-1} \wedge 1/2}^{1/2}
  + \int^{\infty}_{1/2} \rt)  \lt(  s \log \B|  \f{ s+1 }{s-1} \B| - 2\rt)^2 
  \f{x}{  (xs)^{1-3\al}
+ (xs)^{1+\al}  } ds \\
&\teq I_1(x) + I_2(x) + I_3(x).
\eal
\]
We introduce $f(s)$ and it satisfies the following estimate 
\beq\label{eq:kernel_f}
f(s) \teq \B| s \log \B|  \f{ s+1 }{s-1} \B| - 2  \B| \les s^{-1}  , \ \forall s > 2 \ ; \quad 
\B| s \log \B|  \f{ s+1 }{s-1} \B| - 2  \B|  \les 1 , \ \forall s < 1/2 .
\eeq
For $I_3(x)$, we use $ (xs)^{1-3\al}+ (xs)^{1-\al} \geq xs$ and the decay of $f(s)$ in \eqref{eq:kernel_f}
 to obtain
\[
\bal
I_3(x) &\leq \min\lt\{ \int^{\infty}_{1/2}   \lt(  s \log \B|  \f{ s+1 }{s-1} \B| - 2\rt)^2 \f{x}{ xs } ds,
\int^{\infty}_{1/2}   \lt(  s \log \B|  \f{ s+1 }{s-1} \B| - 2\rt)^2 \f{x}{ (xs)^{1-3\al} } ds
 \rt\} \\
 & = \min\lt\{
\int^{\infty}_{1/2}   f(s)^2 \f{1}{ s } ds,  |x|^{3\al}
\int^{\infty}_{1/2}   f(s)^2 \f{1}{ s^{1-3\al} } ds \rt\} 
 \les \min\lt( 1, |x|^{3\al}   \rt).
\eal
\]

Next, we estimate $I_1(x)$. For $s \in [0, 1/2]$, we use the boundedness of $f(s)$ in \eqref{eq:kernel_f} to obtain 
\[
\bal
I_1(x)  &\les \int_0^{x^{-1} \wedge 1/2}  \f{x} {  (xs)^{1-3\al} + (xs)^{1+\al}} ds
\leq \int_0^{x^{-1} \wedge 1/2}  \f{x} {  (xs)^{1-3\al} } ds \\
&= x^{3\al} \int_0^{x^{-1} \wedge 1/2}  s^{-1 + 3\al} ds  = \f{1}{ 3\al} x^{3\al} (x^{-1} \wedge 1/2)^{3\al} \les \f{1}{\al} |x|^{2\al} \wedge 1.
\eal
\]
For $I_2(x)$, by definition, $I_2(x) = 0$ if $x^{-1} \geq 1/2$ (i.e. $ x \leq 2$). For $x  > 2$, we have
\[
\bal
I_2(x)  &\les \int_{x^{-1} }^{1/2}  \f{x} {  (xs)^{1-3\al} + (xs)^{1+\al}} ds
\leq \int_{x^{-1} }^{1/2}  \f{x} {  (xs)^{1+\al} } ds \\
&= x^{-\al} \int_{x^{-1} }^{1/2}  s^{-1  -\al} ds
\leq x^{-\al} \f{1}{\al} (x^{-1} )^{-\al} \leq \f{1}{\al} \leq  \f{1}{\al} |x|^{2\al} \wedge 1.
\eal
\]

Therefore, we conclude
\[
I(x) = I_1(x) + I_2(x) + I_3(x) \les \f{1}{\al} ( |x|^{2\al} \wedge 1 ).
\]
This completes the proof of \eqref{lem:vel_inf2}.

\subsubsection*{Proof of \eqref{lem:w_inf}}
Without loss of generality, we assume $x \geq 0$. Using \eqref{eq:est_asy} and the Cauchy-Schwartz inequality, we get
\beq\label{eq:pf_winf1}
\bal
|\om(x)|& \leq \int_0^x| \om_x| dx  \leq \la \om_x^2, |x|^{1-3\al} + |x|^{1+\al} \ra^{1/2}
\lt(\int_0^{x} \f{1}{ |x|^{1-3\al} + |x|^{1+\al}  } dx\rt)^{1/2} \\
& \les \al^{3/2} \la \om_x^2, \psi \ra^{1/2} \lt(\int_0^{x} \f{1}{ |x|^{1-3\al} + |x|^{1+\al}  } dx\rt)^{1/2}.  
\eal
\eeq
If $x \in[0,1]$, the integral is bounded by
\beq\label{eq:pf_winf2}
\int_0^{x} |x|^{ - 1+3\al}   dx \les \f{1}{\al} |x|^{3\al} \leq \f{1}{\al} |x|^{2\al}
\leq \f{1}{\al}(  |x|^{2\al} \wedge 1 ).
\eeq
Otherwise, it is bounded by
\beq\label{eq:pf_winf3}
\int_0^{1} |x|^{ - 1+3\al}   dx + \int_1^{x}  |x|^{ - 1- \al}  dx \les \f{1}{\al}  \leq \f{1}{\al}(  |x|^{2\al} \wedge 1 ).
\eeq
Therefore, combining \eqref{eq:pf_winf1}, \eqref{eq:pf_winf2} and \eqref{eq:pf_winf3}, we conclude
\[
| \om(x)| \les \al^{3/2} \la \om_x^2, \psi \ra^{1/2} \al^{-1/2} (|x|^{2\al} \wedge 1 )^{1/2}
 = \al \la \om_x^2, \psi \ra^{1/2}(|x|^{\al} \wedge 1 )  . \qedhere
\]

\subsection{Proof of other Lemmas}
\begin{proof}[Proof of Lemma \ref{lem:comp_F}]
Recall the definition of $F(s, \b)$
\[
F(s, \b) \teq \f{ 1 - s^{1+\b} }{1 + s^{1+\b}} \f{2s}{1-s^2} \lt(\log \B|\f{ s+1}{s-1}\B| \rt)^{-1}, \quad s \in [0, 1], \ \b \in [1,2].
\]
For $s \in [0, 1]$, $s^{1+\b}$ is decreasing with respect to $\b$. Hence
\[
\f{1- s^{1+\b}}{1 + s^{1+\b}} \textrm{ is increasing w.r.t. } \b \  \Rightarrow  \ F(s, \b)  \textrm{ is increasing w.r.t.  } \b.
\]
Next, we show that $F(s,1) \leq 1$. Denote
\[
G(s) = \log \B|\f{ s+1}{s-1}\B|  - \f{2s}{1+s^2} ,\quad  s\in[0,1].
\]
Note that
\[
G^{\prime}(s) = \f{2}{1-s^2}  - \f{2}{1+s^2} + \f{4s^2}{ (1+s^2)^2}\geq 0 , \quad \lim_{s \to 0^+} G(s) = 0.
\]
We conclude $G(s) \geq 0 $ for $s \in [0, 1]$. It follows that
\[
F(s,1) =  \f{ 1 - s^{2} }{1 + s^{2}} \f{2s}{1-s^2} \lt(\log \B|\f{ s+1}{s-1}\B| \rt)^{-1}
=  \f{2s }{1 + s^{2}}  \lt(\log \B|\f{ s+1}{s-1}\B| \rt)^{-1}  \leq  1.
\]
The equality is achieved at $s= 0$. Next, we verify that $F(s, \b) \leq 1 + 0.015(\b-1)$. We split $s, \b \in [0,1] \times  [1,2]$ into several domains and prove this inequality separately.
\paragraph{\bf{Case a.} $ (s, \b) \in [0.5, 1] \times [1,2]$  }
 Using the fact that $F(s, \b)$ is increasing w.r.t. $\b$, we know
\[
F(s, \b) \leq F(s, 2) <0.96< 1  \quad \forall s \geq 0.5, \b \in[1,2],
\]
where the inequality $F(s, 2) <1$ for $s  \in [0.5,1]$ can be verified numerically with rigorous error control. For $s$ close to 1, $F(s,\b)$ goes to $0$ since $\log( (1+s)/ (1-s) )$ blows up at $s = 1$.

\paragraph{\bf{Case b.} $(s , \b) \in[0.04, 0.5] \times [1, 2]$  }
In this domain, $F(s, \b)$ is smooth. Denote
\[M(\b) \teq \max_{s \in [0.04, 0.5]} F(s, \b). \]
For $\b = 2$, we can estimate  $F(s,\b)$ using Matlab (evaluating the function at some discrete points with error control and then using the boundedness of  $\pa_sF(s,\b)$ to estimate other points)
\[
\bal
M(2) & =  \max_{s \in [0.04, 0.5]} F(s, 2) \leq 1.0123 \; , \\
 \Rightarrow    M(\b) - 1 &\leq M(2) - 1 \leq 0.0123  \leq  0.015(\b-1)  \quad   \forall \b \in [1.85, 2].
\eal
\]
The constant $0.015$ in \eqref{eq:comp_F} comes from the above inequality $M(2) \leq 1.0123$. Similarly, we can estimate $M(1.85), M(1.5), M(1.1)$ to get
\[
\bal
M(1.85) \leq 1.0071 & \Rightarrow M(\b) - 1 \leq M(1.85) - 1 \leq 0.0071  \leq  0.015(\b-1) \;, \    \forall \b \in [1.5, 1.85]  ,\\
 M(1.5) \leq 1.0007  & \Rightarrow M(\b) - 1 \leq 0.0007  \leq  0.015(\b-1) \;, \quad    \forall \b \in [1.1, 1.5]  ,\\
M(1.1) \leq 0.9989  & \Rightarrow M(\b) - 1 \leq 0  \leq  0.015(\b-1) \;, \quad    \forall \b \in [1, 1.1]  .\\
\eal
\]
Therefore, for $s, \b \in [0.04, 0.5] \times [1,2]$, the above inequalities imply $F(s, \b) \leq M(\b) \leq 1 + 0.015(\b-1)$. All the above inequalities on $\max F(s, \b)$ can be established rigorously using the method discussed in paragraphs (c), (d) in Section \ref{Numerical-verification}.

\paragraph{\bf{Case c.} $(s , \b) \in[0, 0.04] \times [1, 2]$  }
The partial derivative  $\pa_{\b} F(s, \b)$  is given by
\[
\pa_{\b} F(s, \b)  =  - \f{2 s^{1+\b}   \log s   }{ (1 + s^{1+\b})^2 }     \f{2s}{1-s^2} \lt(\log \B|\f{ s+1}{s-1}\B| \rt)^{-1}.
\]
It is not singular near $s = 0$ due to the power $s^{2 + \b}$. Since $s^{1+\b}$ is decreasing with respect to $\b$ and $t / (1+t)^2 $ is increasing for $t \in [0, 1]$,  we get
\[
 \f{ s^{1+\b}} { (1 + s^{1+\b})^2 }    \leq \f{s^2}{ (1+s^2)^2 }.
\]
Notice that $-\log(s) \geq 0  $ for $s\in[0, 1]$. We obtain
\[
0 \leq  \pa_{\b} F(s, \b)   \leq \f{-2 s^2 \log s} { (1 + s^2)^2}   \f{2s}{1-s^2} \lt(\log \B|\f{ s+1}{s-1}\B| \rt)^{-1} \teq H(s).
\]
For $s \in [0,0.04]$, we can estimate $H(s)$ numerically with rigorous error control using the method discussed in paragraphs (c), (d) in Section \ref{Numerical-verification} and obtain $H(s) \leq 0.011$. Therefore for $s \in[0,0.04]$, we yield
\[
F(s, \b) \leq F(s, 1) + 0.011(\b - 1) \leq 1 + 0.015 (\b-1)  .
\]
Combining case (a), (b) and (c) , we conclude the proof of \eqref{eq:comp_F} and Lemma \ref{lem:comp_F}.

\end{proof}

\begin{proof}[ Proof of Lemma \ref{lem:compare_S1} ]
Assume that $S^1$ is $\pi$- periodic. Then we have
\[
u_x(x) = \f{1}{\pi} \int_0^{\pi} \om(y) \cot(x - y) dy , \quad u(x) = - \f{1}{\pi} \int_0^{\pi/2} \log \B|
 \f{\tan y + \tan x}{\tan y - \tan x}  \B| \om(y) dy.
\]
Symmetrizing the convolution kernel as we did in the proof of Lemma \ref{lem:compare} and \eqref{eq:comp_ineq}, we obtain 
\beq\label{eq:comp_ineq_S1}
\bal
I \teq  \f{1}{\pi}\int_0^{\pi/2}  ( u_x \om - a u \om_x)   \lt( \cot x\rt)^{\b} dx
 = \f{1}{\pi} \int_0^{\pi} \om(x) \om(y) K_{sym}(x, y) dx dy,
\eal
\eeq
where the symmetrized kernel $K_{sym}$ is given by 
\beq\label{eq:sym_S10}
\bal
K_{sym}(x, y) &=\f{1+a}{2} (\cot y)^{\b-1} \lt(  \tau (s^{\b-1}+1 )\log \B|  \f{s+1}{s-1}   \B| - (s^{\b-1}-1)\f{2s}{s^2-1} \rt) \\
& +  \f{1+a}{2} (\cot y)^{\b+1} \lt(  \tau (s^{\b+1}+1 )\log \B|  \f{s+1}{s-1}   \B| - (s^{\b+1}-1)\f{2s}{s^2-1} \rt),
\eal
\eeq
with 
\[
\tau = \f{a \b}{1+a},  \quad s = \f{\tan y} {\tan x}  = \f{\cot x}{\cot y}.
\]

Since $a, \;\b$ satisfy the assumption in Lemma \ref{lem:compare}, from the proof of Lemma \ref{lem:compare}, we know that
\beq\label{eq:sym_S11}
 \tau (s^{\b+1}+1 )\log \B|  \f{s+1}{s-1}   \B| - (s^{\b+1}-1)\f{2s}{s^2-1}  \geq 0 , \quad \forall s  \geq 0.
\eeq
Recall $\b > 1$. For $s \in [0, 1]$, we have $s^{\b-1} \geq s^{\b+1}$ and
\beq\label{eq:sym_S12}
\f{s^{\b-1} + 1}{1 - s^{\b-1}}  (1 - s^2)  \geq  \f{s^{\b+1} + 1}{1 - s^{\b+1}}  (1 - s^2) .
\eeq
It is not difficult to show that the above inequality also holds true for $s \geq 1$ if we replace $s$ by $s^{-1}$.
Combining \eqref{eq:sym_S11} and \eqref{eq:sym_S12}, we get
\[
\f{s^{\b-1} + 1}{1 - s^{\b-1}}  (1 - s^2)  \geq  \f{s^{\b+1} + 1}{1 - s^{\b+1}}  (1 - s^2) \geq
\f{1}{\tau} 2s \lt(\log \B|  \f{s+1}{s-1}   \B|\rt)^{-1} ,
\]
which implies 
\beq\label{eq:sym_S13}
\bal
 \tau (s^{\b-1}+1 )\log \B|  \f{s+1}{s-1}   \B| - (s^{\b-1}-1)\f{2s}{s^2-1}  \geq 0 , \quad \forall s  \geq 0.
 \eal
\eeq
Substituting \eqref{eq:sym_S11} and \eqref{eq:sym_S13} in \eqref{eq:sym_S10}, we conclude
\[
K_{sym}(x , y) \geq 0 ,  \quad \forall x, y \in [0, \pi/2].
\]
Finally, noticing that $\om(x) \om(y) \geq 0 $ for all $x, y\in [0, \pi/ 2]$, we prove $I \geq 0$ in \eqref{eq:comp_ineq_S1}.
\end{proof}

\vspace{0.2in}
\noindent
{\bf Acknowledgments.} The authors would like to acknowledge the generous support from the National Science Foundation under Grant No. DMS-1907977 and DMS-1912654. We would like to thank Tarek Elgindi for valuable comments and suggestion related to Remark \ref{rem:wx}. We would also like to thank the two referees for their constructive comments on the original manuscript, which improve the quality of our paper.

\bibliographystyle{plain}
\bibliography{selfsimilar}

\end{document}


\begin{abstract}
In \cite{chen2019finite}, we have presented the weighted $L^2$ and $H^1$ estimates of the linearized equation in the case of the De Gregorio model with $a=1$ and obtained linear stability. The purpose of this supplementary material is to provide relatively sharp estimates for the cross terms in the weighted $H^1$ estimate and the nonlinear terms. These sharp estimates provide an estimate of the upper bound of the residual error, which is not too small but is sufficient for us to close the bootstrap argument in the nonlinear stability analysis. This enables us to use only a personal laptop computer and a mesh with mesh size $h=2 \cdot 10^{-5}$ to construct our approximate self-similar profile with a residual error less than the above upper bound. In addition, we will establish rigorous estimate for the residual error of the approximate self-similar profile in the energy norm, and estimate the time derivative of the solution, which will be used to establish the convergence of the solution to the self-similar solution. We also use the Interval arithmetic software and error analysis for the Trapezoidal rule and Gaussian quadrature to estimate a number of integrals involving our approximate self-similar profile with rigorous error control. We will provide the estimates for the cross terms in Section \ref{sec:cross}, the error term in Section \ref{sec:error}, the nonlinear terms in Section \ref{sec:non}, the energy in the bootstrap argument in Section \ref{sec:boot}, the time derivative of the solution in Section \ref{sec:convg}, and the error of the Gaussian quadrature in Section \ref{sec:GQ}.
\end{abstract}

 \maketitle

\setcounter{section}{6}
\section{Lemmas, Functions and Notations}
We first introduce some Lemmas, functions and notations that were used in \cite{chen2019finite}.

\subsection{Functions}

The weight $\vp$ in the $L^2$ estimate is defined as follows:
\beq\label{eq:wg_a1}
\vp \teq  \lt(- \f{1}{x^3} - \f{e}{x} - \f{f \cdot 2 x}{L^2 - x^2} \rt)
\cdot \lt( \chi_1  \lt(\bar{\om} - \f{x \bar{\om}_x }{5} \rt) + \chi_2 \lt( \bar{\om} -\f{ (x-L) \bar{\om}_x}{3}  \rt)  \rt)^{-1},
\eeq
where $\chi_1, \chi_2 \geq 0$ are some cutoff functions such that $\chi_1 + \chi_2 =1 $ and
\[
\chi_1(x) = \begin{cases}
1  & x \in [0,4] \\
0 & x \in [6, 10] \\
\end{cases} ,\quad
\chi_1(x) = \f{  \exp\lt( \f{1}{x-4} + \f{1}{x-6}  \rt) }{ 1 +  \exp\lt( \f{1}{x-4} + \f{1}{x-6}  \rt)  } \
\forall  x \in[4, 6] \; .
\]
The weight $\psi$ in the $H^1$ estimate is given below:
\beq\label{eq:wg_a1_H1}
\psi =  - \f{1}{\bar{\om}} \lt( \f{1}{x} - \f{x}{L^2} \rt) , \quad x \in [0, L].
\eeq

The normalization conditions for $c_l, c_{\om}$ are defined as follows:
  \beq\label{eq:DGnormal}
c_l =  - \f{ u(L)}{L},  \quad  c_{\om} = c_l \; .
 \eeq

The inner product is defined on the interval $[0, L]$ 
\[
\la f, g \ra \teq  \int_0^L f \cdot g dx,
\]
since the support of $\om, \bar{\om}$ lies in $[-L, L]$. We only consider the half real line due to symmetry.

The nonlinear and the error terms in the weighted $L^2$ and $H^1$ estimates are defined below: 
\beq\label{eq:NF_a1}
\bal
N(\om) = ( c_{\om} + u_x ) \om - (c_l x + u) \om_x, \quad  & F(\bar{\om}) = ( \bar{c}_{\om} + \bar{u}_x ) \bar{\om}- (\bar{c}_l x +\bar{ u}) \bar{\om}_x , \\
 N_1 \teq \la N(\om), \om \vp \ra , \quad  F_1 \teq  \la F(\bar{\om}) , \om  \vp \ra,  \quad 
 & N_2  \teq \la (N(\om))_x, \om_x \psi \ra  , \quad F_2 \teq  \la ( F(\bar{\om}) )_x, \om  \vp \ra .  
 \eal
\eeq

We express $u_x(0), u_x(L)$ and $c_{\om} = c_l$ as the projection of $\om$ onto some functions
\beq\label{eq:defg}
\bal
&c_{\om} = c_l = \la \om , g_{c_{\om}} \ra, \quad u_x(0) = \la \om , g_{u_x(0)} \ra, \quad u_x(L) = \la \om, g_{u_x(L)} \ra, \\
&g_{c_{\om}}  \teq  \f{1}{L \pi }  \log \lt( \f{  L+x }{L-x}\rt)  ,\quad g_{u_x(0)} \teq - \f{2}{\pi x} , \quad g_{u_x(L)} \teq \f{2 x} { \pi (L^2 - x^2)} . \\
\eal
\eeq

\subsection{Parameters and Notations}\label{sec:parameter_notation} 
We will use the following parameters and notations
\[ L=10,\quad M=5,\quad n = 8000,\quad  h_0 =\f{L}{n}= \f{1}{800}, \]
\[N =  400000,\quad h = \f{L}{N} = 2.5 \cdot 10^{-5} ,\quad x_i = ih, i=0,1,..,N. \]
We introduce the following constants which can be estimated accurately
\beq\label{eq:const_a1}
\bal
K_1 &\teq ||\bar{\om}_x||_{L^{\infty}[0,L]}, \quad K_{1 r}   \teq ||\bar{\om}_x||_{L^{\infty}[M,L]},
\quad K_2 \teq || \bar{u}_x||_{L^{\infty}[0,L]}  ,\\
K_3  & \teq || \bar{c}_l + \bar{u}_x||_{L^\infty[0,L]} , \quad  K_4 \teq || \bar{u}_{xx} ||_{L^{\infty}[0,M]},\quad K_{5l} \teq \B|\B| \f{\bar{\om}}{x (x^2 - L^2)} \B|\B|_{L^{\infty}[0,M]},\\ 
K_{5} &\teq \B|\B| \f{\bar{\om}}{ (x^2 - L^2)} \B|\B|_{L^{\infty}[0,L]} ,\quad K_6 \teq || (\bar{\om}_{xx}(x^2 - L^2))_x ||_{L^{\infty}[0.25,L - 0.25]}, \\
K_7 &\teq || \bar{\om}_{xx}(x^2 - L^2) ||_{L^{\infty}[0,L]},  \quad K_8 = || \bar{\om}_{xx}||_{L^{\infty}[0.25,L-0.25]} ,\\
J_1 & \teq \int_0^M \bar{\om}^2_{xxx} dx ,\quad J_2 \teq \int_0^M \f{\bar{\om}^2_{xx}}{x^2} dx,
 \quad J_3 \teq \int_0^L  \f{ (\om_{x} - \om_{x}(0) )^2}{x^2} dx , \\
J_4 &\teq  \int_0^L \bar{\om}^2_{xx}(x^2-L^2)^2dx ,\quad J_5 \teq \int_0^L \bar{\om}^2_{xxx}(x\wedge (L-x))^2 dx ,\\
J_{5r} &\teq \int_M^L \bar{\om}^2_{xxx}(x\wedge (L-x))^2 dx ,\quad J_6 \teq \int_0^L \bar{\om}_x^2 dx ,\\
J_7  &\teq \int_0^L \bar{\om}^2_{xx} dx,\quad  J_{7r}  \teq \int_M^L \bar{\om}^2_{xx} dx.
\eal
\eeq
Using the analytic estimates and the interval arithmetic introduced in \cite[Section 4]{chen2019finite}, we can obtain rigorous upper bounds for the positive quantities above. In particular, for each quantity $q$, our numerical verification using the interval arithmetic will provide an interval $[q_l,q_r]$ that rigorously contain the exact value, where $q_l,q_r$ are two accurate real numbers with 16 digits. We then round up $q_r$ to 4 significant digits to be a strict upper bound of $q$. One can verify that 
\begin{equation}\label{eqt:numeric_estimates}
\bal
& K_1< 1.001,\quad K_{1,r} < 0.8092,\quad  K_2 < 3.606 \quad K_3 <  2.907,\quad K_4 < 0.9296, \\
& K_{5,l}\leq 1.001\times 10^{-2},\quad K_5\leq 4.433\times 10^{-2} , \quad K_6 < 11.13 ,\quad K_7 < 26.65, \quad K_8 < 0.3210, \\
& J_1^{1/2}\leq 0.1683, \quad J_2^{1/2}\leq 0.2141,\quad J_3^{1/2}\leq 0.5757,\quad J_4^{1/2}\leq 53.37,\quad J_{5}^{1/2} \leq 0.5431, \\
& J_{5,r}^{1/2}\leq 0.4673,\quad  J_6^{1/2}\leq 2.096,\quad J_7^{1/2} \leq 0.6965,\quad J_{7,r}^{1/2}\leq 0.4500 ,\\
&|| (x^2 -L^2)^{-1} \psi^{1/2}\vp^{-1/2} ||_{L^{\infty}[0.5,9.5]} < 0.1, \quad ||\psi_x  \psi^{-1/2} \vp^{-1/2} ||_{L^{\infty}[0.5,9.5]} <2.
\eal
\end{equation}

\subsection{Rigorous verification of the numerical values}

All the numerical computations and quantitative verifications are performed by MATLAB (version 2019a) in the double-precision floating-point operation. The MATLAB codes can be found via the link \cite{Matlabcode}.

\subsubsection{Accurate point values of $\bar{u}, \bar{u}_x, \bar{u}_{xx}$}\label{sec:intro_GQ}
Firstly, the approximate profile $\bar \om$ is a piecewise cubic polynomial with $n = 8000$ subintervals, i.e. $[ (i-1)h_0,  i h_0], i = 1,2,.., n $. We use the Hermite basis to obtain the exact values of $\bar \om, \bar \om_x$ on the finer mesh $ x_i = ih, h = 2.5 \cdot 10^{-5}, i =1,2.., N = 400000$. Then we compute $\bar{u}, \bar{u}_x, \bar{u}_{xx}$ on the same grid points $x_i $, 
which involves the Hilbert transform that is a global integral of $\bar \om$.  
To approximate the integral, we use a mixture of analytic integration and the composite $8$-point Legendre-Gauss integral. 
We focus on the computation of $\bar{u}_x = H \bar{\om}$ and $\bar{u}_{xx}, \bar{u}$ are computed similarly. Denote $m= 8$. 
We fix some $i \in \{ 1,2.., n-1, n\}$ and $ x \in [ (i-1) h_0,  i h_0 ]$. We remark that the interval $[(i-1) h_0, i h_0]$ contains $ N / n + 1 = 51$ grid points $x_j$. Since $\bar{\om}$ is odd, we have
\[
\bar{u}_x(x) = H \bar{\om} = \f{1}{\pi} P.V. \int_0^L \bar{\om}(y) ( \f{1}{x-y} -  \f{1}{x+y} ) dy.
\]

Denote $J_0 \teq \{ 0, 1,2,.., n-1\}$ and 
\beq\label{eq:Jset}
\bal
&J_1 \teq \{ j \in J_0: | j -(i-1)| \leq m \}, \quad J_2 \teq J_0 \bsh J_1 ,  \\
 &J_3\teq \{ j \in J_0: | j +(i-1)| \leq m \}, \quad J_4 \teq J_0 \bsh J_3 .\\
\eal
\eeq
$J_1$ and $J_3$ denote the indexes that are close to $i-1$.
Since $[0,L] = \cup_{ 0\leq j \leq n-1} [jh_0, (j+1) h_0]$, we can partition $[0,L]$ into singular region and nonsingular region according to $x-y$ and $x+y$
  \beq\label{eq:GQ_error0}
  \bal
\bar{u}_x(x) &= \f{1}{\pi} P.V. ( \int_{ \cup_{ j \in J_1} [jh_0, (j+1) h_0] } \f{\bar{\om}(y)}{x-y} dy
- \int_{ \cup_{ j \in J_3} [jh_0, (j+1) h_0] } \f{\bar{\om}(y)}{x+y} dy  )  \\
& + \f{1}{\pi}   \int_{ \cup_{ j \in J_2} [jh_0, (j+1) h_0 ] } \f{\bar{\om}(y)}{x-y} dy
- \f{1}{\pi}\int_{ \cup_{ j \in J_4} [jh_0, (j+1) h_0] } \f{\bar{\om}(y)}{x+y} dy  \teq  I_1 + I_2 + I_3. 
\eal
 \eeq
 Since in each interval $[jh_0, (j+1) h_0]$, $\bar{\om}$ is a piecewise polynomial, we have closed-form formula for the integral $\int_{ [jh_0, (j+1)h_0]}  \f{\bar{\om}(y)}{  x \pm y }dy$. We use this  formula to obtain the exact expression of $I_1$. We remark that in the exact expression, we do not have singular term due to the principal integral. For $I_2, I_3$, we cannot apply the exact formula of the integral since evaluating $F( (l+1)h_0 ) - F( l h_0)$ for some regular function $F$ and large $l$ will lead to significant roundoff error. We use the $8$-point Legendre-Gaussian quadrature instead.

Let $A_i, z_i, i=1,2,..,K$ be the weights and nodes in the $K$-point Legendre-Gauss quadrature rule and $\bar{A}_i, \bar{z}_i$ be their numerical approximations. For $f \in C^{2K}[a,b]$, we denote by $GQ_K(f, a,b)$ the Gaussian quadrature and by $NGQ_K(f, a,b)$ its numerical realization 
\beq\label{def:GQ}
\bal
GQ_K(f, a,b) &\teq \f{b-a}{2} \sum_{ i =1}^K  A_i f( \f{b-a}{2} z_i + \f{b+a}{2} ), \\ 
NGQ_K(f, a,b) &\teq\f{b-a}{2} \sum_{ i =1}^K \bar{A}_i f( \f{b-a}{2} \bar{z}_i  + \f{b+a}{2} ).
\eal
\eeq


Since the exact values of $A_i, z_i$ are related to the roots of some high order polynomial, we do not have their closed-form formulas. In the numerical realization of the Gaussian quadrature, we need to use their approximations $\bar{A}_i, \bar{z}_i$ and $NGQ_K(f,a,b)$ instead. We use composite $NGQ_K$ with $K=8$ to compute $I_2, I_3$ in \eqref{eq:GQ_error0}. 

There are two types of errors in this computation. The first type of error is the roundoff error in the computation, which can be addressed using the interval arithmetic. The second type of error is due to the numerical Gaussian quadrature, i.e. $\int_a^b f(x) dx - NGQ_K(f, a, b)$. 
The error estimates of these computations only depend on $|| \pa_x^i \bar \om ||_{L^{\infty}}, i =0,1,2,3$ which have been estimated in \eqref{eq:const_a1}. In particular, we have the following estimate of the second type of error in the computation of $\bar u_x, \bar u , \bar u_{xx}$ 
\[
Error_{GQ}(u_x) < 2 \cdot 10^{-17}, \quad 
Error_{GQ}(u) < 2 \cdot 10^{-19}, \quad 
Error_{GQ}(u_{xx}) < 5 \cdot 10^{-18}.
\]
We will present the detailed derivation of the above estimates in Section \ref{sec:GQ}. 

We remark that although these errors are small, we have to take them into 
account in the interval representations of $\bar{u}_x$ (and also $\bar{u}, \bar{u}_{xx}$). That is, each $\bar u_x(x)$ will be represented by $[\lfloor\bar u_x(x)-\epsilon\rfloor_f,\lceil\bar u_x(x)+\epsilon\rceil_f]$ in any computation using the interval arithmetic, where $\lfloor\cdot\rfloor_f$ and $\lceil\cdot\rceil_f$ stand for the rounding down and rounding up to the nearest floating-point value, respectively. 


\subsubsection{Rigorous estimate of $L^\infty$ norms}\label{sec:RE_L_infty}
 We explain how to rigorously bound the $L^\infty$ norms given in \eqref{eq:const_a1}. Consider a general function $g:[0,L]\rightarrow\mathbb{R}$. Suppose that we are given the grid points values $g(x_i)$ for $i=0,\dots,N$. Let $g^{\max}=\{g^{\max}_i\}_{i=1}^N$ be a sequence such that $g^{\max}_i \geq \max_{x\in[x_{i-1},x_i]} |g(x)|$. That is, $g^{\max}$ dominates the piecewise maximum of the absolute value of $g$. If we can construct such a sequence $g^{\max}$, then we can simply bound $\|g\|_{L^\infty}\leq \max_i g^{\max}_i$.

If we are also given a bound $C$ on the $L^\infty$ norms of $g_x$ (i.e. $\|g_x\|_{L^\infty}\leq C$), then we can choose $g^{\max}_i = |g(x_{i-1})| + hC$. This method is useful for estimating the $L^\infty$ norms of $\bar{\om}$ and its derivatives. Since $\bar{\om}$ is piecewise cubic polynomial, $\bar{\om}_{xxx}$ is piecewise constant and thus $\bar{\om}_{xxx}^{\max}$ can be given by grid point values. With $\bar{\om}_{xxx}^{\max}$ in hand, we can further compute $\bar{\om}_{xx}^{\max}$, $\bar{\om}_x^{\max}$ and $\bar{\om}^{\max}$ recursively.

Or, if we know that $g$ is monotone, then we can choose $g^{\max}_i = \max\{|g(x_{i-1})|,|g(x_i)|\}$. This works for explicit functions such as $g(x) = x^{-1}$.

For example, let us explain how to estimate $K_5=\|g\|_{L^\infty}$ with $g(x) = \bar{\om}(x)/(L^2-x^2)$. We first choose a small number $\varepsilon$ and choose $g^{\max}_i=\bar{\om}^{\max}_i/(L^2-x_i^2)$ for each index $i$ such that $x_i\leq L-\varepsilon$. Note that $(L^2-x_i^2)^{-1}\geq (L^2-x^2)^{-1}$ for $x\in [x_{i-1},x_i]$. For $x> L-\varepsilon$, since $\bar{\om}(L)=0$, we have
\[g(x) = -\frac{1}{L+x}\cdot\frac{\bar{\om}(x)}{x-L} = - \frac{\bar{\om}_x(\xi(x))}{L+x}\quad \text{for some $\xi(x)\in[x,L]\subset[L-\varepsilon,L]$}.\]
Therefore, we can choose $g^{\max}_i = (L+x_{i-1})^{-1}\cdot\max_{x\in[L-\varepsilon,L]}|\bar{\om}_x(x)|$ for each index $i$ such that $x_i> L-\varepsilon$. Note that $\max_{x\in[L-\varepsilon,L]}|\bar{\om}_x(x)|$ can be obtained from $\bar{\om}_x^{\max}$. The parameter $\varepsilon$ needs to be chosen carefully. On the one hand, $\varepsilon$ should be small enough so that the bound $|\bar{w}(\xi(x))|\leq \max_{\tilde{x}\in[L-\varepsilon,L]}|\bar{\om}_x(\tilde{x})|$ is almost tight for $x\in[L-\varepsilon,L]$. On the other, the ratio $\varepsilon/h$ must be large enough so that the estimate $(L^2-x_i^2)^{-1}\geq (L^2-x^2)^{-1}$ is sharp for $x$ close to $L-\varepsilon$. In this example, we choose $\varepsilon=0.1$.

\subsubsection{Rigorous estimate of $K_2$ and $K_3$}
Specially, we explain how to obtain a sequence $\bar{u}_x^{\max}$ that helps bound $K_2 = \|\bar{u}_x\|_{L^\infty[0,L]}$. In particular, we compute that 
\begin{align*}
\max_{x\in[x_{i-1},x_i]} |\bar{u}_x(x)| &\leq |\bar{u}_x(x_{i-1})| + \int_{x_{i-1}}^{x_i}|\bar{u}_{xx}(y)|dy\\
&\leq |\bar{u}_x(x_{i-1})| + \sqrt{h}\cdot\left(\int_{x_{i-1}}^{x_i}|\bar{u}_{xx}(y)|^2dy\right)^{1/2}\leq |\bar{u}_x(x_{i-1})| + \sqrt{h}\cdot\|\bar{u}_{xx}\|_{L^2[0,L]}.
\end{align*}
Buy Lemma~\ref{lem:commute}, we know that $\|\bar{u}_{xx}\|_{L^2[0,L]} = \|\bar{\om}_x\|_{L^2[0,L]}$. Given a bound $C$ on $\|\bar{\om}_x\|_{L^2[0,L]}$, we can choose 
\[(\bar{u}_x^{\max})_i = |\bar{u}_x(x_{i-1})| + \sqrt{h}\|\bar{\om}_x\|_{L^2[0,L]}.\]
An estimate of $K_3$ can be obtained similarly.

\subsubsection{Rigorous estimate of $K_4$} It requires a different strategy to bound $K_4 = \|\bar{u}_{xx}\|_{L^{\infty}[0,M]}$. We first obtain a crude but rigorous bound on $\|\bar{u}_{xxx}\|_{L^{\infty}[0,M]}$. Using Lemma~\ref{lem:commute}, we can compute that 
\begin{align*}
\bar{u}_{xxx}(L^2-x^2) &= \frac{1}{\pi}\mathrm{P.V.}\int_{-L}^L\frac{\bar{\om}_{xx}(y)(L^2-y^2)}{x-y}dy\\
&= \frac{1}{\pi}\mathrm{P.V.}\int_{x-a}^{x+a}\frac{\bar{\om}_{xx}(y)(L^2-y^2)}{x-y}dy + \frac{1}{\pi}\int_{[-L,L]\backslash[x-a,x+a]}\frac{\bar{\om}_{xx}(y)(L^2-y^2)}{x-y}dy\\
&\teq I + II
\end{align*}
for some $a\in[0,L-M]$. We then estimate $I$ and $II$: 
\begin{align*}
|I| &= \frac{1}{\pi}\left|\mathrm{P.V.}\int_{x-a}^{x+a}\frac{\bar{\om}_{xx}(y)(L^2-y^2)}{x-y}dy\right| \\
&= \frac{1}{\pi}\left|\int_0^a\frac{\bar{\om}_{xx}(x+\tau)(L^2-(x+\tau)^2)-\bar{\om}_{xx}(x-\tau)(L^2-(x-\tau)^2)}{\tau}d\tau\right|\\
&\leq \frac{2a}{\pi}\|(\bar{\om}_{xx}(L^2-x^2))_x\|_{L^\infty[0,M+a]}.
\\
|II| &\leq \frac{1}{a\pi} \int_{-L,L} |\bar{\om}_{xx}(y)(L^2-y^2)|dy =  \frac{2}{a\pi}\|\bar{\om}_{xx}(L^2-x^2)\|_{L^1[0,L]}.
\end{align*}
Note that $\|(\bar{\om}_{xx}(L^2-x^2))_x\|_{L^\infty[0,M+a]}$ and $\|\bar{\om}_{xx}(L^2-x^2)\|_{L^1[0,L]}$ can both be rigorously bounded using $\bar{\om}_{xx}^{\max}$ and $\bar{\om}_{xxx}^{\max}$. It then follows that, for $x\in[0,M]$,
\begin{align*}
|\bar{u}_{xxx}(x)|&\leq \frac{1}{L^2-M^2}|\bar{u}_{xxx}(L^2-x^2)|\\
&\leq\frac{1}{L^2-M^2}\left(\frac{2a}{\pi}\|(\bar{\om}_{xx}(L^2-x^2))_x\|_{L^\infty[0,M+a]}+ \frac{2}{a\pi}\|\bar{\om}_{xx}(L^2-x^2)\|_{L^1[0,L]}\right).
\end{align*}
We hence obtain a bound $C$ on $\|\bar{u}_{xxx}\|_{L^{\infty}[0,M]}$. We can then estimate $\|\bar{u}_{xx}\|_{L^{\infty}[0,M]}$ using the preceding techniques. In particular, we will choose $a=4$ in our computation.

\subsubsection{Rigorous estimate of $L^2$ norms}\label{sec:RE_L_2} 
Once we obtain the sequence $g^{\max}$ for a function $g$, we can also use it provide a rigorous bound on $\|g\|_{L^2}$. In fact, we have
\[\|g\|_{L^2[0,L]^2} \leq h\cdot\sum_{i=1}^N(g^{\max}_i)^2.\]

\subsubsection{Rigorous estimate of integrals}\label{sec:RE_inner} 
More generally, we can also rigorously estimate the integral of a function $g$ on $[0,L]$, which is needed in the following. In particular, we want to obtain $c_1,c_2$ such that $c_1\leq \int_0^Lg(x)dx\leq c_2$. To do so, we first construct a pair of sequences $g^{up}=\{g^{up}_i\}_{i=1}^N,g^{low}=\{g^{low}_i\}_{i=1}^N$ such that 
\[g^{up}_i \geq \max_{x\in[x_{i-1},x_i]}g(x)\quad\text{and}\quad g^{low}_i \leq \min_{x\in[x_{i-1},x_i]}g(x).\] 
The rigorous construction of $g^{up}$ or $g^{low}$ is similar to the construction of $g^{\max}$. We omit the details here. Then we have the estimate
\[h\cdot \sum_{i=1}^Ng^{low}_i \leq \int_0^Lg(x)dx\leq h\cdot\sum_{i=1}^Ng^{up}_i.\]

The case where $g(x)=f_1(x)/f_2(x)$ may need a special treatment. For example, if $f_1$ and $f_2$ are both singular at $x=0$, we first need to find a weight function $v(x)$ such that $\tilde{f}_1= vf_1$ and $\tilde{f}_2=vf_2$ are both regular on $[0,L]$. We remark that such $v$ exists in all of our computations. Then we construct $\tilde{f}_1^{up},\tilde{f}_1^{low}$ and $\tilde{f}_2^{up},\tilde{f}_2^{low}$, and use standard interval arithmetic to find $g^{up},g^{low}$. 

For example, in Section~\ref{sec:cross} we will need to estimate the integral of a function like $g(x)=f(x)/\varphi(x)$ on $[0,L]$, where $\varphi$ is given by \eqref{eq:wg_a1}. Note that $\varphi(x)= O(x^{-4})$ near $x=0$ and $\varphi(x) =O((L-x)^{-2})$ near $x=L$. To avoid working with a singular function numerically, we introduce new functions $\tilde{\varphi}(x) = x^4(L-x)^2\varphi(x)$ and $\tilde{f}(x) =  x^4(L-x)^2 f(x)$. We compute $\tilde{\varphi}^{up},\tilde{\varphi}^{low}$ and $\tilde{f}^{up},\tilde{f}^{low}$, and then use interval arithmetic to find $g^{up},g^{low}$.

\section{The cross term in the weighted $H^1$ estimate}\label{sec:cross}

In \cite{chen2019finite}, we have derived the following weighted $H^1$ estimate:
\beq\label{eq:lin_a1_H1}
\bal
\f{1}{2} \f{d}{dt} \la \om_x^2, \psi \ra & =  \B\la   \f{1}{2\psi} (  ( \bar{c}_l x +  \bar{u} ) \psi )_x  ,\om_x^2 \psi \B\ra + \la \bar{u}_{xx} \om, \om_x \psi \ra -\la  (c_l x + u) \bar{\om}_{xx}, \om_x \psi \ra  \\
&+ \la N(\om)_x, \om_x \psi \ra +
\la F(\bar{\om})_x , \om_x  \psi \ra   \teq I + II_2  + N_2 + F_2 \;,
\eal
\eeq
where $I, II_2$ are defined below 
\beq\label{eq:I_II}
I =  \B\la   \f{1}{2\psi} (  ( \bar{c}_l x +  \bar{u} ) \psi )_x  ,\om_x^2 \psi \B\ra + \la \bar{u}_{xx} \om, \om_x \psi \ra , \quad II_2 = -\la  (c_l x + u) \bar{\om}_{xx}, \om_x \psi \ra.
\eeq
We further decompose $I$ into the a damping term and a cross term 
\beq\label{eq:lin_a1_I2}
I = \la D_2(\bar{\om}), \om_x^2 \psi \ra + I_2 , \quad D_2(\bar{\om}) \teq \f{1}{2\psi} (  ( \bar{c}_l x +  \bar{u} ) \psi )_x,   \quad I_2 \teq \la \bar{u}_{xx} \om, \om_x \psi \ra .
\eeq

In this Section, we estimate the cross terms $I_2, II_2$ in the weighted $H^1$ estimate and prove Lemma \ref{lem:cross}.
\begin{lem}\label{lem:cross}
The weighted $H^1$ estimate satisfies
\beq\label{eq:lem_cross}
\f{1}{2} \f{d}{dt}  \la \om_x^2, \psi \ra  = I + II_2 + N_2 + F_2
\leq - 0.25 \la \om_x^2 ,\psi \ra + 7.5 \la \om^2, \vp \ra +  N_2 + F_2,
\eeq
where $I, II_2$ are defined in \eqref{eq:I_II}.
\end{lem}

In \cite{chen2019finite}, we have shown that $\la D_2(\bar{\om}),\om^2 \psi \ra$ in $I$ is a damping term. To control the cross term, one can use the Cauchy-Schwarz inequality and split it into weighted $L^2$ norm of $\om$ and $\om_x$. However, this estimate is not sharp and leads to a large constant in our estimate, which is roughly of order $100 \la \om^2, \vp \ra$, in Lemma \ref{lem:cross}. Then to close the bootstrap argument, it requires the error term to be extremely small, which increases the computational burden. To overcome this difficulty, we will use part of the damping term and some property of the profile to control the cross term, so that we can get a smaller constant in Lemma \ref{lem:cross}, which is about $7.5 \la \om^2 ,\vp \ra$. 

We first consider $I_2 = \la \bar{u}_{xx} \om, \om_x \psi \ra$ in $I$ \eqref{eq:lin_a1_I2}. Note that $\bar{u}_{xx} \psi$ is increasing on the grid points. Formally, using integration by parts, we get
\beq\label{eq:a1_H1_rem2}
I_2  =\la \bar{u}_{xx} \om, \om_x \psi \ra =  - \f{1}{2}\la (\bar{u}_{xx} \psi)_x , \om^2 \ra  \leq 0 .
\eeq
However, since $\bar{u}_{xxx}$ is not continuous, we cannot justify $(\bar{u}_{xx} \psi)_x \geq 0$. To address this problem, we consider the following piecewise linear approximation of $\bar{u}_{xx} \psi$
\[
S_{k, h}(x) \teq \begin{cases}
(\bar{u}_{xx} \psi ) ( 0.5) &  0 \leq x \leq 0.5, \\
 (\bar{u}_{xx} \psi)( x_i) + ( ( \bar{u}_{xx} \psi)( x_{i+1} )- (\bar{u}_{xx} \psi)( x_i) ) \f{ x- x_i}{h} &x \in [x_i , x_{i+1}] \subset[0.5, 9.5]  ,\\
 ( \bar{u}_{xx} \psi) (9.5) & x \geq 9.5,\\
\end{cases}
\]
where $\bar{u}_{xx}(x_i), x_i = ih$ is evaluated using the Hilbert transform (see \eqref{eq:commute4})
\beq\label{eq:cross_hil}
\bar{u}_{xx} = H \om_x , \quad \bar{u}_{xx}(x^2 - L^2) = H(\bar{\om}_x(x^2 - L^2)) .
\eeq
For $x$ close to $0$ or $L$, we construct $S_{k, h}(x)$ to be constant since $\bar{u}_{xx} \psi$ blows up at $x =0, L$. We can verify rigorously that $S_{k, h}(x)$ is monotonically increasing by checking its values on the grid points since $S_{k,h}$ is piecewisely linear. Using the fact that $S_{k, h}(x)$ is monotonically increasing, we can justify \eqref{eq:a1_H1_rem2} as follows
\beq\label{eq:aa_H1_rem22}
\bal
\la \bar{u}_{xx} \om, \om_x \psi \ra &=\la \bar{u}_{xx} \psi  - S_{k, h}(x), \om \om_x \ra  + \la S_{k, h}(x), \om \om_x  \ra   \\
& = \la ( \bar{u}_{xx} \psi  - S_{k, h}(x) ) \vp^{-1/2} \psi^{-1/2},  \vp^{1/2} \om \psi^{1/2} \om_x \ra  - \f{1}{2} \la \pa_x S_{k, h}(x), \om^2  \ra    \\
& \leq ||( \bar{u}_{xx} \psi  - S_{k, h}(x) )\vp^{-1/2} \psi^{-1/2}||_{L^{\infty}}  || \om \vp^{1/2} ||_{L^2}
|| \om_x \psi^{1/2}||_{L^2}- \f{1}{2} \la \pa_x S_{k, h}(x), \om^2  \ra  .
\eal
\eeq
Next we estimate the error of linear interpolation. We have plotted the numerical values of $S_{k,h}$ and $\D(x)$ defined below on the grid points in the first subfigure in Figure \ref{DG_cross} in the Appendix. 
\begin{lem}\label{lem:appro_err} The weighted approximation error satisfies
\beq\label{eq:appro_err}
 || ( \bar{u}_{xx} \psi  - S_{k, h}(x) )\vp^{-1/2} \psi^{-1/2} ||_{L^{\infty}}   \leq 0.1 .
\eeq
\end{lem}

\begin{proof}
Define
\[
\D(x) \teq ( \bar{u}_{xx} \psi  - S_{k, h}(x) )\vp^{-1/2} \psi^{-1/2} .
\]
For $x \leq 0.5$ or $x \geq 9.5$, we use the interval arithmetic to show that the maximal value of $\D(x)$ on the grid points is bounded by 0.09. In \eqref{eq:appro_err3},\eqref{eq:appro_err4}, we show that $\bar{u}_{xx}, \bar{u}_{xxx}(x^2 -L^2)$ grow at most logarithmically near $x = L$.  Using the asymptotic property of the weight
\[
\psi \asymp x^{-2} + 1 ,  \vp \asymp x^{-4} + (x-L)^{-2} \;,
\]
we obtain that
\[ \psi^{1/2} \vp^{-1/2} ,\psi^{-1/2} \vp^{-1/2} \les (x^{-1} + (x-L)^{-1} )^{-1}  ,
\]
which vanishes at $x =0, L$. For $x$ sufficiently close to $x= L$, e.g. $|x-L| < \e$, we can use the fact that $\D(x)$ vanishes at $x=L$ and that $\D(x)$ is at least $C^{1/2}$
to verify $|\D(x)| \leq 0.1$. For $|x - L| > \e$, we can use
the smoothness of $\psi, \vp \in C^{1, 1}$ and $u_{xx} \in C^{1}$ to verify $|\D(x)| \leq 0.1$ for $x $ that is not on a grid point.

For $x \in [x_i, x_{i+1}] \teq I_i \subset [0.5, 9.5]$, we use the standard approximation theory to show that
\beq\label{eq:appro_err2}
\bal
 | \bar{u}_{xx} \psi - S_{k, h}(x) |  &\leq  h \max_{ I_i }
  | (\bar{u}_{xx} \psi)_x |  =h  \max_{I_i} ( |  \bar{u}_{xxx} (x^2 -L^2) \psi(x^2 - L^2)^{-1} |
  +   | \bar{u}_{xx} \psi_x| ) \\
  &\leq  h \max_{I_i} | \bar{u}_{xxx}(x^2 -L^2)| \max_{I_i} |\psi (x^2 - L^2)^{-1}|
  + h  \max_{I_i} | \bar{u}_{xx}  | \max_{I_i}  | \psi_x |.
  \eal
\eeq
Let $\d = 0.25$. From the relation \eqref{eq:cross_hil}, we have for $x \in [0.5, 9.5]$ that
\[
\bal
&\bar{u}_{xxx}(x)(x^2 -L^2) = P.V. \f{1}{\pi} \int_{-L}^L \f{ \bar{\om}_{xx}(y)(y^2 -L^2)}{x - y} dy \\
 = &\f{1}{\pi} \int_{|x-y| < \d} \f{ \bar{\om}_{xx}(y)(y^2 -L^2) -  \bar{\om}_{xx}(x)(x^2 - L^2) }{x-y} dy
 + \f{1}{\pi} \int_{ |x-y| > \d, |y| \leq L } \f{ \bar{\om}_{xx}(y)(y^2 -L^2)}{x - y} dy ,\\
 \eal
\]
which implies
\begin{align}
 & | \bar{u}(x)_{xxx}(x^2 -L^2)   |  \leq  \f{2\d}{\pi}|| ( \bar{\om}_{xx} (x^2 -L^2) )_x||_{L^{\infty}[0.5-\d,9.5 +\d]}
 +  \f{2}{\pi}   || \bar{\om}_{xx}(y^2 -L^2) ||_{\infty}  \log\f{2L}{\d} \notag \\
&\qquad \qquad  \qquad  \qquad =  \f{2\d}{\pi} K_6 + \f{2}{\pi}  K_7 \log\f{2L}{\d} , \label{eq:appro_err3}
 \end{align}
where $K_6,K_7$ are defined in \eqref{eq:const_a1}. Similarly, for $\bar{u}_{xx}= H \bar{\om}_x$, we have
\beq\label{eq:appro_err4}
| \bar{u}_{xx}(x) |  \leq  \f{2\d}{\pi}||  \bar{\om}_{xx} ||_{L^{\infty}[0.5 - \d, 9.5 + \d]}
 +  \f{2}{\pi} || \bar{\om}_x ||_{\infty}  \log\f{2L}{\d}=
 \f{2\d}{\pi} K_8 + \f{2}{\pi}  K_1  \log\f{2L}{\d} \; ,
\eeq
where $K_8, K_1$ are defined in \eqref{eq:const_a1}. Hence, for $x\in [0.5, 9.5]$, $\bar{u}_{xx}, \bar{u}_{xxx}$ are bounded. Taking $\d = L-x$ in \eqref{eq:appro_err3} and \eqref{eq:appro_err4}, we see that both $\bar{u}_{xxx}(x^2-L^2), \bar{u}_{xx}$ can grow at most logarithmically near $x = L$ with another constant in the upper bound. 
For $ x \in[x_i, x_{i+1}] \subset [0.5, 9.5]$, one can verify that $\psi$ and $\vp$ remain smooth over one grid cell. Specifically, we have the following estimate:
\beq\label{eq:appro_err5}
\bal
\max_{I_i} |\psi_x|  \leq 1.1 |\psi_x(x_i) |, \  \max_{I_i} |\psi (x^2 - L^2)^{-1}|  \leq 1.1 |\psi(x_i) (x_i^2 - L^2)^{-1}|.
\eal
\eeq
Substituting the bounds on $\bar{u}_{xx}, \bar{u}_{xxx}$ in \eqref{eq:appro_err3}, \eqref{eq:appro_err4} with $\d = 0.25$ and \eqref{eq:appro_err5} into \eqref{eq:appro_err2}, we obtain
\[
\bal
\max_{[0.5,9.5]}|\D(x)| &\leq 1.1 h
\f{2}{\pi} \lt(  \d K_6 +  K_7\log\f{2L}{\d} \rt) || (x^2 -L^2)^{-1} \psi^{1/2}\vp^{-1/2} ||_{L^{\infty}[0.5,9.5]} \\
&\quad +  1.1 h \f{2}{\pi}\lt(  \d K_8 +   K_1\log\f{2L}{\d}\rt)
||\psi_x  \psi^{-1/2} \vp^{-1/2} ||_{L^{\infty}[0.5,9.5]} \\
&\leq 84h || (x^2 -L^2)^{-1} \psi^{1/2}\vp^{-1/2} ||_{L^{\infty}[0.5,9.5]} + 3.2 h ||\psi_x  \psi^{-1/2} \vp^{-1/2} ||_{L^{\infty}[0.5,9.5]}  \\
& \leq (84 \cdot 0.1 + 3.2 \cdot 2) h < 0.1 , 
\eal
\]
where we have used \eqref{eqt:numeric_estimates} and $h=2.5\times 10^{-5}$.
Combining all the above estimate would give the desired upper bound $|\D(x)| \leq 0.1 , \ \forall x \in [0, L]$.
\end{proof}

Using Lemma \ref{lem:appro_err}, we can estimate \eqref{eq:aa_H1_rem22} below
\beq\label{eq:a1_H1_rem22}
I_2 = \la \bar{u}_{xx} \om, \om_x \psi \ra  \leq 0.1 || \om \vp^{1/2} ||_{L^2} || \om_x \psi^{1/2}||_{L^2}
\leq d_0  || \om \vp^{1/2} ||^2_{L^2} + \f{0.01}{4 d_0} || \om_x \psi^{1/2}||^2_{L^2}.
\eeq

Next, we consider the remaining  linear term $II_2 = \la  (c_l x + u) \bar{\om}_{xx}, \om_x \psi \ra$ in \eqref{eq:lin_a1_H1}. First of all, using integration by parts, we yield
\beq\label{eq:a1_H1_rem31}
II_2 =  - \la  (c_l x + u) \bar{\om}_{xx}, \om_x \psi \ra
= \la (c_l + u_x ) \bar{\om}_{xx} \psi , \om_x \ra + \la (c_l x + u) (\bar{\om}_{xx} \psi)_x, \om \ra
\teq T_1 + T_2.
\eeq
We use a strategy similar to the one in the weighted $L^2$ estimate for $a=1$ in \cite{chen2019finite}.
For $T_1$, we first find a combination of power $x, x^{-1}$ that approximates $\bar{\om}_{xx} \psi$
\[
\bal
T_1  & = \la u_x  \bar{\om}_{xx} \psi , \om \ra + c_l \la \bar{\om}_{xx} \psi, \om \ra
= \B\la u_x  \lt(\f{d_1}{x} - d_2 x \rt), \om \B\ra + \B\la u_x  (\bar{\om}_{xx} \psi - \f{d_1}{x} + d_2 x ), \om \B\ra \\
&+ c_l  \B\la \bar{\om}_{xx} \psi - \f{d_3}{x}, \om \B\ra + c_l d_3 \B\la \f{1}{x} , \om \B\ra,
\eal
\]
where $d_1, d_2, d_3 > 0$ are to be determined. By using the cancellation property \eqref{lem:vel3}, we get
\[
\bal
d_1 \la u_x \om ,x^{-1} \ra  &= d_1 \lt( u_x(0) \la \om ,x^{-1} \ra + \la (u_x - u_x(0)) \om, x^{-1} \ra \rt) \\
&= d_1\lt(  - u_x(0) \f{\pi}{2}  (H\om)(0) + \f{\pi}{4} u_x(0)^2    \rt) =- \f{\pi d_1}{4} 	u_x(0)^2,
\eal
\]
where we have a factor $1/2$ since the integral is from $0$ to $L$ rather than the real line. Using
$xu_x = H( x\om)$ in \eqref{eq:commute3} and \eqref{lem:tric}, we derive
\[
\bal
\la u_x \om , x \ra & = \la H( x \om) (x \om ), x^{-1} \ra = - \f{\pi}{2} H (  H( x \om) (x \om ) )
 \\
 &= -\f{\pi}{4} \lt(  (H(x \om) )^2(0) - (x \om)^2(0) \rt) = -\f{\pi}{4}(H(x \om) )^2(0) = -\f{\pi}{4} (xu_x)^2(0) =0.
 \eal
\]

Using an elementary inequality $2ab \leq a^2 + b^2$, we get
\[
c_l d_3 \la x^{-1} , \om \ra = - \f{\pi}{2} c_l d_3 u_x(0) \leq \f{\pi d_1}{4} u^2_x(0)
+  \f{\pi d_3^2}{4 d_1} c_l^2 .
\]
Using the above estimates, we can reduce $T_1$ to
\beq\label{eq:a1_H1_rem32}
\bal
T_1 &\leq  -\f{\pi d_1}{4} u_x^2(0) + \B\la u_x  (\bar{\om}_{xx} \psi - \f{d_1}{x} + d_2 x ), \om \B\ra + c_l  \B\la \bar{\om}_{xx} \psi - \f{d_3}{x}, \om \B\ra + \f{\pi d_1}{4} u^2_x(0)
+  \f{\pi d_3^2}{4 d_1} c_l^2 \\
& \leq   \B\la u_x  (\bar{\om}_{xx} \psi - \f{d_1}{x} + d_2 x ), \om \B\ra + c_l  \B\la \bar{\om}_{xx} \psi - \f{d_3}{x}, \om \B\ra +  \f{\pi d_3^2}{4 d_1} c_l^2 \\
& = \B\la u_x  (\bar{\om}_{xx} \psi - \f{d_1}{x} + d_2 x ), \om \B\ra + c_l  \B\la \bar{\om}_{xx} \psi - \f{d_3}{x} + \f{\pi d_3^2}{4 d_1} g_{c_{\om}}, \om \B\ra  \\
& \leq || u_x||_2 || (\bar{\om}_{xx} \psi - \f{d_1}{x} + d_2 x )  \om||_2 + c_l  \B\la \bar{\om}_{xx} \psi - \f{d_3}{x} + \f{\pi d_3^2}{4 d_1} g_{c_{\om}}, \om \B\ra  \\
& = || \om||_2 || (\bar{\om}_{xx} \psi - \f{d_1}{x} + d_2 x )  \om||_2 + c_l  \B\la \bar{\om}_{xx} \psi - \f{d_3}{x} + \f{\pi d_3^2}{4 d_1} g_{c_{\om}}, \om \B\ra , \\
\eal
\eeq
where $g_{c_{\om}}$ defined in \eqref{eq:defg} satisfies $c_l = c_{\om}= \la \om , g_{c_{\om}}\ra$ and we have used $|| u_x||_2 = || \om ||_2$.
The second term will be estimated later. 

To estimate $T_2$ in \eqref{eq:a1_H1_rem31}, 
we again use a strategy similar to the one in the weighted $L^2$ estimate for $a=1$ in \cite{chen2019finite} by subtracting the different linear part near $x =0$ and $x = L$.  Let $ M_2 \in (0, L)$ be some parameter to be chosen later.
For $x \in [0, M_2]$, we have
\beq\label{eq:a1_H1_rem33}
\bal
T_{2l} & \teq \la (c_l x + u) (\bar{\om}_{xx} \psi)_x \one_{x \leq M_2}, \om \ra
=  \la ( u - u_x(0)x )  (\bar{\om}_{xx} \psi)_x \one_{x \leq M_2}, \om \ra  \\
&+ (c_l + u_x(0)) \la x (\bar{\om}_{xx} \psi)_x \one_{x \leq M_2}, \om \ra
\teq T_{2l, 1} +T_{2l, 2}.
\eal
\eeq
We first use the Cauchy-Schwarz inequality to split the first term and then apply the Hardy inequality \eqref{eq:hd1} 
\beq\label{eq:a1_H1_rem34}
\bal
 T_{2l,1} &= \B\la (u-u_x(0)) \sqrt{  25 d_4 x^{-6} + 9 d_5 x^{-4} }
,\sqrt{  25 d_4 x^{-6} + 9 d_5 x^{-4}   }^{-1} (\bar{\om}_{xx} \psi)_x \one_{x \leq M_2} \om \B\ra \\
&\leq  \f{1}{4}\la (u-u_x(0))^2 , 25 d_4 x^{-6} + 9 d_5 x^{-4}  \ra
+  \B\la (  25 d_4 x^{-6} + 9 d_5 x^{-4}   )^{-1} ((\bar{\om}_{xx} \psi)_x)^2 \one_{x \leq M_2}, \ \om^2 \B\ra \\
&\leq d_4 \la \om^2, x^{-4}  \ra + d_5 \la \om^2, x^{-2} \ra  +  \B\la (  25 d_4 x^{-6} + 9 d_5 x^{-4}   )^{-1}( (\bar{\om}_{xx} \psi)_x)^2 \one_{x \leq M_2}, \ \om^2 \B\ra ,\\
\eal
\eeq
where $d_4, d_5 > 0$ are to be determined. The constants $25/4, 9/4$ come from \eqref{eq:hd1} with $p=4,2$.
$T_{2l, 2}$ will be estimated later. 

For $x \in [M_2, L]$, we use \eqref{eq:DGnormal} to rewrite $c_l x + u$ and  $T_{2r}$
\beq\label{eq:a1_H1_rem35}
\bal
c_l x + u &= u - u(L) + c_l(x-L)  = u- u(L) - u_x(L)(x-L) + (c_l + u_x(L))(x-L).  \\
T_{2r} &\teq \la (c_l x + u) (\bar{\om}_{xx} \psi)_x \one_{x > M_2}, \om \ra
= \la u- u(L) - u_x(L)(x-L) ,  (\bar{\om}_{xx} \psi)_x \one_{x > M_2} \om \ra   \\
&+ (c_l + u_x(L))  \la (x-L) (\bar{\om}_{xx} \psi)_x \one_{x > M_2} ,\om \ra  \teq T_{2r, 1} + T_{2r, 2}.
\eal
\eeq
Notice that $u- u(L) - u_x(L)(x-L) $ and its derivative vanishes at $x = L$. We apply the Cauchy-Schwarz inequality and then the Hardy inequality similar to \eqref{eq:hd1} with $p=2$ to yield
\[
\bal
&T_{2r,1} = \la ( u- u(L) - u_x(L)(x-L) ) (x-L)^{-2} ,  (x-L)^2 (\bar{\om}_{xx} \psi)_x \one_{x > M_2}, \om \ra  \\
 \leq &\f{9 d_6 }{4} \la ( u- u(L) - u_x(L)(x-L) )^2, (x-L)^{-4}  \ra + \f{1}{9 d_6}  \la ( \  (x-L)^2 (\bar{\om}_{xx} \psi)_x \one_{x > M_2})^2, \om^2 \ra \\
 \leq & d_6 \la (u_x - u_x(L))^2 ,(x-L)^{-2}\ra + \f{1}{9 d_6}  \la ( \  (x-L)^2 (\bar{\om}_{xx} \psi)_x \one_{x  > M_2})^2, \om^2 \ra .\\
 \eal
\]
Furthermore, we apply \eqref{eq:iden22} and the $L^2$ isometry of the Hilbert transform to estimate $\la (u_x - u_x(L))^2 ,(x-L)^{-2}\ra$ 
\[
\la (u_x - u_x(L))^2 ,(x-L)^{-2}\ra \leq \int_{\R}  (u_x - u_x(L))^2 (x-L)^{-2} dx
= \la \om^2 ,(x-L)^{-2} + (x+L)^{-2}  \ra,
\]
which implies 
\beq\label{eq:a1_H1_rem36}
T_{2r,1}  \leq  d_6 \la \om^2, (x-L)^{-2} + (x+L)^{-2} \ra + \f{1}{9 d_6}  \la ( \  (x-L)^2 (\bar{\om}_{xx} \psi)_x \one_{x > M_2})^2, \om^2 \ra .\\
\eeq
We use the damping of $\om_x$ to control part of the above terms. Applying the Hardy inequality 
similar to \eqref{eq:hd1} with $ u - u_x(0) x$ replaced by $\om$ and $p=2$, we obtain
 \beq\label{eq:a1_H1_rem37}
 \la \om^2 , x^{-4} \ra \leq \f{4}{9} \la \om_x^2 , x^{-2} \ra, \quad 0 \leq \f{4d_7}{9} \la \om_x^2 , x^{-2} \ra  - d_7    \la \om^2 , x^{-4} \ra   ,
\eeq
where $d_7 > 0$ is a parameter to be determined.
Therefore, using the relation $II_2 =T_1 + T_2 = T_{1} + T_{2l,1} + T_{2l,2} + T_{2r,1} + T_{2r,2}$ and  combining all the estimates of $I_2, II_2$ in the weighted $H^1$ estimate, 
\eqref{eq:lin_a1_H1}, \eqref{eq:lin_a1_I2}, \eqref{eq:a1_H1_rem22}, \eqref{eq:a1_H1_rem31}, \eqref{eq:a1_H1_rem32}, \eqref{eq:a1_H1_rem33}, \eqref{eq:a1_H1_rem34}, \eqref{eq:a1_H1_rem35}, \eqref{eq:a1_H1_rem36} and \eqref{eq:a1_H1_rem37}, we derive
\beq\label{eq:a1_H1_rem4}
\bal
& I + II_2  = 
 \la D_2(\bar{\om}) , \om_x^2 \psi \ra + I_2 + T_{1} + T_{2l,1} + T_{2l,2} + T_{2r,1} + T_{2r,2} \\
& \leq
\la D_2(\bar{\om}) , \om_x^2 \psi \ra + d_0  || \om \vp^{1/2} ||^2_{L^2} + \f{0.01}{4 d_0} || \om_x \psi^{1/2}||^2_{L^2} \\
&+  || \om||_2 || (\bar{\om}_{xx} \psi - \f{d_1}{x} + d_2 x )  \om||_2 + c_l  \B\la \bar{\om}_{xx} \psi - \f{d_3}{x} + \f{\pi d_3^2}{4 d_1} g_{c_{\om}}, \om \B\ra  \\
 &+ d_4 \la \om^2, x^{-4}  \ra + d_5 \la \om^2, x^{-2} \ra  +  \B\la (  25 d_4 x^{-6} + 9 d_5 x^{-4}   )^{-1}( (\bar{\om}_{xx} \psi)_x)^2 \one_{x \leq M_2}, \ \om^2 \B\ra  \\
 &+ (c_l + u_x(0)) \la x (\bar{\om}_{xx} \psi)_x \one_{x \leq M_2}, \om \ra
+ d_6 \la \om^2, (x-L)^{-2} + (x+L)^{-2} \ra \\
&+ \f{1}{9 d_6}  \la ( \  (x-L)^2 (\bar{\om}_{xx} \psi)_x \one_{x > M_2})^2, \om^2 \ra  + (c_l + u_x(L))  \la (x-L) (\bar{\om}_{xx} \psi)_x \one_{x > M_2} ,\om \ra  \\
&  + \lt(\f{4d_7}{9} \la \om_x^2 , x^{-2} \ra  - d_7    \la \om^2 , x^{-4} \ra  \rt).
\eal
\eeq
The quantities on the right hand side can be classified into three classes:
$(a) \ \la \om_x^2, f  \ra ; \  (b)  \ \la \om^2 , g  \ra $ (except $|| \om||_2 || (\bar{\om}_{xx} \psi - \f{d_1}{x} + d_2 x )  \om||_2 $); (c) Projections of $\om$, i.e. $\la \om, g \ra$. Using the estimate below
\[
|| \om||_2 || (\bar{\om}_{xx} \psi - \f{d_1}{x} + d_2 x )  \om||_2 \leq d_8 || \om||^2_2 + \f{1}{4d_8}
|| (\bar{\om}_{xx} \psi - \f{d_1}{x} + d_2 x )  \om||^2_2,
\]
we can rewrite \eqref{eq:a1_H1_rem4} as
\beq\label{eq:a1_H1_rem41}
\bal
I + II_2& \leq \la S_1, \om_x^2 \psi \ra  + \la S_2, \om^2 \vp \ra +  P(\om) ,\\
\text{where}\quad S_1 &  =   D_2(\bar{\om})+\f{0.01}{4d_0} + \lt( \f{4d_7}{9}  x^{-2}  \rt) \psi^{-1} , \\
S_2 \vp &= d_0 \vp +  d_8 + \f{1}{4d_8} (\bar{\om}_{xx} \psi - \f{d_1}{x} + d_2 x )^2
+ d_4 x^{-4} + d_5 x^{-2}  \\
&+ (  25 d_4 x^{-6} + 9 d_5 x^{-4}   )^{-1}( (\bar{\om}_{xx} \psi)_x)^2 \one_{x \leq M_2}  + d_6 ( (x-L)^{-2} + (x+L)^{-2} ) \\
& + \f{1}{9 d_6}  \lt( (x-L)^2 (\bar{\om}_{xx} \psi)_x \one_{x > M_2} \rt)^2  -d_7   x^{-4} \;,\\
P(\om) & \teq c_l  \la R_1,  \om \ra + (c_l + u_x(0)) \la R_2, \om \ra  +  (c_l + u_x(L))  \la R_3 , \om \ra ,\\
R_1 & = \bar{\om}_{xx} \psi - \f{d_3}{x} + \f{\pi d_3^2}{4 d_1} g_{c_{\om}}  , \  R_2 = x (\bar{\om}_{xx} \psi)_x \one_{x \leq M_2}, \  R_3 = (x-L) (\bar{\om}_{xx} \psi)_x \one_{x > M_2} .
\eal
\eeq
Finally, we estimate $P(\om)$. Using \eqref{eq:defg}, we can rewrite $P(\om)$ as follows
\[
P(\om) = \la g_{c_{\om}} ,\om \ra  \la R_1, \om \ra  + \la g_{c_{\om} } + g_{u_x(0)} , \om \ra \la R_2, \om \ra +  \la g_{c_{\om} } + g_{u_x(L)} , \om \ra \la R_3, \om \ra .
\]
For some function $S_3 \in C([0, L]) , S_3 >0$ to be determined, we want to bound $P(\om)$ by $\la S_3, \om^2 \vp \ra$ as sharp as possible. 
We define 
\beq\label{eq:opt6_H1}
\bal
&z \teq (S_3 \vp)^{1/2} \om , \ \eta_1 \teq  (S_3 \vp)^{-1/2}  g_{c_{\om}} , \ \eta_2 \teq (S_3 \vp)^{-1/2}R_1,  \ \eta_3 \teq (S_3 \vp)^{-1/2} (g_{c_{\om} } + g_{u_x(0)} ),\\
 & \eta_4 \teq (S_3 \vp)^{-1/2} R_2, \quad \eta_5 \teq 
(S_3 \vp)^{-1/2}( g_{c_{\om} } + g_{u_x(L)} ), \quad \eta_6 \teq (S_3 \vp)^{-1/2} R_3.
\eal
\eeq
We want to find the best constant of the following inequality for any $\om \in L^2(\vp)$ 
\beq\label{eq:opt6_H12}
\bal
 \la \eta_1 , z\ra \la \eta_2 , z \ra + \la \eta_3 , z\ra \la \eta_4 , z\ra  + \la \eta_5 , z\ra \la \eta_6 , z\ra & \leq C^* ||z||_2^2 ,\\
\eal
\eeq
which is equivalent to
\beq\label{eq:opt6_H13}
 P(\om)  = \la g_{c_{\om}} ,\om \ra  \la R_1, \om \ra  + \la g_{c_{\om} } + g_{u_x(0)} , \om \ra \la R_2, \om \ra +  \la g_{c_{\om} } + g_{u_x(L)} , \om \ra \la R_3, \om \ra  \leq C^*\la S_3, \om^2 \vp \ra.
\eeq
It is easy to see that $\eta_i$ are generically linearly independent. Suppose that $V \teq  \mathrm{span} \{\eta_1, \eta_2, .., \eta_6 \} $ and $\{ e_j \}_{j=1}^6$ is an orthonormal basis of $V$ under the standard $L^2$ inner product on $[0, L]$. Suppose that $\eta_i$ in terms of the orthonormal basis $\{ e_j \}_{j=1}^6$ has the coordinate $ \th_i = (\th_{i,1},\dots,\th_{i,6})^T \in \R^6 $ (a column vector). We can apply the same analysis as we analyzed the best constant of the projection term in the $L^2$ estimate for $a=1$ in \cite{chen2019finite} to obtain that
\beq\label{eq:opt_H1_C}
C^* =\lam_{\max}  \lt( \f{1}{2} \sum_{i=1}^3 (\th_{2i-1} \th_{2i}^T +  \th_{2i} \th_{2i-1}^T )\rt), \quad
\f{1}{2} \sum_{i=1}^3 (\th_{2i-1} \th_{2i}^T +  \th_{2i} \th_{2i-1}^T ) \in R^{6 \times 6}.
\eeq

\subsection{Rigorous estimate of $C^*$}
Define a matrix
\[A\triangleq \f{1}{2} \sum_{i=1}^3 (\th_{2i-1} \th_{2i}^T +  \th_{2i} \th_{2i-1}^T ) = \frac{1}{2}\Theta_1\Theta_2^T,\]
where $\Theta_1 \triangleq [\theta_1,\theta_2,\theta_3,\theta_4,\theta_5,\theta_6]\in \mathbb{R}^{6\times 6}$ and $\Theta_2 \triangleq [\theta_2,\theta_1,\theta_4,\theta_3,\theta_6,\theta_5]\in \mathbb{R}^{6\times 6}$. Note that $A$ is symmetric, but not necessarily positive semidefinite. $C^*$ is then the largest eigenvalue of $A$. To rigorously estimate $C^*$, we first bound it by the Schatten $p$-norm of $A$:
\begin{equation}\label{C*_bound}
C^* \leq \|A\|_{p} \triangleq \mathrm{Tr}[|A|^p]^{1/p}\quad \text{for all $p\geq 1$}.
\end{equation}
Here $|A| = \sqrt{A^TA} = \sqrt{A^2}$. In particular, if $p$ is an even number, we have $|A|^p = A^p$. Therefore,
\[\mathrm{Tr}[|A|^p] = 2^{-p}\cdot \mathrm{Tr}[(\Theta_1\Theta_2^T)^p] = 2^{-p}\cdot \mathrm{Tr}[(\Theta_2^T\Theta_1)^p] \triangleq 2^{-p}\cdot \mathrm{Tr}[X^p]\]
where $X = \Theta_2^T\Theta_1$. Note that each entry of $X$ is the inner between some $\th_i$ and $\th_j$, $i,j=1,\dots,6$. Moreover, we have
\[\langle \eta_i,\eta_j\rangle = \left\langle \sum_{k=1}^6\th_{i,k}e_k,\sum_{k=1}^6\th_{j,k}e_k\right\rangle = \sum_{k=1}^6 \th_{i,k}\th_{j,k} = \th_i^T\th_j.\]
Therefore, to compute the entries of $X$, we only need to compute the pairwise inner products between $\eta_1,\dots,\eta_6$ (we do not need to compute the coordinate vectors $\th_i$ explicitly). Again, this is done by interval arithmetic. See the discussion in Section~\ref{sec:RE_inner}. So then each entry $X_{ij}$ of $X$ is represented by a pair of numbers that bound it from above and from below. Once we have the estimate of $X$, we can compute an upper bound of $\mathrm{Tr}[X^p]$ stably and rigorously by interval arithmetic, which then gives a bound on $C^*$ via \eqref{C*_bound}. In particular, we choose $p=4$ in our computation.

\subsection{Optimizing the parameters}
To optimize the estimate, we choose 
\[
\bal
d_0 &= 0.15, \ d_1 = 0.11 , \ d_2 = 0.0013, \ d_3 = 0.07, \ d_4 = 4.5 ,\ d_5 = 0.05 ,  \\
d_6 & = 0.03, \ d_7 = 2.5 ,  \ d_8 = 0.0004, \ M_2 = 6.5 .
\eal
\]
After specifying these parameters, the remaining part of the damping term $S_1$ (see \eqref{eq:a1_H1_rem41}) and the estimate of the cross term $S_2$ are completely determined. We plot the numerical values of $S_1, S_2$ on the grid points in the second subfigure of Figure \ref{DG_cross} in the Appendix.

The numerical value of $\max(S_1)$ is less than $-0.3$ and we use a conservative estimate 
\beq\label{eq:a1_H1_rem42}
S_1 < -0.25 ,
\eeq
which can be verified using the smoothness of the profile. It is clear that $S_2 < 7$ and we choose
\beq\label{eq:opt_H13}
S_3 = 7.5 - S_2.
\eeq
The corresponding $\eta_i$ in \eqref{eq:opt6_H1} are determined. The optimal constant in \eqref{eq:opt6_H12} can be computed via \eqref{eq:opt_H1_C}. The numerical value of $C^*$ satisfies $C^* < 0.85$ and we use a conservative estimate $C^* < 1 $, which can be verified rigorously. Plugging this estimate into \eqref{eq:opt6_H13}, we obtain 
\[
P(\om) \leq C^* \la S_3, \om^2 \vp \ra < \la S_3 , \om^2 \ra = \la 7.5 - S_2, \om^2 \vp \ra \;.
\]
Plugging the above estimate and \eqref{eq:a1_H1_rem42} into \eqref{eq:a1_H1_rem41}, we prove 
\[
I + II_2 \leq \la S_1, \om_x^2 \psi \ra  + \la S_2, \om^2 \vp \ra +  P(\om) 
\leq -0.25 \la \om_x^2 ,\psi \ra + 7.5 \la \om^2, \vp \ra,
\]
which completes the proof of Lemma \ref{lem:cross}.

\section{The estimates of the error term}\label{sec:error}
The error terms in the weighted $L^2, H^1$ estimate are given below
\[
\bal
F_1 = \la F(\bar{\om}) , \om \vp \ra,  \quad F_2 = \la  ( F(\bar{\om})  )_x , \om_x \psi  \ra  , \quad F(\bar{\om})= ( \bar{c}_{\om} + \bar{u}_x ) \bar{\om}- (\bar{c}_l x +\bar{ u}) \bar{\om}_x,  \\
\eal
\]
where $\vp, \psi$ are defined in \eqref{eq:wg_a1} and \eqref{eq:wg_a1_H1}, respectively. For some functions 
$\rho_1,\; \rho_2$
to be determined, the Cauchy-Schwarz inequality implies
\[
| F_1| \leq  \la F(\bar{\om})^2, \rho_1 \ra^{1/2}  \la \om^2, \vp^2  \rho^{-1}_1 \ra^{1/2} , \quad
|F_2| \leq \la ( F(\bar{\om})  )^2_x ,\rho_2 \ra^{1/2} \la \om_x^2, \psi^2  \rho^{-1}_2 \ra^{1/2}.
\]
Our goal is to verify that $\la F(\bar{\om})^2, \rho_1 \ra^{1/2}$, $\la ( F(\bar{\om})  )^2_x ,\rho_2 \ra^{1/2} $ are small. Note that we can only evaluate $u, u_x, u_{xx}$ at finitely many points via the Hilbert transform. We will use the composite Trapezoidal rule to approximate the integral and have the following error estimate for the Trapezoidal rule.

\subsection{Error estimate of the Trapezoidal rule}

\begin{lem}[Error estimate for the Trapezoidal rule] \label{lem:Trap}
\begin{align}
 \int_0^M  \f{F^2}{x^4} dx  - T_h \lt(  \f{F^2}{x^4} ,0, M  \rt) & \leq  \f{h^2}{4} \lt( \int_0^M \f{F^2_{xx}}{x^2}dx \rt)^{1/2} \lt( \int_0^M \f{F^2}{x^6} dx\rt)^{1/2} \label{eq:Trap1} ,\\
 \int_M^L   \f{F^2}{ (x-L)^2} dx - T_h \lt(  \f{F^2}{(x-L)^2}  ,M, L  \rt) &\leq  \f{h^2}{4}
\lt( \int_M^L F^2_{xx} dx \rt)^{1/2} \lt( \int_M^L \f{F^2}{(x-L)^4} dx\rt)^{1/2} \label{eq:Trap2} ,\\
\int_0^L F^2  dx -T_h( F^2, 0 ,L)  & \leq \f{h^2}{4} \lt(\int_0^L F^2_{xx} dx \rt) ^{1/2} \lt( \int_0^L F^2 dx \rt)^{1/2}  \label{eq:Trap3}  ,\\
 \int_0^M \f{ F_x^2}{x^2} dx - T_h \lt(  \f{F_x^2}{x^2}, 0, M \rt) &\leq h \lt( \int_0^M \f{F^2_{xx}}{x^2} dx\rt)^{1/2} \lt(  \int_0^M \f{F_x^2}{x^2} \rt)^{1/2} \label{eq:Trap4} ,\\
\int_0^L F_x^2  dx -T_h( F_x^2, 0 ,L)  & \leq h \lt(\int_0^L F^2_{xx} dx \rt) ^{1/2} \lt( \int_0^L F_x^2 dx \rt)^{1/2}  \label{eq:Trap5} \; ,
\end{align}
where $M = 5, L = 10$, $F$ is short for $F(\bar{\om})$ and the Trapezoidal rule $T_h$ is given by
\[
T_h(f, a, b) \teq \sum_{ a \leq ih < b}  \f{ f(ih) + f( (i+1)h )}{2} h.
\]
\end{lem}

Note that the approximate profile $\bar{\om}$ is a piecewise cubic polynomial supported in $[ -L ,L]$ and $\bar \om_x$ is discontinuous at $x = \pm L$. $\bar{u}_{xx}$ grows logarithmically near $x = \pm L$ and the error term $F$ \eqref{eq:NF_a1} is not smooth near $x = \om L$. Moreover, the weight is singular. Hence, we cannot apply the standard error estimate for the Trapezoidal rule. 

\begin{proof}
We first recall the standard error estimate for the Trapezoidal rule:
\[
\int_a^b f(x) dx  - \f{b-a}{2} (f(a)+f(b)) =  \int_a^b f^{\prime \prime}(x) \f{ (x-a)(x-b) }{2} dx.
\]
Denote by $P$ a piecewise quadratic polynomial
\beq\label{eq:trap_P}
P(x) \teq \f{(x- ih)(x- (i+1)h)}{2} , \ \forall x \in [ih, (i+1)h] .
\eeq
Then we have
\[
\bal
\D_1 \teq &\int_0^M  \f{F^2}{x^4} dx  - T_h \lt(  \f{F^2}{x^4} ,0, M  \rt)  = \int_0^{M}
 \lt(\f{F^2}{x^4}\rt)_{xx}  P(x) dx ,\\
\lt(\f{F^2}{x^4}\rt)_{xx} =& \f{2F_{xx} F + 2F^2_x}{x^4} - 16 \f{F_x F}{x^5} + 20 \f{F^2}{x^6}  \\
=& 2 \lt(\f{F_{xx}}{x} - 3 \f{F_x}{x^2} \rt) \f{F}{x^3} +  2 \f{F_x^2}{x^4} - 10 \f{F_x F}{x^5} + 20 \f{F^2}{x^6} \geq 2 \lt(\f{F_{xx}}{x} - 3 \f{F_x}{x^2} \rt) \f{F}{x^3},
\eal
\]
where we have used $2 a^2 + 20b^2 \geq 2 \sqrt{40} |ab| \geq 10 ab  $. From \eqref{eq:trap_P}, we know for $x \in [ih, (i+1) h]$
\[
P(x) \leq 0  , \quad | P(x)| \leq  \f{1}{2} \lt(  \f{(i+1)h - x +  x - ih   }{2} \rt)^2 \leq \f{h^2}{8}.
\]
Combining the above estimates yields
\beq\label{eq:trap_lt1}
\bal
\D_1 &\leq \int_0^M  2 \lt(\f{F_{xx}}{x} - 3 \f{F_x}{x^2} \rt) \f{F}{x^3} P(x) dx
\leq  \f{h^2}{4} \int_0^M  \B| \lt(\f{F_{xx}}{x} - 3 \f{F_x}{x^2} \rt) \f{F}{x^3} \B|  dx  \\
& \leq \f{h^2}{4} \lt( \int_0^M \lt(\f{F_{xx}}{x} - 3 \f{F_x}{x^2} \rt)^2  dx \rt)^{1/2} \lt( \int_0^M \f{F^2}{x^6} dx \rt)^{1/2}.
\eal
\eeq
Using integration by parts yields
\beq\label{eq:trap_lt2}
\bal
&\int_0^M \lt(\f{F_{xx}}{x} - 3 \f{F_x}{x^2} \rt)^2  dx = \int_0^M \f{F^2_{xx}}{x^2} + 9 \f{F_x^2}{x^4} dx
 -\int_0^M \f{3}{x^3} d F^2_x \\
 & = \int_0^M \f{F^2_{xx}}{x^2} + 9 \f{F_x^2}{x^4} dx - 9 \int_0^M \f{F_x^2}{x^4} dx - \f{4F_x^2}{x^3} \B|_0^M
 = \int_0^M \f{F^2_{xx}}{x^2}  - \f{4F^2_x(M)}{M^3}   \leq \int_0^M \f{F^2_{xx}}{x^2},
\eal
\eeq
where we have used the regularity of the profile which satisfies
\[
\lim_{x \to 0+} \f{F_x^2}{x^3} = 0.
\]
Combining \eqref{eq:trap_lt1} and \eqref{eq:trap_lt2} completes the proof of \eqref{eq:Trap1}.

For \eqref{eq:Trap2}, we introduce $G(x) \teq F(L-x)$. After a change of variables, \eqref{eq:Trap2} is equivalent to
\[
 \int_0^M   \f{G^2}{ x^2} dx - T_h \lt(  \f{G^2}{x^2}  ,0, M  \rt) \leq  \f{h^2}{4}
\lt( \int_0^M G_{xx} dx \rt)^{1/2} \lt( \int_0^M \f{G^2}{x^4} dx\rt)^{1/2}.
\]
The proof is very similar to the proof of \eqref{eq:Trap1} and is omitted. The proof of \eqref{eq:Trap3} is similar to the previous proof and is omitted here. The proof of \eqref{eq:Trap4} is based on the following expression of the error term
\[
\int_a^b f(x) dx  = \f{ b-a }{2} ( f(a)+f(b)) - \int_a^b f^{\prime}(x) \lt( x - \f{a+b}{2}  \rt) dx.
\]
Denote by $Q(x)$ a piecewise linear function
\[
Q(x) =  \f{  ih + (i+1) h }{2} - x, \quad x \in [ih, (i+1) h ),\quad |Q(x)| \leq \f{h}{2}.
\]
It follows that the left hand side of \eqref{eq:Trap4} can be written as
\[
\bal
\D_3 & = \int_0^M  \lt(\f{ F^2_x }{x^2} \rt)_{x} Q(x) dx
= 2 \int_0^M  \lt( \f{F_{xx}}{x} - \f{F_x}{x^2} \rt)  \f{F_x}{x}  Q(x) dx  \\
&\leq 2\cdot \f{h}{2} \int_0^M \B| \lt( \f{F_{xx}}{x} - \f{F_x}{x^2} \rt)  \f{F_x}{x}  \B|  dx
\leq h  \lt( \int_0^M \lt( \f{F_{xx}}{x} - \f{F_x}{x^2} \rt)^2 dx   \rt)^{1/2}  \lt(  \int_0^M \f{F_x^2}{x^2}  \rt)^{1/2}  \\
& \leq h  \lt( \int_0^M  \lt(\f{F_{xx}}{x}  \rt)^2 dx   \rt)^{1/2}  \lt(  \int_0^M \f{F_x^2}{x^2}  \rt)^{1/2},
\eal
\]
where we have used integration by parts similar to \eqref{eq:trap_lt2} to get the last inequality.
The proof of \eqref{eq:Trap5} is similar to \eqref{eq:Trap4} and is omitted.
\end{proof}

A corollary of Lemma \ref{lem:Trap} is the following estimates of the error of the approximate profile.
\begin{cor}\label{cor:error_a1}
The error of approximate profile $F(\bar{\om})$ satisfies
\beq\label{eq:error_a1}
\bal
&\int_0^M  \f{F^2}{x^4} dx  \leq T_h \lt(  \f{F^2}{x^4} ,0, M  \rt)  + \f{B^2_1}{15} h^2, \
\int_M^L  \f{F^2}{ (x-L)^2} dx  \leq T_h \lt(  \f{F^2}{(x-L)^2} ,M, L  \rt)  + \f{B^2_2}{3} h^2 ,\\
&\int_0^L F^2 dx\leq  \f{1}{4}\lt( \f{h^2}{4} B_3 + \sqrt{  \f{h^4}{16} B_3^2 + 4 T_h( F^2, 0 ,L) }   \rt)^2,
  \\
  &\int_0^M \f{F_x^2}{x^2} dx\leq  \f{1}{4}\lt( h B_1 + \sqrt{ h^2 B_1^2 + 4 T_h\lt( \f{F_x^2}{x^2}, 0 ,M\rt) }   \rt)^2  ,\\
    &\int_0^L F_x^2 dx\leq  \f{1}{4}\lt( h B_3 + \sqrt{ h^2 B_3^2 + 4 T_h\lt( F_x^2, 0 ,L\rt) }   \rt)^2,
  \\
\eal
\eeq
where $B_i$ defined in \eqref{eq:trap_Fxx2} only depends on $K_i, J_i$ in \eqref{eq:const_a1}.
\end{cor}
Corollary \ref{cor:error_a1} shows that the weighted $L^2, H^1$ errors of the approximate profile can be bounded by approximating the integral using the Trapezoidal rule with an error of order $O(h^2)$.
\begin{proof}
Using the Hardy inequality (the first inequality in \eqref{eq:hd1} with $u-u_x(0) x$ replaced by $F$), we have
\beq\label{eq:trap_lt3}
\bal
 \int_0^{M}  \f{F^2}{x^6} dx& \leq \f{4}{25} \int_0^{M} \f{F_x^2}{x^4} dx \leq \f{4}{25} \cdot \f{4}{9} \int_0^{M} \f{F_{xx}^2}{x^2} dx = \lt(\f{4}{15} \rt)^2 \int_0^{M} \f{F_{xx}^2}{x^2} dx, \\
\int_M^L \f{F^2}{(x-L)^4} dx &\leq \f{4}{9} \int_M^L \f{F^2_x}{(x-L)^2} \leq \f{4}{9} \cdot 4  \int_M^L F^2_{xx} dx = \f{16}{9} \int_M^L F^2_{xx} dx.
\eal
\eeq
Next, we estimate the weighted $L^2$ integral of $F_{xx}$. Note that
\[
F_{xx}  = (F(\bar{\om}))_{xx} = \bar{u}_{xxx} \bar{\om} + \bar{u}_{xx} \bar{\om}_x
-(\bar{c}_l x + \bar{u}) \bar{\om}_{xxx} - (\bar{c}_l + \bar{u}_x) \bar{\om}_{xx}.
\]
We use different estimates for $F_{xx}$ : 
\beq\label{eq:trap_Fxx}
\bal
\f{F_{xx}}{x} & = \bar{u}_{xxx} (x^2 -L^2) \f{\bar{\om}} {x (x^2 -L^2) } + \f{\bar{u}_{xx}}{x} \bar{\om}_x - \f{\bar{c_l} x + \bar{u}}{x} \bar{\om}_{xxx} - (\bar{c}_l + \bar{u}_x) \f{\bar{\om}_{xx}}{x}
\ , x \in [0, M],
\\
 F_{xx}  &=  \bar{u}_{xxx}(x^2 - L^2) \f{\bar{\om}}{x^2 -L^2} + \bar{u}_{xx} \bar{\om}_x
-  \f{ \bar{c}_l x + \bar{u}}{ x \wedge (L-x)} (x \wedge (L-x)) \bar{\om}_{xxx} - (\bar{c}_l + \bar{u}_x) \bar{\om}_{xx}
\ , x \in [0, L].
\eal
\eeq
For the terms involving $\bar{u}_{xx},\bar{u}_{xxx}$, we use \eqref{eq:commute3} and \eqref{eq:commute4} in Lemma \ref{lem:commute} and the $L^2$ isometry of the Hilbert transform $H$
\[
 \bal
&|| \bar{u}_{xxx} (x^2 -L^2) ||_{L^2[0,M]} \leq || \bar{u}_{xxx} (x^2 -L^2) ||_{L^2[0,L]}
= || \bar{\om}_{xx}(x^2-L^2) ||_{L^2[0,L]} = J_4^{1/2} ,\\
&\B| \B| \f{\bar{u}_{xx}}{x} \B|\B|_{L^2[0,M]}   \leq \B| \B| \f{\bar{\om}_x - \bar{\om}_x(0) }{x} \B|\B|_{L^2[0,L]}  =J_3^{1/2} , \quad || \bar{u}_{xx} ||_{L^2[0,L]} = || \bar{\om_x}||_{L^2[0,L]} =J_6^{1/2} .\\
\eal
 \]
 Note that $\bar{c}_lx + u(x) = 0 $ at $x = 0, \pm L$. For the terms involving $\bar{u}, \bar{u}_x$, we use the $L^{\infty}$ estimate
\[
\max_{[0,L]}  \B|\f{\bar{c}_l x + \bar{u}}{x \wedge (L -x)}  \B| =\max\lt(  \max_{[0,M]} \B|\f{\bar{c}_l x + \bar{u}}{x}\B|, \  \max_{[M,L]} \B|\f{\bar{c}_l x + \bar{u}}{x-L}\B| \rt)
\leq \max_{[0,L]} |\bar{c}_l + \bar{u}_x| = K_3 .
\]
For the terms involving $\bar{\om}, \bar{\om}_x$, we also use the $L^{\infty}$ estimate
\[
\bal
 \B|\B| \f{\bar{\om}}{x (x^2 - L^2)} \B|\B|_{L^{\infty}[0,M]}  &= K_{5l}, \quad \B|\B| \f{\bar{\om}}{ (x^2 - L^2)} \B|\B|_{L^{\infty}[0,L]}  = K_5   ,\\ \\
 ||\bar{\om}_x||_{L^{\infty}[0,L]} &= K_1, \quad   ||\bar{\om}_x||_{L^{\infty}[M,L]} 
 =K_{1 r} .
 \eal
\]
For the terms involving $\bar{\om}_{xx}, \bar{\om}_{xxx}$, we use the $L^2$ estimate
\[
\bal
&|| \bar{\om}_{xxx}||_{L^2[0,M]} = J_1^{1/2}, \quad  \B| \B| \f{\bar{\om}_{xx}}{x} \B|\B|_{L^2[0,M]} = J_2^{1/2},
\quad || \bar{\om}_{xxx} (x \wedge (L-x) ) ||_{L^2[0,L]} = J_5^{1/2} ,\\
 &|| \bar{\om}_{xxx} (x \wedge (L-x) ) ||_{L^2[M,L]} = J_{5r}^{1/2}  , \quad || \bar{\om}_{xx} ||_{L^2[0,L]} = J_7^{1/2} , \quad || \bar{\om}_{xx} ||_{L^2[M,L]} = J_{7r}^{1/2}.
\eal
\]
We apply the above estimates and the triangle inequality to yield
\beq\label{eq:trap_Fxx2}
\bal
\B| \B|  \f{F_{xx}}{x } \B| \B|_{L^2[0, M]}
&\leq  K_{5l}J_4^{1/2} + K_1 J_3^{1/2} + K_3 (J_1^{1/2} + J_2^{1/2})  \teq B_1 , \\
|| F_{xx}||_{L^2[M , L]} & \leq  K_{5}  J_4^{1/2}+ K_{1r}  J_6^{1/2}+ K_3 ( J_{5r}^{1/2} + J_{7r}^{1/2})
\teq B_2 , \\ 
|| F_{xx}||_{L^2[0 , L]} & \leq  K_{5}  J_4^{1/2}+ K_{1}  J_6^{1/2}+ K_3 ( J_{5}^{1/2} + J_{7}^{1/2}) \teq B_3 .
\eal
\eeq
Combining \eqref{eq:Trap1}, \eqref{eq:Trap2} in Lemma \ref{lem:Trap}, the Hardy inequality \eqref{eq:trap_Fxx} and the above estimates \eqref{eq:trap_Fxx2}, we obtain the first two inequalities in \eqref{eq:error_a1}.
\[
\bal
\int_0^M  \f{F^2}{x^4} dx  - T_h \lt(  \f{F^2}{x^4} ,0, M  \rt)
&\leq \f{h^2}{15}\int_0^M \f{F^2_{xx}}{x^2 } dx \leq \f{B^2_1}{15} h^2  ,\\
 \int_M^L   \f{F^2}{ (x-L)^2} dx - T_h \lt(  \f{F^2}{(x-L)^2}  ,M, L  \rt) &\leq  \f{h^2}{3}
\int_M^L F^2_{xx} dx  \leq \f{B^2_2}{3} h^2 .
\eal
\]

To obtain the remaining inequalities in \eqref{eq:error_a1}, we note that \eqref{eq:Trap3}, \eqref{eq:Trap4} and \eqref{eq:Trap5} are quadratic inequalities with respect to 
\[
I_1 \teq || F||_{L^2 [0, L]} , \quad  I_2 \teq \B|\B| \f{F_x}{x} \B|\B|_{L^2[M,L]},
\quad I_3 \teq || F_x^2 ||_{L^2[0,L]} ,
\]
and the coefficients on the right hand side of \eqref{eq:Trap3}, \eqref{eq:Trap4} and \eqref{eq:Trap5} are bounded by $B_i$ defined in \eqref{eq:trap_Fxx2}. Using \eqref{eq:Trap3} and \eqref{eq:trap_Fxx2}, we get 
\[
I_1^2 \leq T_h( F^2, 0 ,L)  + \f{h^2}{4}B_3 I_1,
\]
from which we can solve $I_1$ and obtain the following estimate
\[
 I_1 \leq \f{1}{2}\lt( \f{h^2}{4} B_3 + \sqrt{  \f{h^4}{16} B_3^2 + 4 T_h( F^2, 0 ,L) }   \rt).
\]
Taking square on both sides, we prove the third inequality in \eqref{eq:error_a1}. The bound for $I_2, I_3$ can be obtained similarly and we omit their proofs.\\
\end{proof}

Using \eqref{eq:trap_Fxx2} and the rigorous bounds in \eqref{eqt:numeric_estimates}, we have that 
\[B_1< 2.221,\quad B_2< 6.727,\quad B_3< 8.064.\]
We will use the interval arithmetic to verify that 
\begin{align*}
& T_h(F^2,0,L) < 8.445\times 10^{-10},\quad T_h(\frac{F^2}{x^4},0,M) < 7.388 \times 10^{-10}.\\
& T_h(\frac{F^2}{(x-L)^2},M,L) < 6.248 \times 10^{-9},\\
& T_h(F_x^2,0,L) < 9.850 \times 10^{-9},\quad T_h(\frac{F_x^2}{x^2},0,M) < 2.197\times 10^{-9}.
\end{align*}
Then using \eqref{eq:error_a1}, we obtain the following rigorous upper bounds:
\beq\label{eq:concrete_error_a1}
 \bal
 &IF_1 \teq \int_0^M  \f{F^2}{x^4} dx  < 9.443 \times 10^{-10}<10^{-9}, \quad\\
& IF_2 \teq \int_M^L  \f{F^2}{ (x-L)^2} dx  < 1.568\times 10^{-8} < 4 \times 10^{-8},\\
& IF_3 \teq \int_0^L F^2 dx < 8.445\times 10^{-10} < 10^{-9},\\
& IF_4 \teq \int_0^M \f{F_x^2}{x^2} dx <  6.762\times 10^{-9}< 10^{-8}, \\
& IF_5 \teq \int_0^L F_x^2 dx < 5.869\times 10^{-8} < 2\times 10^{-7}.
\eal
\eeq

\subsection{Estimate of the error terms}
We choose the functions $\rho_1, \rho_2$ as follows
\beq\label{eq:err_wg}
\rho_1 =  1.25 x^{-4} \one_{x \leq M} + 0.1 + 0.01 (x-L)^{-2} \one_{x \geq M} , \quad \rho_2 =  x^{-2} \one_{x  \leq M} + 0.02 .
\eeq

Using these weights and the Cauchy-Schwarz inequality, we can estimate the error terms in the weighted $L^2, H^1$ errors \eqref{eq:NF_a1} as follows 
\[
\bal
| F_1|  =  | \la F(\bar{\om}) ,  \om  \vp \ra |
&\leq  \f{1}{4 \tau} \la F(\bar{\om})^2 , \rho_1 \ra   +  \tau \la \om^2 , \vp^2 \rho_1^{-1} \ra \; , \\ 
|F_2| = | \la ( F(\bar{\om}) )_x, \om  \vp \ra  |
&\leq  \f{1}{4 \tau} \la ( F(\bar{\om}) )^2_x ,  \rho_2 \ra + \tau \la \om_x^2 , \psi^2 \rho_2^{-1}  \ra ,
 \eal
\] 
where $\tau$ is to be determined. 
Choose $\mu = 0.02$ and define $F_0$ below 
\[
F_0  \teq \f{1}{4 \tau} \la F(\bar{\om})^2 , \rho_1 \ra  +  \f{\mu}{4\tau} \la ( F(\bar{\om}) )^2_x ,  \rho_2 \ra.
\]
We then obtain
\beq\label{eq:err_a1}
| F_1 + \mu F_2  | \leq F_0 + \tau \la \om^2 , \vp^2 \rho_1^{-1} \ra 
+ \mu \tau \la \om_x^2 , \psi^2 \rho_2^{-1}  \ra .
\eeq
Using the upper bounds of the integrals in \eqref{eq:concrete_error_a1} and the definition of $\rho_i$ in \eqref{eq:err_wg}, we get 
\beq\label{eq:err_tot}
F_0 \leq \f{1}{4 \tau} \lt( 1.25 \cdot IF_1 + 0.01 \cdot IF_2+ 0.1\cdot  IF_3 + 5 \mu \cdot  IF_4 + 0.1 \mu \cdot IF_5 \rt) \leq  \f{1}{4\tau} \cdot 1.760 \cdot 10^{-9}.
\eeq

\section{Estimate of the nonlinear term}\label{sec:non}

\subsection{Estimate of $u_x, \om,  u / x$}
To control the nonlinear term, it suffices to control $||u_x||_{L^{\infty}[0,L]}$ . First of all, we have the following comparison result.
\begin{lem}\label{lem:non_ux1}
The weights $\vp, \psi $ satisfy for $x \in [0, L]$
\beq\label{eq:non_wg}
\bal
\vp(x) &\geq 1.15 \lt( \f{1}{x^4} + \f{0.02}{x^2} \rt) \teq \vp_1(x),\  \vp(x) \geq 0.0085 \lt(  \f{1}{(x-L)^2} +  \f{ 1}{(x+L)^2}  \rt) \teq \vp_2(x) ,\\
\psi(x) & \geq 0.97 \lt( \f{1}{x^2} + 0.01 \rt) \teq \psi_1(x).
\eal
\eeq
 Lemma \ref{lem:commute} and the $L^2$ isometry of the Hilbert transform \eqref{eq:hd1} with $p=2,4$ then imply
\beq\label{eq:non_a11}
\bal
&\la \om^2 ,\vp \ra \geq \la \om^2, \vp_1 \ra = \int_0^L (u_x - u_x(0))^2 \vp_1 dx , \\
&\la \om^2, \vp \ra \geq \la \om^2, \vp_2 \ra = \int_0^L (u_x - u_x(L))^2 \vp_2 dx , \quad  \la \om_x^2 , \psi \ra \geq   \la \om_x^2 , \psi_1  \ra =  \int_0^L u^2_{xx} \psi_1 dx .
\eal
\eeq
\end{lem}
The validity of the inequalities in \eqref{eq:non_wg} can be rigorously verified using the numerical methods discussed in Section~\ref{sec:parameter_notation}. We plot the numerical values of $\vp_1 / \vp, \vp_2 / \vp, \psi_1 /\psi$ on the grid points in the first subfigure in Figure \ref{DG_nonlinear} in the Appendix.

We consider the following functions and energy
\beq\label{eq:non_xi}
\xi_1 \teq x^{-3} + 0.0125 x^{-1} , \quad  \xi_2 \teq (L - x)^{-1} + 0.022 (L-x), \quad E^2(t) \teq \la \om^2, \vp \ra
+ \mu \la \om_x^2, \psi \ra,
\eeq
where $\mu >0$ are to be determined. For $x < 0$, $\xi_2$ should be considered as $\xi_2  = (L - |x|)^{-1} + 0.029 (L-|x|)$. Due to the odd/even symmetry, we only focus on $x >0$ and drop the absolute sign to simplify the notations. 
We have the following estimate for $u_x$.
\begin{lem}\label{lem:non_ux2}
Suppose that $\al_1, \al_2$ satisfies
\beq\label{eq:non_al2}
\xi_1^2 \leq  \mu (\al_1 \vp_1 - \xi_{1, x} ) \al_1 \psi_1,  
\quad  \xi_2^2 \leq  \mu (\al_2 \vp_2 + \xi_{2, x} ) \al_2 \psi_1, 
\eeq
where $\psi_2 = \psi_1$ and $\vp_i$ are defined in Lemma \ref{lem:non_ux1}. Then we have
\beq\label{eq:non_al3}
| u_x(x) - u_x(0) | \leq  \lt( \f{\al_1}{\xi_1(x)} \rt)^{1/2} E(t), \quad 
| u_x(x) - u_x(L) | \leq  \lt( \f{\al_2}{\xi_2(x)} \rt)^{1/2} E(t) .
\eeq
\end{lem}
By definition, $\xi_1(x) >0$ and $\xi_2(x) >0$ for $x\in [0, L]$. We will choose $\al_1 = 5.6$ and  $\al_2 = 500$. For these parameters, we plot the numerical values of the ratio between the left and the right hand sides of \eqref{eq:non_al2} on the grid points in the second subfigure in Figure \ref{DG_nonlinear} in the Appendix.

\begin{proof}
Note that $(u_x- u_x(0))^2 \xi_1$ vanishes at $x = 0$ due to $u_x - u_x(0) =O(x^2)$.
For $x$ close to $0$, we differentiate it and use the Cauchy-Schwarz inequality to yield
\[
\bal
&( u_x - u_x(0) )^2 \xi_1 =  \int_0^x 2 u_{xx} (u_x -u_x(0)) \xi_1 dx + \int_0^x ( u_x - u_x(0) )^2 \xi_{1,x} dx   \\
 \leq &\int_0^x (u_x - u_x(0))^2 (\al_1 \vp_1 -  \xi_{1, x} ) +  u^2_{xx} \xi^2_1 (\al_1 \vp_1 - \xi_{1, x} )^{-1} dx + \int_0^x ( u_x - u_x(0) )^2 \xi_{1,x} dx   \\
  =& \int_0^x (u_x - u_x(0))^2 (\al_1 \vp_1 )  dx+ \int_0^x u^2_{xx} \xi^2_1 (\al_1 \vp_1 - \xi_{1, x} )^{-1} dx \; ,
\eal
\]
where we have used $\al_1 \vp_1 - \xi_{1, x}  > 0$  \eqref{eq:non_al2} when we applied the 
Cauchy-Schwarz inequality. Using the assumption \eqref{eq:non_al2}, we obtain 
\[
( u_x - u_x(0) )^2 \xi_1  \leq
\int_0^x (u_x - u_x(0))^2 (\al_1 \vp_1 )  dx+ \int_0^x u^2_{xx} \mu \al_1 \psi_1 dx .
\]
Combining the above estimates and \eqref{eq:non_a11}, we prove
\[
( u_x - u_x(0) )^2 \xi_1 \leq \al_1 \la \om^2, \vp \ra + \mu \al_1 \la \om_x^2, \psi \ra = \al_1 E^2(t).
 \]
 Taking the square root yields the first estimate in \eqref{eq:non_al3}. 
 For $x$ close to $L$, applying an argument similar to that in our estimate for $(u_x - u_x(0))^2 \xi_1$ yields
 \[
\bal
&(u_x(x) - u_x(L))^2 \xi_2 
= -\int_x^L 2 u_{xx} (u_x(x) - u_x(L)) \xi_2 dx - \int_x^L  (u_x(x) - u_x(L))^2  \xi_{2,x} dx  \\
 \leq & \int_x^L   (u_x(x) - u_x(L))^2 (\al_2 \vp_2 + \xi_{2,x})  + u^2_{xx}  \xi_2^2 (\al_2 \vp_2 + \xi_{2,x})^{-1} dx
  - \int_x^L  (u_x(x) - u_x(L))^2  \xi_{2,x} dx  \\
  = & \int_x^L  (u_x(x) - u_x(L))^2 \al_2 \vp_2  +  u^2_{xx}  \xi_2^2 (\al_2 \vp_2 + \xi_{2,x})^{-1}  dx 
  \leq \int_x^L  (u_x(x) - u_x(L))^2 \al_2 \vp_2  +  u^2_{xx}  \mu \al_2 \psi_1 dx ,
 \eal
 \]
where we have used \eqref{eq:non_al2} in the last inequality. Combining the above estimate and \eqref{eq:non_a11}, we prove 
\[
(u_x - u_x(L))^2 \xi_2 \leq \al_2 \la \om^2, \vp \ra + \mu \al_2 \la \om_x^2 , \psi \ra \leq \al_2 E(t)^2.
\]
 Taking the square root yields the second estimate in \eqref{eq:non_al3}. 
\end{proof}
For $\om$, we have a similar result.

\begin{lem}\label{lem:non_w}
Suppose that the assumptions in Lemma \ref{lem:non_ux2} holds and $x > 0$. We have 
\beq\label{eq:non_w}
| \om(x) | \leq   \lt(  \f{\xi_1(x) }{\al_1} + \f{\xi_2(x)}{\al_2}  \rt)^{-1/2}  E(t) .
\eeq
\end{lem}

\begin{proof}
Using an estimate similar to that in the proof of Lemma \ref{lem:non_ux2}, we have
\[
\bal
\om^2(x) \xi_1(x)  &\leq \al_1 \lt(  \int_0^x \om^2 \vp_1 dx + \mu \int_0^x \om_x^2 \psi_1 dx \rt) , \ \om^2(x) \xi_2(x) \leq \al_2 \lt(  \int_x^L \om^2 \vp_2 dx +  \mu \int_x^L \om_x^2 \psi_1 dx \rt) .
\eal
\]
Using \eqref{eq:non_wg} and the above estimate, we derive
\[
\bal
\om^2(x) \lt(  \f{\xi_1(x)}{\al_1} + \f{\xi_2(x)}{\al_2} \rt) 
&\leq \lt(\int_0^x \om^2  \vp dx +  \mu \int_0^x \om_x^2 \psi dx \rt)
+ \lt(\int_x^L \om^2  \vp dx +  \mu \int_x^L \om_x^2 \psi dx \rt) \\
&=  \la \om^2 , \vp \ra + \mu \la \om_x^2 , \psi \ra  = E^2(t),
\eal
\]
which further implies \eqref{eq:non_w} after taking the square root. 
\end{proof}

A direct result of Lemma \ref{lem:non_ux2} is the following lemma.
\begin{lem}\label{lem:non_u}
Suppose that the assumption in Lemma \ref{lem:non_ux2} holds and $x > 0$. We have 
\beq\label{eq:non_al4}
\B| \f{ u - u_x(0) x  }{ x } \B| \leq  \f{2}{5}\al_1^{1/2} x^{3/2} E(t), \ 
\B| \f{ u - u(L) - u_x(L) (x-L )   }{ x - L} \B| \leq  \f{2}{3} \al_2^{1/2} (L-x)^{1/2} E(t).
\eeq
\end{lem}

\begin{proof}
From the definition of $\xi_1, \xi_2$ in \eqref{eq:non_xi}, we know 
\beq\label{eq:est_xi}
\lt(\f{\al_1}{\xi_1}  \rt)^{1/2}  =  \lt( \f{ \al_1 x^3 }{ 1 + 0.0125 x^2} \rt)^{1/2} \leq \al_1^{1/2} x^{3/2},\quad
\lt(\f{\al_2}{\xi_2}  \rt)^{1/2}  =  \lt( \f{ \al_2 (L-x) }{ 1 + 0.029 (L-x)^2} \rt)^{1/2} 
\leq \al_2^{1/2} (L-x)^{1/2} .
\eeq
Applying the above estimate and \eqref{eq:non_al3}, we yield
\[
\bal
\B|\f{u-u_x(0)x}{x} \B| &= \B| \f{1}{x}\int_0^x  u_x(y) - u_x( 0) dy    \B|
\leq  \f{1}{x} \int_0^x \lt(\f{\al_1}{\xi_1}  \rt)^{1/2}  dy \cdot  E(t)  \\
&\leq \f{\al^{1/2}}{x} E(t) \int_0^x y^{3/2} dy  = \f{2}{5}\al_1^{1/2} x^{3/2} E(t),
\eal
\]
which is the first inequality in \eqref{eq:non_al4}. Using \eqref{eq:est_xi} and \eqref{eq:non_al3}, one can derive the second inequality in \eqref{eq:non_al4} similarly. The factor $2/3 \cdot  (L-x)^{1/2}$ comes from 
\[
\f{1}{L-x}\int_x^L (L-y)^{1/2}  dy = \f{2}{3} (L-x)^{1/2}. 
\]
\end{proof}

For the end points $u_x(0), u_x(L), c_{\om} = -u(L) / L$, we use \eqref{eq:defg} and the Cauchy-Schwarz inequality to obtain the following estimate
\beq\label{eq:non_a15}
\bal
| c_{\om} + u_x(0) | & = | \la g_{c_{\om}} + g_{u_x(0)}, \om \ra |  
\leq \la \om^2 , \vp \ra^{1/2} \la (g_{c_{\om}} + g_{u_x(0)})^2 , \vp^{-1} \ra^{1/2} \leq \g_1 E(t)\;, \\
| c_{\om} + u_x(L) | & = | \la g_{c_{\om}} + g_{u_x(L)}, \om \ra |  
\leq \la \om^2 , \vp \ra^{1/2} \la (g_{c_{\om}} + g_{u_x(L)})^2 , \vp^{-1} \ra^{1/2} \leq \g_2 E(t) \; ,\\	
\eal
\eeq
where we have used $\la \om^2, \vp \ra \leq E(t)$ and the constants $\g_1, \g_2$ are given by 
\beq\label{eq:cons_gam}
\g_1  \teq  \la (g_{c_{\om}} + g_{u_x(0)})^2 , \vp^{-1} \ra^{1/2} , \quad 
\g_2 \teq   \la (g_{c_{\om}} + g_{u_x(L)})^2 , \vp^{-1} \ra^{1/2} .
\eeq

\subsection{Estimate of the nonlinear terms} 

Recall the nonlinear terms $N, N_1, N_2$ in \eqref{eq:NF_a1} and the normalization conditions of $c_l, c_{\om}$  \eqref{eq:DGnormal}. Using integration by part, we have
\beq\label{eq:non_N11}
\bal
N_1 &= \la ( c_{\om} + u_x ) \om - (c_l x + u) \om_x, \om \vp \ra = \B\la \f{ ((c_l x + u) \vp)_x}{2 \vp} + c_{\om} + u_x, \om^2 \vp \B\ra \\
& = \B\la  \f{3}{2} (c_{\om} + u_x )  +  \lt( c_l + \f{u}{x} \rt) \f{x \vp_x}{2 \vp}  ,\om^2 \vp    \B\ra
\teq \la  T , \om^2 \vp \ra \; .
\eal
\eeq
We use different estimates to handle $T$ for $x$ close to $0$ and $x$ close to $L$. For $x$ close to $0$, we have
\[
 T  =   (c_{\om} + u_x(0)) \lt( \f{3}{2} + \f{x\vp_x}{2 \vp} \rt)  + \f{3}{2} (u_x -u_x(0)) 
+ \lt( \f{u}{x} - u_x(0)  \rt) \f{x \vp_x}{2 \vp} .
\]
Using \eqref{eq:non_al3}, \eqref{eq:non_al4} and \eqref{eq:non_a15}, we obtain
\beq\label{eq:non_N12}
| T | \leq 
 \lt( \g_1  \B|  \f{3}{2} + \f{x\vp_x}{2 \vp} \B|+  \f{3}{2} \lt(\f{\al_1}{\xi_1} \rt)^{1/2} + \f{2}{5}\al_1^{1/2} x^{3/2} \B| \f{x \vp_x}{2 \vp} \B| \rt) E(t).
\eeq
For $x$ close to $L$, we use another decomposition to handle $T$
\[
T = ( c_{\om} + u_x(L) ) \lt(  \f{3}{2}  + \f{ (x-L) \vp_x }{ 2\vp} \rt) 
+  \f{3}{2}  ( u_x - u_x(L) ) + \f{  u(x)  - u(L) - u_x(L)(x-L) }{x-L} \f{ (x-L) \vp_x }{ 2\vp}.
\]
Using \eqref{eq:non_al3}, \eqref{eq:non_al4} and \eqref{eq:non_a15}, we obtain
\beq\label{eq:non_N13}
|T |
\leq \lt(  \g_2 \B| \f{3}{2}  + \f{ (x-L) \vp_x }{ 2\vp} \B|  
+ \f{3}{2} \lt(\f{\al_2}{\xi_2} \rt)^{1/2} +\f{2}{3} \al_2^{1/2} (L-x)^{1/2}\B| \f{ (x-L) \vp_x}{2 \vp} \B|
\rt) E(t) .
\eeq
Combining \eqref{eq:non_N11}, \eqref{eq:non_N12} and \eqref{eq:non_N13}, we obtain 
\beq\label{eq:non_N14}
|N_1| \leq  \la Z_1(x) , \om^2 \vp \ra E(t) \; ,
\eeq
where $Z_1(x)$ is given by 
\beq\label{eq:non_N15}
\bal
Z_1(x) \teq &\min\lt( \g_1  \B|  \f{3}{2} + \f{x\vp_x}{2 \vp} \B|+  \f{3}{2} \lt(\f{\al_1}{\xi_1} \rt)^{1/2} + \f{2}{5}\al_1^{1/2} x^{3/2} \B| \f{x \vp_x}{2 \vp} \B| , \rt. \\
 &\lt.      \g_2 \B| \f{3}{2}  + \f{ (x-L) \vp_x }{ 2\vp} \B|  
+ \f{3}{2} \lt(\f{\al_2}{\xi_2} \rt)^{1/2} +\f{2}{3} \al_2^{1/2} (L-x)^{1/2}\B| \f{ (x-L) \vp_x}{2 \vp} \B|       \rt) .
 \eal
\eeq
For $N_2$ in \eqref{eq:NF_a1}, we use integration by parts to obtain 
\beq\label{eq:non_N16}
\bal
N_2 &= \la (( c_{\om} + u_x ) \om - (c_l x + u) \om_x )_x, \ \om_x \psi \ra
= \la  u_{xx} \om - (c_l x + u) \om_{xx} , \om_x \psi \ra  \\
&= \B\la \f{  ( (c_lx+u) \psi)_x }{2 \psi} ,\om_x^2 \psi \B\ra + \la u_{xx} \om, \om_x \psi \ra
\teq N_{2,1} + N_{2,2}.
\eal
\eeq
For the first part, using \eqref{eq:non_N11}, \eqref{eq:non_N12}, \eqref{eq:non_N13} and an estimate similar to that in our estimate for $N_1$, we obtain 
\beq\label{eq:non_N171}
N_{2,1} \leq  \la Z_2(x) , \om_x^2 \psi \ra E(t),
\eeq
where $Z_2(x)$ is given by 
\beq\label{eq:non_N172}
\bal
Z_2(x)  \teq & \min\lt(  \g_1 \B| \f{1}{2} + \f{x\psi_x}{2\psi} \B|  
+ \f{1}{2} \lt( \f{\al_1}{\xi_1} \rt)^{1/2} + \f{2}{5}\al_1^{1/2} x^{3/2} \B| \f{x \psi_x}{2 \psi} \B| ,\rt.\\
& \lt.  \g_2  \B| \f{1}{2} + \f{ (x-L)\psi_x}{2\psi} \B| 
+\f{1}{2} \lt(\f{\al_2}{\xi_2} \rt)^{1/2} +\f{2}{3} \al_2^{1/2} (L-x)^{1/2}\B| \f{ (x-L) \psi_x}{2 \psi} \B|     \rt) .
\eal
\eeq
For $N_{2,2}$, we use \eqref{eq:non_wg}, \eqref{eq:non_a11} and the Cauchy-Schwarz inequality to yield 
\[
 | N_{2,2}  | \leq  \la u_{xx}^2 , \psi_1 \ra^{1/2} \la \om^2, \om_x^2 \psi^2 \psi_1^{-1} \ra^{1/2}
 \leq \la \om_x^2 , \psi_1 \ra^{1/2} \la \om^2, \om_x^2 \psi^2 \psi_1^{-1} \ra^{1/2}.
\]
For some constant $b_3 > 0$ to be determined, we use the above estimate and \eqref{eq:non_w} to derive 
\beq\label{eq:non_N181}
\bal
&|N_{2,2}|  \leq b_3  \la \om_x^2 , \psi_1 \ra E(t)
+ \f{1}{4 b_3}  \la \om^2, \om_x^2 \psi^2 \psi_1^{-1} \ra E(t)^{-1} \\
\leq & b_3  \la \om_x^2 \psi , \psi_1 \psi^{-1} \ra E(t)
+ \f{1}{4 b_3}  \B\la  \lt(  \f{\xi_1(x) }{\al_1} + \f{\xi_2(x)}{\al_2}  \rt)^{-1}, \om_x^2 \psi^2 \psi_1^{-1} \B\ra E(t)  \teq \la Z_3(x) ,\om_x^2  \psi \ra E(t),
\eal
\eeq
where $Z_3(x)$ is given by 
\beq\label{eq:non_N182}
Z_3(x) = b_3 \psi_1 \psi^{-1} + \f{1}{4 b_3} \lt(  \f{\xi_1(x) }{\al_1} + \f{\xi_2(x)}{\al_2}  \rt)^{-1} \psi \psi_1^{-1}.
\eeq
We choose $b_3 = 10$ in the final estimate.

\subsection{Summary of the estimates for the nonlinear terms} Combining the estimate \eqref{eq:non_N14}, \eqref{eq:non_N16}, \eqref{eq:non_N171} and \eqref{eq:non_N181}, we prove 
\beq\label{eq:non_a1}
| N_1 | + \mu|N_2| \leq | N_1| + \mu( |N_{2,1}| + |N_{2,2}| )    
\leq \la Z_1(x) , \om^2 \vp \ra E(t) +  \mu \la Z_2(x) +Z_3(x) , \om_x^2 \psi \ra  E(t)\;,
\eeq
where $Z_1, Z_2$ and $Z_3$ are defined in \eqref{eq:non_N15}, \eqref{eq:non_N172} and \eqref{eq:non_N182}, respectively.

\section{Nonlinear estimate}\label{sec:boot}
In this section, we combine all the estimates to obtain the nonlinear stability. 

\paragraph{\bf{Optimizing the parameters}} We choose the following parameter 
\[
\mu = 0.02, \quad \tau = 0.05 , \quad \al_1 = 5.6, \quad \al_2 = 500.
\]
We can verify that $\al_1, \al_2$ given above satisfy the assumption in Lemma \ref{lem:non_ux2}.

 Combining the weighted $L^2$ estimate 
\[
\f{1}{2}\f{d}{dt} \la \om^2, \vp \ra \leq -0.3 \la \om^2 , \vp \ra + N_1 + F_1,
\]
the weighted $H^1$ estimate \eqref{eq:lem_cross} in Lemma \ref{lem:cross}, the energy $E(t)$ in \eqref{eq:non_xi}, the estimate of error term \eqref{eq:err_a1} and the estimate of nonlinear term \eqref{eq:non_a1}, we have proved rigorously that
\beq\label{eq:boot_a11}
\bal
\f{1}{2} \f{d}{dt} E(t)^2 
&= \f{1}{2} \f{d}{dt} ( \la \om^2 ,\vp \ra + \mu  \la \om_x^2, \psi \ra ) 
\leq  ( -0.3 + 7.5 \mu ) \la \om^2 , \vp \ra 
 - 0.25 \mu \la \om_x^2 , \psi \ra  \\
&+ \lt( F_0  + \tau \la \om^2 , \vp^2 \rho_1^{-1} \ra  
+ \mu \tau \la \om_x^2 , \psi^2 \rho_2^{-1}  \ra  \rt) +
\la Z_1(x) , \om^2 \vp \ra E(t) +  \mu \la Z_2(x) +Z_3(x) , \om_x^2 \psi \ra  E(t). \\
 =& -0.15 \la \om^2 ,\vp \ra  - 0.25 \mu \la \om_x^2, \psi \ra
+ F_0 + \B\la \tau \vp\rho_1^{-1}  + Z_1(x) E(t) , \om^2 \vp\B\ra    \\
&+ \mu \B\la \tau\psi \rho_2^{-1}  + (Z_2(x) + Z_3(x))E(t) , \om_x^2 \psi     \B\ra,
\eal
\eeq
where we have used $\mu = 0.02$ to obtain the last equality. We divide the damping term into two parts as follows to control the error and the nonlinear term
\[
-0.15 \la \om^2 ,\vp \ra  - 0.25 \mu \la \om_x^2, \psi \ra = -0.05 E(t)^2 
 -  0.1 \la \om^2 ,\vp \ra  - 0.2 \mu \la \om_x^2, \psi \ra.
\]
With this decomposition, we can further rewrite \eqref{eq:boot_a11} as follows
\beq\label{eq:boot_a12}
\bal
\f{1}{2} \f{d}{dt} E(t)^2
&\leq -0.05 E(t)^2 + F_0 
+ \B\la -0.1 + \tau \vp\rho_1^{-1}  + Z_1(x) E(t) , \om^2 \vp\B\ra   \\
&+ \mu \B\la  -0.2 + \tau \psi \rho_2^{-1}  + (Z_2(x) + Z_3(x))E(t) , \om_x^2 \psi     \B\ra.
\eal
\eeq
We use the bootstrap argument to complete the proof. We choose the threshold below 
\[
E^* \teq 5 \cdot 10^{-4}.
\]
Suppose that $E(0) < E^*$. To complete the bootstrap argument, it suffices to show that the right hand side of \eqref{eq:boot_a12} is negative at $E(t) = E^*$. In particular, it can be verified rigorously that 
\beq\label{eq:boot_ineq}
\begin{cases}
 F_0  - 0.05 E^{*2}   &< 0  \\
-0.1 + \tau \vp\rho_1^{-1}  + Z_1(x) E^* &< 0 \\
-0.2 + \tau \psi \rho_2^{-1}  + (Z_2(x) + Z_3(x))E^*  & < 0 ,
\end{cases}
\iff 
\begin{cases}
 \f{1}{ 0.2} \cdot 1.389 \cdot 10^{-9} - 0.05 E^{*2}   &< 0  \\
-0.1 + 0.05 \vp\rho_1^{-1}  + Z_1(x) E^* &< 0 \\
-0.2 + 0.05 \psi \rho_2^{-1}  + (Z_2(x) + Z_3(x))E^*  & < 0 ,
\end{cases}
\eeq
where we have used \eqref{eq:err_tot} and substituted $\tau = 0.05$. The first inequality 
comes from a direct calculation. We plot the numerical values of the left hand side of the second and the third quantity on the grid points in Figure \ref{DG_boot} in the Appendix.
Therefore, we prove that the bootstrap argument can be continued. Hence,  $E(t)< E^*$ for all $t >0$ and the nonlinear stability follows.

\section{Convergence to the self-similar solution}\label{sec:convg}

In the previous section, we establish that for odd initial perturbation $\om(0)$ around the approximate steady state $\bar{\om}$ with $E(0)< E^* = 5 \cdot 10^{-4}$, the perturbation $\om(t)$ satisfies the {\it a-priori} estimate $E(t) < E^*$ for all $t > 0$. In this section, we will further establish the following estimate
\beq\label{eq:time}
\f{1}{2}\f{d}{dt} \la  \om_t^2 ,  \vp \ra \leq - 0.15 \la \om_t^2 , \vp \ra.
\eeq

Once this estimate is established, the convergence to the self-similar solution can be proved by an argument similar to that in the case of small $|a|$ we present in Section 4 in \cite{chen2019finite}.

Recall that the linearized equation around the approximate steady states $\bar{\om}$ reads
\[
\om_t +  ( \bar{c}_l x + \bar{u} ) \om_x = ( \bar{c}_{\om} + \bar{u}_x ) \om  + ( u_x + c_{\om} ) \bar{\om} - ( u + c_l x ) \bar{\om}_x +N(\om) + F(\bar{\om}) \;,
\]
where $N(\om)$ and $F(\bar{\om})$ are the nonlinear terms and the error given in \eqref{eq:NF_a1}. We introduce the linearized operator $\CL$
\[
 \CL \om = -  ( \bar{c}_l x + \bar{u} ) \om_x + ( \bar{c}_{\om} + \bar{u}_x ) \om  + ( u_x + c_{\om} ) \bar{\om} - ( u + c_l x ) \bar{\om}_x. 
\]

Then the linearized equation can be simplified as 
\[
\om_t = \CL \om + N(\om) + F(\bar{\om}).
\]

Using the regularity of the profile $\bar{\om}$ and that $\om \in L^2(\vp), \om_x \in L^2(\psi)$, we can obtain $\CL \om , N(\om), F(\bar{\om}) \in L^2(\vp)$, which implies that $\om_t \in L^2(\vp)$. Since $F(\bar{\om})$ is time independent, taking time differentiation, we get 
\[
\pa_t \om_t = \CL \om_t +  \pa_t N(\om) .
\]

Performing weighted $L^2$ estimates, we get 
\beq\label{eq:time_non_L21}
\f{1}{2} \f{d}{dt} \la \om_t^2, \vp \ra = \la \CL \om_t, \om_t \ra + \la  \pa_t N(\om) , \om_t \vp \ra.
\eeq

In Section 4.2 in \cite{chen2019finite}, we have established the a-priori weighted $L^2$ estimate with weight $\vp$ for $\CL$. Since $\om_t \in L^2(\vp)$, we can use this estimate to obtain 
\beq\label{eq:time_non_L22}
\la \CL \om_t , \om_t \ra \leq -0.3 \la \om_t^2 , \vp \ra.
\eeq

It remains to control the nonlinear part. Using the formula \eqref{eq:NF_a1} of $N(\om)$, we have
\[
\pa_t N(\om) 
=\lt(  (c_{\om} + u_x) \om_t - (c_l x + u ) \om_{x, t} \rt) +  \lt( (c_{\om, t} + u_{x, t}) \om -  ( c_{l, t} + u_{t}) \om_x \rt) \teq NT_1 + NT_2.
\]

The estimate of $NT_1$ is the same as that of $N_1$ in \eqref{eq:non_N11}. Using integrations by part, we get 
\[
NT_1 = \la ( c_{\om} + u_x ) \om_t - (c_l x + u) \om_{x,t}, \om_t \vp \ra = \B\la \f{ ((c_l x + u) \vp)_x}{2 \vp} + c_{\om} + u_x, \om_t^2 \vp \B\ra = \la  T , \om_t^2 \vp \ra \; ,
\]
where $T$ is defined in \eqref{eq:non_N11}. Using the estimates \eqref{eq:non_N12}-\eqref{eq:non_N15} and the {\it a-priori} estimate $E(t) < E^*$, we have 
\beq\label{eq:time_non_NT1}
|NT_1 | \leq \la Z_1(x), \om_t^2 \vp \ra  E(t) \leq E^*\la Z_1(x), \om_t^2 \vp \ra .
\eeq

\subsection{Pointwise estimates of $u_t$}

Before estimating $NT_2$, we first establish pointwise estimates 
 of $u_t$ similar to Lemma \ref{lem:non_u}. 
For $x \in [0, L]$, using the Cauchy Schwarz inequality,  we have 
\[
\bal
\B|\f{ u_t(x) - u_{x,t}(0) x }{x} \B| &= \f{1}{x}  \B|\int_0^x  u_{x,t}(y) - u_{x,t}(0)  dy \B|
\leq \f{1}{x}  \lt( \int_0^x \f{  ( u_{x,t}(y) - u_{x,t}(0) )^2 }{y^4} dy \rt)^{1/2}
(\int_0^x y^4 dy)^{1/2} \\
& \leq \f{1}{x}  ( 1.15^{-1} \cdot \la  (u_{x,t}(y) - u_{x,t}(0) )^2, \vp_1 \ra )^{1/2} (\f{ 1}{5} x^5)^{1/2}
\leq (5 \times 1.15)^{-1/2} x^{3/2} \la \om_t^2,  \vp \ra^{1/2} ,
\eal
\]
where we have used $ \vp_1 \geq 1.15 \cdot x^{-4}$ from \eqref{eq:non_wg} in the second inequality and  \eqref{eq:non_a11} in the third inequality. Similarly, for $x \in [0, L]$, we can estimate 
\[
\bal
&\B|\f{ u_t(x) - u_t(L) - u_{x,t}(L) (x-L) }{x-L} \B|
= \f{1}{L-x} \B| \int_x^L u_{x,t}(y) - u_{x, t}(L) dy \B|  \\
\leq &  \f{1}{L-x} \lt(  \int_x^L \f{(u_{x,t}(y) - u_{x, t} (L) )^2}{ (L-x)^2}\rt)^{1/2}
( \int_x^L (L-y)^2 dy)^{1/2} \\
\leq & \f{1}{L-x} \lt( 0.0085^{-1} \cdot \la (u_{x,t}(y) - u_{x, t} (L) )^2, \vp_2 \ra \rt)^{1/2} ( \f{1}{3} (L-x)^3)^{1/2} 
\leq ( 3 \times 0.0085 )^{-1/2} (L-x)^{1/2} \la \om_t^2, \vp \ra^{1/2},
\eal
\]
where we have used  $\vp_2 \geq 0.0085 (L-x)^{-2}$ \eqref{eq:non_wg}  in the second inequality and \eqref{eq:non_a11} in the third inequality. Denote 
\beq\label{eq:time_wg}
\bal
&\kp_1(x) \teq (5 \times 1.15)^{-1/2} x^{3/2}, \quad \kp_2 \teq ( 3 \times 0.0085 )^{-1/2} (L-x)^{1/2},   \\
& \kp_3(x) \teq \lt(  \f{\xi_1(x) }{\al_1} + \f{\xi_2(x)}{\al_2}  \rt)^{-1/2} , \quad \Ups(t) \teq \la \om_t^2, \vp \ra^{1/2}.
\eal
\eeq

We remark that the weight $\kp_3$ comes from \eqref{eq:non_w} in Lemma \ref{lem:non_w}. In particular, we have 
\beq\label{eq:non_w2}
|\om(x) | \leq \kp_3(x) E(t).
\eeq

We can simplify the above estimates on $u_t$ as follows 
\beq\label{eq:non_u2}
 |u_t(x) - u_{x,t}(0) x| \leq \kp_1(x) x \Ups(t), \quad 
 |u_t(x) - u_t(L) - u_{x,t}(L) (x-L) | \leq \kp_2(x) (L-x) \Ups(t).
\eeq

Note that $c_{\om}, u_x(0), u_x(L)$ are 
projections of $\om$ onto some functions. We can apply the estimates in \eqref{eq:non_a15} to $\om_t, c_{\om,t}, u_{x,t}(0), u_{x,t}(L)$ and obtain 
\beq\label{eq:non_u3}
| c_{\om, t} + u_{x,t}(0) | \leq \g_1 \la \om_t^2, \vp \ra^{1/2} = \g_1 \Ups(t), 
\quad 
| c_{\om, t} + u_{x, t}(L) |  \leq \g_2 \la \om_t^2 , \vp \ra^{1/2} = \g_2 \Ups(t), 
\eeq
where $\g_1, \g_2$ are the constants given in \eqref{eq:cons_gam}.

\subsection{Estimates of $NT_2$}

Now, we are in a position to estimate $NT_2$
\[
 NT_2 = (c_{\om, t} + u_{x, t}) \om -  ( c_{l, t} + u_{t}) \om_x . 
\]

The idea of the estimate is quite simple. We use $\la \om_t^2, \vp \ra$ to control $u_t, c_{l, t}$ and the weighted $L^2$ norm of $u_{x,t}$, and use $E(t)$ to control $\om$ pointwisely and the weighted $L^2$ norms of $\om, \om_x$. Since the weight $\vp$ is singular near $x=0$ (of order $x^{-4}$) and $x=L$ (of order $(L-x)^{-2}$), we need different weighted estimates near these two points. Recall the weighted functions $\vp_1, \vp_2$ in Lemma \ref{lem:non_ux1}. Let $\lam > 0$ to be determined. We consider the following partition of unity on $[0,L]$
\beq\label{eq:time_unity}
\eta_1 + \eta_2 = 1, \quad \eta_1 \teq \f{\vp_1}{\vp_1 + \lam \vp_2} , \quad \eta_2 \teq \f{ \lam \vp_2}{\vp_1 + \lam \vp_2} .
\eeq

Observe that $\vp_1$ is singular near $x =0$ and $\vp_2$ is singular near $x=L$. Intuitively, $\eta_1$ mainly supports near $x=0$ and $\eta_2$ supports near $x=L$.

Recall $c_{\om}(t) = c_l(t)$ \eqref{eq:DGnormal}. Clearly, we have $c_{\om, t} = c_{l, t} = -\f{u_t(L)}{L}$. For $x$ near $0$, we have 
\[
NT_2 = (u_{x,t} - u_{x,t}(0)) \om - (u_t - u_{x, t}(0) x) \om_x
+ ( u_{x,t}(0) + c_{\om, t}) (\om - x \om_x).
\]

Applying \eqref{eq:non_w2}, \eqref{eq:non_u2} and \eqref{eq:non_u3}, we obtain 
\beq\label{eq:time_non_NT21}
\bal
 | NT_2 | &\leq  | u_{x,t} - u_{x,t}(0) | \kp_3(x) E(t)
+ \kp_1(x) x |\om_x| \Ups(t) + \g_1 | \om - x \om_x| \Ups(t) \\
& \leq  | u_{x,t} - u_{x,t}(0) | \kp_3(x) E(t) + \g_1 |\om| \Ups(t)+ (\g_1 + \kp_1(x)) x | \om_x| \Ups(t).
\eal
\eeq

For $x$ near $L$, we rewrite $NT_2$ as follows 
\[
NT_2 = (u_{x,t} - u_{x,t}(L) ) \om  + (u_{x,t}(L) + c_{\om, t}) \om
 - (c_{l, t} x + u_t ) \om_x .
\]

We can estimate $c_{l, t} x + u_t$ as follows 
\[
\bal
|c_{l, t} x + u_t |
&= |u_t - u_t(L)  - u_{x,t}(L)(x-L)  + u_t(L) + c_{l, t} x +  u_{x,t}(L)(x-L) | \\
&= | u_t - u_t(L)  - u_{x,t}(L)(x-L)   + (c_{l,t} +u_{x,t}(L) )(x-L) |
\leq ( \kp_2(x) + \g_2) (L-x) \Ups(t),
\eal
\]
where we have applied $c_{l, t} = - \f{u_t(L)}{L}$ in the second equality and \eqref{eq:non_u2}, \eqref{eq:non_u3} in the last inequality.

Applying the above estimate, \eqref{eq:non_w2} and \eqref{eq:non_u3}, we derive 
\beq\label{eq:time_non_NT22}
|NT_2 | \leq |u_{x,t} - u_{x,t}(L) | \kp_3(x) E(t)  + \g_2 |\om| \Ups(t)
+ ( \kp_2(x) + \g_2) (L-x) |\om_x|  \Ups(t).
\eeq

Using a linear combination of the estimates \eqref{eq:time_non_NT21} and \eqref{eq:time_non_NT22} with weights $\eta_1, \eta_2$, we prove 
\beq\label{eq:time_non_NT23}
\bal
|NT_2|  & = \eta_1(x) |NT_2| + \eta_2(x) |NT_2|
\leq 
 | u_{x,t} - u_{x,t}(0) |   \eta_1(x)  \kp_3(x) E(t)  \\
  &+  |u_{x,t} - u_{x,t}(L) | \eta_2(x)  \kp_3(x) E(t) +(  \g_1  \eta_1(x) + \g_2 \eta_2(x) ) |\om| \Ups(t) \\ 
&+ (  (\g_1 + \kp_1(x)) x \eta_1(x)  +  ( \kp_2(x) + \g_2) (L-x) \eta_2(x) ) |\om_x|  \Ups(t)
\teq I_1 + I_2 + I_3 + I_4.
\eal 
\eeq

Using the Cauchy-Schwarz inequality, the {\it a-priori} estimate $E(t) < E^*$ obtained in Section \ref{sec:boot} and \eqref{eq:non_a11} which also holds true for $\om_t, u_{x,t} = H \om_t$
\[
\la (u_{x,t} - u_{x,t}(0))^2 ,\vp_1\ra
\leq \la \om_t^2 ,\ \vp_1 \ra , \quad  \la (u_{x,t} - u_{x,t}(L))^2 ,\vp_2\ra \leq \la \om_t^2 ,\ \vp_2 \ra ,
\]
we obtain 
\[
\bal
\la I_1,  |\om_t| \vp \ra
&\leq E(t) \la (u_{x,t} - u_{x,t}(0))^2 ,\vp_1\ra^{1/2}
\la \eta_1^2 \kp_3(x)^2 \vp_1^{-1}, \om_t^2 \vp^2 \ra^{1/2}
\leq E^* \la \om_t^2, \vp_1 \ra^{1/2} \la \eta_1^2 \kp_3(x)^2 \vp_1^{-1}, \om_t^2 \vp^2 \ra^{1/2} \\
& \leq  E^*
 \tau_1 \la\om_t^2, \vp_1 \ra  + E^* \f{1}{4\tau_1}  \la \eta_1^2 \kp_3(x)^2 \vp_1^{-1}, \om_t^2 \vp^2 \ra , \\
 \la I_2, |\om_t| \vp \ra 
 &\leq E(t) \la ( u_{x,t} - u_{x, t}(L) )^2, \vp_2 \ra^{1/2} 
 \la   \eta_2^2 \kp_3(x)^2 \vp_2^{-1}, \om_t^2 \vp^2 \ra^{1/2} 
\leq E^* \la \om_t^2, \vp_2 \ra^{1/2}\la   \eta_2^2 \kp_3(x)^2 \vp_2^{-1}, \om_t^2 \vp^2 \ra^{1/2}  \\
&\leq  E^* \tau_2 \la \om_t^2, \vp_2 \ra + E^* \f{1}{4 \tau_2} \la   \eta_2^2 \kp_3(x)^2 \vp_2^{-1}, \om_t^2 \vp^2 \ra,
\eal
\]
where $\tau_1, \tau_2 >0$ are some parameters to be determined and we have used the elementary inequality $|ab| \leq c a^2 + \f{1}{4c} b^2$ for any $a, b \in \R$ and $c > 0$ in the third and the sixth inequality. 
Plugging the formulas of $\eta_1, \eta_2$ \eqref{eq:time_unity} in the above estimates and combining them, we have
\beq\label{eq:time_non_I12}
 \la I_1 + I_2, |\om_t| \vp \ra 
 \leq \la Q_1(x) , \om_t^2 \vp \ra, 
\eeq
where 
\beq\label{eq:time_non_Q1}
Q_1(x) = E^* \lt( \tau_1 \f{\vp_1}{\vp} + \tau_2 \f{\vp_2}{\vp} + (\f{1}{4\tau_1} \vp_1 
+ \f{1}{4\tau_2}  \lam^2 \vp_2 )  \f{  \kp_3(x)^2 \vp}{ (\vp_1 + \lam \vp_2)^2}  \rt)
\eeq
and we have  written $\om_t^2 \vp^2$ as $\vp \cdot (\om_t^2 \vp)$ and used $\vp_i = (\vp_i/ \vp) \cdot \vp$ for $i=1,2$. Note that the denominators in $\eta_i$ cancel $\vp_i^{-1}$ for $i=1,2$.
Denote by $Q_2, Q_3$ the coefficients of $I_3, I_4$ in \eqref{eq:time_non_NT23}
\beq\label{eq:time_non_Q23}
\bal
Q_2(x)  &= \g_1 \eta_1(x) + \g_2 \eta_2(x) = \f{\g_1 \vp_1 + \g_2 \lam \vp_2}{ \vp_1 + \lam \vp_2} , \\
Q_3(x) & = (\g_1 + \kp_1(x)) x \eta_1(x)  +  ( \kp_2(x) + \g_2) (L-x) \eta_2(x) \\
&=  \f{ (\g_1 + \kp_1(x)) x \vp_1 +( \kp_2(x) + \g_2) (L-x) \lam \vp_2      }{  \vp_1 + \lam \vp_2}.
\eal
\eeq

To estimate $I_3, I_4$ in \eqref{eq:time_non_NT23}, we use $E^2(t) = \la \om^2, \vp \ra + \mu \la \om_x^2, \psi \ra$ and the {\it a-priori} estimate $E(t) < E^*$ to control the weighted $L^2$ norm. Using the Cauchy-Schwarz inequality, we obtain
\beq\label{eq:time_non_I34}
\bal
\la I_3 , |\om_t| \vp \ra  &= \la Q_2(x) |\om|, |\om_t| \vp \ra \Ups(t)
\leq \la \om^2, \vp \ra^{1/2} \la Q_2(x)^2, \om_t^2 \vp \ra^{1/2} \Ups(t)
\leq E^* (\tau_3 \Ups(t)^2 + \f{1}{4 \tau_3} \la  Q_2(x)^2, \om_t^2 \vp  \ra ) \\
\la I_4 , |\om_t| \vp \ra  &= \la Q_3(x) |\om_x|, |\om_t| \vp \ra \Ups(t) 
\leq \la \om_x^2, \mu \psi \ra^{1/2} \la Q_4(x)^2 , \om^2_t \vp^2 \mu^{-1} \psi^{-1} \ra^{1/2} \Ups(t)  \\
& \leq E^* \la Q_4(x)^2 , \om^2_t \vp^2 \mu^{-1} \psi^{-1} \ra^{1/2} \Ups(t) 
\leq E^* ( \tau_4 \Ups(t)^2 + \f{1}{4 \tau_4}  \la Q_4(x)^2 \vp \mu^{-1} \psi^{-1} , \om^2_t \vp  \ra   ),
\eal
\eeq
where $\tau_3, \tau_4 > 0$ are some parameters to be determined.

\subsection{Summarize the estimates}
Combining the estimates \eqref{eq:time_non_NT1} on $NT_1$ and \eqref{eq:time_non_I12}, \eqref{eq:time_non_I34} on $NT_2$, we prove 
\[
|\la \pa_t N(\om), \om_t \vp \ra | \leq  \la Q(x), \om_t^2 \vp \ra,
\]
where 
\[
Q(x) = E^* Z_1(x) + Q_1(x) + (\tau_3 + \tau_4) E^* + \f{1}{4\tau_3} Q_2^2(x) E^* 
+ \f{1}{4\tau_4} Q_3(x)^2 \vp \mu^{-1} \psi^{-1} E^*,
\]
we have used $\Ups(t)^2 = \la \om_t^2, \vp \ra$ (see \eqref{eq:time_wg}) and $Z_1(x), Q_1, Q_2, Q_3$ are defined in \eqref{eq:non_N15}, \eqref{eq:time_non_Q1} and \eqref{eq:time_non_Q23}. To optimize the estimates, we choose the following parameters 
\[
\lam = 2, \quad \tau_1 = 12, \quad \tau_2 = 32, \quad \tau_3 = 3.3, \quad \tau_4 = 65.
\]

With these parameters, we can show that the explicit function $Q(x)$ satisfies 
\[Q(x) < 0.15 , \]
which along with \eqref{eq:time_non_L21} and \eqref{eq:time_non_L22} imply 
\[
\f{1}{2} \f{d}{dt} \la \om_t^2, \vp \ra \leq -0.15 \la \om_t^2, \vp \ra .
\]

We conclude the proof of \eqref{eq:time}.

\section{Error estimates of the Gaussian quadrature}\label{sec:GQ}
In Section \ref{sec:intro_GQ}, we discussed how to compute the quantities $\bar{u}, \bar{u}_x, \bar{u}_{xx}$ and different types of errors in the computation. 
In this section, we estimate the error in the computation of $I_2, I_3$ in \eqref{eq:GQ_error0} using the numerical Gaussian quadrature.

We use the same notations as those in Section \ref{sec:intro_GQ}. Recall the notations $J_i$ \eqref{eq:Jset}, $GQ_K, NGQ_K$ \eqref{def:GQ}. We fix $i \in {1,2,.., n}$ and $x \in [ (i-1)h_0, ih_0 ], h_0 = L / n$. Firstly, we focus on $I_2$ in \eqref{eq:GQ_error0}. Denote $n_x = \f{x}{h_0}$. 
By definition, we get $n_x \in [i-1 , i]$. Clearly, we have
\beq\label{eq:GQ_error02}
I_2 = \f{1}{\pi} \sum_{j \in J_2} \int_{jh_0}^{ (j+1) h_0}  \f{\bar{\om}(y)}{x-y} dy 
= \f{1}{\pi} \sum_{ j\in J_2} \int_0^1 \f{ \bar{\om}( j h_0 + t h_0)}{ n_x h_0 - jh_0 - th_0  } h_0 dt 
= \f{1}{\pi} \sum_{ j\in J_2} \int_0^1 \f{ \bar{\om}( j h_0 + t h_0)}{ n_x  - j - t  }  dt , 
\eeq
where we have used a change of variable $y = j h_0 + t h_0, t \in [0, 1]$ in the second equality. 
Denote 
\beq\label{eq:GQ_error03}
F_j(t) = \f{ \bar{\om}( j h_0 + t h_0)}{ n_x  - j - t  } , \quad  G_j= \f{ \bar{\om}( j h_0 + t h_0)}{ n_x  + j + t  } , \quad I_{2, j} = \int_0^1 F_j(t) dt.
\eeq

In practice, we use $NGQ_K$ to compute each integral, which leads to the following error
\beq\label{eq:GQ_error1}
\bal
E_j \teq  \int_0^1 F_j(t) dt - NGQ_K( F_j, 0, 1)  = I_{2, j} - NGQ_K( F_j, 0, 1) 
\eal
\eeq
for each $j\in J_2$. We have a similar error in the computation of $I_3$. One approach to estimate the above error is decomposing it into two parts and then estimating each part. The first part is the standard error of Gaussian quadrature $  \int_0^1 F_j(t) d - GQ_K(F_j,0,1)$ that can be estimated using the formula of the error term. The second part is the difference between the Gaussian quadrature and its numerical realization, i.e. $GQ_K(F_j,0,1) - NGQ_K(F_j, 0, 1) $. From the definition \eqref{def:GQ}, this difference can be estimated once we have control on $ A_j - \bar{A}_j, z_j - \bar{z}_j$. We appeal to another approach which avoids estimating $A_j, z_j$.

\subsection{ Simplification of the error}
Though $\bar{A}_j, \bar{z}_j$ are floating-point values, they enjoy some exact identities. The weights $\bar{A}_j$ satisfies $\sum_{j=1}^{K} \bar{A}_j = 2$ without roundoff error and that $\bar{z}_j, \bar A_j$ can be grouped into $K/2$ pairs such that $\bar{z}_j + \bar{z}_{j+K/2} =0, \bar A_j = \bar A_{j+ K/2}$ for $ 1\leq j \leq K/2$ without roundoff error.
We introduce the following quantity that can measure the difference between $(\bar{A}_j, \bar{z}_j)$ and $(A_j, z_j)$
\beq\label{eq:GQ_ek}
  \e_k = \B| \sum_{j=1}^K \bar{A}_j \bar{z}^{2k}_j - \f{2}{2 k+1} \B|, \quad k = 0, 1,2, ..., K-1,  \quad 
  c_K  = \sum_{j=1}^K \bar{A}_j \bar{z}^{2 K}_j .
\eeq

Since $\sum_j \bar{A}_j = 2$, we have $\e_0 = 0$. The factor $\f{2}{2k+1}$ comes from $\sum_{j=1}^K A_j z^{2k}_j  = \int_{-1}^1 t^{2k} dt = \f{2}{ 2k+1}$
for $k\leq K-1$, where we have applied the Gaussian quadrature to the integral of $f = t^{2k}$ and the equality holds since the degree of $f$ is less than $2K-1$. We will show that $\e_k$ is very small, i.e. $< 10^{-15}$.
Applying the Taylor expansion to $F_j(t)$ around $t_0 = \f{1}{2}$ with the Lagrange form of reminder, we have
\[
F_j(t) = \sum_{ 0\leq k \leq 2 K -1} \f{\pa_t^k F_j( t_0) }{k!}\cdot  (t - t_0)^k 
+ \f{\pa_t^{2K} F_j( \xi(t))  }{ (2K)!} (t- t_0)^{2K},
\]
for any $t \in [0,1]$, where $\xi(t) \in (0, 1)$. 
Plugging the above expansion into $I_{2,j} = \int_0^1 F_j(t) dt$, we obtain 
\[
\bal
I_{2,j}  &= \sum_{0\leq k \leq 2K-1} \f{\pa_t^k F_j( t_0) }{k!}  \int_0^1 (t-t_0)^{k}  dt
+\int_0^1 \f{\pa_t^{2K} F_j( \xi(t))  }{ (2K)!} (t- t_0)^{2K} dt \\
&= \sum_{ 0 \leq k \leq K-1} \f{\pa_t^ {2k} F_j( t_0) }{ (2k)!} \f{2}{2 k+1} \cdot 2^{-2 k-1}
+ \int_0^1 \f{\pa_t^{2K} F_j( \xi(t))  }{ (2K)!} (t- t_0)^{2K} dt 
\teq r_{2, j} + s_{2, j},
\eal
\]
where we have used $t_0 = \f{1}{2}, (1-t_0)^{2k+1} =t_0^{2k+1} = 2^{-2k-1}$ and that the integral vanishes when $k$ is odd due to symmetry. Similarly, for $NGQ_K(F_j,0, 1)$ (see \eqref{def:GQ}), we have 
\[
\bal
&NGQ_K(F_j,0, 1)  = \f{1}{2} \sum_{1\leq j \leq K} \bar{A}_j F_j( \f{1}{2} \bar{z}_j + \f{1}{2}) \\
= & \sum_{0\leq k \leq 2K-1} \f{\pa_t^k F_j( t_0) }{k!} \sum_{1\leq j \leq K} \bar{A}_j \f{ \bar{z}^k_j}{2^{k+1}}  
+ \sum_{ 1\leq j \leq K}  
\f{\pa_t^{2K} F_j( \xi( \bar{z}_j / 2 + 1/2 ))  }{ (2K)!}  \f{ \bar{A}_j \bar{z}_j^{2K}}{2^{2K + 1}}  \\
=& \sum_{0\leq k \leq K-1} \f{\pa_t^{2k} F_j( t_0) }{ (2k)!} \sum_{1\leq j \leq K} \bar{A}_j \f{ \bar{z}^{2k}_j}{2^{2k+1}} + \sum_{ 1\leq j \leq K}  \f{\pa_t^{2K} F_j( \xi( \bar{z}_j / 2 + 1/2 ))  }{ (2K)!}  \f{ \bar{A}_j \bar{z}_j^{2K}}{2^{2K + 1}}  \teq \bar{r}_{2,j} + \bar{s}_{2,j},
\eal
\]
where we have used the fact that $\bar{z}_j, \bar A_j$ can be grouped into $K/2$ pairs such that $\bar{z}_j = - \bar z_{j+ K/2}$ and $\bar A_j = \bar A_{j+ K/2}$ in the third equality.
Using the notation $\e_k$, we obtain 
\beq\label{eq:GQ_error12}
| r_{2, j } - \bar{r}_{2, j} |
\leq  \sum_{0\leq k \leq K-1 }\f{ |\pa_t^{2k} F_j( t_0) | }{ (2k)! \cdot 2^{2k+1}} \e_k 
= \sum_{1\leq k \leq K-1 }\f{ |\pa_t^{2k} F_j( t_0) | }{ (2k)! \cdot 2^{2k+1}} \e_k ,
\eeq
where the last equality holds true since $\e_0 = 0$. Note that $\int_0^1 (t-t_0)^{2K} dt = \f{1}{2K+1} 2^{-2K}$. We have a simple estimate for $s_{2,j}$ and $\bar{s}_{2,h}$ 
\beq\label{eq:GQ_error13}
|s_{2,j} | \leq \max_{ t\in [0,1]} \f{ |\pa_t^{2K} F_j( t ) | }  { (2K) !} \cdot \f{1}{(2K+1) 2^{2K}} ,
\quad | \bar{s}_{2,j}| \leq 
\max_{ t\in [0,1] } \f{ |\pa_t^{2K} F_j( t ) | }  { (2K) !}  \cdot \f{c_K}{2^{2K+1}}.
\eeq

It remains to estimate the derivatives $\pa_t^{2k} F_j, k =1,2,..,K$.

\subsection{Estimates of $\pa_t^{2k} F_j$}\label{sec:dj}

Since $\bar{\om}(\cdot)$ is a cubic polynomial on $[jh_0, (j+1) h_0]$, we obtain that $\bar{\om}( j h_0 + \cdot h_0)$ is a cubic polynomial on $[0,1]$. Denote 
\[
d(t, j )\teq n_x - j - t, \quad  d(j) \teq \min_{t \in [0,1 ] } |d(t, j)|.
\]

For any $k \in Z_+$, a directly calculation yields
\[
\pa_t^k  d(t, j)^{-1} = k! \cdot  (n_x - j - t)^{-k-1} = k! \cdot d(t, j)^{-k-1}, \quad 
| \pa_t^k  d(t, j)^{-1}| \leq k! \cdot d(j)^{-k-1},
\] 

Recall the definition of $F_j$ from \eqref{eq:GQ_error03}. Using the fact that $\bar{\om}$ is a cubic polynomial and the Leibniz rule, we derive 
\[
\bal
|\pa_t^{2 k} F_j(t)| &= \B| \sum_{0\leq i \leq 3} { 2 k \choose i } \pa_t^i ( \bar \om(jh_0 + t h_0))
\cdot (2 k- i)!  \cdot d(t,j)^{- (2k-i + 1)} \B|
\eal
\]
for $t \in [0,1]$ and $k \geq 1$, where we have used the convention ${ 2k \choose i } =0$ if $i > 2k$. Notice that for $i=0,1,2,3$, we have $ { 2 k  \choose i } (2 k -i)! = \f{ (2k) !} {i!}$ and
\[
|\pa_t^i (\bar \om(j h_0 + t h_0)) \cdot d(t,j)^{- (2k -i + 1)}|
\leq h_0^i || \pa^i_x \bar \om ||_{L^{\infty}} d(j)^{- (2 k-i+1)} .
\]

Therefore, we can simplify the estimate $\pa_t^{2 k} F_j$ as follows 
\beq\label{eq:GQ_error2}
|\pa_t^{2 k} F_j(t)|
\leq (2 k )! \cdot d(j)^{-2 k +2} ( \sum_{ 0\leq i \leq 3} || \pa_x^i  \bar{\om} ||_{L^{\infty}} h_0^i  d(j)^{i-3} (i!)^{-1}  ). 
\eeq
for $ k \geq 2$. When $k=1$, we do not have the term $\pa_t^i ( \bar \om(jh + th))$ with $i=3$. Using the above estimate, we get 
\beq\label{eq:GQ_error22}
|\pa_t^2 F_j(t) | \leq  \sum_{0\leq i \leq 2} 2! \cdot ( i!)^{-1} h_0^i || \pa^i_x \bar \om ||_{L^{\infty}} d(j)^{- (2 \cdot 1-i+1)} .
\eeq

Next, we further estimate $d(j)$ for each $ j \in J_2$. By definition of $J_2$ \eqref{eq:Jset}, we have 
\[  |j - (i-1)| \geq m+1 ,  \quad  0 \leq j \leq n-1\]
 for any $j \in J_2$, which implies $ 0 \leq j \leq i - m- 2 $ or $j \geq m+i$. We further partition $J_2$ as follows 
\[ J_2^- \teq J_2 \cap [0, i-m-2], \quad J_2^+ \teq J_2 \cap [m+i, n-1 ] .\]

Recall $n_x = \f{x}{h} \in [i-1, i)$ and $i \in \{1,2, .., n\}$. For $j\in J_2^-$, we have $ j \leq i-m-2$ and 
\[
d(j) 
= \min_{t\in [0,1]}  |n_x - j - t|
\geq n_x - j - 1 \geq i-1 - j - 1  = (i-2) - j  \geq m .
\]
Moreover, we have $ (i-2) - j \leq  n - 2 < n$. For $j \in J_2^+$, we have $j \geq m+i$ and 
\[
d(j) 
= \min_{t\in [0,1]}  |n_x - j - t|
\geq j + 0 - n_x \geq  j -i \geq m.
\]
Moreover, we have $j - i \leq  n-1 - i  < n$. Combining two cases, we prove 
\beq\label{eq:GQ_error23}
d(j)^{-1}  \leq \one_{ -n < j -(i-2) \leq -m} | j - (i-2)|^{-1}+  \one_{ n > j - i \geq m } |j-i|^{-1}
\leq m^{-1}.
\eeq
for each $j \in J_2$. 
Hence, we get 
\beq\label{eq:const_Cm}
\sum_{ 0\leq i \leq 3} || \pa_x^i  \bar{\om} ||_{L^{\infty}} h_0^i  d(j)^{i-3} (i!)^{-1}  
\leq \sum_{ 0\leq i \leq 3} || \pa_x^i  \bar{\om} ||_{L^{\infty}} h_0^i  m^{i-3} (i!)^{-1} 
\teq C_m.
\eeq

For any power $l \geq 1$, using \eqref{eq:GQ_error23}, we have 
\[
\sum_{ j \in J_2} d(j)^{-l}
\leq 
 \sum_{- n < j -(i-2) \leq -m} | j - (i-2)|^{- l }
+ \sum_{n > j - i \geq m } |j-i|^{- l}   = 2    \sum_{  m \leq \nu < n } \nu^{- l},
\]
where we have changed the index $j - (i-2)$ to $-\nu$ and $j-i$ to $\nu$ in each summation to obtain the last equality. The estimate of the last summation is straightforward. For $l \geq 2$, we get 
\beq\label{eq:GQ_dj1}
\sum_{ j \in J_2} d(j)^{-l} \leq 2 \sum_{  m \leq \nu < n } \nu^{- l }
\leq 2 \int_{m-1}^{ n} y^{-l} dy = \f{2}{ l -1 } (m-1)^{- l +1} \teq dj_l.
\eeq

For $ l = 1$, we get
\beq\label{eq:GQ_dj2}
\sum_{ j \in J_2} d(j)^{-1} \leq 2 \sum_{  m \leq \nu < n } \nu^{- 1 }
\leq  2 \int_{m-1}^{ n} y^{-l} dy  =  2 \log( \f{n}{m-1}) \teq dj_1.
\eeq

Now, we are in a position to estimate the sum of $\pa_t^{2k} F_j$. 

For $k \geq 2$, plugging \eqref{eq:const_Cm} and \eqref{eq:GQ_dj1} with $l = 2k-2 \geq 2$ into \eqref{eq:GQ_error2}, we obtain 
\beq\label{eq:GQ_error31}
\sum_{j \in J_2} \max_{t\in [0,1]} |\pa_t^{2 k} F_j(t)|
 \leq C_m (2k)! \sum_{j \in J_2} d(j)^{-2k+2}  \leq  C_m (2k)!\cdot  dj_{2k-2}
 \teq  C_{F,k}  .
\eeq

For $k =1 $, plugging \eqref{eq:GQ_dj1}  and \eqref{eq:GQ_dj1} with $l = 2,3$ into \eqref{eq:GQ_error23}, we yield
\beq\label{eq:GQ_error32}
\bal
&\sum_{j \in J_2} \max_{t\in [0,1]} |\pa_t^{2 } F_j(t)|
\leq  \sum_{0\leq i \leq 2} 2! \cdot ( i!)^{-1} h_0^i || \pa^i_x \bar \om ||_{L^{\infty}} 
\sum_{j \in J_2}
d(j)^{- (2 \cdot 1-i+1)} \\
 \leq & 2 || \bar{\om}||_{L^{\infty}} dj_3+ 2 h_0  || \bar \om_x ||_{L^{\infty}} dj_2+ h_0^2  || \bar \om_{xx} ||_{L^{\infty}} dj_1 
\teq C_{F,1} . 
\eal
\eeq

Summing \eqref{eq:GQ_error12} and \eqref{eq:GQ_error13} over $j \in J_2$ and applying the above estimates to \\
$\sum_{j \in J_2} \max_{t \in [0,1]}| \pa_t^{2k} F_j(t)|$,
we prove 
\beq\label{eq:GQ_error4}
\bal
&\sum_{ j \in J_2} | I_{2,j } - NGQ(F_j, 0 , 1)|
\leq \sum_{ j \in J_2} ( | r_{2, j} - \bar r_{2, j}| + | s_{2, j} | + | \bar s_{2, j} | ) \\
\leq &  \sum_{ 1 \leq k \leq K-1} \f{ C_{F,k} \cdot \e_k}{ (2k) ! \cdot 2^{2k+1}}
+  \f{C_{F,K} }{ (2K)! 2^{2K}} ( \f{1}{2K+1} + \f{c_K}{2} ) \teq C_F.
\eal
\eeq


By rewriting $I_3$ from \eqref{eq:GQ_error0} in a form similar to \eqref{eq:GQ_error02} and applying the above argument to the integral of $G_j$ (see \eqref{eq:GQ_error03}), we can estimate the error in the computation of $I_3$ and obtain exactly the same bound as \eqref{eq:GQ_error4}. 
Recall that $I_2 = \f{1}{\pi} \sum_{j \in J_2} I_{2,j}$. (A similar relation can be used in the estimate of error of $I_3$.) Hence, we establish the following error estimate
\[
\B|  I_2 + I_3 - ( \f{1}{\pi} \sum_{ j \in J_2} NGQ_K( F_j, 0, 1) -  \f{1}{\pi} \sum_{j\in J_4} NGQ_K(G_j, 0, 1) )\B|
\leq 
\f{2}{\pi} C_F . 
\]
Compared to \eqref{eq:GQ_error4}, the above upper bound has an extra factor $\f{2}{\pi}$. $\pi^{-1}$ comes from $\pi^{-1}$ in the above summation and $2$ comes from the fact that we have the same error estimate for $I_2, I_3$. Since $C_F$ only depends on some explicit constants, e.g. $\e_k$ \eqref{eq:GQ_ek}, $dj_l, h_0$, and  $ || \pa_x^i \bar \om ||_{L^{\infty}}, i =0, 1,2,3$ which have been estimated in \eqref{eq:const_a1}, we can use Interval arithmetic to obtain rigorous bound on the computation error of $u_x$ 
\beq\label{error_GQ_ux}
Error_{GQ}(u_x) \leq \f{2}{\pi} C_F  < 2 \cdot 10^{-17}.
\eeq

We refer the reader to the MatLab script for the detailed computation.



\subsection{Error estimates in the computation of $\bar u, \bar u_{xx}$}

The error estimates in the computation of $\bar u, \bar u_{xx}$ are similar. In the case of $\bar u$, using a transform similar to that in \eqref{eq:GQ_error02}, we have
\beq\label{eq:GQ_error50}
 \bar u(x) = \f{1}{\pi} \int_0^L \log| \f{x-y}{x+y} | \bar \om(y) dy
 = \f{h_0}{\pi} \sum_{0\leq j\leq n-1} \int_0^{1} \bar \om(jh_0 + th_0) ( \log|n_x - j - t| - \log|n_x + j + t|) dt.
\eeq

We remark that we gain a small factor $h_0$ after the transform. Hence, we need to study 
\[
F^a_j(t) \teq \bar \om(j h_0 + t h_0) \log| n_x - j - t | ,  \quad 
G^a_j(t) \teq \bar \om(j h_0 + t h_0) \log|n_x + j + t|,
\]
which are similar to $F_j, G_j$ in \eqref{eq:GQ_error03}. In the case of $\bar u_{xx}$, we analyze 
\[
F^b_j(t) \teq    \f{ \bar \om_x(j h_0 + t h_0) }{ n_x - j - t}, \quad 
G^b_j(t) \teq   \f{\bar \om_x(j h_0 + t h_0) }{ n_x + j + t},
\]
since $\bar u_{xx}(x) = \f{1}{\pi} \int_0^L \bar \om(y) ( \f{1}{x-y} + \f{1}{x+y}) dy$. Then we apply the previous arguments in the error estimates. We have the following estimates similar to \eqref{eq:GQ_error31}
\beq\label{eq:GQ_error51}
\bal
\sum_{j \in J_2} \max_{t\in [0,1]} |\pa_t^{2k} F^a_j(t)|
&\leq C_m \f{ (2k)!}{2k-3} \sum_{j \in J_2} d(j)^{- 2k +3}
\leq C_m \f{ (2k)!}{2k-3} dj_{2k-3} \teq C_{a,k}  ,  \\
\sum_{j \in J_2} \max_{t\in [0,1]} |\pa_t^{2k} F^b_j(t)|
&\leq C_{m,2}  (2k)!  \sum_{j \in J_2} d(j)^{- 2k +1}
\leq C_{m,2}  (2k)!  dj_{2k-1} \teq C_{b,k}  , 
\eal
\eeq
for $k \geq 2$, where 
\[
C_{m,2} = \sum_{ 0\leq i \leq 2} h_0^i (i!)^{-1} m^{i-2}  || \pa_x^{i+1} \bar \om ||_{L^{\infty}} .
\]

For $k = 1$, we have estimates similar to \eqref{eq:GQ_error32}
\beq\label{eq:GQ_error52}
\bal
\sum_{j \in J_2} \max_{t\in [0,1]} |\pa_t^{2} F^a_j(t)|
& \leq h_0^2 || \bar \om_{xx} ||_{L^{\infty}} dj_0
+ 2 h_0 || \bar \om_x ||_{L^{\infty}} dj_1 
+ || \bar \om||_{L^{\infty}} dj_2 \teq C_{a,1} ,\\
\sum_{j \in J_2} \max_{t\in [0,1]} |\pa_t^{2} F^b_j(t)|
&\leq 2 ( || \bar \om_x ||_{L^{\infty}} dj_3 +  || \bar \om_{xx} ||_{L^{\infty}} h_0 dj_2 
+ \f{1}{2} || \bar \om_{xxx} ||_{L^{\infty}} h_0^2 dj_1) \teq C_{b,1},
\eal
\eeq
where we have used that $ dj_0 = 2 (  (n+1) \log( n+1)  - (n+1) ) $ is the upper bound of $\sum_{ j \in J_2} \log| n_x - j - t| $ for any $t \in [0, 1]$. It can be derived using an argument similar to that in Section \ref{sec:dj}.

Using an argument similar to that in \eqref{eq:GQ_error4} and its following discussion, we obtain the following error bound in the computation of $\bar u$
\[
 Error_{GQ}(u) \leq \f{2 h_0}{\pi}  \cdot ( \sum_{ 1 \leq k \leq K-1} \f{ C_{a,k} \cdot \e_k}{ (2k) ! \cdot 2^{2k+1}}
+  \f{C_{a,K} }{ (2K)! 2^{2K}} ( \f{1}{2K+1} + \f{c_K}{2} ) ) < 2 \cdot 10^{-19},
\] 
where $h_0$ comes from the extra factor in \eqref{eq:GQ_error50} due to the transform. The error bound in the computation of $\bar u_{xx}$ is 
\[
 Error_{GQ}(u_{xx}) \leq \f{2 }{\pi}  \cdot ( \sum_{ 1 \leq k \leq K-1} \f{ C_{b,k} \cdot \e_k}{ (2k) ! \cdot 2^{2k+1}}
+  \f{C_{b,K} }{ (2K)! 2^{2K}} ( \f{1}{2K+1} + \f{c_K}{2} ) ) < 5 \cdot 10^{-18}.
\]

The above two bounds are obtained by an argument similar to that in the derivation of \eqref{error_GQ_ux}.

\appendix
\setcounter{section}{1}
\section{Some useful Lemmas}
The following Lemmas are proved in the Appendix of \cite{chen2019finite} and we have used them in this supplementary material.
\begin{lem}[The Tricomi identity]\label{lem:tric}
 We have
\beq\label{eq:tric}
H( \om H\om)  =  \f{1}{2} (  (H \om)^2 - \om^2  ).
\eeq
\end{lem}

\begin{lem}\label{lem:commute} Suppose that $u_x = H\om$. Then we have
\beq\label{lem:vel0}
\f{u_x - u_x(0)}{x} = H\lt( \f{\om}{x}\rt)  , \textrm{ or equivalently } \ (H\om)(x) = (H\om)(0) + x H\lt( \f{\om}{x}\rt) .
\eeq
Similarly, we have
\beq\label{eq:iden22}
\f{  u_x - u_x(L)}{  x-L } = H\lt( \f{ \om}{  x-L} \rt) , \ u_{xx} = H\om_x, \ x u_{xx} = H(x \om_x).
\eeq
Suppose that in addition $\om$ is odd. Then we have
\beq\label{eq:commute3}
  x^2 u_{xx}  =  H (x^2 \om_x),  \quad x u_x = H(x\om), \quad \f{u_{xx}}{x} = H \lt(  \f{\om_x - \om_x(0)}{x} \rt) .
\eeq
If $\om$ is odd and a piecewise cubic polynomial supported on $[-L,L]$ with $\om(L) = \om(-L) = 0$
($\om^{\prime}, \om^{\prime \prime}$ may not be continuous at $x = \pm L$), then we have
\beq\label{eq:commute4}
 u_{xxx} (x^2 - L^2) = H(  \om_{xx}(x^2 -L^2) )  .
\eeq
\end{lem}

\begin{lem}\label{lem:vel_a1} Suppose $u_x = H\om$.
(a)  We have
\begin{align}
\int_{\R} \f{  ( u_x - u_x(0) ) \om    }{  x  } & = \f{\pi}{2} ( u_x^2(0) + \om^2(0)) \geq 0 .\label{lem:vel3}
\end{align}
In particular,  \eqref{lem:vel3} vanishes if $ u_x(0) = \om(0) =0$.


(b) The Hardy inequality: Suppose that $\om$ is odd and $\om_x(0)=0$. For $p = 2, 4$, we have
\beq\label{eq:hd1}
\int \f{  (u  - u_x(0) x)^2} { |x|^{p+2} }  \leq  \lt(\f{2}{p +1 } \rt)^2   \int \f{  (u_x - u_x(0))^2}{ |x|^{ p } } = \lt(\f{2}{p +1 } \rt)^2  \int   \f{  \om^2}{ |x|^{ p } }.
\eeq
\end{lem}

\section{Numerical values of some functions on the grid points}
We plot the numerical values of some functions on the grid points in this Section. We remark that all the estimates can be verified rigorously using the strategy we outlined in Section 4.3 ``entitled Rigorous verification of the numerical values'' in \cite{chen2019finite} (see page 23). The following figures are used to visualize several estimates.

\subsection{Figure related to Section \ref{sec:cross}}
 \begin{figure}[H]
   \centering
   \includegraphics[width =\textwidth ]{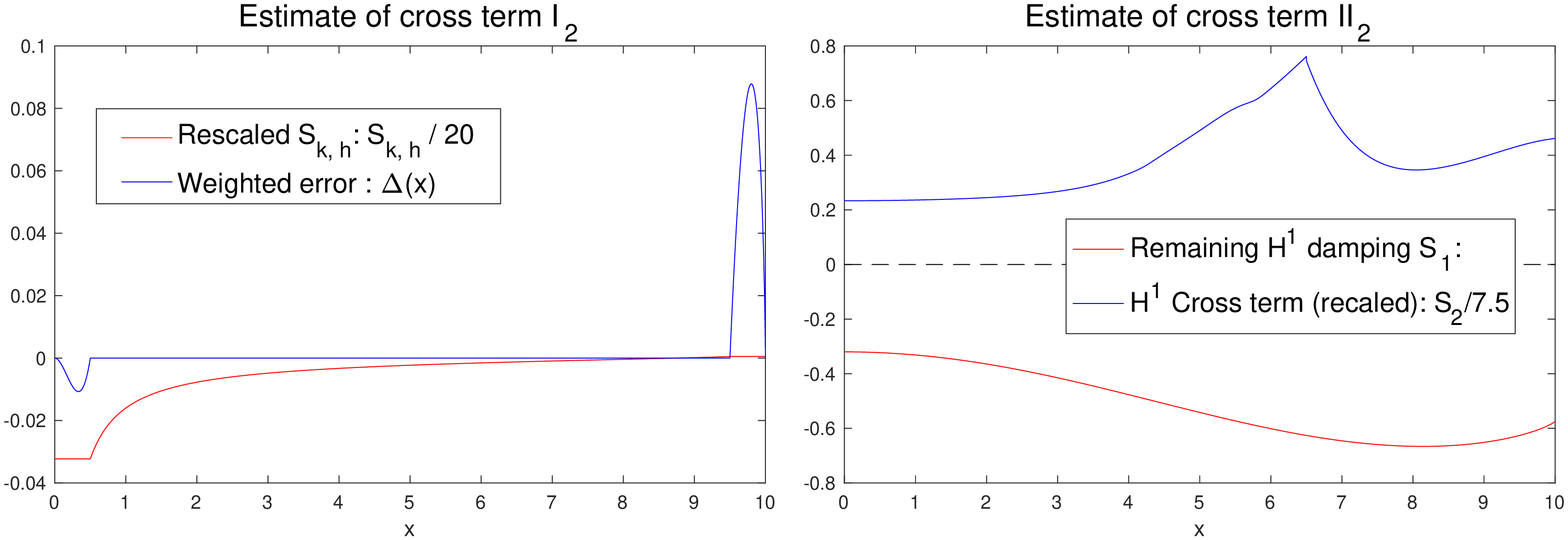}
   \caption{ Left: Numerical values of $S_{k, h}$ and $\D(x)$ on the grid points.
   Right:  Numerical values of $S_1$ and $S_2$ on the grid points.
   We plot the rescaled $S_{k,h}$, i.e. $S_{k, h} / 20$, and the rescaled $S_2$, i.e. $S_2 / 7.5$.
   }
   \label{DG_cross}
 \end{figure}
On the left subfigure of Figure \ref{DG_cross},  the numerical values of $S_{k,h}$ are monotonically increasing and $ |\D(x)| \leq 0.1 $ on the grid points. On the right subfigure, the numerical values of $S_1, S_2$ on the grid points are less than $ -0.3, 7.5$, respectively.

\subsection{Figure related to Section \ref{sec:non}}
 \begin{figure}[H]
   \centering
   \includegraphics[width =\textwidth ]{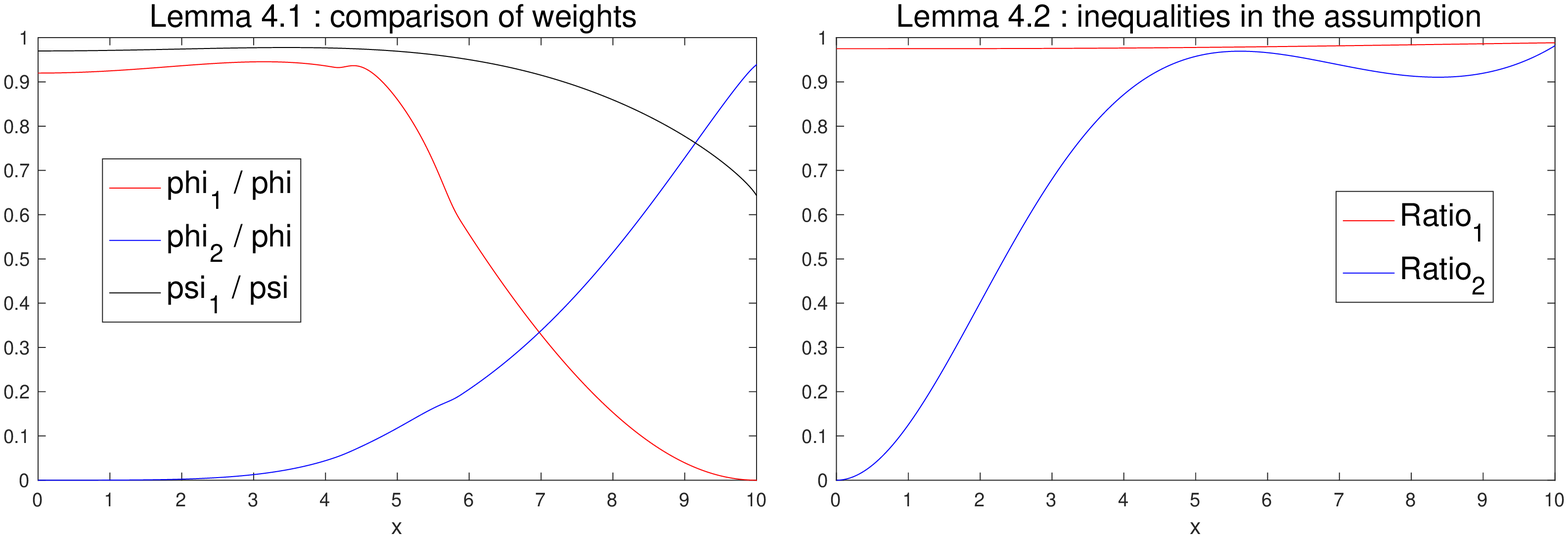}
   \caption{ Comparison of several functions. Left: Numerical values of $\vp_1 / \vp, \vp_2 / \vp, \psi_1 /\psi$ on the grid points. Right: Numerical values of two ratios in \eqref{eq:non_al2} on the grid points.
   Here, $\mathrm{Ratio}_1 =  \xi_1^2 \lt(  \mu (\al_1 \vp_1 - \xi_{1, x} ) \al_1 \psi_1 \rt)^{-1}$ and $\mathrm{Ratio}_2 =  \xi_2^2 \lt(  \mu (\al_2 \vp_2 + \xi_{2, x} ) \al_2 \psi_2 \rt)^{-1}$.
   }
   \label{DG_nonlinear}
 \end{figure}
Numerical values of $\vp_1 / \vp, \vp_2 / \vp, \psi_1 /\psi$ (see Lemma \ref{lem:non_ux1}) and 
$\mathrm{Ratio}_1 =  \xi_1^2 \lt(  \mu (\al_1 \vp_1 - \xi_{1, x} ) \al_1 \psi_1 \rt)^{-1}$, $\mathrm{Ratio}_2 =  \xi_2^2 \lt(  \mu (\al_2 \vp_2 + \xi_{2, x} ) \al_2 \psi_2 \rt)^{-1}$ (see Lemma \ref{lem:non_ux2}) on the grid points are strictly less than $1$.

\subsection{Figure related to Section \ref{sec:boot}}

 \begin{figure}[H]
   \centering
   \includegraphics[width = 0.7\textwidth ]{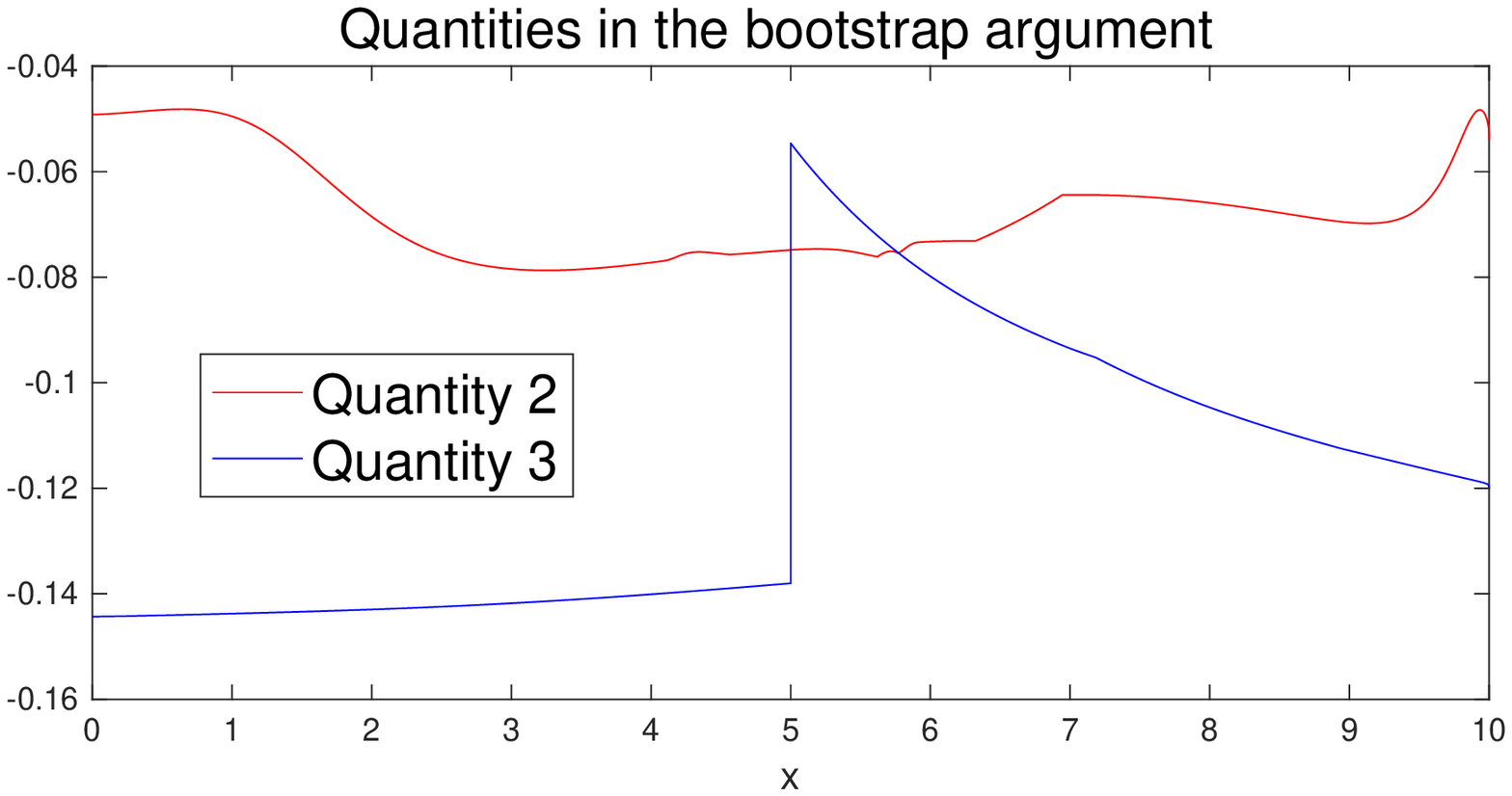} 
   \caption{ Numerical values of the left hand side of the second and the third inequality in \eqref{eq:boot_ineq} on the grid points. 
   }
   \label{DG_boot}
 \end{figure}
Quantity $2, 3$ in Figure \ref{DG_boot} represents the left hand side of the second and the third inequality
in \eqref{eq:boot_ineq}, i.e. 
\[
\bal
\mathrm{Quantity  \ 2}  &\teq -0.1 + 0.05 \vp\rho_1^{-1}  + Z_1(x) E^*   , \\
  \mathrm{Quantity \  3}  & \teq -0.2 + 0.05 \psi \rho_2^{-1}  + (Z_2(x) + Z_3(x))E^* .
\eal
\]
Numerical values of Quantity $2,3 $ on the grid points are less than $-0.04<0$.

\bibliographystyle{plain}
\bibliography{selfsimilar}


